\documentclass[11pt,letterpaper]{article}
\usepackage[letterpaper,margin=1in]{geometry}

\usepackage[T1]{fontenc}
\usepackage[utf8]{inputenc}
\usepackage{amsmath,amsthm}   
\usepackage{newpxtext}
\usepackage{newpxmath}        
\usepackage{bm}               
\usepackage{dsfont}
\usepackage{inconsolata}      
\usepackage{microtype}

\usepackage{subcaption}
\usepackage{algorithm}
\usepackage{algorithmic}
\usepackage{pgfplots}\pgfplotsset{compat=1.18}

\usepackage{natbib}

\usepackage[dvipsnames]{xcolor}
\usepackage{graphicx,multirow,booktabs,array,nicefrac}
\usepackage{comment,cancel}
\usepackage[shortlabels]{enumitem}
\usepackage{mdframed}
\usepackage{framed}
\usepackage{authblk}
\usepackage[colorlinks=true,linkcolor=blue,citecolor=blue,urlcolor=blue]{hyperref}

\setlist{itemsep=2pt,topsep=4pt,parsep=0pt}
\allowdisplaybreaks

\providecommand{\citet}[1]{\cite{#1}}
\providecommand{\citep}[1]{\cite{#1}}

\newtheorem{theorem}{Theorem}[section]
\newtheorem{proposition}[theorem]{Proposition}
\newtheorem{assumption}[theorem]{Assumption}
\newtheorem{lemma}[theorem]{Lemma}
\newtheorem{corollary}[theorem]{Corollary}
\newtheorem{remark}[theorem]{Remark}
\theoremstyle{definition}

\newtheorem{definition}[theorem]{Definition}

\definecolor{ReviewRed}{RGB}{185,20,25}
\definecolor{ReviewGreen}{RGB}{0,112,55}
\definecolor{ReviewBlue}{RGB}{0,0,0}

\title{{{Improved Gradient Complexity for Quasi-Newton Proximal Extragradient Methods via Certified Curvature
Reuse}}
\vspace{4mm}}

\renewcommand*{\Affilfont}{\normalsize}

\makeatletter
\renewcommand\AB@affilsepx{, \protect\Affilfont}
\renewcommand\AB@affilsep{\protect\Affilfont}
\makeatother

\author[$\dagger$,$\ddagger$]{Aryan Mokhtari}
\author[$\ddagger$]{Francisco Patitucci}

\affil[$\dagger$]{Google Research}

\affil[$\ddagger$]{UT Austin}
 
\date{}
\begin{document}
\maketitle

\begin{abstract}
We study how far curvature estimated from gradient differences can accelerate
proximal-extragradient-based quasi-Newton methods, whose strongest global
guarantees---QNPE \citep{jiang2023online}, A-QNPE \citep{jiang2023accelerated}, and
OL-OQN \citep{jiang2025improved}---learn a Hessian approximation online and pay for
the dimension through a regret term. We introduce \emph{certified curvature reuse}: one finite-difference probe per iteration refreshes a group of Hessian columns, the error of the stale columns is certified before each step from the iterates' path length, and the certificate sets the step. Each iteration uses at most three gradients and one regularized linear solve (or trust-region step), with no gradient line search and no inner loop that queries the oracle: in the convex and strongly convex settings, monotone scalar surrogates with controlled polynomial growth select a near-maximal admissible stepsize by a bisection on stored numbers, while in the nonconvex setting a uniform certificate over bounded increments fixes a constant learning rate in closed form. With problem constants fixed, our
results are as follows. In the convex setting, a single run attains the rate
$O(\min\{k^{-2},dk^{-3}\})$ with no logarithmic factor or additive dimension term,
matching accelerated gradient descent (AGD) for $k\lesssim d$ and 
improving on
A-QNPE's $O(\min\{k^{-2},\sqrt{d\log k}\,k^{-5/2}\})$ for
$k\gtrsim d$. In
the strongly convex setting with condition number $\kappa$, a single run matches
AGD's $(1-\Theta(\kappa^{-1/2}))^{k}$ rate for $k\lesssim d\log\kappa$ and afterwards has a superlinear upper envelope, eventually scaling as
$(1+k^{2}/B^{2})^{-k}$ with $B=\tilde O(d\kappa)$ and improving on QNPE in both
phases. In the nonconvex setting, our method finds an $\varepsilon$-stationary
point in $O(\min\{d^{1/3}\varepsilon^{-3/2},\varepsilon^{-7/4}\})$ gradient
queries, improving on OL-OQN's $O(d^{1/4}\varepsilon^{-13/8})$. For sparse Hessians, $d$ is replaced by an effective dimension $\chi\le d$. 
We also compare
our bounds with lazy-Hessian, double-loop, and cutting-plane methods across
accuracy and dimension regimes.
\end{abstract}

\section{Introduction}\label{sec:intro}

We consider the unconstrained minimization problem
\[
\min_{x\in\mathbb R^d} f(x),
\]
where $f$ is twice continuously differentiable with $L_1$-Lipschitz
gradient and $L_2$-Lipschitz Hessian, and the only access to $f$ is a
first-order oracle $x\mapsto(f(x),\nabla f(x))$. Neither Hessians nor
Hessian-vector products are available as oracle operations. We study three settings.
When $f$ is convex and attains its minimum $f^\star$ at some $x^\star$, with
$R\ge\|x_0-x^\star\|$ a bound on the distance from the initial point, the
goal is an $\varepsilon$-accurate point, that is, a point $\hat x$ with
$f(\hat x)-f^\star\le\varepsilon$. When $f$ is in addition
$\mu$-strongly convex, we write $\kappa=L_1/\mu$ for its condition number.
When $f$ is nonconvex and bounded below, with $\Phi_0\ge f(x_0)-\inf_x f(x)$ a bound on the initial objective gap, the goal is an
$\varepsilon$-stationary point, $\|\nabla f(\hat x)\|\le\varepsilon$. In all
three settings the resource is the number of gradient queries.

Accelerated gradient descent (AGD) needs $O(\sqrt{L_1R^2/\varepsilon})$ gradient queries in the convex case and $O(\sqrt{\kappa}\log(L_1R^2/\varepsilon))$ in the strongly convex case \citep{nesterov2018lectures}, and accelerated first-order methods find an $\varepsilon$-stationary point of a nonconvex function with $\tilde O(\varepsilon^{-7/4})$ queries \citep{carmon2017convex}, with the logarithmic factor removed in $O(\varepsilon^{-7/4})$ by \citet{li2022restarted}. The convex bounds are optimal among first-order methods, while in the nonconvex case there is a gap compared to the deterministic first-order lower bound $\Omega(\varepsilon^{-12/7})$ of \citet{carmon2021lower}. That said, these lower bounds hold only in high dimension: the constructions behind them have dimension exceeding the number of queries \citep{nemirovski1983problem,nesterov2018lectures} or growing polynomially in $1/\varepsilon$ \citep{carmon2021lower}. Once the query budget exceeds the dimension, curvature can be estimated from gradient differences and put to use, and for methods that use a constant number of gradients and one linear solve (or trust-region step) per iteration without inner gradient-querying loops, the
question becomes how the achievable complexity depends on both $d$ and
$\varepsilon$.

In the smooth strongly convex setting, quasi-Newton methods have traditionally been the answer to this question. They maintain a matrix that satisfies secant equations along the trajectory, inferring curvature from gradient differences at no additional oracle cost, and the BFGS and DFP updates are among the most used algorithms in smooth optimization \citep{nocedal2006numerical}. Their classical theory is asymptotic, giving local superlinear convergence under the Dennis--Mor\'e condition \citep{dennis1974characterization}. Explicit local superlinear rates for greedy, randomized and classical updates are recent \citep{rodomanov2021greedy,rodomanov2022rates,jin2023nonasymptotic,lin2022explicit}; they require a starting point near the solution and a suitable initial Hessian approximation, and, being local, give no global complexity bound. Global non-asymptotic rates for BFGS with exact and inexact line searches are more recent still \citep{jin2024global,jin2024armijo}. For merely convex functions, none of these classical methods is known to be faster than AGD. Beyond classical secant methods, regularized quasi-Newton variants that approximate cubic regularized Newton steps achieve global accelerated rates for convex functions, but their worst-case $O(k^{-2})$ guarantees still do not improve upon AGD \citep{scieur2024adaptive,kamzolov2026accelerated}.

A recent line of work obtains stronger global guarantees by updating the Hessian approximation through online learning instead of secant equations. In particular, for smooth strongly convex functions, the quasi-Newton proximal extragradient (QNPE) method of \citet{jiang2023online} attains the minimum of a linear and a superlinear rate, $O(\min\{(1-\Theta(\kappa^{-1}))^{k},(1+\sqrt{k/X})^{-k}\})$ with $X=O(d\kappa^2+\kappa L_2^2R^2/\mu^2)$, using at most three gradients per iteration on average; the superlinear rate takes over after only $O(d+L_2^2R^2/(\mu L_1))$ iterations. Its linear rate, however, has $\kappa$ in place of $\sqrt{\kappa}$, so in the first phase QNPE matches gradient descent rather than AGD. In follow-up work for the convex setting, \citet{jiang2023accelerated} proposed A-QNPE, a quasi-Newton variant of the accelerated hybrid proximal extragradient framework of \citet{monteiro2013accelerated}, and proved the rate $O(\min\{k^{-2},\sqrt{d\log(k+1)}\,k^{-5/2}\})$, that is, a gradient complexity of $\tilde O(\min\{\varepsilon^{-1/2},d^{1/5}\varepsilon^{-2/5}\})$. Ignoring logarithmic factors, this matches AGD in the first $d$ iterations and is faster afterward; in terms of accuracy, it matches AGD when $\varepsilon\ge d^{-2}$ and improves on it when $\varepsilon<d^{-2}$ (with the logarithmic factor retained, the crossover is at $k\asymp d\log d$ and $\varepsilon\asymp(d\log d)^{-2}$). It was the first quasi-Newton method provably faster than AGD on convex problems. Finally, for smooth nonconvex $f$, \citet{jiang2025improved} proposed an online-learning-guided optimistic quasi-Newton method (which we call OL-OQN). It combines the online-to-nonconvex conversion of \citet{cutkosky2023optimal}, which turns the regret of an online learner fed with gradients into a stationarity guarantee, with a second online learner that fits the Hessian approximation used as the hint of the optimistic update. OL-OQN needs $O(d^{1/4}\varepsilon^{-13/8})$ gradient queries, below the first-order bound $\varepsilon^{-7/4}$ whenever $d\lesssim\varepsilon^{-1/2}$. 
More recently, \citet{marumo2025parameter} attained the same $O(d^{1/4}\varepsilon^{-13/8})$ complexity with a parameter-free method that combines momentum, quartic regularization, and a scaled Powell--symmetric--Broyden update, requiring no prior knowledge of the Lipschitz constants, the target accuracy, or the iteration horizon.

QNPE, A-QNPE, and OL-OQN provide strong global guarantees with a gradient oracle in the strongly convex, convex, and nonconvex settings, respectively, and serve as our benchmarks. All three learn the Hessian approximation online, so the dimension enters their bounds through a regret term, and nothing indicates that the resulting exponents are the right ones. The goal of this paper is to improve on these bounds in all three settings through a common principle: certifying the curvature model instead of learning it, by probing the Hessian at chosen points and controlling the error of reused probes. Our bounds also exploit known sparsity of the Hessian, with $d$ replaced throughout by an effective dimension $\chi\le d$ (Definition~\ref{def:chi}).

\vspace{2mm}

\noindent\textbf{Contributions.} Our methods for the convex and strongly convex cases build on the accelerated hybrid proximal extragradient (A-HPE) framework, which reduces acceleration to computing approximate proximal steps with small relative error. We compute each step from a positive semidefinite curvature model by one regularized linear solve, yielding a method with no inner gradient-querying loop. The residual of such a step is controlled by the operator-norm error of the model and the length of the step, so a bound on the model error known before the step, combined with an a priori bound on the step length, yields a stepsize that is guaranteed to be admissible: no line search is needed and no gradient is spent on rejected trials. To keep such a bound while reusing the model, we refresh one group of columns per iteration and certify the reused columns by the distance the iterates have traveled since they were computed. The method for the nonconvex case uses the online-to-nonconvex conversion and a trust-region subproblem in place of A-HPE, with the same curvature model and the same certificate. Our contributions are as follows.

\begin{itemize}
\item \textbf{Convex case.} Our method for the convex setting, rolling-probe QNPE (RP-QNPE), uses at most three gradients per iteration. It converges at the rate $O(\min\{k^{-2},\chi k^{-3}\})$ and finds an $\varepsilon$-accurate point with $O(\min\{\varepsilon^{-1/2},\chi^{1/3}\varepsilon^{-1/3}\})$ gradient queries (Theorem~\ref{thm:iterationrate} and Corollary~\ref{thm:complexity}). Even in the dense case $\chi=d$, the least favorable for our bound since $\chi\le d$, this improves on A-QNPE throughout the regime $\varepsilon<d^{-2}$ in which A-QNPE improves on AGD, and matches AGD otherwise (Section~\ref{sec:aqcompare}); for $\chi<d$ the gain is larger (Section~\ref{sec:structurecompare}). The analysis charges the error of stale columns to the path length of the accelerated iterates with no dimension factor, so the model costs $O(\chi)$ in inverse stepsizes each time the accelerated mass doubles; summing over these doublings gives the rate without logarithmic factors.

\item \textbf{Strongly convex case.} Keeping the rolling model throughout the run and using $\mu$ in the accelerated recursion gives a method (SC-RP-QNPE, Algorithm~2) whose objective-error upper bound improves on the accelerated linear guarantee $(1-\Theta(\kappa^{-1/2}))^k$ of AGD after $O(B_\sharp/\kappa)$ iterations and has a superlinear upper envelope afterward, eventually scaling as $(1+k^2/B_\sharp^2)^{-k}$, where $B_\sharp=O(\chi\kappa+\chi\kappa_2\log(2+\kappa))$ and $\kappa_2=L_2R/\mu$ (Theorem~\ref{thm:scmain}). In normalized units the switch occurs after $O(\chi\log\kappa)$ iterations, whereas QNPE has an upper bound running at the slower rate of gradient descent for $O(d+\kappa)$ iterations and then superlinearly with envelope $(1+\sqrt{k/X})^{-k}$, where $X\asymp d\kappa^2+\kappa^3\gg B_\sharp$. At moderate accuracy our method needs $O((B_\sharp\log^2(1/\varepsilon))^{1/3})$ gradients (Corollary~\ref{cor:scqueries}), below the AGD bound once $\kappa\gtrsim(\chi\log\kappa)^2$.

\item \textbf{Nonconvex case.} The model and its certificate do not use convexity. Placing the rolling model in the hint of an optimistic online learner inside the online-to-nonconvex conversion of \citet{cutkosky2023optimal}, using the extrapolated-gradient learner of \citet{patitucci2025double,jiang2025improved}, gives rolling-probe optimistic quasi-Newton (RP-OQN, Algorithm~3), which finds an $\varepsilon$-stationary point with $O(\min\{\chi^{1/3}\varepsilon^{-3/2},\varepsilon^{-7/4}\})$ gradient queries, with explicit constants and no logarithmic factors (Corollary~\ref{cor:ncenvelope}). Here the certificate is needed at every step rather than in sum; it is of order $\chi L_2D$ for increments of length $D$, and the negative term of the optimistic learner absorbs it, leaving only an upper limit on the learning rate. In the dense case, RP-OQN improves on OL-OQN throughout the regime in which OL-OQN improves on first-order methods, and it extends that regime from $d\lesssim\varepsilon^{-1/2}$ to $d\lesssim\varepsilon^{-3/4}$ without a matrix learner, while matching the polynomial exponents of lazy-Hessian methods using gradients alone (Section~\ref{sec:nccompare}).

\end{itemize}

\subsection{Second-order, lazy-Hessian and cutting-plane methods}\label{sec:related}

\noindent\textbf{Second-order and Hessian--vector product methods.} Methods with
access to an exact second-order oracle achieve substantially lower iteration
complexities: $O((L_2R^3/\varepsilon)^{2/7})$ in the convex case
\citep{monteiro2013accelerated,carmon2022optimal} and $O(\varepsilon^{-3/2})$ in the
nonconvex case \citep{nesterov2006cubic}, both matching oracle lower bounds
\citep{arjevani2019oracle,carmon2020lower} and benchmarking the best
$\varepsilon$-dependence achievable by curvature-based methods.  Matrix-free implementations use Krylov methods with Hessian--vector products \citep{agarwal2017finding,carmon2018accelerated,carmon2022optimal}; approximating these products by gradient differences spends a fresh gradient on every inner matrix--vector product and requires control of the approximation error inside the subproblem solver. In the gradient-oracle model the nonconvex rate $\tilde O(\varepsilon^{-7/4})$ is attained by \citet{carmon2017convex}, and rebuilding a finite-difference Hessian at each of the $O(\varepsilon^{-2/7})$ accelerated proximal steps would cost $d$ gradients per step, giving $\tilde O(d\varepsilon^{-2/7})$ if the analysis of \citet{carmon2022optimal} tolerates the finite-difference error (we are not aware of such an analysis). By instead updating a persistent model $B_k$ with one probe per step and solving the subproblem on $B_k$ without further oracle calls, our methods attain $O(\min\{\chi^{1/3}\varepsilon^{-3/2},\varepsilon^{-7/4}\})$ in the nonconvex case and $O(\min\{\varepsilon^{-1/2},\chi^{1/3}\varepsilon^{-1/3}\})$ in the convex case, the latter below $d\varepsilon^{-2/7}$ for all $\varepsilon>d^{-14}$.

\vspace{2mm}

\noindent\textbf{Lazy-Hessian algorithms.} When only gradients are available, a
full Hessian can be approximated from $d$ gradient differences (or $\bar d\le d$
when structure reduces the cost of a Hessian evaluation \citep{chen2025faster});
finite-difference implementations of cubic regularization go back to
\citet{cartis2012oracle}, who obtain $O(d\varepsilon^{-3/2})$ gradient evaluations
in the nonconvex case, and \citet{doikov2023lazy} proposed measuring the complexity
of second-order methods in such equivalent gradient queries, counting one Hessian as
$\bar d$ gradients. Using lazy Hessian updates that recompute a full Hessian once
every $m\asymp\bar d$ steps, cubic regularized Newton achieves
$O(\bar d+\bar d^{1/2}\varepsilon^{-3/2})$ equivalent gradient queries in the
nonconvex case \citep{doikov2023lazy}, with first- and zeroth-order
finite-difference implementations analyzed in \citet{doikov2025first}, and
$O(\bar d+\bar d^{1/2}\varepsilon^{-1/2})$ in the convex case \citep{chayti2024lazy},
as restated by \citet{chen2025faster}. \citet{chen2025faster} subsequently
accelerated this approach: their nonconvex lazy extra-Newton method (NALEN) attains
$O(\bar d+\bar d^{1/3}\varepsilon^{-3/2})$ equivalent gradient queries, and their
convex method (CALEN), which runs the outer Monteiro--Svaiter acceleration of
\citet{carmon2022optimal} and solves each cubic-regularized proximal subproblem via
a restarted inner loop of NALEN, attains
$O((\bar d+\bar d^{13/21}(L_2R^3/\varepsilon)^{2/7})\log\bar d)$ (see also
\citet{chen2026optimal} for inexact second-order oracles). In the nonconvex gradient
oracle model, \citet{adil2025balancing} balanced gradient and periodic
finite-difference Hessian queries to obtain
$\tilde O(d+d^{1/3}\Phi_0\sqrt{L_2}\,\varepsilon^{-3/2})$, matching the polynomial
exponents of NALEN up to polylogarithmic factors.

Our approach differs from lazy-Hessian methods in three respects. First, in the convex and strongly convex settings our methods never pause to build a full Hessian: they refresh one group of columns per iteration in a rolling window and certify the stale columns via the iterates' path length, and during the initial $\chi$ iterations ($k<\chi$), while the first window fills, the spectral cap $\bar E_k=L_1$ keeps every step admissible at the smoothness stepsize, so the accelerated first-order rate holds from the first iteration and the bounds carry no additive $\chi$; lazy-Hessian methods pay $\bar d$ queries before their first step and again at every rebuild. Second, in the convex and strongly convex settings a \emph{single trajectory} attains the minimum of the curvature-accelerated rate and the optimal first-order bound ($O(\varepsilon^{-1/2})$ and $O(\sqrt{\kappa}\log(1/\varepsilon))$) without knowing the target accuracy or switching algorithms in advance; in the nonconvex setting, both our method and lazy-Hessian methods can be paired with a first-order method by an advance parameter switch, and we claim no advantage there. Third, in the convex and strongly convex settings RP-QNPE and SC-RP-QNPE have no gradient line search and no inner loop that queries the oracle: each iteration takes one proximal step with one regularized linear solve, and the stepsize search of Subroutines~1.1 and~2.1 is a bisection on stored scalars (iterative linear solvers, if used, are discussed in Section~\ref{sec:implementation}), whereas accelerated lazy-Hessian methods such as CALEN invoke an inner nonconvex solver at every outer proximal step. For completeness, Sections~\ref{sec:compare}, \ref{sec:sccompare}, and~\ref{sec:nccompare} compare our bounds with these methods under the convention that one Hessian counts as $d$ (or $\bar d$) gradients.

\vspace{2mm}

\noindent\textbf{Cutting-plane methods.} In terms of first-order oracle calls alone,
cutting-plane methods \citep{nemirovski1983problem,lee2015faster,jiang2020improved}
minimize a convex function in $\tilde O(d)$ gradient queries with logarithmic
dependence on $1/\varepsilon$. They represent a fundamentally different algorithmic
paradigm: they maintain a localization set from gradient cuts and solve a
nontrivial auxiliary problem at every step, they do not exploit smoothness or
second-order curvature, and their nonconvex extension requires Lipschitz third
derivatives \citep{hinder2018cutting}. In the dense convex case, none of the
curvature-based gradient bounds of Table~\ref{tab:convexcompare}, ours included,
improves on the better of AGD and cutting planes by more than logarithmic factors:
their curvature branches improve on AGD only for $\varepsilon<d^{-2}$, where they
are at least $d$ (Section~\ref{sec:structurecompare}). Our goal in the dense convex
setting is not to compete with localization methods, but to determine how far
gradient-based curvature can accelerate single-loop proximal extragradient methods
using $O(1)$ gradients and one linear solve per iteration. Crucially, however,
cutting-plane bounds count separation queries in $\mathbb R^d$ and are not known to
improve when the Hessian is sparse, whereas our bounds depend on the dimension
solely through the effective dimension $\chi\le d$. Consequently, for sparse
Hessians with $\chi<d$, our methods fall strictly below the better of AGD and
cutting-plane methods across a broad range of accuracies in the convex case
(Section~\ref{sec:structurecompare}) and condition numbers in the strongly convex
case (Section~\ref{sec:sccompare}). Figures~\ref{fig:convexmaps} and~\ref{fig:scmaps} map these regimes, including the structured instantiation of the lazy-Hessian baselines.

\vspace{2mm}

\noindent\textbf{Concurrent work.} While this paper was being finalized, \citet{zhang2026near} posted a preprint on convex optimization with lazy \emph{exact} second-order oracles, obtaining $\tilde O(\bar d+\bar d^{1/7}(L_2R^3/\varepsilon)^{2/7})$ equivalent gradient queries from an accelerated proximal envelope with an inner lazy-Hessian solver, together with a lower bound for that oracle model. Ignoring logarithmic factors, under the equivalent-cost convention with $\bar d=\chi$, their bound is at most the curvature branch $\chi^{1/3}\varepsilon^{-1/3}$ of our convex rate on the range $\varepsilon<\chi^{-2}$ where that branch is active, and polynomially smaller as $\varepsilon\to0$; ours is proved in the gradient-oracle model, follows a single trajectory with one linear solve per step, and carries no additive dimension term or logarithmic factor. They do not study the nonconvex setting or superlinear convergence in the strongly convex setting.

\section{Preliminaries}\label{sec:prelim}
We first state our notation and assumptions, then define the effective
dimension, and finally recall the accelerated proximal framework and the
properties of it that our methods use.

\subsection{Notation and assumptions}
Throughout, $\|\cdot\|$ is the Euclidean norm, and $\|\cdot\|_{\rm op}$
and $\|\cdot\|_F$ are the operator and Frobenius norms. The Hessian is
$H(x)=\nabla^2f(x)$, the vectors $e_1,\ldots,e_d$ are the coordinate
vectors of $\mathbb R^d$, and $[d]=\{1,\ldots,d\}$.

\begin{assumption}\label{as:main}
The function $f:\mathbb R^d\to\mathbb R$ is convex and twice continuously
differentiable, with known positive upper bounds $L_1,L_2$ such that
\[
0\preceq H(x)\preceq L_1I,\qquad
\|H(x)-H(y)\|_{\rm op}\le L_2\|x-y\|.
\]
A minimizer $x^\star$ exists, and a known $R>0$ satisfies
$\|x_0-x^\star\|\le R$. The oracle returns exact values and gradients at
arbitrary query points. A pattern $\mathcal S$ with
$\operatorname{supp}H(x)\subseteq\mathcal S$ for every $x\in\mathbb R^d$,
and a partition $g_1,\ldots,g_\chi$ of $[d]$ that is structurally
orthogonal for $\mathcal S$ (Definition~\ref{def:chi}), are supplied. When no structure is known,
$\mathcal S=[d]\times[d]$ and the partition into singletons are used, so
that $\chi=d$.
\end{assumption}

\noindent We write $f^\star=f(x^\star)$ and
$\mathcal Z=\{B\in\mathbb S^d:0\preceq B\preceq L_1I\}$, the set of
Hessians allowed by Assumption~\ref{as:main}; the map $\Pi_{\mathcal Z}$
clips the eigenvalues of a symmetric matrix to $[0,L_1]$. The curvature models of Algorithms~1 and~2 are projected onto $\mathcal Z$ so that they satisfy the same spectral bounds. These algorithms use exact spectral projections and linear
solves; Section~\ref{sec:implementation} relaxes both while
keeping gradient queries exact. All $O_\sigma(\cdot)$ constants may depend on a
fixed $\sigma\in(0,1)$; normalized comparisons fix $\sigma$ away from one
and fix $L_1,L_2,R$. The notation $\tilde O$ suppresses logarithmic
factors, not polynomial dependence on those parameters.

\paragraph{Hessian structure and the effective dimension.}
Our guarantees depend on the dimension only through an effective dimension
$\chi\le d$, which equals $d$ for a dense Hessian and can be much smaller
when the sparsity pattern of the Hessian is known. The following
definition makes it precise.

\begin{definition}[Effective dimension]\label{def:chi}
A \emph{pattern} is a symmetric set $\mathcal S\subseteq[d]\times[d]$, and
$\mathcal S_j=\{i:(i,j)\in\mathcal S\}$ denotes the support of column $j$.
A partition $g_1,\ldots,g_\chi$ of $[d]$ is \emph{structurally orthogonal}
for $\mathcal S$ if no two columns in the same group share a row, that is,
$\mathcal S_i\cap\mathcal S_j=\emptyset$ for all distinct $i,j$ in the same
group \citep{curtis1974estimation,coleman1983estimation}. The number $\chi$
of groups is the \emph{effective dimension}. For a group $g$ we write $|g|$
for its size and $v_g=\sum_{j\in g}e_j$ for its indicator vector, and $P_j$
denotes the coordinate projection onto $\mathcal S_j$, that is,
$(P_jw)_i=w_i$ for $i\in\mathcal S_j$ and $(P_jw)_i=0$ otherwise.
\end{definition}

In the dense case the partition into singletons gives $\chi=d$,
$v_{\{j\}}=e_j$ and $P_j=I$, and every statement below reduces to its
dense form; the reader interested only in that case may read $\chi$ as $d$
throughout. When the pattern is known, a partition with the fewest groups
is a minimum coloring of the graph in which two columns are adjacent when
they share a row of $\mathcal S$ \citep{coleman1983estimation}; any proper
coloring may be used, and the bounds are in terms of the supplied $\chi$.
For a Hessian of half-bandwidth $w$, that is, $H_{ij}(x)=0$ whenever
$|i-j|>w$, the residue classes modulo $2w+1$ form a structurally orthogonal
partition, so $\chi\le\min\{d,2w+1\}$; a tridiagonal pattern admits $\chi\le3$ (with three groups necessary and sufficient for the full
tridiagonal pattern when $d\ge3$). The quantity $\chi$ plays, for finite
differences, the role of the Hessian cost parameter $\bar d$ of
\citet{chen2025faster}: a matrix with pattern $\mathcal S$ is formed from
$\chi+1$ gradients by the following lemma, whose proof is in
Appendix~\ref{app:backbone}. The dense case $|g|=1$ is the classical
finite-difference estimate, Lemma~\ref{lem:fd}.

\begin{lemma}\label{lem:group}
Let Assumption~\ref{as:main} hold, let $g$ be a group of the partition,
$x\in\mathbb R^d$ and $\tau>0$, and let
$w=(\nabla f(x+\tau v_g)-\nabla f(x))/\tau$. Then $
(\sum_{j\in g}\|P_jw-H(x)e_j\|^2)^{1/2}\ \le\ \frac{L_2\tau|g|}{2}$.
\end{lemma}
\noindent Thus the columns $H(x)e_j$, $j\in g$, are recovered from one
gradient difference with total Frobenius error $L_2\tau|g|/2$. In the
algorithms, the probe radius of a group is divided by $|g|$, so that every
group carries the same error, and the groups are probed in a fixed cyclic
order.

\subsection{The hybrid proximal extragradient framework}\label{sec:hpe}
Let $f$ be convex and differentiable with a minimizer. A proximal step from
a center $y$ balances the objective against a quadratic penalty for moving
away from $y$. For a parameter $\lambda>0$, the exact proximal point $u$
and its optimality condition are
\[
u=\arg\min_v\left\{f(v)+\frac{1}{2\lambda}\|v-y\|^2\right\}\qquad  \iff 
\qquad u-y+\lambda\nabla f(u)=0.
\]
Computing this
point exactly generally requires solving another nonlinear optimization
problem. The hybrid proximal extragradient (HPE) framework \citep{solodov1999hybrid,monteiro2010complexity} replaces exact
optimality by a relative residual condition, which we state below in the
smooth setting needed for this paper. This is the zero-enlargement
specialization of HPE; we do not need its more general monotone-inclusion
formulation.
\begin{definition}[Relative proximal pair]\label{def:hpepair}
Fix $\sigma\in(0,1)$. Given a center $y$, a pair $(u,\lambda)$ with
$\lambda>0$ is HPE-admissible for $f$ if
\[
\|u-y+\lambda\nabla f(u)\|\le\sigma\|u-y\|.
\]
\end{definition}
The residual on the left vanishes for an exact proximal step, and the
right-hand side measures the distance moved. The required accuracy
therefore adapts to the step itself, rather than being a fixed absolute
tolerance for every subproblem, and the error is measured in the same units
as the progress of the outer algorithm. The extragradient part of HPE then
uses the gradient at this approximate proximal point to correct the
auxiliary iterate, as we describe next. Our task will be to construct such
points using gradients while keeping $\lambda$ large.

\subsubsection{Accelerated HPE: the recursion and its guarantees}
The accelerated hybrid proximal extragradient (A-HPE) method of
\citet{monteiro2013accelerated} couples this relative proximal condition to
an extrapolation and an extragradient correction. The accelerated state consists of a primal point $x_k$, an auxiliary point
$z_k$, and a weight $A_k$. A convex combination produces the next center
$y_k$; an approximate proximal step gives $x_{k+1}$; and its gradient corrects
$z_k$. The identity $a_k^2=\lambda_k A_{k+1}$ makes the potential calculation
close. This gives a useful division of the analysis: the accelerated
recursion converts large admissible steps into a small objective gap,
while the curvature construction determines how large those steps can be.

Initialize $A_0=0$ and $x_0=z_0$. Given a positive stepsize
$\lambda_k$, first choose the additional weight $a_k$ and the accumulated
weight $A_{k+1}$ according to
\begin{equation}\label{eq:aident}
  a_k=\frac{\lambda_k+\sqrt{\lambda_k^2+4\lambda_kA_k}}2,
\qquad A_{k+1}=A_k+a_k,
\end{equation}
which implies $a_k^2=\lambda_kA_{k+1}$. This choice enforces the relation between the weight and the stepsize
needed by the accelerated potential. The same weights determine the
center at which we request the next proximal step:
\[
y_k=\frac{A_kx_k+a_kz_k}{A_{k+1}}.
\]
Thus $y_k$ lies between the current primal iterate and the auxiliary
iterate. Starting from this center, compute $x_{k+1}$ by any procedure
that satisfies the relative proximal condition
\begin{equation}\label{eq:ms}
\|x_{k+1}-y_k+\lambda_k\nabla f(x_{k+1})\|
\le\sigma\|x_{k+1}-y_k\|,\qquad 0<\sigma<1.
\end{equation}
Finally, use the gradient at the accepted point to update the auxiliary
iterate:
\[
z_{k+1}=z_k-a_k\nabla f(x_{k+1}).
\]
This is the extragradient correction: its direction is evaluated at the
new proximal point rather than at the extrapolation center. Together,
these three operations define the undamped A-HPE recursion used by
RP-QNPE.
The A-HPE recursion requires only that each step satisfy \eqref{eq:ms}; it
does not specify how such a step is computed, and its guarantees need no
smoothness assumption. The next proposition collects these guarantees: a
bound on the objective error, a lower bound on the growth of the weights
$A_k$ in terms of the stepsizes, and bounds on the movement of the
iterates. The curvature model and the proximal solver enter only through
the stepsizes $\lambda_k$ for which \eqref{eq:ms} can be certified;
Lemma~\ref{lem:wellposed} establishes this condition for RP-QNPE.
The potential and stepsize-growth estimates below are standard
consequences of the A-HPE analysis of
\citet{monteiro2013accelerated}. We include their proofs,
together with the inverse-moment and path-length consequences
used here, for completeness.

\begin{proposition}\label{prop:hpeestimates}
Consider the preceding recursion with $A_0=0$, $x_0=z_0$, and
$\|x_0-x^\star\|\le R$. Suppose every step satisfies \eqref{eq:ms}.
Let
\[
\Psi_k=A_k(f(x_k)-f^\star)+\tfrac12\|z_k-x^\star\|^2,
\quad C_x=1+\frac{\sigma}{\sqrt{1-\sigma^2}},
\quad \Lambda_N=1+\log(A_N/A_1).
\]
With $r_k=x_{k+1}-y_k$, the following hold for every $N\ge1$:
\begin{enumerate}[label=(\roman*),leftmargin=2em]
\item \emph{Potential and movement energy.}
\[
\Psi_{k+1}\le\Psi_k-\frac{1-\sigma^2}{2}
\frac{A_{k+1}}{\lambda_k}\|r_k\|^2,
\qquad f(x_N)-f^\star\le\frac{R^2}{2A_N},\qquad 
\sum_{k<N}\frac{A_{k+1}}{\lambda_k}\|r_k\|^2
\le\frac{R^2}{1-\sigma^2}.
\]
\item \emph{Growth from stepsize moments.} For every $q>0$,
\[
\sqrt{A_N}\ge\frac12 \sum_{k<N}\sqrt{\lambda_k},
\qquad
A_N\ge\frac{N^{2+1/q}}{4(\sum_{k<N}\lambda_k^{-q})^{1/q}}.
\]
\item \emph{Bounded iterates and path length.}
\[
\|z_k-x^\star\|\le R,\qquad \|x_k-x^\star\|\le C_xR,
\]
\[
Q_N:=\sum_{k<N}(\|y_k-x_k\|+\|x_{k+1}-y_k\|)
\le(1+C_x)R\Lambda_N+\frac{R\sqrt{\Lambda_N}}{\sqrt{1-\sigma^2}}.
\]
\end{enumerate}
\end{proposition}
\begin{proof}
Part (i) is Lemma~\ref{lem:potential}; part (ii) combines
Lemmas~\ref{lem:agrowth} and~\ref{lem:moment}; and part (iii) combines
Lemma~\ref{lem:potential}(ii) with Lemmas~\ref{lem:bounded}
and~\ref{lem:pathlen}. Appendix~\ref{app:backbone}
proves each of these statements directly from the recursion and convexity.
\end{proof}

Each part of the proposition plays a role. Part~(i) reduces the
convergence rate to the growth of $A_N$ and gives the computable stopping
test $A_N\ge R^2/(2\varepsilon)$; its movement-energy bound controls the
total length of the proximal steps. Part~(ii) turns stepsize bounds into
growth of $A_N$: a uniform lower bound on $\lambda_k$ gives the
accelerated-gradient rate, while with $q=1$ it gives
$f(x_N)-f^\star\le 2R^2N^{-3}\sum_{k<N}\lambda_k^{-1}$, so a first inverse
moment growing only logarithmically yields a cubic rate up to a logarithmic
factor (removed in Section~\ref{sec:algo} by localizing to dyadic scales of
$A_k$). Part~(iii)
bounds the iterates and their path length. Section~\ref{sec:algo} uses
all three parts for RP-QNPE, where the path length controls the error of
reused columns and the first inverse moment gives the $k^{-3}$ rate.

\subsubsection{Certified quasi-Newton steps}\label{sec:framework}
Let $F$ be convex with an $\mathcal L$-Lipschitz Hessian. An exact proximal
step from a center $y$ requires solving the nonlinear equation
$r+\lambda\nabla F(y+r)=0$. The Newton proximal extragradient method of
\citet{monteiro2012iteration,monteiro2013accelerated} linearizes the
gradient at $y$ and solves a regularized linear system involving
$\nabla^2F(y)$. We replace this Hessian by a positive semidefinite matrix
$B$ computed from gradients, and take
\[
r=-(B+\lambda^{-1}I)^{-1}\nabla F(y),\qquad u=y+r.
\]
This point is not automatically HPE-admissible, but its residual is
controlled by the accuracy of $B$. If $E\ge\|B-\nabla^2F(y)\|_{\rm op}$,
Taylor's theorem gives
\begin{equation}\label{eq:interfacebound}
\|r+\lambda\nabla F(y+r)\|
\le\lambda\left(E+\frac{\mathcal L}{2}\|r\|\right)\|r\|,
\end{equation}
which is Lemma~\ref{lem:resid}. The term $E$ accounts for replacing the
Hessian by $B$, and $\mathcal L\|r\|/2$ for the change of the Hessian along
the step. Hence the step is admissible whenever
$\lambda(E+\mathcal L\|r\|/2)\le\sigma$. This is the principle on which
our methods are built: a bound $E$ known before the step, together with an
a priori bound on $\|r\|$, yields an admissible stepsize directly, so no
line search is needed and no gradient is spent on rejected trials. The
matrix need not be accurate; a smaller certified error only allows a larger
step. Finite differences of gradients supply such a certificate at the
point where they are computed. Since building a full matrix costs $\chi+1$ gradients,
one at the anchor and one displaced probe per group of
Definition~\ref{def:chi}, the method reuses $B$ over many steps, and the central question is how to keep $E$ valid during
this reuse.

RP-QNPE (Section~\ref{sec:algo}) and its strongly convex counterpart
(Section~\ref{sec:strongly}) take this step with $F=f$ and choose $\lambda$
from explicit bounds on $E$ and $\|r\|$ built from stored quantities; the
nonconvex method (Section~\ref{sec:nonconvex}) uses the same certified
model as the hint of an optimistic online learner, where the certificate
limits the learning rate rather than the stepsize.


\section{Rolling-probe QNPE}\label{sec:algo}

RP-QNPE uses the certified quasi-Newton step of
Section~\ref{sec:framework} with $F=f$, so each step of the accelerated
recursion is a regularized model step. Two choices then determine the
method: how to construct and maintain the matrix $B_k$, and how to choose
the stepsize $\lambda_k$ so that the step is admissible. Since a full
finite-difference matrix costs $\chi+1$ gradients, we query only one new
group of curvature directions per iteration, and the central question is how to
reuse the other directions without losing track of the error they
introduce into the proximal step.

A-QNPE \citep{jiang2023accelerated} makes both choices through a gradient line search. It solves a regularized model system and checks the relative
proximal condition at trial points. A rejected trial supplies a
displacement $s_k$ and a gradient difference $w_k$, and the loss
$\|w_k-Bs_k\|^2/\|s_k\|^2$ measures how well the matrix predicts this
observed gradient change. An online learning procedure updates the matrix,
and its regret bound controls the accumulated directional errors, which
leads to control of the second inverse-stepsize moment
$\sum_k\lambda_k^{-2}$. The line search also requires damping in the
accelerated recursion to accommodate the change from a trial stepsize to
an accepted one.

We make both choices differently, and deterministically. We refresh one
group of columns in a fixed cyclic order and bound the error of the stored columns
by how far their anchors have moved from the current iterate. This gives
an error certificate before any trial step is computed. Its accumulated
value controls the first inverse-stepsize moment $\sum_k\lambda_k^{-1}$,
which by Proposition~\ref{prop:hpeestimates}(ii) yields a faster rate than
the second moment. The same certificate also determines the stepsize
without trial gradients, so the undamped recursion applies directly, and
the matrix construction and the stepsize analysis are linked through a
computable error bound rather than through a regret argument.

It is worth being precise about which feature of the certificate the
rate depends on. Write $\bar E_k$ for the computable bound on
$\|B_k-H(x_k)\|_{\rm op}$ constructed in Block~1. Block~2 will produce
\[
\lambda_k^{-1}\ \le\ \frac{10}{\sigma}\bar E_k
+\big(\text{terms controlled by the path length and the gradient norms}\big),
\]
while Proposition~\ref{prop:hpeestimates}(ii) with $q=1$ gives
$A_N\ge N^3/(4S_1)$ with $S_1=\sum_{k<N}\lambda_k^{-1}$, hence
\[
f(x_N)-f^\star\ \le\ \frac{2R^2}{N^3}\sum_{k<N}\lambda_k^{-1}.
\]
Hence, the design target is the sum $\sum_{k<N}\bar E_k$, and we need it to grow at most logarithmically in $N$, with a
factor $\chi$ permitted; the analysis of Section~\ref{sec:rpguarantee}
uses a localized form of this target, which also removes the logarithm
from the rate. Neither smallness nor monotone decrease of the
individual $\bar E_k$ is needed.
This requirement rules out a fixed model with a nonvanishing error certificate. If the error bound is a positive constant $E>0$, however small, then $\sum_{k<N}\bar E_k=NE$, so
$S_1=O(N)$, $A_N=\Omega(N^2)$, and the guarantee is the
accelerated-gradient rate and nothing more. The certificate has to improve as
the trajectory settles.

\paragraph{The proposed method.}
RP-QNPE has four blocks: refresh one group of columns, choose the stepsize, take a
regularized model step, and update the accelerated state. Its name
reflects this design. As in the quasi-Newton proximal extragradient
methods of \citet{jiang2023online,jiang2023accelerated}, the curvature
matrix enters only through the regularized model step, and the
extragradient correction updates the accelerated state. The prefix refers
to how the matrix is maintained: each iteration spends one
finite-difference probe on a single group of coordinates, and the probes
cycle through the $\chi$ groups, so the model is renewed over a rolling
window of $\chi$ iterations. After choosing $\lambda_k$, the method evaluates the
gradient at the extrapolated point $y_k$, solves the regularized linear
system for the primal step, and evaluates the gradient at its endpoint to
update $z_k$; this gradient is cached for the next iteration. Thus the
probe, the model step, and the extragradient correction cost at most three
gradient queries per iteration. The mass $A_k$ records certified progress
and gives the stopping rule. Algorithm~1 summarizes the method;
Section~\ref{sec:rpblock1} builds the matrix and accounts for the
different ages of its columns, and Section~\ref{sec:rpblock2} derives the
scalar stepsize policy.

\vspace{3mm}

\par\noindent\begin{minipage}{\linewidth}\begin{framed}
\noindent\textbf{Algorithm 1 (RP-QNPE).} Initialize
$x_0=z_0$, $A_0=0$, $M_{-1}=0$, and compute $\nabla f(x_0)$.
For $k=0,1,\ldots$:
\begin{enumerate}[label=\arabic*.,leftmargin=2em]
\item \textbf{Refresh and certify.} If $\nabla f(x_k)=0$, return $x_k$. Refresh the columns of group $g_{i_k}$ with one finite
      difference and form $B_k$ and $\bar E_k$.
\item \textbf{Choose the stepsize.} Compute $\lambda_k$ using only scalar operations as described in Subroutine 1.1. Set $a_k=a(\lambda_k)$ and $A_{k+1}=A_k+a_k$.
\item \textbf{Take the model step.} Evaluate $\nabla f(y_k)$ at
      $y_k=(A_kx_k+a_kz_k)/A_{k+1}$, solve
      $(B_k+\lambda_k^{-1}I)r_k=-\nabla f(y_k)$,
      and set $x_{k+1}=y_k+r_k$.
\item \textbf{Correct the momentum.} Evaluate $\nabla f(x_{k+1})$ and set
      $z_{k+1}=z_k-a_k\nabla f(x_{k+1})$.
\item \textbf{Check progress.} For target accuracy $\varepsilon$, return $x_{k+1}$ once
      $A_{k+1}\ge R^2/(2\varepsilon)$.
\end{enumerate}
\end{framed}\end{minipage}\par\medskip

\subsection{Block 1: refresh and certify the curvature model}\label{sec:rpblock1}

Block~1 must supply two things. First, feasibility, $B_k\in\mathcal Z$;
this is used twice, to make the regularized system well posed and, through
the resolvent bound $\|(B_k+\lambda^{-1}I)^{-1}B_k\|_{\rm op}\le1$ of
Block~2, to keep the momentum displacement at first power in the stepsize
certificate. Second, a quantity $\bar E_k$, computable from stored data
before any new gradient is queried, with
$\|B_k-H(x_k)\|_{\rm op}\le\bar E_k$, whose sum meets the design target discussed previously. 

\paragraph{The construction.}
A gradient difference along a group direction approximates the Hessian
columns of that group: for a group $g$ and any $\tau>0$,
Lemma~\ref{lem:group} gives
\[
\Big(\sum_{j\in g}\Big\|P_j\frac{\nabla f(x+\tau v_g)-\nabla f(x)}{\tau}-H(x)e_j\Big\|^2\Big)^{1/2}
\le\frac{L_2\tau|g|}{2};
\]
in the dense case this is Lemma~\ref{lem:fd}, one coordinate direction per
column. Reconstructing every column at $x_k$ would cost $\chi$ gradients.
We instead refresh one group per iteration and retain the others. A cyclic
order makes every stored column at most $\chi-1$ iterations old once a
cycle is complete, and Hessian Lipschitz continuity converts that age into
the distance the iterates have traveled since the column was queried.
Reuse is therefore productive precisely when the trajectory has settled,
which is the regime the accelerated potential delivers.

Fix $b>0$, $\sigma\in(0,1)$, and $M_{-1}=0$; the headline bounds take
$b=L_2R$. At iteration $k$ let $i_k=(k\bmod\chi)+1$ be the index of the
group to be refreshed and set
\begin{equation}\label{eq:refresh}
u_k=\frac{\nabla f(x_k+\tau_kv_{g_{i_k}})-\nabla f(x_k)}{\tau_k},
\qquad
M_ke_j=
\begin{cases}
P_ju_k, & j\in g_{i_k},\\
M_{k-1}e_j, & j\notin g_{i_k},
\end{cases}
\qquad
B_k=\Pi_{\mathcal Z}\left(\frac{M_k+M_k^\top}{2}\right),
\end{equation}
with a probe radius $\tau_k$ specified below. Thus $M_k$ replaces the
columns of $M_{k-1}$ indexed by $g_{i_k}$ by the block read off from $u_k$
and keeps the other columns; in the dense case $g_{i_k}=\{i_k\}$ and
$P_{i_k}=I$, so column $i_k$ is replaced by $u_k$. The gradient at $x_k$
is cached, so the refresh costs one query. The matrix $M_k$ is not
symmetric in general, since its columns were queried at different
anchors; symmetrization and clipping restore the feasibility required
above and use no gradients. Section~\ref{sec:implementation} gives their
arithmetic cost and a projection-free alternative.

\paragraph{The certificate.}
Write $m_k=\|y_k-x_k\|$, $\delta_k=m_k+\|r_k\|$ (which bounds $\|x_{k+1}-x_k\|$), and $g_k=\|\nabla f(x_k)\|$. For $k\ge\chi$ set
\[
\widehat D_k=\sum_{j=k-\chi+1}^{k-1}\delta_j,\qquad
e_k=\frac{b}{(k-\chi+2)^2},
\]
the path length over the current window and the finite-difference
allowance, which with the radius chosen below bounds twice the
reconstruction error of the stored blocks. Both are available after the
refresh and before the next model step.

\begin{lemma}\label{lem:telescope}
For all $k\ge0$, we have $\|B_k-H(x_k)\|_{\rm op}\le\bar E_k$, where
\begin{equation}\label{eq:rollingcert}
\bar E_k=
\begin{cases}
L_1,&k<\chi,\\[2pt]
\min\{L_1,\ 2L_2\widehat D_k+e_k\},&k\ge\chi.
\end{cases}
\end{equation}\end{lemma}
\noindent The proof is in Appendix~\ref{proof:lem:telescope}.

The cap $L_1$ is valid at every iteration, because $B_k$ and $H(x_k)$ both
lie in $\mathcal Z$, so their difference lies in $[-L_1I,L_1I]$. It is the
only bound available while the window is still incomplete.

The second branch is where the reuse is paid for, and the way it is assembled
is what determines the dependence on $\chi$. Bounding each stale block
separately and combining the $\chi$ block errors would convert a blockwise
bound into an operator-norm bound and lose a factor $\sqrt\chi$ on the drift.
This is avoided by exchanging the order of summation. Up to
finite-difference errors, the staleness of $M_k$ equals
$-\sum_{t=k-\chi+1}^{k-1}\bigl(H(x_{t+1})-H(x_t)\bigr)\Pi_t$, where $\Pi_t$
is the orthogonal projection onto the coordinates whose stored column was
probed at or before iteration $t$. This expresses the staleness through
\emph{full} Hessian increments multiplied by orthogonal projections. Each
$\Pi_t$ has operator norm at most one, so the drift contributes at most
$L_2\sum_t\|x_{t+1}-x_t\|\le L_2\widehat D_k$, with no factor depending on
$\chi$ or $d$. The finite-difference errors are the one place a $\sqrt\chi$
remains: the $\chi$ stored blocks occupy disjoint columns, so their
Frobenius errors add in squares, and $\chi$ blocks with Frobenius error at
most $e_k/(2\sqrt\chi)$ each give Frobenius norm $e_k/2$. Clipping doubles
both contributions by Lemma~\ref{lem:clip}, which gives
\eqref{eq:rollingcert}.

The above result also clarifies how the probe radius sequence $\tau_k$
should be set. Note that the second branch of \eqref{eq:rollingcert} has
two terms. The drift term $2L_2\widehat D_k$ is not a design choice: its
sum is controlled by the accelerated potential, as shown below. The
finite-difference term $e_k$ is entirely a design choice, and it must
satisfy $\sum_ke_k=O(b)$. A constant radius fails this, contributing
$\Theta(N)$ and capping the rate at $N^{-2}$ exactly as a fixed model
would. A summable radius suffices, and we take
\[
\tau_k=\frac{b}{\sqrt\chi\,L_2\,|g_{i_k}|\,(k+1)^2}.
\]
By Lemma~\ref{lem:group}, the block probed at iteration $k$ then has
Frobenius error at most $L_2\tau_k|g_{i_k}|/2=b/(2\sqrt\chi\,(k+1)^2)$,
independent of the size of the group: the factor $\sqrt\chi L_2|g_{i_k}|$
cancels the conversion above, so the finite-difference term is exactly
$e_k=b/(k-\chi+2)^2$, determined by the oldest block in the window and
free of $\chi$, $L_2$, and the group sizes, and these terms sum to at most
$2b$. Any exponent larger than one keeps this contribution an additive
constant (exponent one would add a logarithm); the choice $2$ gives the
explicit bound $2b$. In the
dense case the radius is $\tau_k=b/(\sqrt d\,L_2(k+1)^2)$.

The construction was designed around a sum, and the following lemma evaluates
it. Its proof uses only the window structure, not the stepsize schedule, so
the promised bound is available before Block 2 is specified.

\begin{lemma}\label{lem:sum}
Let $Q_N=\sum_{k<N}\delta_k$. Then $
\sum_{k<N}\bar E_k\le\chi L_1+2\chi L_2Q_N+2b$,
and the sum is also at most $NL_1$.
\end{lemma}
\noindent The proof is in Appendix~\ref{proof:main:lem:sum}. The mechanism is
that a displacement is charged only while it remains in a window: the
condition $k-\chi+1\le j\le k-1$ is equivalent to $j+1\le k\le j+\chi-1$, so each
$\delta_j$ is counted at most $\chi-1$ times. The finite-difference allowances
sum to at most $2b$ by the previous paragraph.
Combining with the path-length estimate
$Q_N=O_\sigma(R\Lambda_N)$ of Proposition~\ref{prop:hpeestimates}(iii) and
taking $b=L_2R$ gives
\begin{equation}\label{eq:certsum}
\sum_{k<N}\bar E_k\ =\ O_\sigma\big(\chi(L_1+L_2R)\,\Lambda_N\big),
\end{equation}
which is the $\tilde O(\chi)$ required by the design target. The term
$\chi L_1$ is the cost of learning the initially missing columns. The
proof of Theorem~\ref{thm:iterationrate} uses a localized form of this
estimate: within a dyadic scale of the mass $A_k$, the path length is
$O_\sigma(R)$, and charging the first $\chi$ certificates of the scale to
the cap $L_1$ leaves windows that lie inside the scale, so the
certificates of each scale sum to $O_\sigma(\chi(L_1+L_2R))$, with no
logarithm.

\subsection{Block 2: choose a step from stored scalar data}\label{sec:rpblock2}

As mentioned earlier, we seek the largest
stepsize we can take at every iteration. The only limit is
that the pair $(y_k,x_{k+1})$ must still satisfy the HPE condition
\eqref{eq:ms}. By Lemma~\ref{lem:resid} applied at $y_k$, a sufficient
condition is
\begin{equation}\label{eq:targetcond}
\lambda_k\left(\|B_k-H(y_k)\|_{\rm op}
+\frac{L_2}{2}\|r_k\|\right)\ \le\ \sigma .
\end{equation}
Block 2 is therefore an optimization in a single scalar: maximize $\lambda$
subject to \eqref{eq:targetcond}.

Neither quantity in \eqref{eq:targetcond} is available when the choice is
made, and both depend on the quantity being chosen. The extrapolation point
$y_k=y(\lambda)$ moves with $\lambda$, so $\|B_k-H(y_k)\|_{\rm op}$ is not the certified error
$\bar E_k$ at $x_k$; and $r_k=r(\lambda)$ is the solution of a linear system
whose right-hand side is $\nabla f(y(\lambda))$, a gradient we have not paid
for. 
Testing \eqref{eq:ms} by backtracking with $y(\lambda)$ recomputed
at each trial would require up to two gradients per trial.
A-QNPE instead uses a line-search mechanism with a damping
correction to accommodate the change from its trial stepsize
to the accepted one.
We instead replace each of the two terms by an upper bound that is an explicit
function of $\lambda$ and of stored data, and solve the resulting scalar
problem by bisection.

For a candidate $\lambda>0$, let $a(\lambda)$ be the positive solution of
$a^2=\lambda(A_k+a)$, namely
$a(\lambda)=\big(\lambda+\sqrt{\lambda^2+4\lambda A_k}\big)/2$, and let
\[
y(\lambda)=\frac{A_kx_k+a(\lambda)z_k}{A_k+a(\lambda)},
\qquad
r(\lambda)=-(B_k+\lambda^{-1}I)^{-1}\nabla f(y(\lambda)).
\]
These are the points Algorithm~1 forms once $\lambda_k=\lambda$ is fixed.
Since $y(\lambda)-x_k=\frac{a(\lambda)}{A_k+a(\lambda)}(z_k-x_k)$, the
distance from the gradient anchor $x_k$ to $y(\lambda)$ is known in closed
form in terms of $\zeta_k=\|z_k-x_k\|$:
\begin{equation}\label{eq:mudef}
m(\lambda)=\|y(\lambda)-x_k\|=\frac{a(\lambda)}{A_k+a(\lambda)}\zeta_k
=\frac{\zeta_k}{2A_k}\Big(\sqrt{\lambda^2+4\lambda A_k}-\lambda\Big),
\end{equation}
with $m\equiv0$ when $A_k=0$, since $x_0=z_0$. The next lemma bounds
both terms of \eqref{eq:targetcond} by explicit functions of $\lambda$
built only from $m(\lambda)$, $\bar E_k$, and the cached
$g_k=\|\nabla f(x_k)\|$.

\begin{lemma}\label{lem:scalarbounds}
For every $\lambda>0$,
\begin{align}
\|B_k-H(y(\lambda))\|_{\rm op}\ &\le\ E_k(\lambda):=\min\{L_1,\ \bar E_k+L_2m(\lambda)\},
\label{eq:Etransport}\\
\|r(\lambda)\|\ &\le\ \widehat\rho_k(\lambda):=\lambda g_k
+m(\lambda)\left(1+\lambda\bar E_k+\frac{L_2\lambda}{2}m(\lambda)\right).
\label{eq:rhodef}
\end{align}
\end{lemma}
\noindent The proof is in Appendix~\ref{proof:lem:scalarbounds}.
The first bound is the certificate at $x_k$, inflated by the distance
$m(\lambda)$ to the extrapolated point and capped at $L_1$. The second
bounds the model step by three terms: the gradient at the anchor, the
displacement to the extrapolated point, and the curvature error along that
displacement. The displacement enters without a factor $\lambda L_1$
because $B_k\succeq0$; this is the second use of $B_k\in\mathcal Z$
promised in Block~1.

Substituting the bounds of Lemma~\ref{lem:scalarbounds} into
\eqref{eq:targetcond} gives
$\widehat\Psi_k(\lambda)=\lambda E_k(\lambda)+\frac{L_2\lambda}{2}\widehat\rho_k(\lambda)$,
which dominates the left-hand side of \eqref{eq:targetcond}; hence
$\widehat\Psi_k(\lambda)\le\sigma$ suffices for acceptance. Expanded, with
$m=m(\lambda)$,
\begin{equation}\label{eq:psiexpanded}
\widehat\Psi_k(\lambda)=\lambda\min\{L_1,\bar E_k+L_2m\}
+\frac{L_2g_k}{2}\lambda^2
+\frac{L_2}{2}\lambda m
+\frac{L_2\bar E_k}{2}\lambda^2 m
+\frac{L_2^2}{4}\lambda^2 m^2 .
\end{equation}
Note that every term is nondecreasing in $\lambda$, so $\widehat\Psi_k(\lambda)\le\sigma$
is a monotone test, and a nearly largest admissible stepsize can be found by
a scalar search.

\vspace{2mm}
\par\noindent\begin{minipage}{\linewidth}\begin{framed}
\noindent\textbf{Subroutine 1.1 (Stepsize selection).} Input: $A_k$,
$\zeta_k$, $g_k$, $\bar E_k$. Evaluate $\widehat\Psi_k$ through
\eqref{eq:psiexpanded} and \eqref{eq:mudef}.
\begin{enumerate}[label=\arabic*.,leftmargin=2em]
\item If $\widehat\Psi_k(\sigma/L_1)\ge\sigma$, return $\lambda_k=\sigma/L_1$.
\item Set $\lambda^-=\sigma/L_1$ and $\lambda^+=\sigma/\bar E_k$.
\item While $\lambda^+/\lambda^->2^{1/4}$: set
      $\lambda=\sqrt{\lambda^-\lambda^+}$; if $\widehat\Psi_k(\lambda)<\sigma$,
      set $\lambda^-=\lambda$, otherwise set $\lambda^+=\lambda$.
\item Return $\lambda_k=\lambda^-$.
\end{enumerate}
\end{framed}\end{minipage}\par\medskip

\paragraph{The logic of Subroutine 1.1.}
Two certificates are available for the stepsize. The smoothness step
$\sigma/L_1$ is admissible whatever the model does, by the second
certificate of Lemma~\ref{lem:resid}, so it needs no test. 
The model test certifies every $\lambda>0$ with
$\widehat\Psi_k(\lambda)\le\sigma$.
The function $\widehat\Psi_k$ is continuous and strictly
increasing, tends to zero as $\lambda\downarrow0$, and
diverges as $\lambda\to\infty$, as verified in
Appendix~\ref{proof:main:lem:wellposed}.
Consequently, there is a unique positive root
$\lambda_k^\circ$ of $\widehat\Psi_k(\lambda)=\sigma$,
and the model test certifies exactly the interval
$(0,\lambda_k^\circ]$.
The longest certified step is therefore
$\max\{\sigma/L_1,\lambda_k^\circ\}$. The root is not available in closed
form, but a single evaluation at the smoothness step tells which of the two
is larger: by monotonicity, $\widehat\Psi_k(\sigma/L_1)\ge\sigma$ if and
only if $\lambda_k^\circ\le\sigma/L_1$. In that case the model cannot
certify a step longer than the smoothness step, and Step~1 returns it.
Otherwise $\lambda_k^\circ>\sigma/L_1$, and the subroutine approximates the
root from below. This requires a step at which the test is guaranteed to
fail. The natural candidate is $\sigma/\bar E_k$: it is the step the test
would allow if the certified model error were the only error, and the
remaining terms of $\widehat\Psi_k$ can only shorten it. Indeed, since
$\bar E_k\le L_1$, the first term of \eqref{eq:psiexpanded} alone gives
$\widehat\Psi_k(\sigma/\bar E_k)\ge\sigma$. Its location depends on
the certificate, not on the gradient norm, which keeps the search short.
Steps~2--3 bisect on $\log\lambda$ and maintain
$\widehat\Psi_k(\lambda^-)<\sigma\le\widehat\Psi_k(\lambda^+)$. In both
cases $\lambda_k$ is within a factor $2^{1/4}$ of
$\max\{\sigma/L_1,\lambda_k^\circ\}$, and it never falls below the
globally safe step.

\begin{lemma}\label{lem:wellposed}
Subroutine~1.1 uses no gradient queries and at most
\[
1+\left\lceil
\log_2\!\left(
\max\left\{1,\frac{W_k}{\log(2^{1/4})}\right\}
\right)
\right\rceil
\]
evaluations of $\widehat\Psi_k$, where
$W_k=\log(L_1/\bar E_k)$. The returned stepsize
satisfies the relative condition \eqref{eq:ms}, and
\[
\frac{\sigma}{L_1}\ \le\ \lambda_k\ \le\ \frac{\sigma}{\bar E_k},
\qquad \widehat\Psi_k(\lambda_k)\ \ge\ \frac{\sigma}{2}.
\]
\end{lemma}
\noindent The proof is in Appendix~\ref{proof:main:lem:wellposed}.
For $k<\chi$ we have $\bar E_k=L_1$, so $W_k=0$ and Step~1 returns the
smoothness step without search. For $k\ge\chi$,
$\bar E_k\ge\min\{L_1,e_k\}$ with $e_k=b/(k-\chi+2)^2$ by \eqref{eq:rollingcert}, so
$W_k\le\max\{0,2\log(k-\chi+2)+\log(L_1/b)\}$ and the search costs
$O(\log\log k)$ scalar evaluations. The bound
$\widehat\Psi_k(\lambda_k)\ge\sigma/2$ is the form in which near-maximality
enters the rate analysis; acceptance alone would not suffice, since every
small enough $\lambda$ passes the test.

\subsection{Guarantee and proof outline}\label{sec:rpguarantee}

The analysis has a simple skeleton and three places where it could fail.
The skeleton is the A-HPE accounting of Section~\ref{sec:prelim}: once every
step is admissible, inverse stepsizes convert into mass and mass into
accuracy, $f(x_N)-f^\star\le R^2/(2A_N)$. For the whole run,
Proposition~\ref{prop:hpeestimates}(ii) gives $A_N\gtrsim N^3/S_1$ with
$S_1=\sum_{k<N}\lambda_k^{-1}$, and \eqref{eq:certsum} bounds the leading
part of $S_1$ by $\tilde O(\chi)$, as the design target requires. Used
globally, this conversion loses a logarithm, because the path length of the
whole run is $O(R\Lambda_N)$; we perform it locally instead. Group the
iterations into dyadic scales of the mass: for $s=2^jA_1$, the scale $I$
consists of the $n$ consecutive iterations with $s\le A_{k+1}<2s$. Within a
scale, $\sum_{k\in I}\sqrt{\lambda_k}\le2\sqrt{2s}$, and H\"older's
inequality gives $n^3\le8sS_I$ with $S_I=\sum_{k\in I}\lambda_k^{-1}$.
Everything therefore rests on one statement, $S_I=O(\chi)$ on every scale:
the price of the curvature model, in inverse stepsizes, is one installment
of size $\chi$ per doubling of the mass, however long the method runs. The
obstacles are to make each step admissible without a gradient line search,
to keep that price proportional to $\chi$ rather than to $\chi\sqrt\chi$ or
$n$, and to close a circularity between stepsizes and gradient norms.

Admissibility is Lemma~\ref{lem:wellposed}. The acceptance condition
involves two quantities unavailable when the stepsize is chosen, the model
error at the extrapolated point and the length of the model step, and both
depend on the stepsize itself. We replace them by explicit upper bounds
built from stored scalars: the certificate is transported from $x_k$ to
$y_k$, and the step is bounded through the resolvent, with no factor
$\lambda L_1$ because $B_k\succeq0$. The surrogate is monotone in $\lambda$,
so a scalar bisection returns a nearly largest certified step, and the
guarantee $\widehat\Psi_k(\lambda_k)\ge\sigma/2$ is what the rate analysis
inverts: each $\lambda_k^{-1}$ is bounded by five terms, led by the
certified error $\bar E_k$ (Lemma~\ref{lem:fiveterm}). The price of the
model is thus $\sum_k\bar E_k$. A block of columns that is $\chi-1$
iterations old is wrong by the Hessian drift since it was probed; charging
the $\chi$ blocks separately costs a factor $\sqrt\chi$ in operator norm,
and a certificate that does not shrink along the trajectory costs
$\Theta(N)$, which caps the rate at $N^{-2}$. The exchange of summation in
Lemma~\ref{lem:telescope} regroups the stale blocks into full Hessian
increments times orthogonal projections, so the drift costs $L_2$ times the
path length with no dimension factor, and Lemma~\ref{lem:sum} charges each
displacement at most $\chi-1$ times. The path length is $O(R\Lambda_N)$ by
the accelerated potential, but only $O(R)$ within a scale, where
$\sum_{k\in I}a_k/A_{k+1}\le1+\log2$. This is the mechanism behind the
method: the trajectory settles, and stale columns become accurate,
precisely in the regime that acceleration produces. The result is
\eqref{eq:certsum}, $\sum_k\bar E_k=\tilde O(\chi)$, and its local form:
once the first $\chi$ certificates of a scale are charged to the cap $L_1$,
the remaining windows lie inside the scale, and
$\sum_{k\in I}\bar E_k=O(\chi)$.

The circularity is the term in $\sqrt{g_k}$: the stepsize bound involves
the gradient norm at the anchor, gradient norms are controlled by the energy
of the potential, and the energy is spent at a rate set by the stepsizes
themselves. Within a scale the circularity is harmless. Every weight there
satisfies $a_k^2=\lambda_kA_{k+1}\ge\sigma s/L_1$, so spending the energy
budget a second time gives
$\sum_{k\in I}\sqrt{g_k}\lesssim\sqrt{L_1R}+L_1^{1/4}R^{1/2}n^{3/4}s^{-1/4}$,
and $n^3\le8sS_I$ becomes an inequality in which $n$ enters the right-hand
side only through $n^{3/4}$. Absorbing it gives $n=O_\sigma((Ks)^{1/3})$
with $K=\chi(L_1+L_2R)$. The scales double, so these counts form a
geometric series dominated by the last scale, and
$N=O_\sigma((KA_N)^{1/3})$ with no logarithm. Two independent lower bounds
on the same $A_N$ follow: the floor $\lambda_k\ge\sigma/L_1$ gives
$A_N=\Omega(N^2/L_1)$ unconditionally, and the scale count gives
$A_N=\Omega(N^3/K)$. Theorem~\ref{thm:iterationrate} is their minimum, for
the same iterates. Appendix~\ref{app:rolling} gives the complete proofs,
including an explicit, non-asymptotic form of the bound,
\eqref{eq:fp-explicit-rate}. Throughout, $b=L_2R$.

\begin{theorem}\label{thm:iterationrate}
Under Assumption~\ref{as:main}, for all $k\ge1$,  Algorithm~1 with Subroutine~1.1 satisfies
\begin{equation}\label{eq:iterationheadline}
f(x_k)-f^\star\le O_\sigma\!\left(
\min\left\{\frac{L_1R^2}{k^2},
\frac{\chi(L_1+L_2R)R^2}{k^3}\right\}\right).
\end{equation}
Moreover, every iteration uses at most three gradients.
\end{theorem}
\noindent The proof is in Appendix~\ref{proof:iterationrate}. For fixed
$L_1,L_2,R,\sigma$, the bound reads $O(\min\{k^{-2},\chi k^{-3}\})$, with no
logarithmic factor, and the $k^{-3}$ branch is the smaller one once
$k\gtrsim\chi$. The following corollary gives the corresponding total query
count, including stopping and initialization.

\begin{corollary}\label{thm:complexity}
Under Assumption~\ref{as:main}, for any $\varepsilon>0$ and $b=L_2R$,
Algorithm~1 with Subroutine~1.1 returns an $\varepsilon$-accurate point in
$G_\varepsilon$ gradient queries, where
\begin{equation}\label{eq:rpcomplexity}
G_\varepsilon=O_\sigma\left(1+
\min\left\{\sqrt{\frac{L_1R^2}{\varepsilon}},
\left(\frac{\chi(L_1+L_2R)R^2}{\varepsilon}\right)^{1/3}
\right\}\right).
\end{equation}
\end{corollary}
\noindent The proof is in Appendix~\ref{proof:thm:complexity}, which also
gives the bound with explicit constants, \eqref{eq:fp-explicit-queries}.
The stopping rule is implementable because $A_k$ is computed by the
method; it does not require knowing $f^\star$. There is no additive $\chi$
in the bound: the method starts optimizing while the curvature model is
being built, and the smoothness branch remains available throughout.
The normalized consequence makes the dependence on dimension and
accuracy explicit.

The above result implies that, for fixed $L_1,L_2,R,\sigma$, the RP-QNPE method has gradient complexity of $
G_\varepsilon=O\left(1+
\min\{\varepsilon^{-1/2},\chi^{1/3}\varepsilon^{-1/3}\}\right)$, with
$\chi=d$ in the dense case.

\subsection{Comparison with existing guarantees}\label{sec:compare}
Throughout this section $L_1$, $L_2$, $R$ and $\sigma$ are fixed,
$0<\varepsilon\le1$, and logarithmic factors are suppressed unless
displayed; a \emph{normalized} bound is the resulting function of the
dimension and $\varepsilon$ (or $k$). 
 All comparisons are between proved upper bounds on gradient
queries or equivalent gradient costs under the stated Hessian
cost convention; matrix arithmetic is treated in
Section~\ref{sec:implementation}. 
Table~\ref{tab:convexcompare} collects the bounds and Figure~\ref{fig:convexregimes} draws the bounds against $1/\varepsilon$ in the dense case and on a sparse slice, and Figure~\ref{fig:convexmaps} maps the smallest known bound as a function of accuracy and sparsity. There and below, $\bar d$ denotes the Hessian-cost parameter of the exact-Hessian lazy analyses, instantiated as $d$ in the dense case and, where indicated, as $\chi$, the finite-difference cost of Lemma~\ref{lem:group}. Prior bounds are stated in terms of $d$, so
Sections~\ref{sec:aqcompare} and~\ref{sec:lazycompare} compare in the dense
case $\chi=d$, which suffices because our bound is nondecreasing in $\chi$;
Section~\ref{sec:structurecompare} quantifies the additional gain when
$\chi<d$.

\begin{table}[t]
\centering\small
\begin{tabular}{@{}lll@{}}
\toprule
Method & Oracle & Gradient complexity\\
\midrule
AGD \citep{nesterov2018lectures} & gradient & $O(\varepsilon^{-1/2})$\\
Cutting planes \citep{lee2015faster,jiang2020improved} & gradient & $\tilde O(d)$\\
A-QNPE \citep{jiang2023accelerated} & gradient & $\tilde O(\min\{\varepsilon^{-1/2},d^{1/5}\varepsilon^{-2/5}\})$\\
\midrule
LazyCRN \citep{chayti2024lazy} & Hessian $=\bar d$ gradients & $O(\bar d+\bar d^{1/2}\varepsilon^{-1/2})$\\
CALEN \citep{chen2025faster} & Hessian $=\bar d$ gradients & $\tilde O(\bar d+\bar d^{13/21}\varepsilon^{-2/7})$\\
\midrule
RP-QNPE (Corollary~\ref{thm:complexity}) & gradient & $O(\min\{\varepsilon^{-1/2},\chi^{1/3}\varepsilon^{-1/3}\})$\\
\bottomrule
\end{tabular}
\caption{\small Convex case, with $L_1$, $L_2$, $R$ fixed and logarithmic factors suppressed;
$\chi\le d$ is the effective dimension of Definition~\ref{def:chi} ($\chi=d$ for a dense Hessian).
The LazyCRN and CALEN bounds concern exact-Hessian implementations, counted at $\bar d$ gradients
per Hessian as in \citet{chen2025faster}: $\bar d=d$ as published, $\bar d=\chi$
(Lemma~\ref{lem:group}) as a hypothetical instantiation.}
\label{tab:convexcompare}
\end{table}

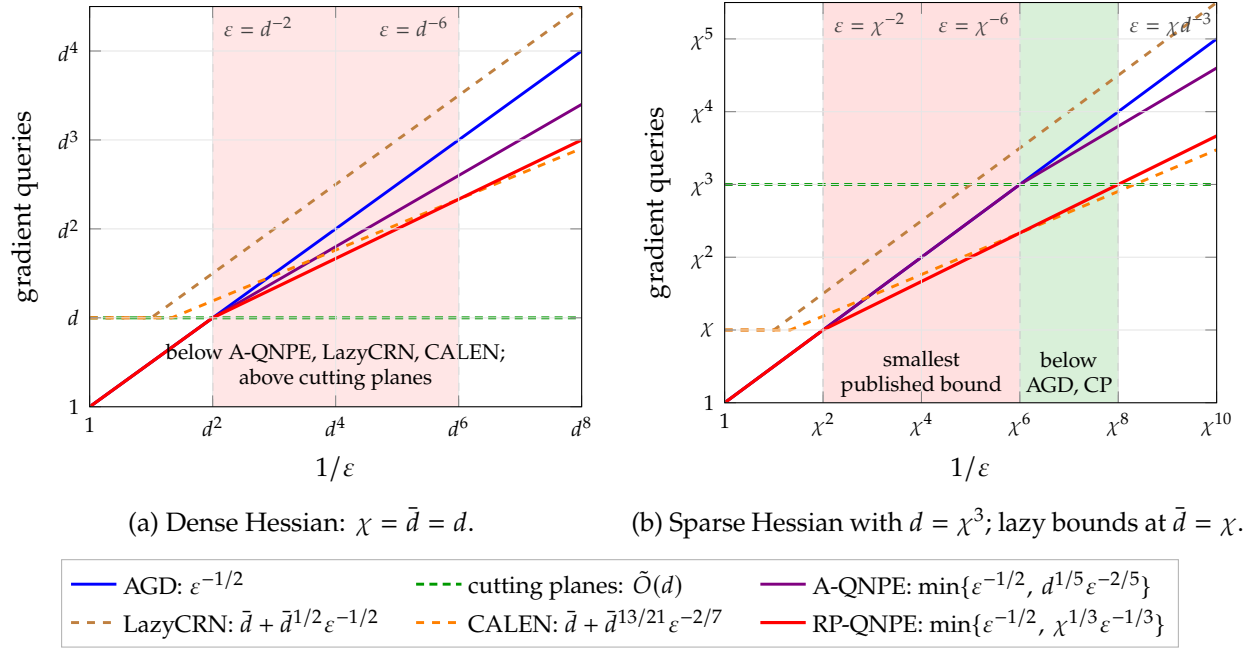
\begin{figure}[t]
\centering
\begin{subfigure}[t]{0.49\textwidth}
\centering
\begin{tikzpicture}
\begin{axis}[width=\linewidth, height=0.85\linewidth, grid=major, grid style={gray!20}, axis on top,
  xmin=0, xmax=8, ymin=0, ymax=4.5,
  xtick={0,2,4,6,8}, xticklabels={$1$,$d^{2}$,$d^{4}$,$d^{6}$,$d^{8}$},
  ytick={0,...,4}, yticklabels={$1$,$d$,$d^{2}$,$d^{3}$,$d^{4}$},
  xlabel={$1/\varepsilon$}, ylabel={gradient queries},
  tick label style={font=\scriptsize}, label style={font=\small},
  every axis plot/.append style={line width=1.1pt},
  legend to name=leg:convexregimes, legend columns=3, legend cell align={left},
  legend style={font=\footnotesize, draw=gray!60,
                /tikz/every even column/.append style={column sep=0.4cm}}]
\fill[red!10] (axis cs:2,0) rectangle (axis cs:6,4.5);
\addplot[blue] coordinates {(0,0) (8,4)};
  \addlegendentry{AGD: $\varepsilon^{-1/2}$}
\addplot[green!60!black, densely dashed] coordinates {(0,1) (8,1)};
  \addlegendentry{cutting planes: $\tilde O(d)$}
\addplot[violet] coordinates {(0,0) (2,1) (8,3.4)};
  \addlegendentry{A-QNPE: $\min\{\varepsilon^{-1/2},\,d^{1/5}\varepsilon^{-2/5}\}$}
\addplot[brown, dashed] coordinates {(0,1) (1,1) (8,4.5)};
  \addlegendentry{LazyCRN: $\bar d+\bar d^{1/2}\varepsilon^{-1/2}$}
\addplot[orange, dashed] coordinates {(0,1) (1.3333,1) (8,2.9048)};
  \addlegendentry{CALEN: $\bar d+\bar d^{13/21}\varepsilon^{-2/7}$}
\addplot[red, very thick] coordinates {(0,0) (2,1) (8,3)};
  \addlegendentry{RP-QNPE: $\min\{\varepsilon^{-1/2},\,\chi^{1/3}\varepsilon^{-1/3}\}$}
\draw[dashed, gray!70] (axis cs:2,0) -- (axis cs:2,4.5) node[pos=1, below right, font=\scriptsize, text=gray!50!black] {$\varepsilon=d^{-2}$};
\draw[dashed, gray!70] (axis cs:6,0) -- (axis cs:6,4.5) node[pos=1, below left,  font=\scriptsize, text=gray!50!black] {$\varepsilon=d^{-6}$};
\node[font=\scriptsize, align=center] at (axis cs:4,0.45) {below A-QNPE, LazyCRN, CALEN;\\above cutting planes};
\end{axis}
\end{tikzpicture}
\caption{Dense Hessian: $\chi=\bar d=d$.}
\label{fig:convexregimes-dense}
\end{subfigure}\hfill
\begin{subfigure}[t]{0.49\textwidth}
\centering
\begin{tikzpicture}
\begin{axis}[width=\linewidth, height=0.85\linewidth, grid=major, grid style={gray!20}, axis on top,
  xmin=0, xmax=10, ymin=0, ymax=5.5,
  xtick={0,2,4,6,8,10}, xticklabels={$1$,$\chi^{2}$,$\chi^{4}$,$\chi^{6}$,$\chi^{8}$,$\chi^{10}$},
  ytick={0,...,5}, yticklabels={$1$,$\chi$,$\chi^{2}$,$\chi^{3}$,$\chi^{4}$,$\chi^{5}$},
  xlabel={$1/\varepsilon$}, ylabel={gradient queries},
  tick label style={font=\scriptsize}, label style={font=\small},
  every axis plot/.append style={line width=1.1pt}]
\fill[red!12]            (axis cs:2,0) rectangle (axis cs:6,5.5);
\fill[green!60!black!15] (axis cs:6,0) rectangle (axis cs:8,5.5);
\addplot[blue] coordinates {(0,0) (10,5)};
\addplot[green!60!black, densely dashed] coordinates {(0,3) (10,3)};
\addplot[violet] coordinates {(0,0) (6,3) (10,4.6)};
\addplot[brown, dashed] coordinates {(0,1) (1,1) (10,5.5)};
\addplot[orange, dashed] coordinates {(0,1) (1.3333,1) (10,3.4762)};
\addplot[red, very thick] coordinates {(0,0) (2,1) (10,3.6667)};
\draw[dashed, gray!70] (axis cs:2,0) -- (axis cs:2,5.5) node[pos=1, below right, font=\scriptsize, text=gray!50!black] {$\varepsilon=\chi^{-2}$};
\draw[dashed, gray!70] (axis cs:6,0) -- (axis cs:6,5.5) node[pos=1, below left,  font=\scriptsize, text=gray!50!black] {$\varepsilon=\chi^{-6}$};
\draw[dashed, gray!70] (axis cs:8,0) -- (axis cs:8,5.5) node[pos=1, below right, font=\scriptsize, text=gray!50!black] {$\varepsilon=\chi d^{-3}$};
\node[font=\scriptsize, align=center] at (axis cs:4,0.4) {smallest\\published bound};
\node[font=\scriptsize, align=center] at (axis cs:7,0.4) {below\\AGD, CP};
\end{axis}
\end{tikzpicture}
\caption{Sparse Hessian with $d=\chi^{3}$; lazy bounds at $\bar d=\chi$.}
\label{fig:convexregimes-sparse}
\end{subfigure}

\medskip
\pgfplotslegendfromname{leg:convexregimes}
\caption{Convex case: gradient complexity against $1/\varepsilon$, with both axes in
powers of $d$ in (a) and of $\chi$ in (b); constants and logarithmic factors are suppressed. The dashed brown and orange curves represent the exact-Hessian
analyses of LazyCRN and CALEN counted at $\bar d$ gradients per
Hessian, with $\bar d=d$ in (a) and the hypothetical instantiation $\bar d=\chi$ in (b). The curves are the bounds of
Table~\ref{tab:convexcompare}. On the shaded range of (a) RP-QNPE is below A-QNPE, LazyCRN and CALEN
(and above cutting planes); in (b) it is the smallest bound shown on the red range and below AGD and
cutting planes, but above CALEN at $\bar d=\chi$, on the green range.}
\label{fig:convexregimes}
\end{figure}

\subsubsection{RP-QNPE versus A-QNPE}\label{sec:aqcompare}
The normalized rates are $\min\{k^{-2},dk^{-3}\}$ for RP-QNPE and
$\min\{k^{-2},\sqrt{d\log(k+1)}\,k^{-5/2}\}$ for A-QNPE. Ignoring logarithmic factors, both equal $k^{-2}$
for $k\le d$, and for every $k>d$ the RP-QNPE bound $dk^{-3}$ is strictly below both branches of the A-QNPE bound (its ratio to $\sqrt{d\log(k+1)}\,k^{-5/2}$ is $\sqrt{d/(k\log(k+1))}<1$), by a factor that tends to zero. Since both methods use a
constant number of gradients per iteration on average, the same holds for the
polynomial gradient complexities $\min\{\varepsilon^{-1/2},d^{1/3}\varepsilon^{-1/3}\}$
and $\min\{\varepsilon^{-1/2},d^{1/5}\varepsilon^{-2/5}\}$ (with logarithmic factors suppressed): both curvature branches are at least $\varepsilon^{-1/2}$ when $\varepsilon\ge d^{-2}$,
where both bounds equal that of AGD, and for $\varepsilon<d^{-2}$ their
ratio is $(d^2\varepsilon)^{1/15}<1$ (retaining the logarithmic factor in A-QNPE places its crossover with AGD at $\varepsilon\asymp(d\log d)^{-2}$). Thus RP-QNPE improves on A-QNPE
throughout the regime in which A-QNPE improves on AGD, and matches AGD
outside it; the comparison concerns polynomial exponents, not constants.

\subsubsection{Hessian-based lazy methods}\label{sec:lazycompare}
LazyCRN and CALEN  use an
exact Hessian oracle, so we compare them with RP-QNPE under the convention
that one Hessian counts as $d$ gradients. Under this convention, every lazy
bound carries an additive $d$ for the Hessian computed before the first
step and has no $\varepsilon^{-1/2}$ branch, whereas the RP-QNPE bound
$\min\{\varepsilon^{-1/2},d^{1/3}\varepsilon^{-1/3}\}$ has no additive term.
Consequently, RP-QNPE is strictly smaller than LazyCRN at every accuracy
(LazyCRN reads $d+d^{1/2}\varepsilon^{-1/2}$, which exceeds
$\varepsilon^{-1/2}$ throughout). Compared with CALEN, whose bound is
$\tilde O(d+d^{13/21}\varepsilon^{-2/7})$, RP-QNPE is smaller for
$\varepsilon>d^{-6}$, first through its $\varepsilon^{-1/2}$ branch (which
beats the additive $d$ for $\varepsilon>d^{-2}$ and beats
$d^{13/21}\varepsilon^{-2/7}$ down to $\varepsilon\asymp d^{-26/9}$) and
then through its $d^{1/3}\varepsilon^{-1/3}$ branch on
$d^{-6}<\varepsilon\le d^{-2}$, while CALEN is smaller for
$\varepsilon<d^{-6}$ by virtue of its $\varepsilon^{-2/7}$ exponent. For structured Hessians, instantiating the lazy bounds at $\bar d=\chi$, the finite-difference cost of Lemma~\ref{lem:group}, transfers all of these crossovers with $d$ replaced by $\chi$; this instantiation is hypothetical, since the lazy analyses are for exact Hessians; Figure~\ref{fig:convexregimes}(b) draws it dashed, and Figure~\ref{fig:convexmaps} maps the resulting regimes as functions of the accuracy and of the sparsity.

\subsubsection{Cutting-plane methods}\label{sec:structurecompare}

In the dense case $\chi=d$, every curvature bound in Table~\ref{tab:convexcompare} is at least $d$ on the range $\varepsilon<d^{-2}$ where it improves on AGD's $\varepsilon^{-1/2}$, so none improves on the better of AGD and cutting-plane methods \citep{lee2015faster,jiang2020improved} by more than logarithmic factors. Here we focus on structured scenarios ($\chi<d$) in which our bound improves on both.

\begin{figure}[t]
\centering
\begin{subfigure}[t]{0.49\textwidth}
\centering
\begin{tikzpicture}
\begin{axis}[width=\linewidth, height=0.72\linewidth, grid=major, grid style={gray!30}, axis on top,
  xmin=0, xmax=3.5, ymin=0, ymax=1,
  xtick={0,1,2,3}, xticklabels={$1$,$d$,$d^{2}$,$d^{3}$},
  ytick={0,0.25,0.5,0.75,1}, yticklabels={$1$,$d^{1/4}$,$d^{1/2}$,$d^{3/4}$,$d$},
  xlabel={$1/\varepsilon$}, ylabel={effective dimension $\chi$},
  tick label style={font=\scriptsize}, label style={font=\small},
  legend to name=leg:convexmaps, legend columns=2, legend cell align={left},
  legend style={font=\footnotesize, draw=gray!60}]
\fill[blue!25]   (axis cs:0,0) -- (axis cs:2,1) -- (axis cs:0,1) -- cycle;
\fill[red!40]    (axis cs:0,0) -- (axis cs:2,1) -- (axis cs:2.5714,0.4286) -- cycle;
\fill[red!15]    (axis cs:0,0) -- (axis cs:2.5714,0.4286) -- (axis cs:3,0) -- cycle;
\fill[orange!35] (axis cs:3,0) -- (axis cs:2.5714,0.4286) -- (axis cs:3.5,0) -- cycle;
\fill[green!30]  (axis cs:2,1) -- (axis cs:3.5,1) -- (axis cs:3.5,0) -- (axis cs:2.5714,0.4286) -- cycle;
\addlegendimage{area legend, fill=blue!25, draw=none}
  \addlegendentry{AGD smallest (RP-QNPE ties it)}
\addlegendimage{area legend, fill=green!30, draw=none}
  \addlegendentry{cutting planes smallest}
\addlegendimage{area legend, fill=red!40, draw=none}
  \addlegendentry{RP-QNPE smallest published bound}
\addlegendimage{area legend, fill=red!15, draw=none}
  \addlegendentry{RP-QNPE below AGD and cutting planes, above CALEN at $\bar d=\chi$}
\addlegendimage{area legend, fill=orange!35, draw=none}
  \addlegendentry{CALEN at $\bar d=\chi$ smallest (hypothetical)}
\addlegendimage{black, dotted, thick}
  \addlegendentry{slice $d=\chi^{3}$ shown in Figure~\ref{fig:convexregimes}(b)}
\draw[dashed, blue!50!black] (axis cs:0,0) -- (axis cs:2,1)
  node[pos=0.55, sloped, above, font=\scriptsize] {$\varepsilon=\chi^{-2}$};
\draw[dashed, red!60!black] (axis cs:0,0) -- (axis cs:2.5714,0.4286)
  node[pos=0.6, sloped, below, font=\scriptsize] {$\varepsilon=\chi^{-6}$};
\draw[dashed, green!40!black] (axis cs:2,1) -- (axis cs:3,0)
  node[pos=0.3, sloped, above, font=\scriptsize] {$\varepsilon=\chi d^{-3}$};
\draw[dashed, orange!70!black] (axis cs:2.5714,0.4286) -- (axis cs:3.5,0)
  node[pos=0.55, sloped, above, font=\tiny] {CALEN($\chi$) $=$ CP};
\draw[dotted, thick, black] (axis cs:0,0.3333) -- (axis cs:3.5,0.3333)
  node[pos=0.12, above, font=\scriptsize] {$\chi=d^{1/3}$};
\end{axis}
\end{tikzpicture}
\caption{Axes in powers of $d$; the top edge $\chi=d$ is the dense case.}
\label{fig:convexmap-d}
\end{subfigure}\hfill
\begin{subfigure}[t]{0.49\textwidth}
\centering
\begin{tikzpicture}
\begin{axis}[width=\linewidth, height=0.72\linewidth, grid=major, grid style={gray!30}, axis on top,
  xmin=0, xmax=10, ymin=1, ymax=4,
  xtick={0,2,4,6,8,10}, xticklabels={$1$,$\chi^{2}$,$\chi^{4}$,$\chi^{6}$,$\chi^{8}$,$\chi^{10}$},
  ytick={1,2,3,4}, yticklabels={$\chi$,$\chi^{2}$,$\chi^{3}$,$\chi^{4}$},
  xlabel={$1/\varepsilon$}, ylabel={dimension $d$},
  tick label style={font=\scriptsize}, label style={font=\small}]
\fill[blue!25]   (axis cs:0,1) -- (axis cs:2,1) -- (axis cs:2,4) -- (axis cs:0,4) -- cycle;
\fill[red!40]    (axis cs:2,1) -- (axis cs:6,2.3333) -- (axis cs:6,4) -- (axis cs:2,4) -- cycle;
\fill[red!15]    (axis cs:6,2.3333) -- (axis cs:10,3.6667) -- (axis cs:10,4) -- (axis cs:6,4) -- cycle;
\fill[orange!35] (axis cs:6,2.3333) -- (axis cs:10,3.4762) -- (axis cs:10,3.6667) -- cycle;
\fill[green!30]  (axis cs:2,1) -- (axis cs:10,1) -- (axis cs:10,3.4762) -- (axis cs:6,2.3333) -- cycle;
\draw[dashed, blue!50!black] (axis cs:2,1) -- (axis cs:2,4)
  node[pos=0.5, sloped, above, font=\scriptsize] {$\varepsilon=\chi^{-2}$};
\draw[dashed, red!60!black] (axis cs:6,2.3333) -- (axis cs:6,4)
  node[pos=0.5, sloped, above, font=\scriptsize] {$\varepsilon=\chi^{-6}$};
\draw[dashed, green!40!black] (axis cs:2,1) -- (axis cs:10,3.6667)
  node[pos=0.3, sloped, below, font=\scriptsize] {$\varepsilon=\chi d^{-3}$};
\draw[dashed, orange!70!black] (axis cs:6,2.3333) -- (axis cs:10,3.4762)
  node[pos=0.6, sloped, below, font=\tiny] {CALEN($\chi$) $=$ CP};
\draw[dotted, thick, black] (axis cs:0,3) -- (axis cs:10,3)
  node[pos=0.1, above, font=\scriptsize] {$d=\chi^{3}$};
\end{axis}
\end{tikzpicture}
\caption{Axes in powers of $\chi$; the bottom edge $d=\chi$ is the dense case.}
\label{fig:convexmap-chi}
\end{subfigure}

\medskip
\pgfplotslegendfromname{leg:convexmaps}
\caption{Convex case: the smallest known bound as a function of the accuracy and of the sparsity,
with constants and logarithmic factors suppressed and the lazy-Hessian bounds instantiated at the
finite-difference cost $\bar d=\chi$ of Lemma~\ref{lem:group} (a hypothetical instantiation of
exact-Hessian analyses). The two panels show the same regimes in different coordinates: in (a)
$1/\varepsilon=d^{a}$ and $\chi=d^{c}$, in (b) $1/\varepsilon=\chi^{a}$ and $d=\chi^{\gamma}$. The
dotted line is the slice $d=\chi^{3}$ of Figure~\ref{fig:convexregimes}(b), along which the regimes
are AGD, RP-QNPE, RP-QNPE below AGD and cutting planes, CALEN($\chi$), and cutting planes in that
order. The RP-QNPE region is empty on the dense edge and widens as the Hessian becomes sparser;
the three boundaries meet at $\varepsilon=\chi^{-6}$, $d=\chi^{7/3}$.}
\label{fig:convexmaps}
\end{figure}

\noindent The cutting-plane bound does not depend on the Hessian pattern, whereas
our curvature bound is $\chi^{1/3}\varepsilon^{-1/3}$, which is smaller
than $d$ whenever $\varepsilon>\chi d^{-3}$ (up to logarithmic factors).
Since our curvature branch is active for $\varepsilon<\chi^{-2}$, it
improves on both AGD and cutting planes throughout the range
$\chi d^{-3}<\varepsilon<\chi^{-2}$, which is nonempty whenever $\chi<d$;
for a tridiagonal Hessian ($\chi=3$), this range is
$d^{-3}\lesssim\varepsilon\lesssim1$. These are comparisons with AGD and cutting planes, neither of which uses the pattern. Against lazy-Hessian bounds instantiated at $\bar d=\chi$, the window on which RP-QNPE is the smallest published bound shrinks to $\max\{\chi d^{-3},\chi^{-6}\}<\varepsilon<\chi^{-2}$, which for $\chi=3$ and $d\ge13$ is, up to constants and logarithms, $10^{-3}\lesssim\varepsilon\lesssim10^{-1}$ (Figure~\ref{fig:convexmaps}). We exclude two structures that
reduce the problem dimension for every method alike: a known subspace
containing the range of every $H(x)$, and a disconnected sparsity pattern
(which makes $f$ separable into components of size at most $s_{\max}$, so
$d$ is replaced by $s_{\max}$). This compares proved upper bounds, as we are not aware of localization methods or lower bounds that exploit Hessian sparsity
; the strongly convex counterpart is in Section~\ref{sec:sccompare}.

\section{Extension to the strongly convex case}\label{sec:strongly}
Under strong convexity, the accelerated recursion can use the modulus
$\mu$ directly, and a small objective gap certifies a small distance to
the minimizer. The method of this section uses both facts. It keeps the
rolling curvature model of Section~\ref{sec:algo} throughout the run,  and it replaces the accelerated recursion of
Section~\ref{sec:prelim} by the strongly convex A-HPE recursion
of \citet[Algorithm~1]{marques2022variants}, specialized to
a differentiable objective and zero enlargement error.
Strong convexity permits a larger relative residual;
see also \citet[Section~5]{barre2022approximate}.
Its stepsize is chosen, as in Subroutine~1.1, by a scalar search over
stored quantities. The certificate that drives the search has a branch
that contains no gradient norm, and this is what makes the accumulated
price of the model independent of the length of the run. The method has an accelerated linear guarantee matching AGD
and a global superlinear upper envelope for the same iterates
(Section~\ref{sec:scguarantee}).
Section~\ref{sec:sccompare} compares these guarantees with existing ones
and with the simpler alternative of restarting the convex method, which
is analyzed in Appendix~\ref{app:strongly}.

\begin{assumption}\label{as:sc}
In addition to Assumption~\ref{as:main}, $f$ is $\mu$-strongly convex for a
known $\mu>0$, that is, $\mu I\preceq H(x)$ for every $x$. In particular,
the minimizer $x^\star$ is unique.
\end{assumption}
\noindent Throughout this section,
\[
\kappa=\frac{L_1}{\mu}\ge1,\qquad
\kappa_2=\frac{L_2R}{\mu},\qquad
\mathcal L_\varepsilon=\log_+\frac{L_1R^2}{2\sigma\varepsilon},\qquad
\mathcal S_\varepsilon=1+\log_+\frac{L_1R^2}{\varepsilon},
\]
where $\log_+=\max\{0,\log\}$. The quantity $\kappa_2$ is dimensionless
and measures the variation of the Hessian over the initial ball against
the strong convexity; in normalized units ($L_1=L_2=R=1$) it coincides
with $\kappa$. The method uses $\mu$, as does the restart scheme of
Appendix~\ref{app:strongly}; a known lower bound $0<\hat\mu\le\mu$ may
replace it, and the bounds then hold with $\mu$ replaced by $\hat\mu$.

\subsection{The method}\label{sec:scmethod}
Algorithm~2 retains the three building blocks of RP-QNPE from
Section~\ref{sec:algo}: it maintains the same rolling finite-difference
model $B_k$ and error certificate $\bar E_k$, takes the same regularized
quasi-Newton step $r_k=-(B_k+\lambda_k^{-1}I)^{-1}\nabla f(y_k)$ of
Section~\ref{sec:framework}, and selects $\lambda_k$ by a gradient-free
scalar bisection. Strong convexity enters in two places: it modifies the
A-HPE recursion so that the weights $A_k$ grow geometrically even for a
constant stepsize, and it relaxes the HPE admissibility condition
\eqref{eq:ms} by a factor $\sqrt{1+\mu\lambda_k}$, which allows $\lambda_k$
to grow without bound as the iterates approach $x^\star$ and the model
error shrinks.

\paragraph{The accelerated recursion.}
To incorporate $\mu$-strong convexity into the A-HPE framework of
Section~\ref{sec:prelim}, we follow the strongly convex extension of
\citet[Algorithm~1]{marques2022variants} and \citet[Section~5]{barre2022approximate},
which strengthens the quadratic term of the potential from
$\frac12\|z_k-x^\star\|^2$ to $\frac{1+\mu A_k}{2}\|z_k-x^\star\|^2$. As in Section~\ref{sec:prelim}, the
state consists of a primal iterate $x_k$, an auxiliary iterate $z_k$, and a
weight $A_k$, initialized at $A_0=0$ and $z_0=x_0$. Given a stepsize
$\lambda_k>0$, the new weight $a_k>0$ is chosen as the positive root of
\begin{equation}\label{eq:sc-a}
a_k^2=\lambda_k\big[A_k(1+\mu A_k)+a_k(1+2\mu A_k)\big],
\end{equation}
which reduces to \eqref{eq:aident} when $\mu=0$, and the updated weight and
extrapolation center are
\begin{equation}\label{eq:sc-weights}
A_{k+1}=A_k+a_k,\qquad
\eta_k=\frac{a_k(1+\mu A_k)}{1+\mu A_{k+1}},\qquad
W_k=A_k+\eta_k,\qquad
y_k=\frac{A_kx_k+\eta_kz_k}{W_k}.
\end{equation}
Here $y_k$ is a convex combination of $x_k$ and $z_k$, with
$W_k=A_{k+1}/(1+\mu\lambda_k)$ (Lemma~\ref{lem:sc-weights}). From the
center $y_k$, we take the regularized quasi-Newton step of
Section~\ref{sec:framework} and follow
\begin{equation}\label{eq:sc-update}
r_k=-(B_k+\lambda_k^{-1}I)^{-1}\nabla f(y_k),\qquad
x_{k+1}=y_k+r_k,\qquad
z_{k+1}=z_k+\frac{a_k}{1+\mu A_{k+1}}\big[\mu(x_{k+1}-z_k)-\nabla f(x_{k+1})\big],
\end{equation}
where the additional term $\mu(x_{k+1}-z_k)$ pulls $z_k$ toward $x_{k+1}$
and vanishes when $\mu=0$. Strong convexity also relaxes the relative
proximal condition \eqref{eq:ms}: we call the step \emph{admissible} if
\begin{equation}\label{eq:sc-hpe}
\|r_k+\lambda_k\nabla f(x_{k+1})\|\le\sigma\sqrt{1+\mu\lambda_k}\,\|r_k\|.
\end{equation}
Under \eqref{eq:sc-hpe}, the potential
\begin{equation}\label{eq:sc-potential}
\Psi_k=A_k\big(f(x_k)-f^\star\big)+\frac{1+\mu A_k}{2}\|z_k-x^\star\|^2
\end{equation}
satisfies
$\Psi_{k+1}+\frac{(1-\sigma^2)A_{k+1}}{2\lambda_k}\|r_k\|^2\le\Psi_k$
(Lemma~\ref{lem:sc-contract}); the factor $\sqrt{1+\mu\lambda_k}$ in
\eqref{eq:sc-hpe} comes directly from the $\mu$-strong convexity inequality
between $x_{k+1}$ and $x^\star$. Since $\Psi_0\le R^2/2$, every admissible
run satisfies $f(x_k)-f^\star\le R^2/(2A_k)$, and strong convexity then
gives $\|x_k-x^\star\|\le R/\sqrt{\mu A_k}$.

\paragraph{The curvature model.}
The matrix $B_k$ and the certificate $\bar E_k$ are constructed exactly as
in Block~1 of Section~\ref{sec:rpblock1}: one group of columns is refreshed
per iteration by \eqref{eq:refresh} with $b=L_2R$, and $\bar E_k$ is given
by \eqref{eq:rollingcert} using the actual step displacement
$\delta_j=\|x_{j+1}-x_j\|$ in place of $\|y_j-x_j\|+\|r_j\|$. Since
Lemma~\ref{lem:telescope} uses $\delta_j$ only through
$\|H(x_{j+1})-H(x_j)\|_{\rm op}\le L_2\delta_j$, it applies verbatim and
gives $B_k\in\mathcal Z$ and $\|B_k-H(x_k)\|_{\rm op}\le\bar E_k\le L_1$.
Because strong convexity makes the total path length $\sum_j\delta_j$
finite over an infinite run, the rolling model never needs to be reset.

\paragraph{The certified stepsize.}
By Lemma~\ref{lem:resid} applied at $y_k$, the relaxed HPE condition
\eqref{eq:sc-hpe} holds whenever
\[
\lambda_k\Big(\|B_k-H(y_k)\|_{\rm op}+\frac{L_2}{2}\|r_k\|\Big)
\ \le\ \sigma\sqrt{1+\mu\lambda_k}.
\]
As in Block~2 (Section~\ref{sec:rpblock2}), neither $\|B_k-H(y_k)\|_{\rm op}$
nor $\|r_k\|$ is available before $\lambda_k$ is chosen, so we bound both by
explicit functions of a candidate stepsize $\lambda>0$. Let $a(\lambda)$,
$A'(\lambda)=A_k+a(\lambda)$, $\eta(\lambda)$, $W(\lambda)$ and $y(\lambda)$
be the quantities \eqref{eq:sc-a}--\eqref{eq:sc-weights} formed with
$\lambda_k=\lambda$, and write $\zeta_k=\|z_k-x_k\|$ and
$g_k=\|\nabla f(x_k)\|$. Using $W(\lambda)=A'(\lambda)/(1+\mu\lambda)$ and
$\eta(\lambda)a(\lambda)=\lambda A'(\lambda)(1+\mu A_k)/(1+\mu\lambda)$, the
displacement from the anchor $x_k$ to the trial center $y(\lambda)$ is known
in closed form:
\[
m_k(\lambda):=\|y(\lambda)-x_k\|=\frac{\eta(\lambda)}{W(\lambda)}\,\zeta_k
=\frac{\lambda(1+\mu A_k)}{a(\lambda)}\,\zeta_k,
\]
so $\|B_k-H(y(\lambda))\|_{\rm op}\le\min\{L_1,\bar E_k+L_2m_k(\lambda)\}$.
For the trial step length $\|r(\lambda)\|$, we take the minimum of two
upper bounds: the resolvent bound of Lemma~\ref{lem:scalarbounds}, which
holds unconditionally, and the movement-energy bound
$\|r_k\|\le R\sqrt{\lambda_k/((1-\sigma^2)A_{k+1})}$ implied by the
potential decrease $\frac{(1-\sigma^2)A_{k+1}}{2\lambda_k}\|r_k\|^2\le\Psi_k\le R^2/2$
for any admissible step, giving
\begin{equation}\label{eq:sc-rho}
\rho_k(\lambda)=\min\Big\{
\lambda g_k+m_k(\lambda)\Big(1+\lambda\bar E_k+\frac{L_2\lambda}{2}m_k(\lambda)\Big),\ \
R\sqrt{\frac{\lambda}{(1-\sigma^2)A'(\lambda)}}\Big\}.
\end{equation}
Combining the model-error and step-length bounds yields the certificate and
the scalar test function
\begin{equation}\label{eq:sc-test}
\omega_k(\lambda)=\min\Big\{L_1,\ \bar E_k+L_2\big(m_k(\lambda)+\rho_k(\lambda)\big)\Big\},
\qquad
\phi_k(\lambda)=\frac{\lambda\,\omega_k(\lambda)}{\sqrt{1+\mu\lambda}} .
\end{equation}
The function $\phi_k$ is continuous and strictly increasing
(Lemma~\ref{lem:sc-mono}), 
and along the run $\phi_k(\lambda)\le\sigma$ certifies \eqref{eq:sc-hpe} for the step with stepsize $\lambda$ (Lemma~\ref{lem:sc-admissible}); although the second branch of \eqref{eq:sc-rho} holds a priori only for admissible steps, a continuation argument on $(0,\lambda]$ in Lemma~\ref{lem:sc-admissible} shows that $\phi_k(\lambda)\le\sigma$ simultaneously certifies admissibility and that step-length bound. Two values bracket its root: the
smoothness step $\underline\lambda$, the positive root of
$\lambda L_1=\sigma\sqrt{1+\mu\lambda}$, which satisfies
$\phi_k(\underline\lambda)\le\sigma$ and $\underline\lambda\ge\sigma/L_1$;
and the model step $\bar\lambda_k$, the positive root of
$\lambda\bar E_k=\sigma\sqrt{1+\mu\lambda}$, which satisfies
$\phi_k(\bar\lambda_k)\ge\sigma$. 
Lemma~\ref{lem:sc-mono} also gives
$\phi_k(c\lambda)\le c^4\phi_k(\lambda)$ for $c\ge1$.
Thus reducing the bracket ratio to $2^{1/4}$ ensures
$\sigma/2\le\phi_k(\lambda_k)\le\sigma$, using only
evaluations of stored scalar formulas.

\vspace{2mm}
\par\noindent\begin{minipage}{\linewidth}\begin{framed}
\noindent\textbf{Subroutine 2.1 (Stepsize selection).} Input: $A_k$,
$\zeta_k$, $g_k$, $\bar E_k$. Evaluate $\phi_k$ through
\eqref{eq:sc-rho}--\eqref{eq:sc-test}.
\begin{enumerate}[label=\arabic*.,leftmargin=2em]
\item Set $\lambda^-=\underline\lambda$ and $\lambda^+=\bar\lambda_k$.
\item While $\lambda^+/\lambda^->2^{1/4}$: set
      $\lambda=\sqrt{\lambda^-\lambda^+}$; if $\phi_k(\lambda)\le\sigma$,
      set $\lambda^-=\lambda$, otherwise set $\lambda^+=\lambda$.
\item Return $\lambda_k=\lambda^-$.
\end{enumerate}
\end{framed}\end{minipage}\par\medskip

\par\noindent\begin{minipage}{\linewidth}\begin{framed}
\noindent\textbf{Algorithm 2 (SC-RP-QNPE: RP-QNPE for strongly convex functions).}
Initialize $x_0=z_0$, $A_0=0$, $M_{-1}=0$ and $b=L_2R$, and compute
$\nabla f(x_0)$. For $k=0,1,\ldots$:
\begin{enumerate}[label=\arabic*.,leftmargin=2em]
\item \textbf{Refresh and certify.} If $\nabla f(x_k)=0$, return $x_k$. Refresh the
      columns of group $g_{i_k}$ and form $B_k$ by \eqref{eq:refresh}, and
      form $\bar E_k$ by \eqref{eq:rollingcert} with $\delta_j=\|x_{j+1}-x_j\|$.
\item \textbf{Choose the stepsize} $\lambda_k$ by Subroutine~2.1, and form
      $a_k$, $A_{k+1}$, $\eta_k$, $W_k$ and $y_k$ by
      \eqref{eq:sc-a}--\eqref{eq:sc-weights}.
\item \textbf{Take the model step.} Evaluate $\nabla f(y_k)$, compute
      $r_k$, and set $x_{k+1}=y_k+r_k$ as in \eqref{eq:sc-update}.
\item \textbf{Correct the momentum.} Evaluate $\nabla f(x_{k+1})$ and update
      $z_{k+1}$ by \eqref{eq:sc-update}.
\item \textbf{Check progress.} For target accuracy $\varepsilon$, return
      $x_{k+1}$ once $A_{k+1}\ge R^2/(2\varepsilon)$.
\end{enumerate}
\end{framed}\end{minipage}\par\medskip

Each iteration queries the probe gradient, $\nabla f(y_k)$ and
$\nabla f(x_{k+1})$, which is cached as the next anchor gradient, so it uses
at most three gradients; Subroutine~2.1 uses none. The search is the
analogue of Subroutine~1.1, with two differences. First, the strongly
convex recursion relaxes the acceptance threshold from $\sigma$ to
$\sigma\sqrt{1+\mu\lambda}$, which permits stepsizes that grow without
bound as the model improves. Second, the step-length bound
\eqref{eq:sc-rho} has a gradient-free branch, which is the one the
analysis uses. It is of order $(a(\lambda)/A'(\lambda))R$ while the weights
grow polynomially, and of order
$\sqrt{\mu\lambda/(1+\mu\lambda)}\,R/\sqrt{\mu A_k}$ once they grow
geometrically; along the run, both are summable up to a logarithm in
$\kappa$. The resolvent branch is not needed for the guarantees: it adapts
the step to the observed gradient and can only enlarge the accepted
stepsize.

\subsection{Guarantees}\label{sec:scguarantee}
We now state the convergence and complexity guarantees of Algorithm~2. In
the convex case, the stepsize rule bounded the inverse-stepsize sum
$\sum_{k<N}\lambda_k^{-1}$ by the accumulated certificates and displacements
over each dyadic scale of $A_k$. Under strong convexity, the search guarantee
$\sigma/2\le\phi_k(\lambda_k)\le\sigma$ instead controls the
moment $\sum_{k<N}\sqrt{1+\mu\lambda_k}/\lambda_k$.
The stepsize floor guarantees geometric growth of $A_k$,
and the accelerated recursion gives a total path-length bound
$T_\sharp R$ independent of the horizon. Specifically, define the per-step
accelerated contraction factor $q$, the dimensionless path-length factor
$T_\sharp$, and the resulting moment constants $C_\sharp$ and $B_\sharp$ by
\[
q=1-\sqrt{\frac{\sigma}{\kappa+\sigma}},\qquad
T_\sharp=\Big(1+C_x+\frac{1}{\sqrt{1-\sigma^2}}\Big)
\Big(2+\log_+\frac{1}{\mu\underline\lambda}\Big)+4\sqrt2+\frac{2}{\sqrt{1-\sigma^2}},
\]
\[
C_\sharp=\frac{2}{\sigma}\Big[\chi L_1+2L_2R+(2\chi+1)L_2RT_\sharp\Big],
\qquad
B_\sharp=\frac{C_\sharp}{\mu}
=O_\sigma\big(\chi\kappa+\chi\kappa_2\log(2+\kappa)\big),
\]
where $C_x=1+\sigma/\sqrt{1-\sigma^2}$ is the constant of
Proposition~\ref{prop:hpeestimates}, and
$\log_+\frac{1}{\mu\underline\lambda}\le\log\frac{\kappa}{\sigma}$. The
next theorem shows that the moment sum is uniformly bounded by $C_\sharp$
for all $N\ge1$ and converts this bound into three simultaneous upper
bounds on the objective gap.

\begin{theorem}\label{thm:scmain}
Under Assumptions~\ref{as:main} and~\ref{as:sc}, Algorithm~2 with
Subroutine~2.1 is well defined, every step satisfies \eqref{eq:sc-hpe},
$\lambda_k\ge\underline\lambda\ge\sigma/L_1$, and
\begin{equation}\label{eq:sc-moment}
\sum_{k<N}\frac{\sqrt{1+\mu\lambda_k}}{\lambda_k}\ \le\ C_\sharp
\qquad\text{for every }N\ge1 .
\end{equation}
Consequently, after every completed iteration $N\ge1$, with $n=N-1$,
\begin{equation}\label{eq:sc-rate}
f(x_N)-f^\star\le\frac{L_1R^2}{2\sigma}\,
\min\Big\{q^{\,n},\ \Big(1+\frac{n^2}{B_\sharp^2}\Big)^{-n},\
\exp\Big(-\frac{1}{\sqrt2}\min\Big\{\frac n2,\ \Big(\frac n2\Big)^{3/2}B_\sharp^{-1/2}\Big\}\Big)\Big\}.
\end{equation}
Every iteration uses at most three gradients, and Subroutine~2.1 uses
$O(1+\log_+\log(\bar\lambda_k/\underline\lambda))$ evaluations of $\phi_k$ and
no gradient.
\end{theorem}

\noindent The proof is in Appendix~\ref{app:sc}. The three branches of
\eqref{eq:sc-rate} come from three lower bounds on the same weight $A_N$.
The first uses only the floor $\lambda_k\ge\underline\lambda$: since
$A_k/A_{k+1}\le1-\sqrt{\mu\lambda_k/(1+\mu\lambda_k)}$, the weights grow
geometrically, and $q^n=\exp(-\Theta_\sigma(n/\sqrt\kappa))$ is the rate of
AGD. The other two use the moment bound \eqref{eq:sc-moment}, whose
right-hand side does not depend on $N$; on an infinite run it forces
$\lambda_k\to\infty$. The second branch is superlinear: for fixed
parameters its logarithm is $-2n\log n+O(n)$. For $n\lesssim B_\sharp$, the third branch has an exponent
of order $-n^{3/2}B_\sharp^{-1/2}$.
For sufficiently large $\kappa$, it improves on the linear
branch once $n\gtrsim B_\sharp/\kappa$.

The proof bounds the accumulated accelerated gain from below
using the moment estimate, treating small and large
$\mu\lambda_k$ separately. The theorem bounds the objective error from
above; it does not assert Q-superlinear convergence of the errors
themselves.

The weight $A_k$ gives the stopping rule of Algorithm~2, as in the convex
case: $f(x_k)-f^\star\le R^2/(2A_k)$, so $A_k\ge R^2/(2\varepsilon)$ certifies
the requested accuracy. If $\varepsilon\ge L_1R^2/2$, the initial point is
already $\varepsilon$-accurate.

\begin{corollary}\label{cor:scqueries}
Under the assumptions of Theorem~\ref{thm:scmain}, let
$0<\varepsilon<L_1R^2/2$ and
\[
n_\varepsilon=\Big\lceil\min\Big\{
\frac{\mathcal L_\varepsilon}{\log(1/q)},\
\max\big\{2\sqrt2\,\mathcal L_\varepsilon,\ 2^{4/3}\mathcal L_\varepsilon^{2/3}B_\sharp^{1/3}\big\},\
\frac{\mathcal L_\varepsilon}{\log\big(1+(\mathcal L_\varepsilon/B_\sharp)^{2/3}\big)}
\Big\}\Big\rceil .
\]
Algorithm~2 returns an $\varepsilon$-accurate point within $1+n_\varepsilon$
iterations and at most $4+3n_\varepsilon$ gradient queries. In particular,
\begin{equation}\label{eq:sc-complexity}
G_\varepsilon=O_\sigma\Big(1+\min\Big\{\sqrt\kappa\,\mathcal L_\varepsilon,\
\mathcal L_\varepsilon+\big(B_\sharp\mathcal L_\varepsilon^2\big)^{1/3},\
\frac{\mathcal L_\varepsilon}{\log\big(1+(\mathcal L_\varepsilon/B_\sharp)^{2/3}\big)}\Big\}\Big).
\end{equation}
\end{corollary}
\noindent The proof is in Appendix~\ref{proof:scqueries}. The first branch
is the bound of AGD. The middle branch, of order
$(B_\sharp\mathcal L_\varepsilon^2)^{1/3}$, governs moderate accuracies, and
the last branch gives $O(\log(1/\varepsilon)/\log\log(1/\varepsilon))$ as
$\varepsilon\downarrow0$ for fixed problem parameters.

\begin{remark}[An explicit stepsize]\label{rem:scexplicit}
Let $D_0=R$, $D_k=R/\sqrt{\mu A_k}$ for $k\ge1$, and
$c_\sigma=2+1/\sqrt{1-\sigma^2}$. Along the run, $m_k(\lambda)\le2D_k$ and
$\rho_k(\lambda)\le D_k/\sqrt{1-\sigma^2}$ for every $\lambda$, so
$\omega_k(\lambda)\le\omega_k^{\rm ex}:=\min\{L_1,\bar E_k+c_\sigma L_2D_k\}$.
The certificate $\omega_k^{\rm ex}$ does not depend on $\lambda$, and the
stepsize can be taken as the explicit root of
$\lambda\,\omega_k^{\rm ex}=\sigma\sqrt{1+\mu\lambda}$, with no search. This
rule is admissible, and Theorem~\ref{thm:scmain} holds for it with
$B_\sharp$ replaced by
$O_\sigma(\chi\kappa+\chi\kappa_2\log(2+\kappa)+\kappa_2\sqrt\kappa\log(2+\kappa))$
(Appendix~\ref{proof:scexplicit}). The additional term comes from the bound
$\sum_kD_k=O_\sigma(R\sqrt\kappa\log(2+\kappa))$.
Quadratic weight growth gives a harmonic bound on the radii,
while geometric growth controls the tail. Subroutine~2.1 replaces these radii by the actual
displacement $m_k(\lambda)$ and by a step bound that is small whenever the
new weight is small relative to $A_k$.
\end{remark}


\subsection{Comparison with existing guarantees}\label{sec:sccompare}
Table~\ref{tab:sccompare} collects the guarantees; Figure~\ref{fig:scaccuracy} draws the bounds against the accuracy at $\kappa=d$, and Figure~\ref{fig:scmaps} maps the smallest known bound over condition numbers and accuracies. To compare polynomial
factors we normalize $L_1=L_2=R=1$, so that $\kappa_2=\kappa$,
$\mathcal L_\varepsilon\asymp\mathcal S_\varepsilon\asymp\log(1/\varepsilon)$
and $B_\sharp\asymp\chi\kappa\log(2+\kappa)$; we write $\mathcal L$ for
$\mathcal L_\varepsilon$ and ignore constants.

\begin{table}[t]
\centering\small
\begin{tabular}{@{}ll@{}}
\toprule
Method & Gradient complexity\\
\midrule
AGD \citep{nesterov2018lectures} & $O\big(\sqrt{\kappa}\,\mathcal S_\varepsilon\big)$\\[3pt]
QNPE \citep{jiang2023online} & $O\!\left(\min\left\{\kappa\,\mathcal S_\varepsilon,\;
\dfrac{\mathcal S_\varepsilon}{\log\!\big(1+(\mathcal S_\varepsilon/(d\kappa^2+\kappa\kappa_2^2))^{1/3}\big)}\right\}\right)$\\[4pt]
Cutting planes \citep{jiang2020improved} & {$\tilde O(d\,\mathcal S_\varepsilon)$}\\
\midrule
{Restarted CALEN \citep{chen2025faster}} & {$O\big((\bar d\,\mathcal S_\varepsilon+\bar d^{13/21}\kappa_2^{2/7})\log\bar d\big)$}\\
\midrule
{Restarted Algorithm~1 (Theorem~\ref{thm:screstart})} & {$O\big(\mathcal S_\varepsilon+\min\{\sqrt\kappa\,\mathcal S_\varepsilon,\ (\chi\kappa)^{1/3}\mathcal S_\varepsilon+(\chi\kappa_2)^{1/3}\}\big)$}\\[3pt]
SC-RP-QNPE (Corollary~\ref{cor:scqueries}) &
$O\!\left(\min\left\{\sqrt\kappa\,\mathcal S_\varepsilon,\;
\mathcal S_\varepsilon+(B_\sharp\mathcal S_\varepsilon^2)^{1/3},\;
\dfrac{\mathcal S_\varepsilon}{\log\big(1+(\mathcal S_\varepsilon/B_\sharp)^{2/3}\big)}\right\}\right)$\\
\bottomrule
\end{tabular}
\caption{\small Strongly convex case, with $\kappa=L_1/\mu$, $\kappa_2=L_2R/\mu$,
$\mathcal S_\varepsilon=1+\log_+(L_1R^2/\varepsilon)$ and
$B_\sharp=O(\chi\kappa+\chi\kappa_2\log(2+\kappa))$; the constants of the SC-RP-QNPE row depend
only on $\sigma$, and $\chi\le d$ is the effective dimension of Definition~\ref{def:chi}
($\chi=d$ for a dense Hessian). {Restarted CALEN counts one Hessian as $\bar d$
gradients: $\bar d=d$ as published, $\bar d=\chi$ (Lemma~\ref{lem:group}) as a hypothetical
instantiation.}}
\label{tab:sccompare}
\end{table}

\begin{figure}[t]
\centering
\begin{subfigure}[t]{0.49\textwidth}
\centering
\begin{tikzpicture}
\begin{axis}[width=\linewidth, height=0.85\linewidth, grid=major, grid style={gray!20}, axis on top,
  xmin=0, xmax=5, ymin=0, ymax=6,
  xtick={0,...,5}, xticklabels={$1$,$d$,$d^{2}$,$d^{3}$,$d^{4}$,$d^{5}$},
  ytick={0,...,6}, yticklabels={$1$,$d$,$d^{2}$,$d^{3}$,$d^{4}$,$d^{5}$,$d^{6}$},
  xlabel={accuracy $\mathcal L=\log(1/\varepsilon)$}, ylabel={gradient queries},
  tick label style={font=\scriptsize}, label style={font=\small},
  every axis plot/.append style={line width=1.1pt},
  legend to name=leg:scaccuracy, legend columns=2, legend cell align={left},
  legend style={font=\scriptsize, draw=gray!60,
                /tikz/every even column/.append style={column sep=0.6cm}}]
\fill[blue!8]  (axis cs:0,0)   rectangle (axis cs:0.5,6);
\fill[red!14]  (axis cs:0.5,0) rectangle (axis cs:3,6);
\fill[red!6]   (axis cs:3,0)   rectangle (axis cs:5,6);
\addplot[blue] coordinates {(0,0.5) (5,5.5)};
  \addlegendentry{AGD: $\sqrt{\kappa}\,\mathcal L$}
\addplot[green!60!black] coordinates {(0,1) (5,6)};
  \addlegendentry{cutting planes: $\tilde O(d\,\mathcal L)$}
\addplot[violet] coordinates {(0,1) (3,3) (5,5)};
  \addlegendentry{QNPE: $\min\{\kappa\mathcal L,\ \mathcal L/\log(1+(\mathcal L/X)^{1/3})\}$, $X=d\kappa^{2}+\kappa^{3}$}
\addplot[orange, dashed] coordinates {(0,1) (5,6)};
  \addlegendentry{restarted CALEN: $\tilde O(\bar d\,\mathcal L)$}
\addplot[red, very thick] coordinates {(0,0.5) (0.5,1) (2,2) (5,5)};
  \addlegendentry{SC-RP-QNPE: $\min\{\sqrt{\kappa}\mathcal L,\ \mathcal L+(B_\sharp\mathcal L^{2})^{1/3},\ \mathcal L/\log(1+(\mathcal L/B_\sharp)^{2/3})\}$, $B_\sharp=\chi\kappa$}
\addplot[brown] coordinates {(0,0.5) (5,5.5)};
\addlegendentry{restarted RP-QNPE:$\min\{\sqrt\kappa,(\chi\kappa)^{1/3}\}\mathcal L$}
\draw[dashed, gray!70] (axis cs:0.5,0) -- (axis cs:0.5,6) node[pos=1, below right, font=\scriptsize, text=gray!50!black] {$\mathcal L=d/\sqrt{\kappa}$};
\draw[dashed, gray!70] (axis cs:2,0) -- (axis cs:2,6) node[pos=1, below right, font=\scriptsize, text=gray!50!black] {$\mathcal L=B_\sharp$};
\draw[dashed, gray!70] (axis cs:3,0) -- (axis cs:3,6) node[pos=1, below right, font=\scriptsize, text=gray!50!black] {$\mathcal L=X$};
\node[font=\scriptsize, inner sep=1pt, rotate=90] at (axis cs:0.25,1.6) {ties AGD};
\node[font=\scriptsize, align=center, inner sep=1pt] at (axis cs:1.75,0.45) {strictly smallest};
\node[font=\scriptsize, align=center, inner sep=1pt] at (axis cs:4,0.45) {same polynomial order as QNPE\\normalized expression ratio $\to2$};
\end{axis}
\end{tikzpicture}
\caption{Dense: $\chi=\bar d=d$, at $\kappa=d$.}
\label{fig:scaccuracy-dense}
\end{subfigure}\hfill
\begin{subfigure}[t]{0.49\textwidth}
\centering
\begin{tikzpicture}
\begin{axis}[width=\linewidth, height=0.85\linewidth, grid=major, grid style={gray!20}, axis on top,
  xmin=0, xmax=5, ymin=0, ymax=6,
  xtick={0,...,5}, xticklabels={$1$,$\chi$,$\chi^{2}$,$\chi^{3}$,$\chi^{4}$,$\chi^{5}$},
  ytick={0,...,6}, yticklabels={$1$,$\chi$,$\chi^{2}$,$\chi^{3}$,$\chi^{4}$,$\chi^{5}$,$\chi^{6}$},
  xlabel={accuracy $\mathcal L=\log(1/\varepsilon)$}, ylabel={gradient queries},
  tick label style={font=\scriptsize}, label style={font=\small},
  every axis plot/.append style={line width=1.1pt}]
\fill[red!6]  (axis cs:0,0) rectangle (axis cs:1,6);
\fill[red!14] (axis cs:1,0) rectangle (axis cs:5,6);
\addplot[blue] coordinates {(0,1.5) (4.5,6)};
\addplot[green!60!black] coordinates {(0,3) (3,6)};
\addplot[violet] coordinates {(0,3) (4.5,6)};
\addplot[orange, dashed] coordinates {(0,1) (5,6)};
\addplot[red, very thick] coordinates {(0,1.3333) (4,4) (5,5)};
\addplot[brown] coordinates {(0,1.3333) (4.6667,6)};
\draw[dashed, gray!70] (axis cs:1,0) -- (axis cs:1,6) node[pos=1, below right, font=\scriptsize, text=gray!50!black] {$\mathcal L=\kappa/\chi^{2}$};
\draw[dashed, gray!70] (axis cs:4,0) -- (axis cs:4,6) node[pos=1, below right, font=\scriptsize, text=gray!50!black] {$\mathcal L=B_\sharp$};
\node[font=\scriptsize, inner sep=1pt, rotate=90] at (axis cs:0.5,3) {above CALEN's stage term only};
\node[font=\scriptsize, align=center, inner sep=1pt] at (axis cs:3,0.45) {strictly smallest};
\end{axis}
\end{tikzpicture}
\caption{Sparse: $d=\chi^{3}$, $\bar d=\chi$, at $\kappa=d$.}
\label{fig:scaccuracy-sparse}
\end{subfigure}

\medskip
\pgfplotslegendfromname{leg:scaccuracy}
\caption{Strongly convex case: gradient complexity against the accuracy
$\mathcal L=\log(1/\varepsilon)$ at $\kappa=d$, in powers of $d$ (a) and of $\chi$ (b);
constants and logarithms suppressed (the bounds are those of Table~\ref{tab:sccompare}), so a branch $\mathcal L/\log(1+(\mathcal L/Y)^{\theta})$ is
drawn as $\mathcal L^{1-\theta}Y^{\theta}$ for $\mathcal L<Y$ and as $\mathcal L$ for
$\mathcal L>Y$. Dashed: restarted CALEN at $\bar d$ gradients per Hessian, up to a transient
($\bar d=d$ in (a), where it coincides with cutting planes; hypothetically $\bar d=\chi$ in (b)).
SC-RP-QNPE shows its three phases, with the superlinear one starting at
$\mathcal L=B_\sharp\asymp\chi\kappa$, whereas QNPE's starts only at $\mathcal L=X$; beyond
$\mathcal L=X$ both bounds are $\Theta(\mathcal L/\log\mathcal L)$, the ratio of the displayed normalized expressions approaches $2$.}
\label{fig:scaccuracy}
\end{figure}
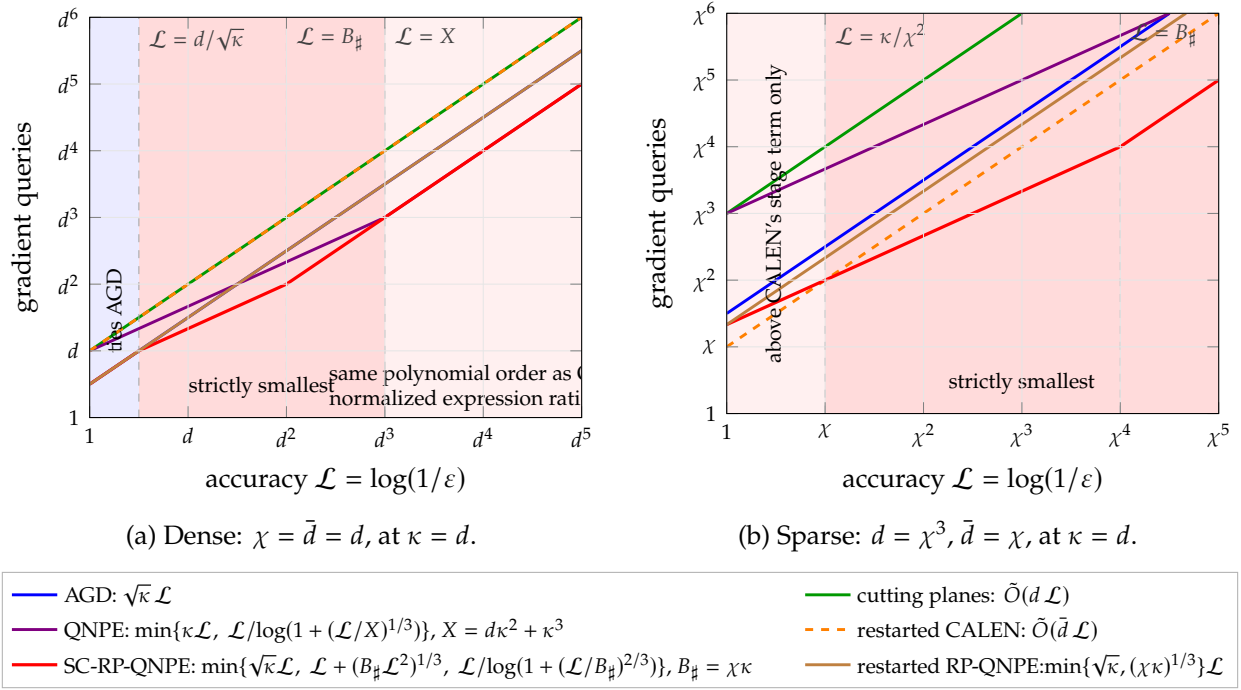

\subsubsection{AGD and QNPE}\label{sec:scagd}
Compared with AGD, whose complexity is $O(\sqrt\kappa\,\mathcal L)$
\citep{nesterov2018lectures}, SC-RP-QNPE is never slower because the first
branch of \eqref{eq:sc-complexity} matches AGD at every accuracy. The
curvature branches improve on AGD once the iteration count exceeds
$O(B_\sharp/\kappa)=O(\chi\log\kappa)$: in terms of accuracy, the middle
branch $\mathcal L+(B_\sharp\mathcal L^2)^{1/3}$ is smaller than
$\sqrt\kappa\,\mathcal L$ as soon as
$\mathcal L\gtrsim\chi\log\kappa/\sqrt\kappa$, and when
$\kappa\gtrsim(\chi\log\kappa)^2$ this threshold is at most a constant, so
SC-RP-QNPE  improves on AGD at every nontrivial accuracy.

Compared with QNPE \citep{jiang2023online}, SC-RP-QNPE improves the guaranteed
bound in every regime. QNPE's upper bound has an initial linear rate
$(1-\Theta(\kappa^{-1}))^k$ matching gradient descent rather than AGD, and
transitions after $O(d+\kappa_2^2/\kappa)$ iterations to a superlinear
envelope $(1+c\sqrt{k/X})^{-k}$ with $X=O(d\kappa^2+\kappa\kappa_2^2)$ (so
$X\asymp d\kappa^2+\kappa^3$ in normalized units). Even in the dense case
$\chi=d$, SC-RP-QNPE replaces the unaccelerated linear phase $\kappa\mathcal L$
by the accelerated bound $\sqrt\kappa\,\mathcal L$, and at moderate
accuracies ($\mathcal L\le B_\sharp$) its curvature
branch improves from $(X\mathcal L^2)^{1/3}$ to
$\mathcal L+(B_\sharp\mathcal L^2)^{1/3}$, with
$B_\sharp\asymp d\kappa\log\kappa\ll X\asymp d\kappa^2+\kappa^3$. At high
accuracies, the eventual superlinear envelope improves from
$(1+c\sqrt{k/X})^{-k}$ to $(1+k^2/B_\sharp^2)^{-k}$ in two ways: the base becomes bounded away from one at
$k\asymp B_\sharp\asymp d\kappa\log\kappa$, rather than
$k\asymp X\asymp d\kappa^2+\kappa^3$, and beyond these
scales the respective bases decay as $k^{-2}$ and $k^{-1/2}$
for fixed problem parameters
(because for $\mu\lambda_k\gg1$ the normalized summand of
\eqref{eq:sc-moment} satisfies
$\mu^{-1}\lambda_k^{-1}\sqrt{1+\mu\lambda_k}\asymp(\mu\lambda_k)^{-1/2}$,
controlling the half-moment of the large-step contraction factor
$(\mu\lambda_k)^{-1}$, whereas QNPE's regret analysis bounds the
inverse-stepsize second moment $\sum_k(\mu\lambda_k)^{-2}$).

\subsubsection{Restarted methods}\label{sec:screstartcompare}
CALEN is designed for convex functions, and its strongly convex guarantee
\citep[Theorem~I.1]{chen2025faster} is obtained by restarting it whenever
the distance to $x^\star$ has provably halved: with one Hessian counted as
$\bar d$ gradients, restarted CALEN needs
$O((\bar d\,\mathcal S_\varepsilon+\bar d^{13/21}\kappa_2^{2/7})\log\bar d)$
equivalent gradient queries. Each stage pays the additive $\bar d$ of the
Hessian computed at its start, so in the dense case the bound is
$\tilde O(d\,\mathcal L)$ plus a transient. It is smaller than
Corollary~\ref{cor:scqueries} only when
$\kappa\gtrsim d^2\max\{1,\mathcal L/\log\kappa\}$ (up to the transient),
and SC-RP-QNPE is smaller otherwise.

Similarly, restarting our own convex method, Algorithm~1, in the same way
gives
$O(\mathcal S_\varepsilon+\min\{\sqrt\kappa\,\mathcal S_\varepsilon,(\chi\kappa)^{1/3}\mathcal S_\varepsilon+(\chi\kappa_2)^{1/3}\})$
gradient queries (Theorem~\ref{thm:screstart} in
Appendix~\ref{app:strongly}). Restarting pays the
inverse-stepsize moment at every stage (giving $(\chi\kappa)^{1/3}\mathcal L$
in normalized units), whereas Algorithm~2 bounds the strongly convex moment
\eqref{eq:sc-moment} once over the whole run (giving
$(B_\sharp\mathcal L^2)^{1/3}\asymp(\chi\kappa\log\kappa)^{1/3}\mathcal L^{2/3}$
before its superlinear phase): consequently, Corollary~\ref{cor:scqueries}
improves the accuracy dependence and dominates restarting as soon as
$\mathcal L\gtrsim\log\kappa$, whereas at coarser accuracies
($\mathcal L\lesssim\log\kappa$) the direct bound is at most
a factor $(\log\kappa)^{1/3}$ worse than the restarted bound.

\subsubsection{Cutting plane methods}\label{sec:scstructure}
Cutting-plane methods need $\tilde O(d\,\mathcal S_\varepsilon)$ gradient
queries irrespective of $\kappa$. Our bound is below this whenever its AGD
branch is, $\kappa\lesssim d^2$, or its middle branch is,
$(\chi\kappa\log\kappa)^{1/3}\mathcal L^{2/3}\lesssim d\,\mathcal L$, that
is, $\kappa\lesssim d^3\mathcal L/(\chi\log\kappa)$. In the dense case the
two conditions combine to $\kappa\lesssim d^2\max\{1,\mathcal L/\log\kappa\}$,
and cutting planes are smaller for larger $\kappa$. For $\chi<d$ the second
condition widens by the factor $d/\chi$, because the cutting-plane bound
does not depend on the Hessian pattern while ours depends on the dimension
only through $\chi$; in particular, up to logarithmic factors, on the range
$\chi^2\lesssim\kappa\lesssim d^3/\chi$ our bound is below both the AGD and
the cutting-plane bounds.

\begin{figure}[t]
\centering
\begin{subfigure}[t]{0.49\textwidth}
\centering
\begin{tikzpicture}
\begin{axis}[width=\linewidth, height=0.75\linewidth, grid=major, grid style={gray!30}, axis on top,
  xmin=0, xmax=4, ymin=0, ymax=3,
  xtick={0,...,4}, xticklabels={$1$,$d$,$d^{2}$,$d^{3}$,$d^{4}$},
  ytick={0,...,3}, yticklabels={$1$,$d$,$d^{2}$,$d^{3}$},
  xlabel={condition number $\kappa$}, ylabel={$\mathcal L=\log(1/\varepsilon)$},
  tick label style={font=\scriptsize}, label style={font=\small},
  legend to name=leg:scmaps, legend columns=2, legend cell align={left},
  legend style={font=\scriptsize, draw=gray!60,
                /tikz/every even column/.append style={column sep=0.6cm}}]
\fill[blue!25]  (axis cs:0,0) -- (axis cs:2,0) -- (axis cs:0,1) -- cycle;
\fill[green!30] (axis cs:2,0) -- (axis cs:4,0) -- (axis cs:4,2) -- cycle;
\fill[red!40]   (axis cs:0,1) -- (axis cs:2,0) -- (axis cs:4,2) -- (axis cs:4,3) -- (axis cs:0,3) -- cycle;
\addlegendimage{area legend, fill=blue!25, draw=none}
  \addlegendentry{AGD smallest ($=$ first branch of SC-RP-QNPE)}
\addlegendimage{area legend, fill=green!30, draw=none}
  \addlegendentry{cutting planes smallest}
\addlegendimage{area legend, fill=red!40, draw=none}
  \addlegendentry{SC-RP-QNPE strictly smallest}
\addlegendimage{area legend, fill=red!15, draw=none}
  \addlegendentry{below AGD and cutting planes only (not below restarted CALEN at $\bar d=\chi$)}
\draw[dashed, blue!50!black] (axis cs:0,1) -- (axis cs:2,0)
  node[midway, sloped, above, font=\scriptsize] {$\mathcal L=d/\sqrt{\kappa}$};
\draw[dashed, green!40!black] (axis cs:2,0) -- (axis cs:4,2)
  node[midway, sloped, above, font=\scriptsize] {$\mathcal L=\kappa/d^{2}$};
\draw[dashed, gray!60!black] (axis cs:0,1) -- (axis cs:2,3)
  node[midway, sloped, above, font=\scriptsize] {$\mathcal L=B_\sharp=d\kappa$};
\draw[dotted, black] (axis cs:2,0) -- (axis cs:2,3) node[pos=0.08, right, font=\scriptsize] {$\kappa=d^{2}$};
\end{axis}
\end{tikzpicture}
\caption{Dense Hessian: $\chi=\bar d=d$.}
\label{fig:scmap-dense}
\end{subfigure}\hfill
\begin{subfigure}[t]{0.49\textwidth}
\centering
\begin{tikzpicture}
\begin{axis}[width=\linewidth, height=0.75\linewidth, grid=major, grid style={gray!30}, axis on top,
  xmin=0, xmax=12, ymin=0, ymax=4,
  xtick={0,2,...,12}, xticklabels={$1$,$\chi^{2}$,$\chi^{4}$,$\chi^{6}$,$\chi^{8}$,$\chi^{10}$,$\chi^{12}$},
  ytick={0,...,4}, yticklabels={$1$,$\chi$,$\chi^{2}$,$\chi^{3}$,$\chi^{4}$},
  xlabel={condition number $\kappa$}, ylabel={$\mathcal L=\log(1/\varepsilon)$},
  tick label style={font=\scriptsize}, label style={font=\small}]
\fill[blue!25]  (axis cs:0,0) -- (axis cs:2,0) -- (axis cs:0,1) -- cycle;
\fill[green!30] (axis cs:8,0) -- (axis cs:12,0) -- (axis cs:12,4) -- cycle;
\fill[red!15]   (axis cs:2,0) -- (axis cs:8,0) -- (axis cs:12,4) -- (axis cs:6,4) -- cycle;
\fill[red!40]   (axis cs:0,1) -- (axis cs:2,0) -- (axis cs:6,4) -- (axis cs:0,4) -- cycle;
\draw[dashed, blue!50!black] (axis cs:0,1) -- (axis cs:2,0)
  node[midway, sloped, above, font=\scriptsize] {$\mathcal L=\chi/\sqrt{\kappa}$};
\draw[dashed, red!60!black] (axis cs:2,0) -- (axis cs:6,4)
  node[midway, sloped, above, font=\scriptsize] {$\mathcal L=\kappa/\chi^{2}$};
\draw[dashed, green!40!black] (axis cs:8,0) -- (axis cs:12,4)
  node[midway, sloped, above, font=\scriptsize] {$\mathcal L=\chi\kappa/d^{3}$};
\draw[dashed, gray!60!black] (axis cs:0,1) -- (axis cs:3,4)
  node[midway, sloped, above, font=\scriptsize] {$\mathcal L=B_\sharp=\chi\kappa$};
\draw[dotted, black] (axis cs:6,0) -- (axis cs:6,4) node[pos=0.08, right, font=\scriptsize] {$\kappa=d^{2}$};
\end{axis}
\end{tikzpicture}
\caption{Sparse Hessian with $d=\chi^{3}$; restarted CALEN at $\bar d=\chi$.}
\label{fig:scmap-sparse}
\end{subfigure}

\medskip
\pgfplotslegendfromname{leg:scmaps}
\caption{Strongly convex case: the smallest known bound as a function of $\kappa$ and of the
accuracy $\mathcal L=\log(1/\varepsilon)$, in powers of $d$ (a) and of $\chi$ (b); constants and
logarithms suppressed. The dashed lines are the crossovers of SC-RP-QNPE with AGD
($\mathcal L=\chi/\sqrt\kappa$), with cutting planes ($\mathcal L=\chi\kappa/d^{3}$) and, in (b),
with restarted CALEN instantiated at $\bar d=\chi$ ($\mathcal L=\kappa/\chi^{2}$, a hypothetical
instantiation of an exact-Hessian analysis); above $\mathcal L=B_\sharp$ its superlinear branch is
active. QNPE is never the smallest bound and is omitted. At $\mathcal L=O(1)$ the red region of
(a) shrinks to the point $\kappa\asymp d^{2}$, whereas in (b) it covers $\chi^{2}\lesssim\kappa\lesssim d^{3}/\chi$.}
\label{fig:scmaps}
\end{figure}

Table~\ref{tab:sccompare} and Figure~\ref{fig:scmaps} collect these comparisons.
In the moderate-accuracy regime, SC-RP-QNPE improves on both
AGD and cutting planes when
$\mathcal L\gtrsim\max\{\chi\log\kappa/\sqrt\kappa,\
\chi\kappa\log\kappa/d^3\}$.
Replacing the second term by $\kappa\log\kappa/\chi^2$
also ensures improvement over the stage term of restarted
CALEN instantiated at $\bar d=\chi$, and hence over its
full bound; this condition need not be necessary because
CALEN also carries a transient.
In the dense case, the comparison with AGD and cutting planes
gives the band
$(d\log\kappa/\mathcal L)^2\lesssim\kappa
\lesssim d^2\mathcal L/\log\kappa$,
which is nonempty once $\mathcal L\gtrsim\log\kappa$.
Within this band and the moderate-accuracy regime, the gain
over the better of AGD and cutting planes is at most a factor
$(\mathcal L/\log\kappa)^{1/3}$.
For $\chi<d$, the band widens to
$(\chi\log\kappa/\mathcal L)^2\lesssim\kappa
\lesssim d^3\mathcal L/(\chi\log\kappa)$,
which for $\mathcal L\asymp\log\kappa$ gives the range
$\chi^2\lesssim\kappa\lesssim d^3/\chi$ above.

\section{Extension to nonconvex optimization}\label{sec:nonconvex}
The curvature model of Block~1 does not use convexity. Lemma~\ref{lem:group}
and the exchange of summation in the proof of Lemma~\ref{lem:telescope}
use only the pattern $\mathcal S$ and the Lipschitz continuity of the
Hessian, and together they bound the error of a rolling finite-difference
model at the current point by the path traveled over the current window and the finite-difference error.
What uses convexity is everything downstream: the accelerated recursion,
its potential, and the path-length estimate that makes the summed
certificate logarithmic. In this section we replace that machinery by the
online-to-nonconvex conversion of \citet{cutkosky2023optimal}, in the extrapolated-gradient form
of \citet{patitucci2025double,jiang2025improved}, and ask what the rolling model
contributes there. The answer is a gradient complexity of $O(\chi^{1/3}\varepsilon^{-3/2}+\chi)$ {\color{ReviewBlue}(the additive $\chi$ is a warm-up, dominated by the first term on the range where the model is used; Remark~\ref{rem:ncadditive})} for finding an $\varepsilon$-stationary point of a function with Lipschitz gradient and Lipschitz Hessian and, for
the same conversion and learner with the model switched off, the
first-order rate $O(\varepsilon^{-7/4})$; together,
$O(\min\{\chi^{1/3}\varepsilon^{-3/2},\varepsilon^{-7/4}\})$.

The conversion has a natural slot for curvature information. It drives an
online learner whose linear losses are gradients of $f$ along the
trajectory, and an \emph{optimistic} learner may use a prediction, or
hint, of its next loss. A quasi-Newton hint predicts the next gradient by
linearizing at the current point with a model Hessian \citep{jiang2025improved}. The question is
which property of the model the regret bound rewards, and the answer
differs from the convex case: not a certificate whose sum over the run is
small, but a certificate that holds at \emph{every} step with a small
constant. The per-step inequality of an optimistic learner contains a
negative term, proportional to the squared change of the increment, that
is otherwise discarded; a pointwise bound on the prediction error in terms
of that change lets the negative term absorb the error entirely, and the
only trace of the model that remains is an upper limit on the learning
rate. The rolling model supplies exactly such a bound, with a constant
proportional to $\chi L_2D$, where $D$ is the radius of the learner's
increments. Sections~\ref{sec:ncconversion} and~\ref{sec:nclearner} set up
the conversion and the learner and isolate the design target,
Section~\ref{sec:ncalgorithm} places the rolling model in the hint and
certifies it, Section~\ref{sec:ncguarantee} states the guarantees,
Section~\ref{sec:ncimplementation} discusses the arithmetic, and
Section~\ref{sec:nccompare} compares the guarantees with existing bounds.

\subsection{Setting}\label{sec:ncsetting}
\begin{assumption}\label{as:nonconvex}
The function $f:\mathbb R^d\to\mathbb R$ is twice continuously
differentiable and bounded below, with known positive constants $L_1,L_2$
such that
\[
\|H(x)\|_{\rm op}\le L_1,\qquad
\|H(x)-H(y)\|_{\rm op}\le L_2\|x-y\|\qquad\text{for all }x,y\in\mathbb R^d .
\]
The oracle returns exact gradients at arbitrary query points. A pattern
$\mathcal S$ with $\operatorname{supp}H(x)\subseteq\mathcal S$ for every
$x$, and a partition $g_1,\ldots,g_\chi$ of $[d]$ that is structurally
orthogonal for $\mathcal S$, are supplied as in Assumption~\ref{as:main};
when no structure is known, $\chi=d$.
\end{assumption}
\noindent Write $f^\star=\inf_xf(x)$ and $\Phi_0=f(x_0)-f^\star$ for the
initial gap. The goal is a point $\bar w$ with
$\|\nabla f(\bar w)\|\le\varepsilon$, and the resource counted is the
number of gradient queries. Neither $R$ nor $x^\star$ plays a role: a
minimizer need not exist, and progress is measured by the decrease of
$f$, which is at most $\Phi_0$ over any run. The set $\mathcal Z$ and the
projection $\Pi_{\mathcal Z}$ of the convex method play no role either;
the model used below is the raw symmetrized finite-difference matrix,
and the analysis asks nothing of it beyond symmetry and an error bound.
Lemma~\ref{lem:group} holds under Assumption~\ref{as:nonconvex} with the
proof given in Appendix~\ref{app:backbone}, since that proof uses only
the pattern and the Lipschitz constant of the Hessian; we use it below
without further comment.

\subsection{From regret to stationarity}\label{sec:ncconversion}
 We use the deterministic midpoint formulation of
\citet{jiang2025improved}, which builds on the
online-to-nonconvex conversion of \citet{cutkosky2023optimal}.
The conversion runs for $N=KT$ iterations grouped into $K$ episodes
$\mathcal N_k=\{(k-1)T+1,\ldots,kT\}$ of length $T$. An online learner
chooses increments $\Delta_n$ with $\|\Delta_n\|\le D$, and the iterates
and two auxiliary points are
\begin{equation}\label{eq:ncpoints}
x_n=x_{n-1}+\Delta_n,\qquad
w_n=x_{n-1}+\tfrac12\Delta_n,\qquad
z_n=x_n+\tfrac12\Delta_n,\qquad n\ge1,
\end{equation}
with $\Delta_0=0$ and $z_0=x_0$. The point $w_n$ is the midpoint of the
$n$th step; the gradient there is the learner's $n$th loss. The point
$z_n$ is where the midpoint of the next step would fall if the increment
were repeated; it is where the hint and the curvature probe of the next
step are evaluated. The next midpoint lies within $D$ of it,
\begin{equation}\label{eq:ncsn}
w_{n+1}=z_n+s_n,\qquad s_n=\tfrac12(\Delta_{n+1}-\Delta_n),\qquad
\|s_n\|\le D,
\end{equation}
and consecutive probe points satisfy
\begin{equation}\label{eq:nczstep}
z_{n+1}-z_n=\tfrac32\Delta_{n+1}-\tfrac12\Delta_n,\qquad
\|z_{n+1}-z_n\|\le2D .
\end{equation}
Two facts tie the losses to the objective. Evaluating the gradient at the
midpoint makes the first-order expansion of $f$ along a step exact up to
a third-order term, and averaging the midpoints of an episode produces a
point whose gradient is close to the average loss of the episode. 
The following estimates follow the arguments of
\citet[Lemmas~2.1--2.2]{jiang2025improved};
see also \citet[Appendix~A]{patitucci2025double}.
We include proofs with our constants for completeness.

\begin{lemma}\label{lem:ncconversion}
Let Assumption~\ref{as:nonconvex} hold and $\|\Delta_n\|\le D$ for all
$n$. Then for every $n\ge1$,
\begin{equation}\label{eq:ncdescent}
f(x_n)-f(x_{n-1})\le\langle\nabla f(w_n),\Delta_n\rangle+\frac{L_2D^3}{24},
\end{equation}
and for every episode $k$, with
$\bar w^k=\frac1T\sum_{n\in\mathcal N_k}w_n$ and
$G^k=\sum_{n\in\mathcal N_k}\nabla f(w_n)$,
\begin{equation}\label{eq:ncavg}
\|\nabla f(\bar w^k)\|\le\frac{\|G^k\|}{T}+\frac{L_2T^2D^2}{2}.
\end{equation}
\end{lemma}
\noindent The proof is in Appendix~\ref{proof:ncconversion}.
Summing \eqref{eq:ncdescent} over the run shows that the learner's
cumulative loss $\sum_n\langle\nabla f(w_n),\Delta_n\rangle$ is at least
$-\Phi_0$ minus the third-order terms: the learner cannot keep descending.
Regret compares this cumulative loss with that of a fixed comparator in
each episode, and the comparator that makes the comparison informative is
the direction opposite to the episode's summed gradient, scaled to the
radius:
\begin{equation}\label{eq:nccomparator}
u^k=-D\,\frac{G^k}{\|G^k\|}\ \ (u^k=0\text{ if }G^k=0),\qquad
\mathrm{Reg}_N=\sum_{k=1}^K\sum_{n\in\mathcal N_k}
\langle\nabla f(w_n),\Delta_n-u^k\rangle .
\end{equation}
The comparator's loss in episode $k$ is $-D\|G^k\|$, so a small regret
forces the summed gradients to be small. The output is the average of the
episode with the smallest summed gradient, $\bar w^{k^\star}$ with
$k^\star\in\arg\min_k\|G^k\|$; it is determined by gradients the method
has already queried and costs no additional oracle call. The next proposition adapts the regret-to-stationarity
argument of \citet[Proposition~2.3]{jiang2025improved}
to the output rule above.

\begin{proposition}\label{prop:ncmaster}
Under Assumption~\ref{as:nonconvex}, for every $K\ge1$ and $N=KT$,
\begin{equation}\label{eq:ncmaster}
\|\nabla f(\bar w^{k^\star})\|\le
\frac{\Phi_0+\mathrm{Reg}_N}{DN}+\frac{L_2D^2}{24}+\frac{L_2T^2D^2}{2}.
\end{equation}
\end{proposition}
\noindent The proof is in Appendix~\ref{proof:ncmaster}. The inequality
holds for any increments in the ball of radius $D$, however they are
chosen; the learner enters only through $\mathrm{Reg}_N$. The parameters
have distinct roles. The radius $D$ trades the first term against the last
two, which are the discretization errors of the midpoint rule and of the
averaging. The episode length $T$ enters the averaging error and the regret, as the next section shows, through the number $K=N/T$ of
comparator changes. The task of the learner is to make
$\mathrm{Reg}_N/(DN)$ small, and the quality of its hint determines how
small.

\subsection{The optimistic learner and the design target}\label{sec:nclearner}
The learner is optimistic online gradient descent on the ball
\citep{rakhlin2013online}. 
We use the optimistic quasi-Newton update and hint construction
of \citet[Section~3.1]{jiang2025improved}. Before choosing $\Delta_{n+1}$ it forms a hint
$h_{n+1}$ for the next loss $\nabla f(w_{n+1})$, and it corrects for the
error $\theta_n=\nabla f(w_n)-h_n$ of the previous hint:
\begin{equation}\label{eq:nclearner}
\Delta_{n+1}=\Pi_D\big(\Delta_n-\eta\,\theta_n-\eta\,h_{n+1}\big),
\qquad \theta_0=0,
\end{equation}
where $\Pi_D$ is the projection onto $\{\|\Delta\|\le D\}$ and $\eta>0$ is
the learning rate. The hint is a quasi-Newton prediction. The next loss is
the gradient at $w_{n+1}=z_n+s_n$ by \eqref{eq:ncsn}, and we predict it
by linearizing at $z_n$ with a symmetric model $B_n$ of the Hessian:
\begin{equation}\label{eq:nchint}
h_{n+1}=\nabla f(z_n)+B_ns_n=\nabla f(z_n)+\tfrac12B_n(\Delta_{n+1}-\Delta_n).
\end{equation}
The prediction error of this hint is
\begin{equation}\label{eq:nctheta}
\theta_{n+1}=\nabla f(w_{n+1})-h_{n+1}
=\nabla f(z_n+s_n)-\nabla f(z_n)-B_ns_n ,
\end{equation}
the error of $B_n$ as a predictor of the gradient change along $s_n$. Its squared norm is the prediction loss used by the
matrix learner of \citet{jiang2025improved}.
Since $h_{n+1}$ depends on
$\Delta_{n+1}$, the update \eqref{eq:nclearner} is an implicit equation.
Following \citet[Lemma~3.1 and Definition~3.1]{jiang2025improved},
we express this implicit update through the stationarity
conditions of a trust-region subproblem. The following lemma
also records the equivalent formulation with an inexact residual.

\begin{lemma}\label{lem:nctrust}
Let $B_n$ be symmetric, $\eta,D>0$, and
\begin{equation}\label{eq:nctr}
\phi_n(\Delta)=\langle\nabla f(z_n)+\theta_n,\Delta\rangle
+\tfrac14\langle B_n(\Delta-\Delta_n),\Delta-\Delta_n\rangle
+\frac{1}{2\eta}\|\Delta-\Delta_n\|^2 .
\end{equation}
For $\nu\ge0$, a point $\Delta$ with $\|\Delta\|\le D$ satisfies
\begin{equation}\label{eq:nckkt}
\|\nabla\phi_n(\Delta)+\varrho\Delta\|\le\nu
\quad\text{for some }\varrho\ge0\text{ with }\varrho(\|\Delta\|-D)=0
\end{equation}
if and only if
\begin{equation}\label{eq:ncinexact}
\Delta=\Pi_D\big(\Delta_n-\eta\theta_n-\eta h_{n+1}+\eta q\big)
\quad\text{for some }q\text{ with }\|q\|\le\nu ,
\end{equation}
where $h_{n+1}$ is \eqref{eq:nchint} evaluated at $\Delta_{n+1}=\Delta$.
In particular, the solutions of the implicit update \eqref{eq:nclearner}
are exactly the KKT points of $\min_{\|\Delta\|\le D}\phi_n(\Delta)$, the
case $\nu=0$ of \eqref{eq:nckkt}. The problem has a global minimizer,
every global minimizer is a KKT point, and hence \eqref{eq:nclearner} has
a solution.
\end{lemma}
\noindent The proof is in Appendix~\ref{proof:nctrust}. We call a point
satisfying \eqref{eq:nckkt} a \emph{$\nu$-approximate solution} of the
subproblem at step $n$; it is what a trust-region solver returns when run
to residual $\nu$ in the stationarity condition, and $\nu=0$ is the exact
update. Only the stationarity condition is required: a KKT point that is
not the global minimizer of $\phi_n$ over the ball is admissible, a point
we return to in Section~\ref{sec:ncimplementation}. The problem
$\min_{\|\Delta\|\le D}\phi_n$ is a classical trust-region subproblem
\citep{more1983computing} in the stored matrix $\frac12B_n+\eta^{-1}I$,
which may be indefinite when $\eta$ is large, the regime of interest
below. No gradient is queried in solving it.
The next lemma is the one-step inequality of
\citet[Lemma~B.1]{jiang2025improved}, written in our notation.
It includes the subproblem error and retains the negative term.
We reproduce its proof for completeness.

\begin{lemma}\label{lem:nconestep}
Let $\Delta_{n+1}$ be a $\nu$-approximate solution of the subproblem at
step $n$. Then for every $u$ with $\|u\|\le D$,
\begin{equation}\label{eq:nconestep}
\begin{aligned}
\langle\nabla f(w_{n+1}),\Delta_{n+1}-u\rangle
\le{}&\frac{\|\Delta_n-u\|^2-\|\Delta_{n+1}-u\|^2}{2\eta}
+\langle\theta_{n+1},\Delta_{n+1}-u\rangle-\langle\theta_n,\Delta_n-u\rangle\\
&+\eta\|\theta_n\|^2-\frac{\|\Delta_{n+1}-\Delta_n\|^2}{4\eta}+2D\nu .
\end{aligned}
\end{equation}
\end{lemma}
\noindent The proof is in Appendix~\ref{proof:nconestep}. Summing over an
episode with $u=u^k$ telescopes the first two groups of terms, and
summing over episodes gives the regret bound below. Its first form keeps
the prediction errors and the negative term side by side; its second form
is what a pointwise certificate makes of them. 
Retaining the negative term and absorbing prediction errors
through a bound on consecutive increments follows the analysis
of \citet[Lemmas~3.1--3.2]{patitucci2025double}.
Here we state the argument for a general certificate $\Gamma$
and include the subproblem residual.

\begin{lemma}\label{lem:ncregret}
Suppose every step uses a $\nu$-approximate solution of the subproblem.
Then
\begin{equation}\label{eq:ncregret}
\mathrm{Reg}_N\le\frac{4KD^2}{\eta}
+3\eta\sum_{n=0}^{N}\|\theta_n\|^2
-\frac1{4\eta}\sum_{n=0}^{N-1}\|\Delta_{n+1}-\Delta_n\|^2+2DN\nu .
\end{equation}
If, in addition, there is $\Gamma>0$ such that
\begin{equation}\label{eq:nccert}
\|\theta_{n+1}\|\le\Gamma\,\|\Delta_{n+1}-\Delta_n\|
\qquad\text{for every }n\ge0,
\end{equation}
and $\eta\le1/(2\sqrt3\,\Gamma)$, then
\begin{equation}\label{eq:ncregretabs}
\mathrm{Reg}_N\le\frac{4KD^2}{\eta}+2DN\nu .
\end{equation}
\end{lemma}
\noindent The proof is in Appendix~\ref{proof:ncregret}.

Under the certificate \eqref{eq:nccert}, the regret
\eqref{eq:ncregretabs} is, up to the subproblem error, that of an
optimistic learner with perfect hints, $4KD^2/\eta$, and the hint is
felt only through the admissible learning rate $\eta\asymp1/\Gamma$.
Substituting into \eqref{eq:ncmaster} gives
\[
\|\nabla f(\bar w^{k^\star})\|\le\frac{\Phi_0}{DN}+\frac{4D}{\eta T}
+2\nu+\frac{L_2D^2}{24}+\frac{L_2T^2D^2}{2},
\]
and with $\eta\asymp1/\Gamma$ the second term is of order $\Gamma D/T$.
Balancing it against the averaging error $L_2T^2D^2/2$ gives
$T\asymp(\Gamma/(L_2D))^{1/3}$, and both terms are then of order
$L_2D^2\,(\Gamma/(L_2D))^{2/3}$.  The figure of merit is therefore the
ratio  $\Gamma/(L_2D)$, which we normalize as
\[
 \omega=\max\left\{1,\frac{\Gamma}{L_2D}\right\}.
\]

Accounting for $T\ge1$, we therefore choose
$T\asymp\omega^{1/3}$ and
$\eta\asymp1/(\omega L_2D)$.
 If $\omega$ does not depend on $D$, choosing $D\asymp\sqrt{\varepsilon/(L_2\omega^{2/3})}$
makes the discretization errors $\varepsilon/2$, and
$\Phi_0/(DN)\le\varepsilon/4$ then requires
$N\asymp\Phi_0\sqrt{L_2}\,\omega^{1/3}\varepsilon^{-3/2}$ iterations
(Theorem~\ref{thm:ncmaster} below makes the constants explicit). Three
hints calibrate the scale. With no model, $B_n=0$, the prediction error
is a gradient difference over a step of length $\|s_n\|$, so
$\Gamma=L_1/2$ and $\omega=\max\{1,L_1/(2L_2D)\}$ depends on $D$; the balance then
gives the first-order rate $\varepsilon^{-7/4}$
(Theorem~\ref{thm:ncsmooth}). With the exact Hessian, $B_n=H(z_n)$, the error is second order in $\|s_n\|$, so $\Gamma=L_2D/4$, Theorem~\ref{thm:ncmaster} applies with $\omega=1$, and the count is $O(\Phi_0\sqrt{L_2}\,\varepsilon^{-3/2})$, the iteration
count of cubic regularization \citep{nesterov2006cubic}. The rolling
model will give $\|B_n-H(z_n)\|_{\rm op}\le(2\chi-1)L_2D$ at every step,
hence $\Gamma\le\chi L_2D$ and $\omega=\chi$, independent of $D$.

The design target is thus a \emph{uniform} pointwise bound on
$\|B_n-H(z_n)\|_{\rm op}$ of order $L_2D$, with a factor as small as
possible. Unlike in Section~\ref{sec:algo}, no decay of the certificate
along the run is needed, because the conversion operates at the fixed
scale $D$ rather than along a settling trajectory.
An aggregate prediction-error bound alone does not guarantee
the cancellation in \eqref{eq:ncregret} unless it also relates
the error to the squared changes of the increments.
To see the cost of discarding this relationship, retain only
the consequence
$\sum_{n=0}^{N}\|\theta_n\|^2\le4\Gamma^2D^2N$
of \eqref{eq:nccert}, $\theta_0=0$, and
$\|\Delta_{n+1}-\Delta_n\|\le2D$, and drop the negative term.
Balancing $4KD^2/\eta$ against $12\eta\Gamma^2D^2N$ then gives
$\eta\asymp1/(\Gamma\sqrt T)$ and a bound on the regret,
excluding the subproblem error, larger by a factor $\sqrt T$
than the bound obtained from \eqref{eq:ncregretabs}
with $\eta\asymp1/\Gamma$.
This distinction
separates the present analysis
from \citet{jiang2025improved}, where $B_n$
is produced by an online learner in matrix space whose guarantee is a
regret bound on the prediction losses, an aggregate statement over the
run rather than a bound at each step; their analysis obtains
$O(d^{1/4}\varepsilon^{-13/8})$.
The rolling finite-difference construction supplies the required
pointwise bound by controlling both probe error and drift since
each group was refreshed. This bound enables the cancellation
and the improved rate below.

\begin{theorem}\label{thm:ncmaster}
Let Assumption~\ref{as:nonconvex} hold, let the symmetric matrices $B_n$
satisfy the certificate \eqref{eq:nccert} with $\Gamma\le\omega L_2D$ for
some $\omega\ge1$, and run the conversion with the optimistic learner,
$\nu$-approximate subproblem solutions, and
\[
\eta=\frac{1}{2\sqrt3\,\omega L_2D},\qquad
T=\big\lceil(16\sqrt3\,\omega)^{1/3}\big\rceil\le4.1\,\omega^{1/3}.
\]
Then for every $K\ge1$ and $N=KT$,
\begin{equation}\label{eq:ncmasterbound}
\|\nabla f(\bar w^{k^\star})\|\le\frac{\Phi_0}{DN}+2\nu
+13\,\omega^{2/3}L_2D^2 .
\end{equation}
\end{theorem}
\noindent The proof is in Appendix~\ref{proof:thm:ncmaster}. The theorem
is stated for an arbitrary model satisfying the certificate; it remains to
construct one with $\omega$ independent of $D$.

\subsection{The algorithm: a rolling model in the hint}\label{sec:ncalgorithm}
The model is the rolling finite-difference matrix of Block~1, built at
the probe points $z_n$ with a constant probe radius and read out by
symmetrization alone. All $\chi$ groups are probed once at the starting
point; afterwards one group is refreshed per iteration in a fixed cyclic
order. With $i_n=((n-1)\bmod\chi)+1$ for $n\ge1$ and
$\tau_g=2D/(\sqrt\chi\,|g|)$ for every group $g$, the refresh at
iteration $n\ge1$ is
\begin{equation}\label{eq:ncrefresh}
u_n=\frac{\nabla f(z_n+\tau_gv_g)-\nabla f(z_n)}{\tau_g},\quad g=g_{i_n},
\qquad
M_ne_j=
\begin{cases}
P_ju_n, & j\in g_{i_n},\\
M_{n-1}e_j, & j\notin g_{i_n},
\end{cases}
\qquad
B_n=\frac{M_n+M_n^\top}{2},
\end{equation}
and the warm-up applies the block formula at $z_0=x_0$ to every group
to form $M_0$. As in Section~\ref{sec:algo}, the stored matrix $M_n$ is
kept raw and unsymmetrized, so that a refresh replaces exactly the
columns of one group and leaves the others untouched. Unlike there, no
projection is applied at the read-out: the convex method needs
$B_k\succeq0$ for its resolvent bounds, whereas the analysis of this
section uses $B_n$ only through its symmetry and its error at $z_n$, and
the symmetrized raw matrix has the smaller error. The probe radius is
divided by $|g|$ so that every block carries the same finite-difference
error, $L_2\tau_g|g|/2=L_2D/\sqrt\chi$ in Frobenius norm by
Lemma~\ref{lem:group}, and the $\chi$ blocks together have Frobenius
error $L_2D$, one step of Hessian drift. A constant radius suffices here,
in contrast to the summable radii of Section~\ref{sec:algo}, for the
reason given above: the certificate need not improve along the run.

The method has a second variant in which the model is switched off,
$B_n=0$ for all $n$, no probes are taken, and the hint is
$h_{n+1}=\nabla f(z_n)$. We call the two variants the \emph{model
variant} and the \emph{gradient variant}. In the gradient variant the hint
does not depend on $\Delta_{n+1}$, so the update \eqref{eq:nclearner} is
explicit, $\Delta_{n+1}=\Pi_D(\Delta_n-\eta\theta_n-\eta\nabla f(z_n))$,
and $\nu=0$ at no cost. The gradient variant is the doubly optimistic gradient
method of \citet{patitucci2025double} within the online-to-nonconvex conversion of
\citet{cutkosky2023optimal} (which removes the inner fixed-point loop and logarithmic factors of \citet{cutkosky2023optimal}; see also \citet{jiang2025improved}), and we include it because the same code
provides the fallback when the model does not pay, and because its
explicit constants are needed to select between the two in advance.

\vspace{2mm}
\par\noindent\begin{minipage}{\linewidth}\begin{framed}
\noindent\textbf{Algorithm 3 (RP-OQN: rolling-probe optimistic
quasi-Newton method).} Input: $x_0$, a radius $D$, an episode length
$T$, a learning rate $\eta$, a number of episodes $K$, a subproblem
accuracy $\nu$, and the variant (model or gradient). Set $N=KT$,
$\Delta_0=0$, $z_0=x_0$, $\theta_0=0$, and, in the model variant,
$\tau_g=2D/(\sqrt\chi\,|g|)$ for every group $g$.
\begin{enumerate}[label=\arabic*.,leftmargin=2em]
\item \textbf{Warm-up.} Evaluate $\nabla f(x_0)$. In the model variant,
      evaluate $\nabla f(x_0+\tau_gv_g)$ for every group $g$ and store
      the $\chi$ blocks in $M_0$ by the block formula of
      \eqref{eq:ncrefresh}.
\item For $n=0,1,\ldots,N-1$:
\begin{enumerate}[label=(\alph*),leftmargin=1.8em]
\item \textbf{Probe and read out.} If $n\ge1$, evaluate $\nabla f(z_n)$
      and, in the model variant, refresh the columns of group $g_{i_n}$
      by \eqref{eq:ncrefresh}. Set $B_n=(M_n+M_n^\top)/2$ in the model
      variant and $B_n=0$ in the gradient variant.
\item \textbf{Choose the increment.} Compute a $\nu$-approximate solution
      $\Delta_{n+1}$ of $\min_{\|\Delta\|\le D}\phi_n(\Delta)$, with
      $\phi_n$ as in \eqref{eq:nctr}; in the gradient variant,
      $\Delta_{n+1}=\Pi_D(\Delta_n-\eta\theta_n-\eta\nabla f(z_n))$.
\item \textbf{Step and observe.} Set $x_{n+1}=x_n+\Delta_{n+1}$,
      $w_{n+1}=x_n+\frac12\Delta_{n+1}$ and
      $z_{n+1}=x_{n+1}+\frac12\Delta_{n+1}$; evaluate $\nabla f(w_{n+1})$
      and set
      $\theta_{n+1}=\nabla f(w_{n+1})-\nabla f(z_n)-\frac12B_n(\Delta_{n+1}-\Delta_n)$.
\end{enumerate}
\item \textbf{Return} $\bar w^{k^\star}$, the average of the midpoints
      $w_n$, $n\in\mathcal N_{k^\star}$, of an episode with the smallest
      $\|G^k\|$.
\end{enumerate}
\end{framed}\end{minipage}\par\medskip

\paragraph{Cost.} In the model variant, iteration $n\ge1$ queries three
gradients: at the probe point $z_n$, at the displaced probe point
$z_n+\tau_gv_g$, and at the new midpoint $w_{n+1}$; the gradient at
$z_0=x_0$ and the $\chi$ warm-up probes are queried once. A run of $N$
iterations therefore costs $3N+\chi-1$ gradient queries. In the gradient
variant, iteration $n\ge1$ queries two gradients, at $z_n$ and at
$w_{n+1}$, and a run costs $2N$. The remaining operations are arithmetic
on stored data: the running sums $G^k$ and the episode averages
$\bar w^k$, the symmetrization, and, in the model variant, the
trust-region solve of step~2(b), whose cost is the subject of
Section~\ref{sec:ncimplementation}. None of it queries the oracle.

\paragraph{The certificate.}
The pointwise error of the model is bounded by the same exchange of
summation as in Lemma~\ref{lem:telescope}, with $z_j$ in place of $x_j$
and with the probe-point displacements \eqref{eq:nczstep} in place of the
$\delta_j$. Because the window has at most $\chi-1$ displacements of
length at most $2D$, the bound is a constant multiple of $\chi L_2D$ at
every iteration, including the first ones: the warm-up anchors every
initial block at $z_0$, so there is no initial phase in which the model
is incomplete.

\begin{lemma}\label{lem:ncpointwise}
Under Assumption~\ref{as:nonconvex}, the matrices of the model variant of
Algorithm~3 satisfy, for every $n\ge0$,
\begin{equation}\label{eq:ncpointwise}
\|B_n-H(z_n)\|_{\rm op}
\le L_2\!\!\sum_{j=\max\{0,n-\chi+1\}}^{n-1}\!\!\|z_{j+1}-z_j\|
+L_2D
\le(2\chi-1)L_2D .
\end{equation}
In particular $\|B_n\|_{\rm op}\le L_1+(2\chi-1)L_2D$.
\end{lemma}
\noindent The proof is in Appendix~\ref{proof:ncpointwise}.  The error of the hint along the realized step also involves the averaged Hessian $\bar H_n=\int_0^1H(z_n+t s_n)\,dt$. Its difference from $H(z_n)$ has operator norm at most $L_2\|s_n\|/2\le L_2D/2$, which gives the certificate below. In the gradient variant the hint error
is a gradient difference over a step of length $\|s_n\|$.

\begin{lemma}\label{lem:nccertificate}
Under Assumption~\ref{as:nonconvex}, Algorithm~3 satisfies
\eqref{eq:nccert} for every $n\ge0$ with
\begin{equation}\label{eq:ncgamma}
\Gamma=\chi L_2D\ \ \text{(model variant)},\qquad
\Gamma=\tfrac12L_1\ \ \text{(gradient variant)}.
\end{equation}
\end{lemma}
\noindent The proof is in Appendix~\ref{proof:nccertificate}. The two
certificates are the two readings of Theorem~\ref{thm:ncmaster}: the
model variant has $\omega=\chi$ independent of $D$, and the gradient
variant  can use
$\omega=\max\{1,L_1/(2L_2D)\}$. The model variant is the better one
when $\chi L_2D$ is a fixed fraction of $L_1$ or less, that is, when the
model is more accurate than the trivial bound; the next subsection makes
the crossover explicit.

\subsection{Guarantees}\label{sec:ncguarantee}
\begin{theorem}\label{thm:ncmodel}
Let Assumption~\ref{as:nonconvex} hold and run the model variant of
Algorithm~3 with
\[
D=\sqrt{\frac{\varepsilon}{26\,L_2\chi^{2/3}}},\qquad
\eta=\frac{1}{2\sqrt3\,\chi L_2D},\qquad
T=\big\lceil(16\sqrt3\,\chi)^{1/3}\big\rceil,\qquad
\nu\le\frac{\varepsilon}{8}.
\]
Then for every $K\ge1$,
$\|\nabla f(\bar w^{k^\star})\|\le\frac{\Phi_0}{DKT}+\frac34\varepsilon$.
In particular, if $K\ge\frac{4\Phi_0}{DT\varepsilon}$, the output
satisfies $\|\nabla f(\bar w^{k^\star})\|\le\varepsilon$, and with
$K=\max\{1,\lceil4\Phi_0/(DT\varepsilon)\rceil\}$ the total number of
gradient queries is at most
\begin{equation}\label{eq:ncmodelcount}
G_\varepsilon\ \le\ 62\,\Phi_0\sqrt{L_2}\,\chi^{1/3}\varepsilon^{-3/2}
+13\,\chi^{1/3}+\chi .
\end{equation}
\end{theorem}
\noindent The proof is in Appendix~\ref{proof:ncmodel}. The parameters
require $L_2$ and $\chi$; the number of episodes requires a known upper bound
on the initial gap $\Phi_0=f(x_0)-f^\star$ (for instance, from evaluating $f(x_0)$ and a known lower bound on $f^\star$), and any such upper bound may be
used in place of $\Phi_0$ in \eqref{eq:ncmodelcount}. The theorem holds
for every $\chi$ and every $\varepsilon$; it is the better of the two
guarantees in order on the range identified after the next one.

\begin{theorem}\label{thm:ncsmooth}
Let Assumption~\ref{as:nonconvex} hold, let $\varepsilon\le L_1^2/L_2$,
and run the gradient variant of Algorithm~3 with
\[
D=\frac{\varepsilon^{3/4}}{9\,L_1^{1/2}L_2^{1/4}},\qquad
\eta=\frac{1}{\sqrt3\,L_1},\qquad
T=\Big\lceil\Big(8\sqrt3\,\frac{L_1}{L_2D}\Big)^{1/3}\Big\rceil
\le6\Big(\frac{L_1^2}{L_2\varepsilon}\Big)^{1/4}.
\]
Then for every $K\ge1$,
$\|\nabla f(\bar w^{k^\star})\|\le\frac{\Phi_0}{DKT}+\frac34\varepsilon$,
and with $K=\max\{1,\lceil4\Phi_0/(DT\varepsilon)\rceil\}$ the output
satisfies $\|\nabla f(\bar w^{k^\star})\|\le\varepsilon$ after at most the following number of gradient queries:
\begin{equation}\label{eq:ncsmoothcount}
G_\varepsilon\ \le\ 72\,\Phi_0L_1^{1/2}L_2^{1/4}\varepsilon^{-7/4}
+12\Big(\frac{L_1^2}{L_2\varepsilon}\Big)^{1/4}.
\end{equation}
\end{theorem}
\noindent The proof is in Appendix~\ref{proof:ncsmooth}. The condition
$\varepsilon\le L_1^2/L_2$ only ensures $\omega=L_1/(2L_2D)\ge1$ in
Theorem~\ref{thm:ncmaster}. The two leading terms agree, up to an
absolute constant, where
$\chi^{1/3}\sqrt{L_2}\,\varepsilon^{-3/2}=L_1^{1/2}L_2^{1/4}\varepsilon^{-7/4}$,
that is, at $\chi=(L_1^2/(L_2\varepsilon))^{3/4}$; there the certificate
$\chi L_2D$ of the model variant, with the radius of
Theorem~\ref{thm:ncmodel}, equals $L_1/\sqrt{26}$. The model variant
thus gives the better asymptotic order on its  advantageous regime, where its
certificate improves on the trivial one by a constant factor or more, and
this leading-order selection can be made in advance from $\chi$, $L_1$,
$L_2$ and $\varepsilon$, without knowledge of $\Phi_0$.

\begin{corollary}[Gradient complexity in the nonconvex case]\label{cor:ncenvelope}
Under Assumption~\ref{as:nonconvex}, for fixed $L_1,L_2,\Phi_0$ and
$0<\varepsilon\le L_1^2/L_2$, Algorithm~3 in its model variant
with the parameters of Theorem~\ref{thm:ncmodel} if
$\chi\le(L_1^2/(L_2\varepsilon))^{3/4}$, and in its gradient variant with
the parameters of Theorem~\ref{thm:ncsmooth} otherwise, finds an
$\varepsilon$-stationary point with
\begin{equation}\label{eq:ncenvelope}
G_\varepsilon=O\big(\min\{\chi^{1/3}\varepsilon^{-3/2},\ \varepsilon^{-7/4}\}\big)
\end{equation}
gradient queries, with $\chi=d$ in the dense case.
\end{corollary}
\begin{proof}
Write $\beta=L_1^2/(L_2\varepsilon)\ge1$. If $\chi\le\beta^{3/4}$, then
$\chi=\chi^{1/3}\chi^{2/3}\le\chi^{1/3}\beta^{1/2}$, which for fixed
constants is $O(\chi^{1/3}\varepsilon^{-1/2})$, so the two lower-order
terms of \eqref{eq:ncmodelcount} are $O(\chi^{1/3}\varepsilon^{-3/2})$;
and $\chi^{1/3}\varepsilon^{-3/2}\le\beta^{1/4}\varepsilon^{-3/2}=O(\varepsilon^{-7/4})$,
so $\chi^{1/3}\varepsilon^{-3/2}$ is the smaller of the two leading
terms. If $\chi>\beta^{3/4}$, then $\varepsilon^{-7/4}=O(\chi^{1/3}\varepsilon^{-3/2})$
is the smaller one, and the lower-order term of \eqref{eq:ncsmoothcount}
is $O(\varepsilon^{-1/4})$.
\end{proof}

\begin{remark}\label{rem:ncadditive}
The term $\chi$ in \eqref{eq:ncmodelcount} is the warm-up, which fills
the window before the first step. RP-QNPE pays no such term in the convex
case because the smoothness step $\sigma/L_1$ is admissible at every
iteration, so the method can optimize while the window fills. Here the
learning rate is fixed once from a certificate that must hold at every
step, and the certificate $\Gamma=\chi L_2D$ requires a complete model
from $n=0$ on. Filling the window while optimizing would require an analysis
that combines the initial gradient phase with the subsequent
model phase. The present analysis uses a complete warm-up
to ensure that the same certificate and learning rate apply
throughout the run. For fixed $\Phi_0>0$, however, the warm-up term is absorbed into the normalized complexity envelope \eqref{eq:ncenvelope}: on the range where the model variant is
selected, $\chi\le\chi^{1/3}L_1/\sqrt{L_2\varepsilon}$, which is below the
first term of \eqref{eq:ncmodelcount} as soon as
$\Phi_0\ge L_1\varepsilon/(62L_2)$, and the gradient variant has no
additive term at all. The
lazy second-order methods discussed below carry an additive $d$ of the
same origin, the cost of one Hessian.
\end{remark}

\begin{remark}\label{rem:ncexact}
With exact Hessians in the hint, Theorem~\ref{thm:ncmaster} applies with
$\omega=1$ and gives $O(1+\Phi_0\sqrt{L_2}\,\varepsilon^{-3/2})$ iterations,
the rate of cubic regularization, which is optimal for second-order
methods \citep{carmon2020lower}. The factor $\chi^{1/3}$ in
\eqref{eq:ncmodelcount} is therefore the price of measuring the Hessian
one group per iteration instead of all at once. It arises as
$\omega^{1/3}$ with $\omega=\chi$, because the pointwise error of a
model whose columns are up to $\chi-1$ steps old is of the order of
$\chi$ steps of drift, and the master formula converts a factor in the
certificate into its cube root in the iteration count. This is the same
exponent as the laziness penalty of the lazy extra-Newton method of
\citet{chen2025faster}, obtained here with the rolling finite-difference model and a first-order oracle.
\end{remark}

\subsection{Comparison with existing guarantees}\label{sec:nccompare}
Table~\ref{tab:nccompare} collects the bounds for finding an
$\varepsilon$-stationary point under Lipschitz gradient and Hessian, with
$L_1$, $L_2$ and $\Phi_0$ fixed, and Figure~\ref{fig:ncregimes} draws them against $1/\varepsilon$. As in Section~\ref{sec:compare}, prior
bounds are stated in terms of $d$, so we compare in the dense case
$\chi=d$, and all comparisons are between proved upper bounds. Against
the first-order rate, the $\tilde O(\varepsilon^{-7/4})$ of
\citet{carmon2017convex} and the $O(\varepsilon^{-7/4})$ of
\citet{li2022restarted}, the model variant of Algorithm~3 is smaller
whenever $\chi\lesssim\varepsilon^{-3/4}$, and its gradient variant
recovers $O(\varepsilon^{-7/4})$ otherwise.

\begin{table}[t]
\centering\small
\setlength{\tabcolsep}{4pt}
\begin{tabular}{@{}lll@{}}
\toprule
Method & Oracle & Gradient complexity\\
\midrule
Convex until proven guilty \citep{carmon2017convex} & gradient & $\tilde O(\varepsilon^{-7/4})$\\
Restarted AGD \citep{li2022restarted} & gradient & $O(\varepsilon^{-7/4})$\\
OL-OQN \citep{jiang2025improved} & gradient & $O(d^{1/4}\varepsilon^{-13/8})$\\
\citet{marumo2025parameter} & gradient & $O(d^{1/4}\varepsilon^{-13/8})$\\
{Finite-difference CRN \citep{cartis2012oracle}} & {gradient} & {$O(d\,\varepsilon^{-3/2})$}\\
\citet{adil2025balancing} & gradient & $\tilde O(d+d^{1/3}\varepsilon^{-3/2})$\\
\midrule
CRN \citep{nesterov2006cubic} & Hessian $=\bar d$ gradients & $O(\bar d\,\varepsilon^{-3/2})$\\
Lazy cubic Newton \citep{doikov2023lazy} & Hessian $=\bar d$ gradients & $O(\bar d+\bar d^{1/2}\varepsilon^{-3/2})$\\
NALEN \citep{chen2025faster} & Hessian $=\bar d$ gradients & $O(\bar d+\bar d^{1/3}\varepsilon^{-3/2})$\\
\midrule
RP-OQN (Corollary~\ref{cor:ncenvelope}) & gradient & $O(\min\{\chi^{1/3}\varepsilon^{-3/2}+\chi,\ \varepsilon^{-7/4}\})$\\
\bottomrule
\end{tabular}
\caption{\small Nonconvex case: gradient queries to find a point with $\|\nabla f\|\le\varepsilon$,
with $L_1$, $L_2$, $\Phi_0$ fixed and logarithmic factors suppressed in the $\tilde O$ rows;
$\chi\le d$ is the effective dimension of Definition~\ref{def:chi} ($\chi=d$ for a dense Hessian).
Hessian-based methods are listed for reference, one Hessian counted as $\bar d$ gradients as in
Table~\ref{tab:convexcompare} {($\bar d=d$ as published, $\bar d=\chi$ a hypothetical
instantiation)}; CRN stands for cubic regularized Newton. {\color{ReviewBlue}The additive terms are dominated by the curvature terms whenever
$d\le\varepsilon^{-9/4}$. For RP-OQN, its additive $\chi$ is
dominated throughout the model-selection range
$\chi\le\varepsilon^{-3/4}$ (Remark~\ref{rem:ncadditive}).}}
\label{tab:nccompare}
\end{table}

\begin{figure}[t]
\centering
\begin{subfigure}[t]{0.49\textwidth}
\centering
\begin{tikzpicture}
\begin{axis}[width=\linewidth, height=0.85\linewidth, grid=major, grid style={gray!20}, axis on top,
  xmin=0, xmax=3, ymin=0, ymax=6,
  xtick={0,1,2,3}, xticklabels={$1$,$d$,$d^{2}$,$d^{3}$},
  ytick={0,...,6}, yticklabels={$1$,$d$,$d^{2}$,$d^{3}$,$d^{4}$,$d^{5}$,$d^{6}$},
  xlabel={$1/\varepsilon$}, ylabel={gradient queries},
  tick label style={font=\scriptsize}, label style={font=\small},
  every axis plot/.append style={line width=1.1pt},
  legend to name=leg:ncregimes, legend columns=2, legend cell align={left},
  legend style={font=\scriptsize, draw=gray!60,
                /tikz/every even column/.append style={column sep=0.6cm}}]
\fill[blue!8] (axis cs:0,0) rectangle (axis cs:1.3333,6);
\fill[red!10] (axis cs:1.3333,0) rectangle (axis cs:3,6);
\addplot[blue] coordinates {(0,0) (3,5.25)};
  \addlegendentry{first order (Li et al.): $\varepsilon^{-7/4}$}
\addplot[violet] coordinates {(0,0.25) (3,5.125)};
  \addlegendentry{OL-OQN, Marumo et al.: $d^{1/4}\varepsilon^{-13/8}$}
\addplot[teal] coordinates {(0,1) (0.4444,1) (3,4.8333)};
  \addlegendentry{Adil et al.\ (dense FD): $d+d^{1/3}\varepsilon^{-3/2}$}
\addplot[orange, dashed] coordinates {(0,1) (0.4444,1) (3,4.8333)};
  \addlegendentry{NALEN: $\bar d+\bar d^{1/3}\varepsilon^{-3/2}$}
\addplot[olive, dashed] coordinates {(0,1) (0.3333,1) (3,5)};
  \addlegendentry{lazy CRN: $\bar d+\bar d^{1/2}\varepsilon^{-3/2}$}
\addplot[brown, dashed] coordinates {(0,1) (3,5.5)};
  \addlegendentry{CRN: $\bar d\,\varepsilon^{-3/2}$}
\addplot[red, line width=2.6pt, opacity=0.5] coordinates {(0,0) (1.3333,2.3333) (3,4.8333)};
  \addlegendentry{RP-OQN: $\min\{\chi+\chi^{1/3}\varepsilon^{-3/2},\,\varepsilon^{-7/4}\}$}
\draw[dashed, gray!70] (axis cs:1.3333,0) -- (axis cs:1.3333,6) node[pos=1, below right, font=\scriptsize, text=gray!50!black] {$\varepsilon=d^{-4/3}$};
\draw[dashed, gray!70] (axis cs:2,0) -- (axis cs:2,6) node[pos=1, below right, font=\scriptsize, text=gray!50!black] {$\varepsilon=d^{-2}$};
\node[font=\scriptsize, align=center, inner sep=1pt] at (axis cs:0.67,0.45) {ties first order;\\below the rest};
\node[font=\scriptsize, align=center, inner sep=1pt] at (axis cs:2.17,0.45) {ties NALEN, Adil et al.;\\below first order};
\end{axis}
\end{tikzpicture}
\caption{Dense Hessian: $\chi=\bar d=d$.}
\label{fig:ncregimes-dense}
\end{subfigure}\hfill
\begin{subfigure}[t]{0.49\textwidth}
\centering
\begin{tikzpicture}
\begin{axis}[width=\linewidth, height=0.85\linewidth, grid=major, grid style={gray!20}, axis on top,
  xmin=0, xmax=3, ymin=0, ymax=6,
  xtick={0,1,2,3}, xticklabels={$1$,$\chi$,$\chi^{2}$,$\chi^{3}$},
  ytick={0,...,6}, yticklabels={$1$,$\chi$,$\chi^{2}$,$\chi^{3}$,$\chi^{4}$,$\chi^{5}$,$\chi^{6}$},
  xlabel={$1/\varepsilon$}, ylabel={gradient queries},
  tick label style={font=\scriptsize}, label style={font=\small},
  every axis plot/.append style={line width=1.1pt}]
\fill[blue!8] (axis cs:0,0) rectangle (axis cs:1.3333,6);
\fill[red!10] (axis cs:1.3333,0) rectangle (axis cs:3,6);
\addplot[blue] coordinates {(0,0) (3,5.25)};
\addplot[violet] coordinates {(0,0.75) (3,5.625)};
\addplot[teal] coordinates {(0,3) (1.3333,3) (3,5.5)};
\addplot[orange, dashed] coordinates {(0,1) (0.4444,1) (3,4.8333)};
\addplot[olive, dashed] coordinates {(0,1) (0.3333,1) (3,5)};
\addplot[brown, dashed] coordinates {(0,1) (3,5.5)};
\addplot[red, line width=2.6pt, opacity=0.5] coordinates {(0,0) (1.3333,2.3333) (3,4.8333)};
\draw[dashed, gray!70] (axis cs:1.3333,0) -- (axis cs:1.3333,6) node[pos=1, below right, font=\scriptsize, text=gray!50!black] {$\varepsilon=\chi^{-4/3}$};
\node[font=\scriptsize, align=center, inner sep=1pt] at (axis cs:0.67,0.45) {ties first order;\\below the rest};
\node[font=\scriptsize, align=center, inner sep=1pt] at (axis cs:2.17,0.45) {ties NALEN($\chi$);\\below Adil et al.};
\end{axis}
\end{tikzpicture}
\caption{Sparse Hessian, $d=\chi^{3}$, $\bar d=\chi$.}
\label{fig:ncregimes-sparse}
\end{subfigure}

\medskip
\pgfplotslegendfromname{leg:ncregimes}
\caption{Nonconvex case: gradient complexity against $1/\varepsilon$, in powers of $d$
(a) and of $\chi$ (b); constants and logarithms suppressed; the bounds are those of
Table~\ref{tab:nccompare}. Solid: gradient-oracle bounds; dashed:
exact-Hessian bounds at $\bar d$ gradients per Hessian ($\bar d=d$ in (a), hypothetically
$\bar d=\chi$ in (b)). RP-OQN (thick, translucent) is the lower envelope of the first-order rate
(blue band: gradient variant) and of NALEN (red band: model variant), which in (b) lies below the
dense finite-difference bound of Adil et al.\ by a factor of at least $(d/\chi)^{1/3}$.}
\label{fig:ncregimes}
\end{figure}
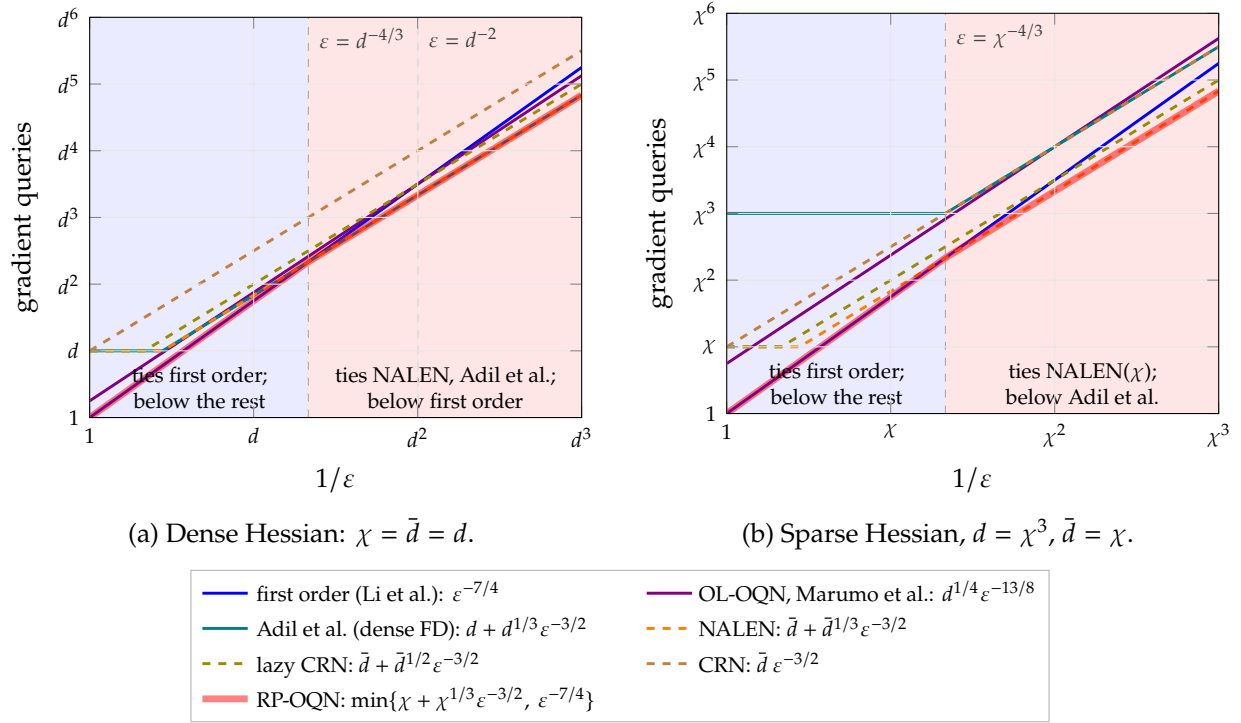

\subsubsection{Quasi-Newton methods}\label{sec:ncqncompare}
In the dense case the normalized bounds are
$\min\{d^{1/3}\varepsilon^{-3/2},\varepsilon^{-7/4}\}$ for RP-OQN and
$d^{1/4}\varepsilon^{-13/8}$ for OL-OQN \citep{jiang2025improved}, which
uses the same conversion and the same form of hint. The ratio of our model
branch to the OL-OQN bound is $(d^{2/3}\varepsilon)^{1/8}$, below one for
$d<\varepsilon^{-3/2}$, and our gradient branch is below it for
$d>\varepsilon^{-1/2}$, so Algorithm~3, with the variant selected in advance as in Corollary~\ref{cor:ncenvelope}, has the smaller bound for every $d$ and every $\varepsilon<1$. Measured against the first-order rate, OL-OQN
improves on $\varepsilon^{-7/4}$ only when $d<\varepsilon^{-1/2}$ and
RP-OQN when $d<\varepsilon^{-3/4}$; on the former range the ratio of the
two bounds is at most $\varepsilon^{1/12}$. RP-OQN is also simpler: it
needs no matrix learner and no separation oracle, spends three gradients
per iteration (one more than OL-OQN), and its only nontrivial computation
is the trust-region solve of Section~\ref{sec:ncimplementation}, which
OL-OQN performs as well. The same comparison of exponents applies to
\citet{marumo2025parameter}, whose method attains the OL-OQN bound without
prior knowledge of $L_1$, $L_2$ or $\varepsilon$, which Algorithm~3
requires. As in Section~\ref{sec:aqcompare}, this compares exponents, not
constants.

\subsubsection{Second-order and lazy methods}\label{sec:nclazycompare}

As in Section~\ref{sec:lazycompare}, the Hessian-based rows compare costs
under the convention of \citet{chen2025faster} that one Hessian evaluation
counts as $\bar d$ gradients ($\bar d=d$ in the dense case). Cubic
regularization \citep{nesterov2006cubic} then costs $O(d\,\varepsilon^{-3/2})$,
which RP-OQN improves by a factor of at least $d^{2/3}$ for every $d$ and
$\varepsilon$. Lazy cubic Newton \citep{doikov2023lazy} costs
$O(d+d^{1/2}\varepsilon^{-3/2})$ and improves on $\varepsilon^{-7/4}$ only
when $d<\varepsilon^{-1/2}$; RP-OQN is below it by a factor of at least
$d^{1/6}$ for every $d$ and $\varepsilon$. The lazy extra-Newton method
NALEN of \citet{chen2025faster}, which shares the conversion and the
trust-region subproblem of Algorithm~3 and recomputes an exact Hessian
every $\bar d$ steps, has the standalone bound
$O(\bar d+\bar d^{1/3}\varepsilon^{-3/2})$, matching the order of our
model variant (Theorem~\ref{thm:ncmodel}) when $\bar d=\chi$: the two
bounds coincide for $\chi\le\varepsilon^{-3/4}$, while for
$\chi>\varepsilon^{-3/4}$ the standalone NALEN bound lacks the
$\varepsilon^{-7/4}$ branch of Corollary~\ref{cor:ncenvelope} (though, as
noted in Section~\ref{sec:related}, pairing NALEN with a first-order
method via the same advance parameter switch recovers the same polynomial
envelope). \citet{adil2025balancing}, who rebuild a finite-difference
Hessian periodically and query only gradients in the dense case, obtain
$\tilde O(d+d^{1/3}\varepsilon^{-3/2})$, matching the polynomial order of
our model variant up to logarithmic factors, and their general theorem
likewise recovers $O(\varepsilon^{-7/4})$ when run without curvature. We
do not claim an improvement over the polynomial exponents of these two
methods: the $\varepsilon$-exponent $3/2$ is that of cubic regularization,
optimal for second-order methods, and the dimension exponent $1/3$ is the
laziness penalty of Remark~\ref{rem:ncexact}, obtained here from gradients
alone with no logarithmic factors and an explicit structured dependence on
$\chi$.

{\color{ReviewBlue} Figure~\ref{fig:ncregimes} summarizes the comparison.
In the dense case, RP-OQN matches the best polynomial
envelope among the bounds compared here: it matches
NALEN and \citet{adil2025balancing} on
$d\le\varepsilon^{-3/4}$ and the first-order rate otherwise.
For structured Hessians, its model variant matches the hypothetical
NALEN instantiation at $\bar d=\chi$, while improving
on the displayed dense finite-difference bound of
\citet{adil2025balancing}.}
What it adds is a gradient-oracle proof of the exponents $\chi^{1/3}\varepsilon^{-3/2}$ with explicit constants and no logarithmic factors, the dependence on the pattern through $\chi$, and a strict improvement over the quasi-Newton bounds $d^{1/4}\varepsilon^{-13/8}$ of OL-OQN and \citet{marumo2025parameter} at every $d$ and $\varepsilon<1$.

\subsubsection{Lower bounds}\label{sec:nclower}
The lower bounds $\Omega(\varepsilon^{-12/7})$ for deterministic
first-order methods \citep{carmon2021lower} and $\Omega(\varepsilon^{-3/2})$
for second-order methods \citep{carmon2020lower} are proved for
constructions whose dimension grows polynomially in $1/\varepsilon$, so
they do not contradict a dimension-dependent bound such as
$\chi^{1/3}\varepsilon^{-3/2}$, which falls below $\varepsilon^{-12/7}$
only when $\chi<\varepsilon^{-9/14}$.

\section{Implementation and total computational cost}\label{sec:implementation}

The gradient bounds of Sections~\ref{sec:algo}--\ref{sec:nonconvex} count
every oracle call, including finite-difference probes, under exact gradient
evaluations and Assumptions~\ref{as:main} and~\ref{as:nonconvex}, as appropriate. Algorithms~1 and~2 are stated with
exact spectral projections and regularized linear solves for simplicity of
exposition. Algorithm~3 is different: its model variant requires a
trust-region subproblem, but its analysis already permits
$\nu$-approximate solutions satisfying \eqref{eq:nckkt}; its gradient
variant requires no trust-region solve.
For Algorithms~1 and~2, neither exact spectral projections nor exact
linear solves are essential. A shift-and-scale repair replaces the
projection while preserving the curvature certificates and their
constants, and a relative residual stopping rule preserves the HPE
guarantees once a fixed fraction of the tolerance is reserved for the
linear solver (Section~\ref{sec:cvximplementation}). Algorithm~3 needs
no spectral projection, and its model-variant subproblem can be solved
to the required accuracy by projected gradient with a computable
stopping rule (Section~\ref{sec:ncimplementation}).
Section~\ref{sec:arithmetic} records the resulting arithmetic cost
of each method.

\subsection{Algorithms 1 and 2: projection-free models and inexact linear solves}\label{sec:cvximplementation}

\paragraph{What the analysis uses.}
The curvature model enters the proofs through exactly two properties.
\begin{enumerate}[label=(P\arabic*),leftmargin=3em]
\item \emph{Feasibility}, $B\in\mathcal Z$. The lower bound $B\succeq0$
gives the resolvent bounds $\|(B+\lambda^{-1}I)^{-1}\|_{\rm op}\le\lambda$
and $\|(B+\lambda^{-1}I)^{-1}B\|_{\rm op}\le1$ used in
Lemma~\ref{lem:scalarbounds}; the upper bound
$B\preceq L_1I$ gives the cap $\bar E_k\le L_1$ and makes the smoothness
step $\sigma/L_1$ admissible through the second certificate of
Lemma~\ref{lem:resid}.
\item \emph{A certified error at the current point},
$\|B_k-H(x_k)\|_{\rm op}\le\bar E_k$ in Lemma~\ref{lem:telescope}.
\end{enumerate}
No proof uses that $B$ is the nearest point of $\mathcal Z$ to the raw
matrix. The projection appears in exactly one place, the proof of
Lemma~\ref{lem:telescope}, and only through
Lemma~\ref{lem:clip}: if the symmetrized raw matrix $S$ satisfies
$\|S-H\|_{\rm op}\le e$ for some $H\in\mathcal Z$, then
$\|\Pi_{\mathcal Z}(S)-H\|_{\rm op}\le2e$. Any map with these two
properties can therefore replace $\Pi_{\mathcal Z}$ without changing
anything else. The following one has them and costs $O(d^2)$ arithmetic
operations.

\paragraph{Certified spectral repair.}
Given a symmetric $S$ and a bound $e\ge0$ with $\|S-H\|_{\rm op}\le e$ for
some $H\in\mathcal Z$, define
\begin{equation}\label{eq:repair}
\mathcal R_e(S)=\frac{L_1}{L_1+2e}\,(S+eI).
\end{equation}
The spectrum of $S$ lies in $[-e,L_1+e]$; the shift moves its lower end to
$0$ and the scaling moves its upper end to $L_1$, so
$\mathcal R_e(S)\in\mathcal Z$, and Lemma~\ref{lem:repair} shows
$\|\mathcal R_e(S)-H\|_{\rm op}\le2e$, the same factor as for the
projection. The bound $e$ is available before the step is taken: the proof
of Lemma~\ref{lem:telescope} bounds the symmetrized raw error for
$k\ge\chi$ by $e_k^{\rm raw}=L_2\widehat D_k+e_k/2$ with
$e_k=b/(k-\chi+2)^2$ from \eqref{eq:rollingcert}, a function of stored
quantities; set
$B_k=\mathcal R_{e_k^{\rm raw}}\big((M_k+M_k^\top)/2\big)$ and
$\bar E_k=\min\{L_1,2e_k^{\rm raw}\}$, which is \eqref{eq:rollingcert},
and use $B_k=0$, whose error is at most $L_1$, while the window is
incomplete, storing the probes as before. The certificate is unchanged,
hence so are Lemmas~\ref{lem:sum}, \ref{lem:scalarbounds}
and~\ref{lem:wellposed} and every theorem built on them, with identical
constants; Algorithm~2 uses the model only through
Lemma~\ref{lem:telescope}, so the repair applies to it unchanged as well.
If an eigensolver is used instead, any feasible $\widetilde B_k\in\mathcal Z$
with a certified $\|\widetilde B_k-B_k\|_{\rm op}\le\xi_k$ can be
accommodated by adding $\xi_k$ to the certificate; a summable budget
$\xi_k\le b/(k+1)^2$ preserves the bounds
(Lemma~\ref{lem:approxprojection}). Truncating an uncertified partial
decomposition is not sufficient.

\paragraph{Inexact linear solves.}
For a computed displacement $r$ at a center $y$ with stepsize $\lambda$
and model $B$, let $w=(B+\lambda^{-1}I)r+\nabla F(y)$ be the linear
residual. Fix $0<\delta<1$ and stop the solver as soon as
\begin{equation}\label{eq:linearstop}
\lambda\|w\|\le\delta\|r\| ,
\end{equation}
a test that costs one matrix-vector product and no gradient. The analysis
uses a step only through its HPE residual and its length, and both
degrade gracefully. With $r_*$ the exact solution,
$r=r_*+(B+\lambda^{-1}I)^{-1}w$, so $\|r\|\le\|r_*\|/(1-\delta)$; and
Taylor's formula gives the exact identity
\[
r+\lambda\nabla F(y+r)
=\lambda w+\lambda\Big(\int_0^1\nabla^2F(y+sr)\,ds-B\Big)r ,
\]
whose second term is the model error already controlled by
Lemma~\ref{lem:resid}. Hence the residual coefficient of the inexact step
is at most $\delta$ plus the coefficient of the exact analysis, evaluated
at the slightly longer step (Lemma~\ref{lem:inexactresidual}). It remains
to reserve the fraction $\delta$. 
For the inexact implementation of Subroutine~1.1, set
$\delta=\sigma/4$, $\sigma'=\sigma/2$, and
$\lambda_{\rm b}'=3\sigma/(4L_1)$.
First evaluate $\widehat\Psi_k(\lambda_{\rm b}')$.
If $\widehat\Psi_k(\lambda_{\rm b}')\ge\sigma'$,
return $\lambda_k=\lambda_{\rm b}'$.
Otherwise, initialize
$\lambda^-=\lambda_{\rm b}',\;
\lambda^+=\frac{\sigma'}{\bar E_k},$
and perform the same geometric bisection as in
Subroutine~1.1, using the threshold $\sigma'$.
Return $\lambda_k=\lambda^-$ once
$\lambda^+/\lambda^-\le2^{1/4}$.
At the floor the coefficient is $\delta+\lambda L_1=\sigma$;
at a bisected step it is at most $\sigma/4+(\sigma/2)/(1-\sigma/4)\le\sigma$.
The returned stepsize satisfies $\lambda_k\ge3\sigma/(4L_1)$ and
$\widehat\Psi_k(\lambda_k)\ge\sigma/4$ in place of $\sigma/L_1$ and
$\sigma/2$. The proofs of Theorem~\ref{thm:iterationrate} and
Corollary~\ref{thm:complexity} use these two facts, HPE admissibility, and
the unchanged certificate sum, so only their $\sigma$-dependent constants
move. The same holds for the restarted strongly convex guarantees of
Appendix~\ref{app:strongly}. Part~(iii) below extends the
implementation guarantee to Algorithm~2.

\begin{theorem}[Implementations preserving the gradient bounds]
\label{thm:implementation}
Let Assumption~\ref{as:main} hold.
\begin{enumerate}[label=(\roman*)]
\item With the certified repair \eqref{eq:repair} in place of
$\Pi_{\mathcal Z}$, using the raw bounds specified above, all lemmas and
theorems of Section~\ref{sec:algo} hold for the repaired matrices with the
same statements and constants.
\item If in addition every linear system is solved up to
\eqref{eq:linearstop}, with $\delta=\sigma/4$, test value $\sigma/2$ and
floor $3\sigma/(4L_1)$ in Subroutine~1.1, then every step remains
HPE-admissible with its original parameter, and
Theorem~\ref{thm:iterationrate} and Corollary~\ref{thm:complexity} hold as
stated, with their $O_\sigma$ constants multiplied by absolute factors.

\item Under Assumption~\ref{as:sc}, the certified repair also
preserves the guarantees of Algorithm~2 with identical constants.
For inexact solves, impose \eqref{eq:linearstop} with
$\delta=\sigma/4$ and run Subroutine~2.1 with threshold $\sigma/2$
and endpoints $\underline\lambda_{\rm in}$,
$\bar\lambda_{k,{\rm in}}$ defined by
$\; \underline\lambda_{\rm in}L_1
 =\frac{\sigma}{2}\sqrt{1+\mu\underline\lambda_{\rm in}},
 \;
 \bar\lambda_{k,{\rm in}}\bar E_k
 =\frac{\sigma}{2}\sqrt{1+\mu\bar\lambda_{k,{\rm in}}}.$
Keep the original $\sigma$ in the certificate and outer update.
Then Theorem~\ref{thm:scmain} and Corollary~\ref{cor:scqueries}
retain their convergence and gradient-complexity orders,
with constants modified by absolute factors.
\end{enumerate}
Neither modification queries the gradient oracle.
\end{theorem}
\noindent The proof is in Appendix~\ref{proof:implementation}. In
particular, exact linear algebra is not what removes the gradient line
search: the scalar certificate guarantees acceptance before the trial
point is queried, and inexact solves leave this mechanism intact. As for
meeting \eqref{eq:linearstop}, Lemma~\ref{lem:chebyshev} gives a Chebyshev
iteration that does so within
\begin{equation}\label{eq:lineariterations}
m\le1+\Big\lceil\frac{\sqrt{\kappa_B}}{2}
\log\frac{2(\kappa_B+\delta)}{\delta}\Big\rceil,
\qquad \kappa_B=1+\lambda L_B,
\end{equation}
matrix-vector products in exact arithmetic, where $L_B\ge\|B\|_{\rm op}$
is known, $L_B=L_1$ for a feasible model.

\subsection{Algorithm 3: no projection and the trust-region subproblem}\label{sec:ncimplementation}

\paragraph{What the analysis uses.}
As for Algorithms~1 and~2, we separate what the analysis uses from what
the algorithm box computes. The model enters the proofs through two
properties only: $B_n$ is symmetric, and it satisfies the pointwise
certificate \eqref{eq:nccert}. The increment enters through one property
only: it is a $\nu$-approximate solution of the subproblem in the sense of
\eqref{eq:nckkt}, with $\nu\le\varepsilon/8$. Two consequences follow.

\paragraph{No projection.}
Nothing requires $B_n$ to lie in a prescribed set, so the spectral
projection of Section~\ref{sec:algo}, and the repair that replaces it in
Section~\ref{sec:cvximplementation}, have no counterpart here; the read-out
is the symmetrization, which costs $O(d^2)$, or $O(|\mathcal S|)$ when the
pattern is exploited. 
Should a bound $\|B_n\|_{\rm op}\le L_1$ be wanted
for other reasons, the scaling
$\frac{L_1}{L_1+e_n}B_n$, with $e_n$ the
right-hand side of \eqref{eq:ncpointwise},
provides it. Since
$\|B_n\|_{\rm op}\le L_1+e_n$,
the scaled matrix has operator norm at most
$L_1$ and error at most $2e_n$ at $z_n$.
This permits a certificate twice as large.
Applying Theorem~\ref{thm:ncmaster} with
$\omega=2\chi$ and retuning $D$, $\eta$ and $T$
increases the leading complexity term
by a factor $2^{1/3}$.
Neither a
projection nor a separation oracle is needed; the separation oracle of
\citet{jiang2025improved} serves their matrix learner, which has no
analogue here.

\paragraph{The subproblem.}
The trust-region problem $\min_{\|\Delta\|\le D}\phi_n$ in the matrix
$A_n=\frac12B_n+\eta^{-1}I$ is the one computation of the model variant
that is not a vector operation, and it is intrinsic to the method: with
$\eta\asymp1/(\chi L_2D)$, which is large compared with $1/L_1$ when the
model pays, $A_n$ is in general indefinite, and the implicit optimistic
update is a quasi-Newton trust-region step in the stored model. On the
range $\chi L_2D\le L_1/4$ where the model variant is active,  choosing $\eta=1/L_1$ would make $A_n\succeq\frac14L_1I$ and the
fixed-point iteration of \eqref{eq:nclearner} a contraction, but also
return the regret to $4KD^2L_1$, that is, to the rate of the gradient
variant.  The subproblem has the same algebraic form as that of
\citet{jiang2025improved}, so the same solution methods apply.
\begin{enumerate}[label=(\roman*),leftmargin=2.2em]
\item \emph{Exact linear algebra.} Under an idealized real-arithmetic
model, one symmetric eigen-decomposition of $B_n$ per iteration separates
the problem in the eigenbasis; the multiplier $\varrho$ of
Lemma~\ref{lem:nctrust} is then found from a scalar secular equation and
the increment from a diagonal solve, with standard handling of the
singular boundary (``hard'') case in the minimal eigenspace
\citep{more1983computing}. This costs $O(d^3)$ arithmetic and $O(d^2)$
memory per iteration and gives $\nu=0$.
\item \emph{Matrix-vector products with a certificate.} Since only a
stationary point is required, projected gradient on $\phi_n$ suffices, and
its stopping rule certifies \eqref{eq:nckkt} without any eigenvalue
computation and without the hard case of the trust-region problem.
\end{enumerate}

\begin{lemma}[Projected gradient on the subproblem]\label{lem:ncsolver}
Let $A_n=\frac12B_n+\eta^{-1}I$, let $0<\alpha\le1/\|A_n\|_{\rm op}$, and
iterate, from $\Delta^{(0)}=\Delta_n$,
\begin{equation}\label{eq:ncpg}
\Delta^{(t+1)}=\Pi_D\big(\Delta^{(t)}-\alpha\nabla\phi_n(\Delta^{(t)})\big),
\qquad
\nabla\phi_n(\Delta)=\nabla f(z_n)+\theta_n+A_n(\Delta-\Delta_n).
\end{equation}
If $\|\Delta^{(t+1)}-\Delta^{(t)}\|\le\alpha\nu/2$, then $\Delta^{(t+1)}$
is a $\nu$-approximate solution of the subproblem at step $n$. For
$\nu>0$, this stopping rule is met within
$1+\big\lfloor8\big(\phi_n(\Delta_n)-\min_{\|\Delta\|\le D}\phi_n\big)/(\alpha\nu^2)\big\rfloor$
iterations, and
$\phi_n(\Delta_n)-\min_{\|\Delta\|\le D}\phi_n
\le2D\|\nabla f(z_n)+\theta_n\|+\|B_n\|_{\rm op}D^2$.
\end{lemma}
\noindent The proof is in Appendix~\ref{proof:ncsolver}. Each iteration of
\eqref{eq:ncpg} costs one product with $B_n$, $O(d^2)$ operations in the
dense case and $O(|\mathcal S|+d)$ with the pattern, and the bound
$\|A_n\|_{\rm op}\le\frac12\|B_n\|_{\rm op}+\eta^{-1}$ needed for $\alpha$
is available from Lemma~\ref{lem:ncpointwise}. The count is $O(\nu^{-2})$
rather than logarithmic; a Krylov trust-region solver improves this to
$\tilde O(\nu^{-1/2})$ at the price of an approximate extreme eigenvector.
The solver of \citet[Lemma~4.4]{jiang2025improved} returns a
$\nu$-approximate solution with $\tilde O(\sqrt{\mathcal B_nD/\nu})$
matrix-vector products with probability at least $1-q$ per call, where
$\mathcal B_n\ge\max\{\lambda_{\max}(A_n)-\lambda_{\min}(A_n),\lambda_{\max}(A_n)\}$,
and $\mathcal B_n=\|B_n\|_{\rm op}+\eta^{-1}$ is admissible; a union bound
over the $N$ calls, with $q=q_0/N$, is absorbed into the logarithmic
factor. With $\nu=\varepsilon/8$ and the parameters of
Theorem~\ref{thm:ncmodel} on the range
$\chi\le(L_1^2/(L_2\varepsilon))^{3/4}$ of Corollary~\ref{cor:ncenvelope},
Lemma~\ref{lem:ncpointwise} gives $\mathcal B_n\le3L_1$, so each iteration
uses $\tilde O\big(L_1^{1/2}L_2^{-1/4}\chi^{-1/6}\varepsilon^{-1/4}\big)$
products, and the run uses
$\tilde O\big((\Phi_0L_1^{1/2}L_2^{1/4}\varepsilon^{-3/2}+L_1^{1/2}L_2^{-1/4})\chi^{1/6}\varepsilon^{-1/4}\big)$,
which for fixed positive constants is $\tilde O(\chi^{1/6}\varepsilon^{-7/4})$.
In the gradient variant there is no subproblem: the update is the explicit
projection of Algorithm~3, step~2(b).

\subsection{Total arithmetic cost}\label{sec:arithmetic}
Let $c_g$ be the cost of one gradient and $N$ the number of iterations.
With exact linear algebra, Algorithms~1 and~2 diagonalize the symmetrized
model once per iteration, clip its eigenvalues, and apply the shifted
inverse in the same eigenbasis, and the model variant of Algorithm~3
diagonalizes $B_n$ once per iteration and solves the secular equation.
With the repair (or the symmetrization) and an iterative solver that uses
$m_k$ matrix-vector products at iteration $k$, every matrix operation
costs $O(d^2)$:
\[
\begin{array}{lll}
&\text{exact linear algebra}&\text{iterative solves}\\[2pt]
\text{Algorithms~1 and~2:}&
O\big((3N+1)c_g+Nd^3\big)&
O\big((3N+1)c_g+d^2\textstyle\sum_{k<N}(1+m_k)\big)\\[2pt]
\text{Algorithm~3, model variant:}&
O\big((3N+\chi)c_g+Nd^3\big)&
O\big((3N+\chi)c_g+d^2\textstyle\sum_{k<N}(1+m_k)\big)
\end{array}
\]
The gradient variant of Algorithm~3 stores no matrix and costs
$O(2Nc_g+Nd)$. All other variants use $O(d^2)$ memory, and the scalar
searches of Subroutines~1.1 and~2.1 add $O(\log\log k)$ evaluations of
stored scalar formulas per iteration (Lemma~\ref{lem:wellposed} and
Theorem~\ref{thm:scmain}). No variant queries the oracle outside the
counts of Theorems~\ref{thm:iterationrate}, \ref{thm:scmain},
\ref{thm:ncmodel} and~\ref{thm:ncsmooth}. Because spectral clipping and
the changing regularization shifts $\lambda^{-1}I$ complicate the reuse of
factorizations across steps, we do not assume an amortized factorization
cost, and the iterative variant is the natural one; its per-solve counts
are \eqref{eq:lineariterations} for the linear solves and
Lemma~\ref{lem:ncsolver} or a Krylov solver for the subproblem. The
condition numbers $\kappa_B$ in \eqref{eq:lineariterations} need not stay
bounded, since $\lambda_k$ grows as the rolling certificate shrinks; the
counts above therefore do not assert an unconditional arithmetic
improvement over factorization, or over the matrix-vector complexity of
A-QNPE and OL-OQN. They offer a choice between direct solves and certified
iterative solves, with the same gradient guarantees in either case.

\paragraph{Structured Hessians.}
Grouping reduces gradient queries, not dense matrix work: the spectral
projection of a sparse matrix is dense in general, so with exact
projection the arithmetic bounds above are unchanged. With the certified
repair \eqref{eq:repair}, however, the raw matrix keeps the pattern
$\mathcal S$ of Definition~\ref{def:chi}, since every stored column
$P_ju_k$ is supported on $\mathcal S_j$, and the repair adds only a
multiple of the identity, so $B_k$ has at most $|\mathcal S|+d$ nonzeros;
the symmetrization of Algorithm~3 keeps the pattern as well. A
matrix-vector product then costs $O(|\mathcal S|+d)$, each iterative solve
costs $O((|\mathcal S|+d)\,m)$ in place of $O(d^2m)$ with the same
iteration count $m$, and storage is $O(|\mathcal S|+d)$.

\section{Conclusion}
We introduced certified curvature reuse for quasi-Newton methods: the
curvature model carries an operator-norm error bound that is available
before each step, so the bound itself fixes the stepsize, no gradient line search is needed and no gradient is spent on rejected trials, and the error of reused curvature is charged to the path length of the iterates. This gives a convex rate of
$O(\min\{k^{-2},\chi k^{-3}\})$; {\color{ReviewBlue} in the strongly convex case, an accelerated linear guarantee,
with curvature-based improvement after $O(B_\sharp/\kappa)$
iterations for sufficiently large $\kappa$, and an eventual
superlinear upper envelope;} and in the nonconvex case, where the
certificate enters the hint of an optimistic learner,
$O(\min\{\chi^{1/3}\varepsilon^{-3/2},\varepsilon^{-7/4}\})$ gradient
queries. These bounds improve on QNPE, A-QNPE and OL-OQN, and depend on
the dimension only through the effective dimension $\chi\le d$, so a known
Hessian sparsity pattern directly reduces the number of gradient queries. 

\bibliographystyle{abbrvnat}

\bibliography{refs}

\clearpage
\appendix
\section{Shared analytical tools}\label{app:backbone}
This appendix proves the general estimates used in the rest of the paper.
We first establish the A-HPE statements collected in
Proposition~\ref{prop:hpeestimates}, without assuming global smoothness.
We then prove the finite-difference, clipping, and residual certificates
used to construct admissible steps. The algorithm-specific arguments in
Appendix~\ref{app:rolling} invoke these tools with the objective $f$ and
the radius bounds of Section~\ref{sec:algo}; Appendix~\ref{app:nonconvex}
reuses the finite-difference certificate for the nonconvex method.

\subsection{The accelerated hybrid proximal extragradient backbone}\label{sec:ahpe}

Let $F$ be a convex differentiable function with a minimizer $x^\star$,
write $F^\star=F(x^\star)$, and assume $\|x_0-x^\star\|\le R$.
Use the undamped recursion of Section~\ref{sec:prelim}, with $F$ in
place of $f$, $A_0=0$, $x_0=z_0$, and positive stepsizes. Every step is
assumed to satisfy
\[
\|r_k+\lambda_k\nabla F(x_{k+1})\|\le\sigma\|r_k\|,
\qquad 0<\sigma<1.
\]
We retain the geometric notation of that section and write
$\nabla F_{k+1}=\nabla F(x_{k+1})$. No smoothness constant is used in
this subsection. We first prove the potential and mass-growth estimates,
then the movement and boundedness estimates that complete
Proposition~\ref{prop:hpeestimates}.
\begin{lemma}[Potential decrease; cf.~\citet{monteiro2013accelerated}]\label{lem:potential}
Define $\Psi_k:=A_k\big(F(x_k)-F^\star\big)+\tfrac12\|z_k-x^\star\|^2$. Then for every $k\ge0$,
\[
\Psi_{k+1}\ \le\ \Psi_k-\frac{1-\sigma^2}{2}\cdot\frac{A_{k+1}}{\lambda_k}\,\|r_k\|^2 .
\]
Consequently, for $N\ge1$, with $A_0=0$ and $\Psi_0\le\tfrac12R^2$,
\[
\textbf{(i)}\ \ F(x_N)-F^\star\le\frac{R^2}{2A_N},\qquad \textbf{(ii)}\ \ \|z_k-x^\star\|\le R\ \ \forall k,\qquad \textbf{(iii)}\ \sum_{k=0}^{N-1}\frac{A_{k+1}}{\lambda_k}\|r_k\|^2\le\frac{R^2}{1-\sigma^2}.
\]
\end{lemma}

\begin{proof}
Expand the momentum step:
\begin{equation}
\tfrac12\|z_{k+1}-x^\star\|^2=\tfrac12\|z_k-x^\star\|^2-a_k\langle \nabla F_{k+1},z_k-x^\star\rangle+\tfrac{a_k^2}{2}\|\nabla F_{k+1}\|^2. \label{eq:momentumexpand}
\end{equation}
Split $z_k-x^\star=(z_k-x_{k+1})+(x_{k+1}-x^\star)$. By convexity, $\langle \nabla F_{k+1},x_{k+1}-x^\star\rangle\ge F(x_{k+1})-F^\star$, hence
\begin{equation}
a_k\langle \nabla F_{k+1},z_k-x^\star\rangle\ \ge\ a_k\langle \nabla F_{k+1},z_k-x_{k+1}\rangle+a_k\big(F(x_{k+1})-F^\star\big). \label{eq:convexfirst}
\end{equation}
From $y_k=\frac{A_kx_k+a_kz_k}{A_{k+1}}$ we get $a_kz_k=A_{k+1}y_k-A_kx_k$, so, using $A_{k+1}=A_k+a_k$,
\[
a_k(z_k-x_{k+1})=A_{k+1}y_k-A_kx_k-a_kx_{k+1}=A_{k+1}(y_k-x_{k+1})+A_k(x_{k+1}-x_k).
\]
By convexity again, $\langle \nabla F_{k+1},x_{k+1}-x_k\rangle\ge F(x_{k+1})-F(x_k)$, so
\begin{equation}
a_k\langle \nabla F_{k+1},z_k-x_{k+1}\rangle\ \ge\ A_{k+1}\langle \nabla F_{k+1},y_k-x_{k+1}\rangle-A_k\big(F(x_k)-F(x_{k+1})\big). \label{eq:convexsecond}
\end{equation}
Insert \eqref{eq:convexfirst}--\eqref{eq:convexsecond} into \eqref{eq:momentumexpand} and use the identity (valid since $A_{k+1}=A_k+a_k$)
\[
A_k\big(F(x_k)-F(x_{k+1})\big)-a_k\big(F(x_{k+1})-F^\star\big)=A_k\big(F(x_k)-F^\star\big)-A_{k+1}\big(F(x_{k+1})-F^\star\big),
\]
which yields
\[
\Psi_{k+1}\ \le\ \Psi_k-A_{k+1}\langle \nabla F_{k+1},y_k-x_{k+1}\rangle+\tfrac{a_k^2}{2}\|\nabla F_{k+1}\|^2 .
\]
Substitute $a_k^2=\lambda_kA_{k+1}$ and set $v:=\lambda_k^{-1}(y_k-x_{k+1})=-\lambda_k^{-1}r_k$. The last two terms equal
\[
A_{k+1}\lambda_k\Big(-\langle \nabla F_{k+1},v\rangle+\tfrac12\|\nabla F_{k+1}\|^2\Big)=\frac{A_{k+1}\lambda_k}{2}\Big(\|\nabla F_{k+1}-v\|^2-\|v\|^2\Big).
\]
Condition \eqref{eq:ms} reads $\lambda_k\|\nabla F_{k+1}-v\|\le\sigma\lambda_k\|v\|$, so the bracket is $\le(\sigma^2-1)\|v\|^2=-(1-\sigma^2)\|r_k\|^2/\lambda_k^2$, giving the claim.

For the consequences: telescoping from $k=0$ to $N-1$ with $A_0=0$, $x_0=z_0$,
\[
A_N\big(F(x_N)-F^\star\big)+\tfrac12\|z_N-x^\star\|^2+\frac{1-\sigma^2}{2}\sum_{k<N}\frac{A_{k+1}}{\lambda_k}\|r_k\|^2\ \le\ \tfrac12R^2,
\]
and each of (i)--(iii) follows by discarding the other nonnegative terms.
\end{proof}

\begin{lemma}[Growth of $A_N$]\label{lem:agrowth}
$\displaystyle \sqrt{A_N}\ \ge\ \tfrac12\sum_{k=0}^{N-1}\sqrt{\lambda_k}$, hence $A_N\ge\tfrac14\big(\sum_{k=0}^{N-1}\sqrt{\lambda_k}\big)^2$.
\end{lemma}

\begin{proof}
By \eqref{eq:aident}, $A_{k+1}-A_k=a_k=\sqrt{\lambda_kA_{k+1}}$, so
\[
\sqrt{A_{k+1}}-\sqrt{A_k}=\frac{A_{k+1}-A_k}{\sqrt{A_{k+1}}+\sqrt{A_k}}=\frac{\sqrt{\lambda_kA_{k+1}}}{\sqrt{A_{k+1}}+\sqrt{A_k}}\ \ge\ \frac{\sqrt{\lambda_kA_{k+1}}}{2\sqrt{A_{k+1}}}=\frac{\sqrt{\lambda_k}}{2},
\]
using $A_k\le A_{k+1}$. Telescoping and $A_0=0$ give the claim.
\end{proof}

\begin{lemma}[Moment lemma]\label{lem:moment}
Let $q>0$ and $S_q^{(n)}:=\sum_{k=0}^{n-1}\lambda_k^{-q}$. Then for every $n\ge1$,
\[
A_n\ \ge\ \tfrac14\,n^{\,2+1/q}\,\big(S_q^{(n)}\big)^{-1/q}.
\]
\end{lemma}

\begin{proof}
Put $p_1:=\frac{2q+1}{2q}$, $p_2:=2q+1$; then $\frac1{p_1}+\frac1{p_2}=\frac{2q}{2q+1}+\frac1{2q+1}=1$, so $(p_1,p_2)$ are conjugate. Set $\alpha_k:=\lambda_k^{\,q/(2q+1)}$, $\beta_k:=\lambda_k^{-q/(2q+1)}$, so $\alpha_k\beta_k=1$ and
\[
\alpha_k^{p_1}=\lambda_k^{\frac{q}{2q+1}\cdot\frac{2q+1}{2q}}=\lambda_k^{1/2},\qquad \beta_k^{p_2}=\lambda_k^{-q}.
\]
H\"older gives $n=\sum_{k<n}\alpha_k\beta_k\le\big(\sum_{k<n}\sqrt{\lambda_k}\big)^{\frac{2q}{2q+1}}\big(S_q^{(n)}\big)^{\frac1{2q+1}}$. Raising to the power $\frac{2q+1}{2q}$,
\[
\sum_{k<n}\sqrt{\lambda_k}\ \ge\ n^{\frac{2q+1}{2q}}\big(S_q^{(n)}\big)^{-\frac1{2q}} .
\]
Square and apply Lemma~\ref{lem:agrowth}. The sequence form of H\"older used here follows by normalizing both sequences and summing $uv\le u^{p_1}/p_1+v^{p_2}/p_2$. The scalar inequality follows by maximizing $uv-u^{p_1}/p_1$ over $u\ge0$, whose maximizer is $u=v^{1/(p_1-1)}$.
\end{proof}

\begin{corollary}\label{cor:moment2}
With $S_1:=S_1^{(N)}$ and $q=1$: for every $n\le N$, $\ A_n\ \ge\ \dfrac{n^3}{4S_1}$, since $S_1^{(n)}\le S_1$.
\end{corollary}
\begin{proof}Apply Lemma~\ref{lem:moment} with $q=1$ and use $S_1^{(n)}\le S_1$.\end{proof}

\begin{lemma}[Logarithmic sum]\label{lem:logsum}
Assume $A_1>0$. Then with $\Lambda_N:=1+\log\frac{A_N}{A_1}$,
\[
\sum_{k=0}^{N-1}\frac{a_k}{A_{k+1}}\ \le\ \Lambda_N,\qquad \sum_{k=0}^{N-1}\Big(\frac{a_k}{A_{k+1}}\Big)^2\ \le\ \Lambda_N .
\]
\end{lemma}

\begin{proof}
Since $A_0=0$, $a_0=A_1$ and $a_0/A_1=1$. For $k\ge1$, $0<A_k<A_{k+1}$ and
\[
\frac{a_k}{A_{k+1}}=\int_{A_k}^{A_{k+1}}\frac{dt}{A_{k+1}}\le\int_{A_k}^{A_{k+1}}\frac{dt}{t}=\log\frac{A_{k+1}}{A_k}.
\]
Summing gives $\le1+\log(A_N/A_1)$. Each term lies in $[0,1]$ since $a_k\le A_{k+1}$, so squares are dominated termwise.
\end{proof}

\begin{lemma}[Summability of the steps]\label{lem:stepsum}
\[
\sum_{k=0}^{N-1}\|r_k\|\ \le\ \frac{R}{\sqrt{1-\sigma^2}}\,\sqrt{\Lambda_N}.
\]
\end{lemma}

\begin{proof}
By Cauchy--Schwarz,
\[
\sum_k\|r_k\|=\sum_k\Big(\sqrt{\tfrac{A_{k+1}}{\lambda_k}}\,\|r_k\|\Big)\Big(\sqrt{\tfrac{\lambda_k}{A_{k+1}}}\Big)\ \le\ \Big(\sum_k\tfrac{A_{k+1}}{\lambda_k}\|r_k\|^2\Big)^{1/2}\Big(\sum_k\tfrac{\lambda_k}{A_{k+1}}\Big)^{1/2}.
\]
The first factor is $\le R/\sqrt{1-\sigma^2}$ by Lemma~\ref{lem:potential}(iii). For the second, \eqref{eq:aident} gives $\frac{\lambda_k}{A_{k+1}}=\big(\frac{a_k}{A_{k+1}}\big)^2$, and Lemma~\ref{lem:logsum} bounds the sum by $\Lambda_N$.
\end{proof}

\begin{lemma}[Proximal nonexpansiveness]\label{lem:prox_nonexpansive}
Let $F:\mathbb{R}^d\to\mathbb{R}$ be a convex continuously differentiable function and $\lambda>0$. The proximal operator $\operatorname{prox}_{\lambda F}$ is nonexpansive.
\end{lemma}

\begin{proof}
Let $x_1, x_2 \in \mathbb{R}^d$ and define $p_1 = \operatorname{prox}_{\lambda F}(x_1)$ and $p_2 = \operatorname{prox}_{\lambda F}(x_2)$. 
The quadratic term makes the proximal objective coercive and strictly convex: a convex differentiable function has an affine lower bound, which the quadratic dominates. Thus each proximal minimizer exists uniquely. Its first-order condition gives $p_1 + \lambda \nabla F(p_1) = x_1$ and $p_2 + \lambda \nabla F(p_2) = x_2$.
Subtracting the two equations and taking the inner product with $p_1 - p_2$ yields
\[
\|p_1 - p_2\|^2 + \lambda \langle \nabla F(p_1) - \nabla F(p_2), p_1 - p_2 \rangle = \langle x_1 - x_2, p_1 - p_2 \rangle.
\]
By convexity of $F$, $\langle \nabla F(p_1) - \nabla F(p_2), p_1 - p_2 \rangle \ge 0$. Thus,
\[
\|p_1 - p_2\|^2 \le \langle x_1 - x_2, p_1 - p_2 \rangle \le \|x_1 - x_2\| \|p_1 - p_2\|,
\]
which implies $\|p_1 - p_2\| \le \|x_1 - x_2\|$.
\end{proof}

\begin{lemma}[Boundedness of the iterates]\label{lem:bounded}
For all $k\ge0$,
\[
\|x_k-x^\star\|\ \le\ C_xR,\qquad C_x:=1+\frac{\sigma}{\sqrt{1-\sigma^2}},\qquad\text{hence}\qquad \zeta_k=\|z_k-x_k\|\ \le\ (1+C_x)R .
\]
$C_x$ depends only on $\sigma$, not on $N$, $\varepsilon$, $d$, or any logarithm.
\end{lemma}

\begin{proof}
\textit{Step 1 (one-step inexact-prox estimate).} Let $p_k:=\operatorname{prox}_{\lambda_k F}(y_k)$, the unique solution of $p+\lambda_k\nabla F(p)=y_k$, and $w_k:=x_{k+1}-y_k+\lambda_k\nabla F(x_{k+1})$, so $\|w_k\|\le\sigma\|r_k\|$ by \eqref{eq:ms}. Then $x_{k+1}+\lambda_k\nabla F(x_{k+1})=y_k+w_k$, i.e.\ $x_{k+1}=\operatorname{prox}_{\lambda_k F}(y_k+w_k)$. The proximal operator of a convex function is nonexpansive by Lemma~\ref{lem:prox_nonexpansive}, so
\[
\|x_{k+1}-p_k\|=\big\|\operatorname{prox}_{\lambda_k F}(y_k+w_k)-\operatorname{prox}_{\lambda_k F}(y_k)\big\|\le\|w_k\|\le\sigma\|r_k\| .
\]
Also $\nabla F(x^\star)=0$ gives $\operatorname{prox}_{\lambda_k F}(x^\star)=x^\star$, hence $\|p_k-x^\star\|\le\|y_k-x^\star\|$ and
\begin{equation}
\|x_{k+1}-x^\star\|\ \le\ \|y_k-x^\star\|+\sigma\|r_k\| . \label{eq:proxradius}
\end{equation}

\textit{Step 2 (weighted telescoping).} Since $y_k$ is a convex combination of $x_k$ and $z_k$,
\[
\|y_k-x^\star\|\le\frac{A_k}{A_{k+1}}\|x_k-x^\star\|+\frac{a_k}{A_{k+1}}\|z_k-x^\star\| .
\]
Substituting into \eqref{eq:proxradius} and multiplying by $A_{k+1}>0$,
\[
A_{k+1}\|x_{k+1}-x^\star\|\ \le\ A_k\|x_k-x^\star\|+a_k\|z_k-x^\star\|+\sigma A_{k+1}\|r_k\| .
\]
Summing over $k=0,\dots,N-1$, the first two families telescope (and $A_0=0$):
\[
A_N\|x_N-x^\star\|\ \le\ \sum_{j<N}a_j\|z_j-x^\star\|+\sigma\sum_{j<N}A_{j+1}\|r_j\| .
\]
By Lemma~\ref{lem:potential}(ii), $\|z_j-x^\star\|\le R$, and $\sum_ja_j=A_N$, so the first sum is at most $A_NR$.

\textit{Step 3 (the weighted step sum).} Using $a_j=\sqrt{\lambda_jA_{j+1}}$,
\[
A_{j+1}\|r_j\|=\Big(\sqrt{\tfrac{A_{j+1}}{\lambda_j}}\,\|r_j\|\Big)\cdot a_j,
\]
so by Cauchy--Schwarz and Lemma~\ref{lem:potential}(iii),
\[
\sum_{j<N}A_{j+1}\|r_j\|\ \le\ \frac{R}{\sqrt{1-\sigma^2}}\Big(\sum_ja_j^2\Big)^{1/2}\ \le\ \frac{R}{\sqrt{1-\sigma^2}}A_N,
\]
since $\sum_ja_j^2\le\big(\sum_ja_j\big)^2=A_N^2$ for nonnegative $a_j$. Combining and dividing by $A_N>0$ gives $\|x_N-x^\star\|\le C_xR$; $N$ was arbitrary. Finally $\zeta_k\le\|z_k-x^\star\|+\|x_k-x^\star\|\le(1+C_x)R$.
\end{proof}

\begin{lemma}[Interleaved path length]\label{lem:pathlen}
With $ m_k=\|y_k-x_k\|$ and $\delta_k= m_k+\|r_k\|$ define $Q_N:=\sum_{k=0}^{N-1}\delta_k$. {\color{ReviewBlue}(Throughout the convex analysis we write $m_k$ and $m(\lambda)$ for this displacement, as in Section~\ref{sec:algo}, reserving $\mu$ for the strong-convexity modulus.)} Then
\[
Q_N\ \le\ (1+C_x)R\,\Lambda_N+\frac{R\sqrt{\Lambda_N}}{\sqrt{1-\sigma^2}}\ =\ O_\sigma\!\big(R\,\Lambda_N\big).
\]
\end{lemma}

\begin{proof}
From $y_k=\frac{A_kx_k+a_kz_k}{A_{k+1}}$ and $A_{k+1}=A_k+a_k$,
\begin{equation}
y_k-x_k=\frac{a_k}{A_{k+1}}(z_k-x_k),\qquad\text{so}\qquad  m_k=\frac{a_k}{A_{k+1}}\zeta_k\le(1+C_x)R\,\frac{a_k}{A_{k+1}} \label{eq:ydiff}
\end{equation}
by Lemma~\ref{lem:bounded}. Summing and using Lemma~\ref{lem:logsum} bounds $\sum_k m_k\le(1+C_x)R\Lambda_N$; Lemma~\ref{lem:stepsum} bounds $\sum_k\|r_k\|$. The order statement uses that $C_x$ depends only on $\sigma$ and $\sqrt{\Lambda_N}\le\Lambda_N$ (as $\Lambda_N\ge1$).
\end{proof}

\paragraph{Recovering the preliminary proposition.}
Taking $F=f$ gives Proposition~\ref{prop:hpeestimates}: its potential
and energy estimates are Lemma~\ref{lem:potential}; its mass bounds are
Lemmas~\ref{lem:agrowth} and~\ref{lem:moment}; and its geometric bounds
are Lemma~\ref{lem:potential}(ii) and Lemmas~\ref{lem:bounded} and~\ref{lem:pathlen}.

\subsection{Curvature reconstruction and model-step certificates}
The following tools concern a single curvature model. Finite differences
control its raw approximation error, clipping restores spectral feasibility,
and the residual estimate converts model accuracy into HPE admissibility.
They do not prescribe when the model is refreshed.

\begin{lemma}[Finite-difference error]\label{lem:fd}
For a function $F$ with $\mathcal L$-Lipschitz Hessian, any $\tau>0$, and any unit vector $v$,
\[
\left\|\frac{\nabla F(y+\tau v)-\nabla F(y)}\tau
                  -\nabla^2F(y)v\right\|\le\frac{\mathcal L\tau}{2}.
\]
\end{lemma}
\begin{proof}
Fix $\tau>0$ and consider the map $s\mapsto\nabla F(y+s\tau v)$ on $[0,1]$.
It is continuously differentiable, since $\nabla^2F$ is Lipschitz and hence
continuous, with derivative $\tau\nabla^2F(y+s\tau v)v$. By the fundamental
theorem of calculus,
\[
\nabla F(y+\tau v)-\nabla F(y)=\tau\int_0^1\nabla^2F(y+s\tau v)v\,ds .
\]
Dividing by $\tau$ and writing $\nabla^2F(y)v=\int_0^1\nabla^2F(y)v\,ds$,
the finite-difference error becomes
\[
\frac{\nabla F(y+\tau v)-\nabla F(y)}{\tau}-\nabla^2F(y)v
=\int_0^1\big(\nabla^2F(y+s\tau v)-\nabla^2F(y)\big)v\,ds .
\]
For each $s\in[0,1]$, the Lipschitz property of the Hessian gives
$\|\nabla^2F(y+s\tau v)-\nabla^2F(y)\|_{\rm op}\le\mathcal L\|s\tau v\|
=\mathcal Ls\tau$, because $\|v\|=1$. Taking the norm of the integral,
bounding it by the integral of the norm, and using this estimate,
\[
\left\|\frac{\nabla F(y+\tau v)-\nabla F(y)}{\tau}-\nabla^2F(y)v\right\|
\le\int_0^1\mathcal Ls\tau\,ds=\frac{\mathcal L\tau}{2}. \qedhere
\]
\end{proof}

\begin{proof}[Proof of Lemma~\ref{lem:group}]
Write $H=H(x)$ and $\epsilon:=w-Hv_g$. The identity in the proof of
Lemma~\ref{lem:fd} holds for any direction, so with $F=f$ and $v=v_g$,
$\epsilon=\int_0^1\big(H(x+s\tau v_g)-H\big)v_g\,ds$, and
Assumption~\ref{as:main} gives
$\|\epsilon\|\le\int_0^1L_2s\tau\|v_g\|^2\,ds=L_2\tau|g|/2$, because
$\|v_g\|^2=|g|$. Fix $j\in g$. Since $\operatorname{supp}H\subseteq\mathcal S$,
the column $He_{j'}$ is supported on $\mathcal S_{j'}$ for every $j'$. For
$j'\in g\setminus\{j\}$ the sets $\mathcal S_{j'}$ and $\mathcal S_j$ are
disjoint by structural orthogonality, so $P_jHe_{j'}=0$, while
$P_jHe_j=He_j$. Hence
$P_jw=P_jHv_g+P_j\epsilon=\sum_{j'\in g}P_jHe_{j'}+P_j\epsilon=He_j+P_j\epsilon$,
that is, $P_jw-He_j=P_j\epsilon$. The projections $P_j$, $j\in g$, have
pairwise disjoint supports, so
$\sum_{j\in g}\|P_jw-He_j\|^2=\sum_{j\in g}\|P_j\epsilon\|^2\le\|\epsilon\|^2
\le(L_2\tau|g|/2)^2$.
\end{proof}

\begin{lemma}[Clipping is $2$-stable in operator norm]\label{lem:clip}
If $H\in\mathcal Z$ and $\|S-H\|_{\rm op}\le e$ for symmetric $S$, then
$\|\Pi_{\mathcal Z}(S)-H\|_{\rm op}\le2e$.
The same conclusion holds for clipping to the PSD cone when $H\succeq0$.
Also, projection onto $\mathcal Z$ is nonexpansive in Frobenius norm.
\end{lemma}

\begin{proof}
Let $S=U\operatorname{diag}(s_1,\dots,s_d)U^\top$ be an eigendecomposition,
so that $\Pi_{\mathcal Z}(S)=U\operatorname{diag}(\hat s_1,\dots,\hat s_d)U^\top$
with $\hat s_i=\min\{\max\{s_i,0\},L_1\}$. Since $0\preceq H\preceq L_1I$ and
$-eI\preceq S-H\preceq eI$, adding the two relations gives
$-eI\preceq S\preceq(L_1+e)I$, so every $s_i$ lies in $[-e,L_1+e]$. Clipping
leaves $s_i\in[0,L_1]$ unchanged, sends $s_i\in[-e,0)$ to $0$, and sends
$s_i\in(L_1,L_1+e]$ to $L_1$; in every case $|s_i-\hat s_i|\le e$. Hence
$S-\Pi_{\mathcal Z}(S)=U\operatorname{diag}(s_i-\hat s_i)U^\top$ has operator
norm $\max_i|s_i-\hat s_i|\le e$, and the triangle inequality yields
$\|\Pi_{\mathcal Z}(S)-H\|_{\rm op}\le\|\Pi_{\mathcal Z}(S)-S\|_{\rm op}
+\|S-H\|_{\rm op}\le2e$. For clipping to the PSD cone, $\hat s_i=\max\{s_i,0\}$
and the same argument uses only $s_i\ge-e$, which follows from $H\succeq0$.

For the last claim, we first check that $\Pi_{\mathcal Z}$ is the Frobenius
projection onto $\mathcal Z$, a closed convex set. Let $Y\in\mathcal Z$ have
eigenvalues $y_1\ge\dots\ge y_d$, all in $[0,L_1]$, and order the $s_i$
decreasingly. The Hoffman--Wielandt inequality for symmetric matrices
\citep[Cor.~6.3.8]{horn2013matrix} gives
$\|S-Y\|_F^2\ge\sum_i(s_i-y_i)^2\ge\sum_i(s_i-\hat s_i)^2
=\|S-\Pi_{\mathcal Z}(S)\|_F^2$, where the second inequality holds because
$\hat s_i$ is the point of $[0,L_1]$ nearest to $s_i$ and $y_i\in[0,L_1]$.
Since $\Pi_{\mathcal Z}(S)\in\mathcal Z$, it is the nearest point of
$\mathcal Z$ to $S$ in Frobenius norm. Metric projections onto nonempty
closed convex subsets of a Hilbert space are nonexpansive
\citep[Prop.~4.16]{bauschke2017convex}, which proves the claim.
\end{proof}

\begin{lemma}[Two elementary certificates]\label{lem:resid}
For a convex differentiable function $F$ with $\mathcal L$-Lipschitz Hessian,
$B\succeq0$, $\lambda>0$, and
$r=-(B+\lambda^{-1}I)^{-1}\nabla F(y)$,
\[
\|r+\lambda\nabla F(y+r)\|
\le\lambda\left(\|B-\nabla^2F(y)\|_{\rm op}
                    +\frac{\mathcal L}{2}\|r\|\right)\|r\|.
\]
If in addition $0\preceq\nabla^2F(v),B\preceq LI$ everywhere on
$[y,y+r]$, then
\[
\|r+\lambda\nabla F(y+r)\|\le\lambda L\|r\|.
\]
\end{lemma}
\begin{proof}
Since $B\succeq0$, the matrix $B+\lambda^{-1}I$ is invertible, and the
definition of $r$ reads $(B+\lambda^{-1}I)r=-\nabla F(y)$, that is,
$r+\lambda\nabla F(y)=-\lambda Br$. By the fundamental theorem of calculus
applied to $t\mapsto\nabla F(y+tr)$,
\[
\nabla F(y+r)=\nabla F(y)+\bar Hr,
\qquad
\bar H:=\int_0^1\nabla^2F(y+tr)\,dt ,
\]
where $\bar H$ is symmetric. Combining the two identities,
\[
r+\lambda\nabla F(y+r)
=\big(r+\lambda\nabla F(y)\big)+\lambda\bar Hr
=\lambda(\bar H-B)r ,
\]
so $\|r+\lambda\nabla F(y+r)\|\le\lambda\|\bar H-B\|_{\rm op}\|r\|$, and
it remains to bound $\|\bar H-B\|_{\rm op}$ in two ways. For the first
bound, write $\bar H-B=\big(\bar H-\nabla^2F(y)\big)+\big(\nabla^2F(y)-B\big)$.
Since $\bar H-\nabla^2F(y)=\int_0^1\big(\nabla^2F(y+tr)-\nabla^2F(y)\big)\,dt$
and $\|\nabla^2F(y+tr)-\nabla^2F(y)\|_{\rm op}\le\mathcal Lt\|r\|$ for
$t\in[0,1]$ by the Lipschitz property of the Hessian,
\[
\|\bar H-\nabla^2F(y)\|_{\rm op}\le\int_0^1\mathcal Lt\|r\|\,dt
=\frac{\mathcal L}{2}\|r\| ,
\]
and the triangle inequality gives
$\|\bar H-B\|_{\rm op}\le\|B-\nabla^2F(y)\|_{\rm op}+\frac{\mathcal L}{2}\|r\|$,
which is the first certificate. 

For the second, suppose
$0\preceq\nabla^2F(v)\preceq LI$ for all $v\in[y,y+r]$. For any $u$,
$\langle u,\bar Hu\rangle=\int_0^1\langle u,\nabla^2F(y+tr)u\rangle\,dt$
lies in $[0,L\|u\|^2]$, so $0\preceq\bar H\preceq LI$. Together with
$0\preceq B\preceq LI$ this gives $-LI\preceq\bar H-B\preceq LI$. Hence, $\|\bar H-B\|_{\rm op}\le L$ and
$\|r+\lambda\nabla F(y+r)\|\le\lambda L\|r\|$.
\end{proof}


\section{Analysis of RP-QNPE}\label{app:rolling}

This appendix proves the results of Section~\ref{sec:algo} in the
order in which they are stated. Appendices~\ref{proof:lem:telescope}
and~\ref{proof:main:lem:sum} prove the rolling certificate of
Lemma~\ref{lem:telescope} and its summed form, Lemma~\ref{lem:sum}.
Appendix~\ref{sec:schedule} records two facts about the scalar data of
Subroutine~1.1, the monotonicity of the displacement $ m(\lambda)$ and the
resolvent bound on $\|r(\lambda)\|$, from which Lemma~\ref{lem:scalarbounds}
and the acceptance guarantee of Lemma~\ref{lem:wellposed} follow, and
inverts the near-maximality guarantee into the pointwise bound of
Lemma~\ref{lem:fiveterm}. Appendix~\ref{proof:iterationrate} proves
Theorem~\ref{thm:iterationrate} by applying the A-HPE estimates of
Appendix~\ref{app:backbone} with $F=f$ on each dyadic scale of the mass
$A_k$, bounding the number of iterations in each scale, and summing these
counts over the scales, and Appendix~\ref{proof:thm:complexity} converts
the resulting mass bound into the gradient count of
Corollary~\ref{thm:complexity}.

\subsection{The rolling curvature certificate}\label{sec:model}
\subsubsection{Proof of Lemma~\ref{lem:telescope}}
\label{proof:lem:telescope}
The argument below establishes the certificate \eqref{eq:rollingcert}.

\begin{proof}
Write $H_j:=H(x_j)$. By Assumption~\ref{as:main}, $f$ is convex with
$L_1$-Lipschitz gradient and $L_2$-Lipschitz Hessian, so
$H_j\in\mathcal Z=\{B\in\mathbb S^d:0\preceq B\preceq L_1I\}$ and
$\|H_{j+1}-H_j\|_{\rm op}\le L_2\|x_{j+1}-x_j\|$. Since $x_{j+1}=y_j+r_j$ and
$\delta_j= m_j+\|r_j\|$ with $ m_j=\|y_j-x_j\|$, we have
$\|x_{j+1}-x_j\|\le\delta_j$, hence for every $j\ge0$
\begin{equation}\label{eq:hessstep}
\|H_{j+1}-H_j\|_{\rm op}\le L_2\delta_j .
\end{equation}

The expressions $B_k=\Pi_{\mathcal Z}(\cdot)\in\mathcal Z$ and
$H_k\in\mathcal Z$ give $-L_1I\preceq B_k-H_k\preceq L_1I$, so
$\|B_k-H_k\|_{\rm op}\le L_1$ for all $k\ge0$. This is the first branch of
\eqref{eq:rollingcert}, $\bar E_k=L_1$ for $k<\chi$, and it justifies the
minimum in the second branch,
$\bar E_k=\min\{L_1,\,2L_2\widehat D_k+e_k\}$ for $k\ge\chi$.
It remains to show $\|B_k-H_k\|_{\rm op}\le2L_2\widehat D_k+e_k$
for $k\ge\chi$, where $e_k=b/(k-\chi+2)^2$.

Consider a time index $k$ such that $k\ge\chi$. The refresh
index $i_j=(j\bmod\chi)+1$ takes every value in $\{1,\dots,\chi\}$ exactly once as
$j$ runs over the window $\{k-\chi+1,\dots,k\}$. For each group $g$ of the
partition, let $k_g$ be the iteration in this window with $g_{i_{k_g}}=g$;
we call $k_g$ the \emph{anchor} of the group $g$ and of each of its
columns, and for a column $i$ we write $k_i$ for the anchor of the group
containing $i$. By the update rule \eqref{eq:refresh} at iteration $k_g$,
the columns of $M_{k_g}$ indexed by $g$ are $P_iu_{k_g}$, $i\in g$, with
$u_{k_g}=(\nabla f(x_{k_g}+\tau_{k_g}v_g)-\nabla f(x_{k_g}))/\tau_{k_g}$, and
at iterations $j=k_g+1,\dots,k$ we have $g_{i_j}\ne g$, so \eqref{eq:refresh}
copies these columns unchanged. Hence $M_ke_i=P_iu_{k_g}$ for every
$i\in g$. Lemma~\ref{lem:group}, applied with $x=x_{k_g}$ and
$\tau=\tau_{k_g}$, gives
\[
M_ke_i=H_{k_i}e_i+\phi_i\quad(i\in[d]),\qquad
\Big(\sum_{i\in g}\|\phi_i\|^2\Big)^{1/2}\le\frac{L_2\tau_{k_g}|g|}{2}
\quad\text{for every group } g .
\]
Summing $(M_ke_i-H_ke_i)e_i^\top$ over $i$ and setting
$\Delta_k:=\sum_i\phi_ie_i^\top$,
\begin{equation}\label{eq:rawsplit}
M_k-H_k=\sum_{i=1}^d(H_{k_i}-H_k)e_ie_i^\top+\Delta_k .
\end{equation}

Note that, for each $i$,
$H_{k_i}-H_k=-\sum_{j=k_i}^{k-1}(H_{j+1}-H_j)$, an empty sum if $k_i=k$.
Substituting in \eqref{eq:rawsplit} and grouping, for each $j$, the columns
with $k_i\le j$,
\begin{equation}\label{eq:exchange}
\sum_{i=1}^d(H_{k_i}-H_k)e_ie_i^\top
=-\sum_{j=k-\chi+1}^{k-1}(H_{j+1}-H_j)\,\Pi_j,
\qquad
\Pi_j:=\sum_{i:\ k_i\le j}e_ie_i^\top,
\end{equation}
where the lower limit is $\min_ik_i=k-\chi+1$. Each $\Pi_j$ is diagonal with
entries in $\{0,1\}$, an orthogonal projection, so $\|\Pi_j\|_{\rm op}\le1$.

Next, using \eqref{eq:exchange}, \eqref{eq:hessstep}, and
the definition $\widehat D_k=\sum_{j=k-\chi+1}^{k-1}\delta_j$, we have
\[
\Big\|\sum_{i=1}^d(H_{k_i}-H_k)e_ie_i^\top\Big\|_{\rm op}
\le\sum_{j=k-\chi+1}^{k-1}\|H_{j+1}-H_j\|_{\rm op}\|\Pi_j\|_{\rm op}
\le L_2\sum_{j=k-\chi+1}^{k-1}\delta_j
=L_2\widehat D_k .
\]
Each Hessian increment appears once, multiplied by a projection; this is why
no factor of $\chi$ or $d$ arises.

The radius $\tau_j=b/(\sqrt\chi\,L_2|g_{i_j}|(j+1)^2)$ gives
$L_2\tau_{k_g}|g|/2=b/(2\sqrt\chi\,(k_g+1)^2)\le b/(2\sqrt\chi\,(k-\chi+2)^2)$
for every group $g$, because $k_g\ge k-\chi+1$. The $\chi$ groups index
disjoint sets of columns, so the squared Frobenius norms of the
corresponding blocks of $\Delta_k$ add, and
\[
\|\Delta_k\|_{\rm op}\le\|\Delta_k\|_F
=\Big(\sum_{g}\sum_{i\in g}\|\phi_i\|^2\Big)^{1/2}
\le\Big(\chi\cdot\frac{b^2}{4\chi\,(k-\chi+2)^4}\Big)^{1/2}
=\frac{b}{2(k-\chi+2)^2}=\frac{e_k}{2}.
\]
With \eqref{eq:rawsplit},
$\|M_k-H_k\|_{\rm op}\le e:=L_2\widehat D_k+e_k/2$.

Next we need to transfer the bounds to $B$ instead of $M$. Let $S_k=(M_k+M_k^\top)/2$, so
that $B_k=\Pi_{\mathcal Z}(S_k)$ by \eqref{eq:refresh}. As $H_k$ is
symmetric, $S_k-H_k=\tfrac12(M_k-H_k)+\tfrac12(M_k-H_k)^\top$, and
transposition preserves the operator norm, so $\|S_k-H_k\|_{\rm op}\le e$.
Lemma~\ref{lem:clip} states that if $H\in\mathcal Z$ and $S$ is symmetric
with $\|S-H\|_{\rm op}\le e$, then $\|\Pi_{\mathcal Z}(S)-H\|_{\rm op}\le2e$.
Applied with $H=H_k$ and $S=S_k$,
\[
\|B_k-H_k\|_{\rm op}\le2e=2L_2\widehat D_k+e_k ,
\]
which leads to the second branch of \eqref{eq:rollingcert}.
\end{proof}

\subsubsection{Proof of Lemma~\ref{lem:sum}}
\label{proof:main:lem:sum}

\begin{proof}
By \eqref{eq:rollingcert}, $\bar E_k=L_1$ for
$k<\chi$, while for $k\ge\chi$
\[
\bar E_k=\min\{L_1,\,2L_2\widehat D_k+e_k\}
\le2L_2\widehat D_k+\frac{b}{(k-\chi+2)^2},
\]
where $e_k=b/(k-\chi+2)^2$. Hence
\begin{equation}\label{eq:sumsplit}
\sum_{k<N}\bar E_k
\le\chi L_1+2L_2\sum_{k=\chi}^{N-1}\widehat D_k
+b\sum_{k=\chi}^{N-1}\frac1{(k-\chi+2)^2}.
\end{equation}
For the drift term, insert $\widehat D_k=\sum_{j=k-\chi+1}^{k-1}\delta_j$ and
exchange the order of summation:
\[
\sum_{k=\chi}^{N-1}\widehat D_k
=\sum_{k=\chi}^{N-1}\ \sum_{j=k-\chi+1}^{k-1}\delta_j
=\sum_j\delta_j\cdot\#\{k:\ \chi\le k\le N-1,\ k-\chi+1\le j\le k-1\}.
\]
The condition $k-\chi+1\le j\le k-1$ is equivalent to $j+1\le k\le j+\chi-1$,
so for each $j$ at most $\chi-1$ indices $k$ qualify, and any such $k$ forces
$j\le k-1\le N-2$. Therefore
\[
\sum_{k=\chi}^{N-1}\widehat D_k\le(\chi-1)\sum_{j<N}\delta_j=(\chi-1)Q_N .
\]
For the finite-difference term, substitute $m=k-\chi+2$ and use
$\sum_{m\ge2}m^{-2}\le\int_1^\infty x^{-2}\,dx=1$:
\[
b\sum_{k=\chi}^{N-1}\frac1{(k-\chi+2)^2}=b\sum_{m=2}^{N-\chi+1}\frac1{m^2}\le b .
\]
Inserting both bounds in \eqref{eq:sumsplit},
$\sum_{k<N}\bar E_k\le\chi L_1+2(\chi-1)L_2Q_N+b\le\chi L_1+2\chi L_2Q_N+2b$.
Finally, $\bar E_k\le L_1$ for every $k$, either directly or because the
second branch is a minimum with $L_1$; summing gives
$\sum_{k<N}\bar E_k\le NL_1$.
\end{proof}

\subsection{Stepsize selection and acceptance}\label{sec:schedule}
\begin{lemma}[Monotonicity and transport]\label{lem:mono}
Fix an iteration $k$ of Algorithm~1, with mass $A_k\ge0$,
$\zeta_k=\|z_k-x_k\|$, model $B_k\in\mathcal Z$, and certificate $\bar E_k$
satisfying $\|B_k-H(x_k)\|_{\rm op}\le\bar E_k$ (Lemma~\ref{lem:telescope}).
For $\lambda>0$ let $a(\lambda)=\big(\lambda+\sqrt{\lambda^2+4\lambda A_k}\big)/2$
be the positive root of $a^2=\lambda(A_k+a)$, let
$y(\lambda)=\big(A_kx_k+a(\lambda)z_k\big)/\big(A_k+a(\lambda)\big)$, and let
$ m(\lambda)=\|y(\lambda)-x_k\|=\frac{a(\lambda)}{A_k+a(\lambda)}\,\zeta_k$
be the displacement of \eqref{eq:mudef}. Then:
\begin{enumerate}[label=(\roman*)]
\item $a$ is continuous and strictly increasing on $(0,\infty)$, with
$a(\lambda)\to0$ as $\lambda\to0^+$ and $a(\lambda)\to\infty$ as
$\lambda\to\infty$;
\item $m$ is continuous and nondecreasing, $m(\lambda)\to0$ as
$\lambda\to0^+$, and $m(c\lambda)\le c\, m(\lambda)$ for every $c\ge1$;
\item (transport) for every $\lambda>0$,
\[
\|B_k-H(y(\lambda))\|_{\rm op}\le\min\{L_1,\ \bar E_k+L_2 m(\lambda)\}
=E_k(\lambda),
\qquad
\Big\|B_k-\int_0^1H\big(x_k+s(y(\lambda)-x_k)\big)\,ds\Big\|_{\rm op}
\le\bar E_k+\frac{L_2}{2} m(\lambda).
\]
\end{enumerate}
\end{lemma}

\begin{proof}
\emph{(i)} Both $\lambda$ and $\sqrt{\lambda^2+4\lambda A_k}$ are
continuous and strictly increasing in $\lambda>0$ (recall $A_k\ge0$), hence
so is $a(\lambda)$. Since $a(\lambda)\ge\lambda$, $a(\lambda)\to\infty$ as
$\lambda\to\infty$. Since $\sqrt{\lambda^2+4\lambda A_k}\le\lambda+2\sqrt{\lambda A_k}$,
we have $a(\lambda)\le\lambda+\sqrt{\lambda A_k}\to0$ as $\lambda\to0^+$.

\emph{(ii)} Write $m(\lambda)=\zeta_k\,\phi(a(\lambda))$ with
$\phi(a)=a/(A_k+a)=1-A_k/(A_k+a)$. The map $\phi$ is continuous and
nondecreasing on $(0,\infty)$, and $a(\cdot)$ is continuous and increasing
by (i), so $m$ is continuous and nondecreasing. If $A_k>0$, then
$\phi(a)\to0$ as $a\to0^+$ and $a(\lambda)\to0$ as $\lambda\to0^+$, so
$m(\lambda)\to0$. If $A_k=0$, then $k=0$ and $x_0=z_0$, so $\zeta_0=0$
and $m\equiv0$. For the scaling inequality, divide the defining relation
$a^2=\lambda(A_k+a)$ by $a(A_k+a)$ to get $a/(A_k+a)=\lambda/a$; hence
\[
\frac{m(\lambda)}{\lambda}=\frac{\zeta_k}{a(\lambda)},
\]
which is nonincreasing in $\lambda$ by (i). For $c\ge1$ this gives
$m(c\lambda)/(c\lambda)\le m(\lambda)/\lambda$, that is,
$m(c\lambda)\le c\,m(\lambda)$.

\emph{(iii)} Let $v=y(\lambda)-x_k$, so $\|v\|=m(\lambda)$. By the
triangle inequality, the certificate $\|B_k-H(x_k)\|_{\rm op}\le\bar E_k$,
and the $L_2$-Lipschitz continuity of $H$,
\[
\|B_k-H(y(\lambda))\|_{\rm op}
\le\|B_k-H(x_k)\|_{\rm op}+\|H(x_k)-H(x_k+v)\|_{\rm op}
\le\bar E_k+L_2 m(\lambda).
\]
Moreover $B_k\in\mathcal Z$ by construction and $H(y(\lambda))\in\mathcal Z$
by Assumption~\ref{as:main}, so $-L_1I\preceq B_k-H(y(\lambda))\preceq L_1I$
and $\|B_k-H(y(\lambda))\|_{\rm op}\le L_1$. Taking the smaller of the two
bounds gives the first claim. For the second, with
$\bar H=\int_0^1H(x_k+sv)\,ds$,
\[
\|H(x_k)-\bar H\|_{\rm op}
=\Big\|\int_0^1\big(H(x_k)-H(x_k+sv)\big)\,ds\Big\|_{\rm op}
\le\int_0^1L_2s\|v\|\,ds=\frac{L_2}{2}{\color{ReviewBlue}m}(\lambda),
\]
and the triangle inequality gives
$\|B_k-\bar H\|_{\rm op}\le\|B_k-H(x_k)\|_{\rm op}+\|H(x_k)-\bar H\|_{\rm op}
\le\bar E_k+\frac{L_2}{2}{\color{ReviewBlue}m}(\lambda)$.
\end{proof}

\begin{lemma}[The resolvent certificate]\label{lem:rescert}
Fix an iteration $k$ of Algorithm~1 and $\lambda>0$. Let
$y(\lambda)=\big(A_kx_k+a(\lambda)z_k\big)/\big(A_k+a(\lambda)\big)$ and
${\color{ReviewBlue}m}(\lambda)=\|y(\lambda)-x_k\|$ be as in Lemma~\ref{lem:mono}, let
$B_k\in\mathcal Z$ be the current model with certificate
$\|B_k-H(x_k)\|_{\rm op}\le\bar E_k$, and let $g_k=\|\nabla f(x_k)\|$. Then
the model step $r(\lambda)=-(B_k+\lambda^{-1}I)^{-1}\nabla f(y(\lambda))$
satisfies
\[
\|r(\lambda)\|\ \le\ \widehat\rho_k(\lambda)
:=\lambda g_k+{\color{ReviewBlue}m}(\lambda)\Big(1+\lambda\bar E_k+\frac{L_2\lambda}{2}{\color{ReviewBlue}m}(\lambda)\Big),
\]
which is the bound \eqref{eq:rhodef}.
\end{lemma}

\begin{proof}
Let $v=y(\lambda)-x_k$, so that $\|v\|={\color{ReviewBlue}m}(\lambda)$, and let
$\bar H=\int_0^1H(x_k+sv)\,ds$. By the fundamental theorem of calculus,
$\nabla f(y(\lambda))=\nabla f(x_k)+\bar Hv$, and adding and subtracting
$B_kv$,
\[
\nabla f(y(\lambda))=\nabla f(x_k)+B_kv+(\bar H-B_k)v .
\]
Applying $-(B_k+\lambda^{-1}I)^{-1}$ to each term,
\begin{equation}\label{eq:rsplit}
r(\lambda)=-(B_k+\lambda^{-1}I)^{-1}\nabla f(x_k)
-(B_k+\lambda^{-1}I)^{-1}B_k\,v
-(B_k+\lambda^{-1}I)^{-1}(\bar H-B_k)\,v .
\end{equation}
We bound the three operators. Since $B_k\in\mathcal Z$, it is symmetric
with eigenvalues $b_i\ge0$. In an eigenbasis of $B_k$, the matrix
$(B_k+\lambda^{-1}I)^{-1}$ is diagonal with entries
$1/(b_i+\lambda^{-1})\le\lambda$, and $(B_k+\lambda^{-1}I)^{-1}B_k$ is
diagonal with entries $b_i/(b_i+\lambda^{-1})\in[0,1)$. Both matrices are
symmetric, so
\[
\|(B_k+\lambda^{-1}I)^{-1}\|_{\rm op}\le\lambda,
\qquad
\|(B_k+\lambda^{-1}I)^{-1}B_k\|_{\rm op}\le1 .
\]
The second bound is where $B_k\succeq0$ is used: the displacement $v$
enters \eqref{eq:rsplit} without a factor $\lambda L_1$. For the third
operator, Lemma~\ref{lem:mono}(iii) states that
$\|B_k-\bar H\|_{\rm op}\le\bar E_k+\frac{L_2}{2}{\color{ReviewBlue}m}(\lambda)$. Taking norms
in \eqref{eq:rsplit} and using $\|v\|={\color{ReviewBlue}m}(\lambda)$ and
$\|\nabla f(x_k)\|=g_k$,
\[
\|r(\lambda)\|
\le\lambda g_k+{\color{ReviewBlue}m}(\lambda)
+\lambda\Big(\bar E_k+\frac{L_2}{2}{\color{ReviewBlue}m}(\lambda)\Big){\color{ReviewBlue}m}(\lambda)
=\lambda g_k+{\color{ReviewBlue}m}(\lambda)\Big(1+\lambda\bar E_k+\frac{L_2\lambda}{2}{\color{ReviewBlue}m}(\lambda)\Big),
\]
which is $\widehat\rho_k(\lambda)$.
\end{proof}

\subsubsection{Proof of Lemma~\ref{lem:scalarbounds}}
\label{proof:lem:scalarbounds}
\begin{proof}
The first bound is the transport statement of Lemma~\ref{lem:mono}; the
second is Lemma~\ref{lem:rescert}.
\end{proof}

\subsubsection{Proof of Lemma~\ref{lem:wellposed}}
\label{proof:main:lem:wellposed}

\begin{proof}
Throughout, $\widehat\Psi_k$ is the function \eqref{eq:psiexpanded},
\[
\widehat\Psi_k(\lambda)=\lambda\min\{L_1,\bar E_k+L_2{\color{ReviewBlue}m}\}
+\frac{L_2g_k}{2}\lambda^2+\frac{L_2}{2}\lambda{\color{ReviewBlue}m}
+\frac{L_2\bar E_k}{2}\lambda^2{\color{ReviewBlue}m}+\frac{L_2^2}{4}\lambda^2{\color{ReviewBlue}m}^2,
\qquad {\color{ReviewBlue}m}={\color{ReviewBlue}m}(\lambda),
\]
where ${\color{ReviewBlue}m}(\lambda)=\zeta_k a(\lambda)/(A_k+a(\lambda))$ is given by
\eqref{eq:mudef}. Its value depends only on $\lambda$, the stored scalars
$A_k,\zeta_k,g_k,\bar E_k$, and $L_1,L_2$, so Subroutine~1.1 queries no
gradient. We also record that $0<\bar E_k\le L_1$: the upper bound is the
cap in \eqref{eq:rollingcert}, and the lower bound holds because the
second branch of \eqref{eq:rollingcert} contains the allowance
$e_k=b/(k-\chi+2)^2>0$ and $b>0$. In particular $W_k=\log(L_1/\bar E_k)$ is
finite and nonnegative, and $\sigma/\bar E_k<\infty$.

We first show that $\widehat\Psi_k$ is continuous, strictly increasing, and onto $(0,\infty)$. By Lemma~\ref{lem:mono}(ii), ${m}$ is continuous and
nondecreasing with ${m}(\lambda)\to0$ as $\lambda\downarrow0$ (and
${m}\equiv0$ if $A_k=0$). Hence every term of $\widehat\Psi_k$ is
continuous, nondecreasing, and tends to zero as $\lambda\downarrow0$. For strict monotonicity, let $E_k(\lambda)=\min\{L_1,\bar E_k+L_2{m}(\lambda)\}$ be as in \eqref{eq:Etransport}; since $\bar E_k\le L_1$, $E_k$ is nondecreasing with $E_k(\lambda)\ge\bar E_k$.
For $\lambda_2>\lambda_1>0$,
\[
\lambda_2{E_k}(\lambda_2)-\lambda_1{E_k}(\lambda_1)
=\lambda_2\big({E_k}(\lambda_2)-{E_k}(\lambda_1)\big)+(\lambda_2-\lambda_1){E_k}(\lambda_1)
\ge(\lambda_2-\lambda_1)\bar E_k>0,
\]
so the first term is strictly increasing and the others are nondecreasing.
Finally $\widehat\Psi_k(\lambda)\ge\lambda{\color{ReviewBlue}E_k}(\lambda)\ge\lambda\bar E_k\to\infty$
as $\lambda\to\infty$. Therefore $\widehat\Psi_k(\lambda)=\sigma$ has a unique
root $\lambda_k^\circ>0$, the set $\{\lambda:\widehat\Psi_k(\lambda)\le\sigma\}$
is the interval $(0,\lambda_k^\circ]$, and we will use repeatedly that
\begin{equation}\label{eq:psilower}
\widehat\Psi_k(\lambda)\ \ge\ \lambda\bar E_k\qquad\text{for all }\lambda>0 .
\end{equation}

Next we verify the bracket maintained by Subroutine~1.1. By \eqref{eq:psilower},
$\widehat\Psi_k(\sigma/\bar E_k)\ge\sigma$. If $W_k=0$, then $\bar E_k=L_1$
and \eqref{eq:psilower} gives $\widehat\Psi_k(\sigma/L_1)\ge\sigma$, so
Step~1 of the subroutine returns $\lambda_k=\sigma/L_1$ after one
evaluation. If Step~1 does not return, then
$\sigma>\widehat\Psi_k(\sigma/L_1)\ge\sigma\bar E_k/L_1$ by
\eqref{eq:psilower}, so $\bar E_k<L_1$ and $W_k>0$. Step~2 then sets
$\lambda^-=\sigma/L_1<\lambda^+=\sigma/\bar E_k$, and the two facts just
established give the invariant
\begin{equation}\label{eq:bracket}
\widehat\Psi_k(\lambda^-)<\sigma\le\widehat\Psi_k(\lambda^+),
\qquad\text{equivalently}\qquad \lambda^-<\lambda_k^\circ\le\lambda^+ .
\end{equation}
Each pass of Step~3 evaluates $\widehat\Psi_k$ at $\lambda=\sqrt{\lambda^-\lambda^+}$
and replaces $\lambda^-$ by $\lambda$ if $\widehat\Psi_k(\lambda)<\sigma$,
and $\lambda^+$ by $\lambda$ otherwise; either way \eqref{eq:bracket} is
preserved. Moreover $\log(\lambda/\lambda^-)=\log(\lambda^+/\lambda)
=\tfrac12\log(\lambda^+/\lambda^-)$, so each pass halves
$\log(\lambda^+/\lambda^-)$.

The evaluation count follows from the halving. The quantity $\log(\lambda^+/\lambda^-)$
equals $W_k$ when the loop is entered and $W_k2^{-m}$ after $m$ passes.
The loop runs while $\lambda^+/\lambda^->2^{1/4}$, that is, while
$W_k2^{-m}>\log2^{1/4}$, so the number of passes is the smallest integer
$m\ge0$ with $2^m\ge W_k/\log2^{1/4}$, namely
$\max\{0,\lceil\log_2(W_k/\log2^{1/4})\rceil\}
=\lceil\log_2\max\{1,W_k/\log2^{1/4}\}\rceil$. Adding the single evaluation
of Step~1 gives the bound of the lemma; when Step~1 returns, the bound holds
with the maximum equal to $1$.

We now show that the returned stepsize satisfies \eqref{eq:ms}. Algorithm~1
sets $a_k=a(\lambda_k)$, $y_k=y(\lambda_k)$, $r_k=r(\lambda_k)$ and
$x_{k+1}=y_k+r_k$, and \eqref{eq:ms} asks for
$\|r_k+\lambda_k\nabla f(y_k+r_k)\|\le\sigma\|r_k\|$. If Step~1 returns
$\lambda_k=\sigma/L_1$, we use the second certificate of
Lemma~\ref{lem:resid}: if $0\preceq B\preceq LI$ and
$0\preceq\nabla^2F\preceq LI$ on $[y,y+r]$, then
$\|r+\lambda\nabla F(y+r)\|\le\lambda L\|r\|$. It applies with $F=f$,
$B=B_k\in\mathcal Z$ and $L=L_1$, since $0\preceq H\preceq L_1I$ everywhere
by Assumption~\ref{as:main}, and gives
$\|r_k+\lambda_k\nabla f(y_k+r_k)\|\le\lambda_kL_1\|r_k\|=\sigma\|r_k\|$.
Otherwise Step~4 returns $\lambda_k=\lambda^-$, and
$\widehat\Psi_k(\lambda_k)<\sigma$ by \eqref{eq:bracket}. The first
certificate of Lemma~\ref{lem:resid} states
$\|r+\lambda\nabla f(y+r)\|\le\lambda\big(\|B-H(y)\|_{\rm op}
+\frac{L_2}{2}\|r\|\big)\|r\|$, and Lemma~\ref{lem:scalarbounds} gives
$\|B_k-H(y_k)\|_{\rm op}\le E_k(\lambda_k)$ and
$\|r_k\|\le\widehat\rho_k(\lambda_k)$. Since the right-hand side of the
certificate is increasing in both quantities, and
$\widehat\Psi_k(\lambda)=\lambda E_k(\lambda)+\frac{L_2\lambda}{2}\widehat\rho_k(\lambda)$
by definition,
\[
\|r_k+\lambda_k\nabla f(y_k+r_k)\|
\le\lambda_k\Big(E_k(\lambda_k)+\frac{L_2}{2}\widehat\rho_k(\lambda_k)\Big)\|r_k\|
=\widehat\Psi_k(\lambda_k)\|r_k\|\le\sigma\|r_k\| .
\]

The stepsize bounds are immediate. If Step~1 returns, $\lambda_k=\sigma/L_1
\le\sigma/\bar E_k$ because $\bar E_k\le L_1$. Otherwise $\lambda^-$ is only
ever replaced by a larger value and $\lambda^+$ only by a smaller one, so
at termination $\sigma/L_1\le\lambda^-<\lambda^+\le\sigma/\bar E_k$, and
$\lambda_k=\lambda^-$ satisfies the claim.

Finally we prove near-maximality. If Step~1 returns,
$\widehat\Psi_k(\lambda_k)\ge\sigma$ by the test in Step~1. Otherwise let
$c=\lambda^+/\lambda^-$ at termination; the exit condition of Step~3 gives
$1<c\le2^{1/4}$. We compare $\widehat\Psi_k(c\lambda)$ with
$\widehat\Psi_k(\lambda)$ term by term, using
${\color{ReviewBlue}m}(c\lambda)\le c\,{\color{ReviewBlue}m}(\lambda)$ from Lemma~\ref{lem:mono}(ii) and the
elementary inequality $\min\{L_1,\bar E_k+cX\}\le c\min\{L_1,\bar E_k+X\}$
for $X\ge0$ and $c\ge1$ (if the right-hand minimum is $L_1$, the left side
is at most $L_1\le cL_1$; otherwise $c(\bar E_k+X)\ge\bar E_k+cX$). Replacing
$\lambda$ by $c\lambda$ multiplies the first term,
$\lambda\min\{L_1,\bar E_k+L_2{\color{ReviewBlue}m}\}$, by at most $c\cdot c=c^2$; the second,
$\frac{L_2g_k}{2}\lambda^2$, by $c^2$; the third, $\frac{L_2}{2}\lambda{\color{ReviewBlue}m}$,
by at most $c^2$; the fourth, $\frac{L_2\bar E_k}{2}\lambda^2{\color{ReviewBlue}m}$, by at most
$c^3$; and the fifth, $\frac{L_2^2}{4}\lambda^2{\color{ReviewBlue}m}^2$, by at most $c^4$.
Since $c\ge1$, all factors are at most $c^4\le2$, so
$\widehat\Psi_k(c\lambda)\le2\widehat\Psi_k(\lambda)$ for every $\lambda>0$.
Applying this with $\lambda=\lambda_k=\lambda^-$ and $c\lambda_k=\lambda^+$,
and using \eqref{eq:bracket},
\[
\sigma\le\widehat\Psi_k(\lambda^+)=\widehat\Psi_k(c\lambda_k)
\le2\widehat\Psi_k(\lambda_k),
\]
which is $\widehat\Psi_k(\lambda_k)\ge\sigma/2$.
\end{proof}

\subsubsection{Inverting near-maximality}\label{proof:fiveterm}
The rate analysis uses Lemma~\ref{lem:wellposed} through the following
pointwise bound on the inverse stepsize, which inverts the guarantee
$\widehat\Psi_k(\lambda_k)\ge\sigma/2$.

\begin{lemma}[Pointwise inverse-stepsize bound]\label{lem:fiveterm}
At every iteration of Algorithm~1 with Subroutine~1.1, let
${\color{ReviewBlue}m}_k={\color{ReviewBlue}m}(\lambda_k)=\|y_k-x_k\|$, $g_k=\|\nabla f(x_k)\|$ and
$E_k=E_k(\lambda_k)=\min\{L_1,\bar E_k+L_2{\color{ReviewBlue}m}_k\}$. Then
\begin{equation}\label{eq:fiveterm}
\lambda_k^{-1}\ \le\ \frac{10E_k}{\sigma}
+\sqrt{\frac{5L_2g_k}{\sigma}}+\frac{5L_2{\color{ReviewBlue}m}_k}{\sigma}
+\sqrt{\frac{5L_2\bar E_k{\color{ReviewBlue}m}_k}{\sigma}}
+\frac{L_2{\color{ReviewBlue}m}_k}{2}\sqrt{\frac{10}{\sigma}} .
\end{equation}
\end{lemma}

\begin{proof}
By \eqref{eq:psiexpanded}, $\widehat\Psi_k(\lambda_k)$ is the sum of the five
nonnegative terms
\[
\lambda_kE_k,\quad \frac{L_2g_k}{2}\lambda_k^2,\quad
\frac{L_2}{2}\lambda_k{\color{ReviewBlue}m}_k,\quad
\frac{L_2\bar E_k}{2}\lambda_k^2{\color{ReviewBlue}m}_k,\quad
\frac{L_2^2}{4}\lambda_k^2{\color{ReviewBlue}m}_k^2 .
\]
Their sum is at least $\sigma/2$ by Lemma~\ref{lem:wellposed}, so at least
one of them is at least $\sigma/10$. Solving each alternative for
$\lambda_k^{-1}$:
$\lambda_kE_k\ge\sigma/10$ gives $\lambda_k^{-1}\le10E_k/\sigma$;
$\frac{L_2g_k}{2}\lambda_k^2\ge\sigma/10$ gives
$\lambda_k^{-1}\le\sqrt{5L_2g_k/\sigma}$;
$\frac{L_2}{2}\lambda_k{\color{ReviewBlue}m}_k\ge\sigma/10$ gives
$\lambda_k^{-1}\le5L_2{\color{ReviewBlue}m}_k/\sigma$;
$\frac{L_2\bar E_k}{2}\lambda_k^2{\color{ReviewBlue}m}_k\ge\sigma/10$ gives
$\lambda_k^{-1}\le\sqrt{5L_2\bar E_k{\color{ReviewBlue}m}_k/\sigma}$; and
$\frac{L_2^2}{4}\lambda_k^2{\color{ReviewBlue}m}_k^2\ge\sigma/10$ gives
$\lambda_k^{-1}\le\frac{L_2{\color{ReviewBlue}m}_k}{2}\sqrt{10/\sigma}$. Whichever alternative
holds, $\lambda_k^{-1}$ is at most the sum of the five right-hand sides,
which is \eqref{eq:fiveterm}.
\end{proof}

\subsection{Proof of Theorem~\ref{thm:iterationrate}}\label{proof:iterationrate}
Throughout this proof and the next, $N\ge1$ is a completed horizon,
$b=L_2R$, ${\color{ReviewBlue}m}_k$, $g_k$ and $E_k$ are as in Lemma~\ref{lem:fiveterm},
$\delta_k={\color{ReviewBlue}m}_k+\|r_k\|$, and
\[
 K=\chi(L_1+L_2R),\qquad \lambda_b=\frac{\sigma}{L_1},\qquad
 q_k=\sqrt{\frac{A_{k+1}}{\lambda_k}}\,\|r_k\|.
\]
By Lemma~\ref{lem:wellposed}, every iteration produces a pair
$(y_k,x_{k+1})$ satisfying \eqref{eq:ms} with $\lambda_k\ge\lambda_b$, so the
lemmas of Appendix~\ref{app:backbone} apply with $F=f$. With the weight
recursion, Lemma~\ref{lem:potential}(ii),(iii) and Lemma~\ref{lem:bounded}
give
\begin{equation}\label{eq:fp-basic}
\begin{gathered}
 \lambda_k\ge\lambda_b,\qquad a_k^2=\lambda_k A_{k+1},\qquad
 A_{k+1}=A_k+a_k,\qquad A_0=0,\\
 \sum_{k<N}q_k^2\le\frac{R^2}{1-\sigma^2},\qquad
 \|x_k-x^\star\|\le C_xR,\qquad
 \|z_k-x^\star\|\le R,\qquad
 C_x=1+\frac{\sigma}{\sqrt{1-\sigma^2}}.
\end{gathered}
\end{equation}
In particular $A_1=a_0=\lambda_0\ge\lambda_b>0$, and the mass $A_k$ is
strictly increasing. We prove the two branches of
\eqref{eq:iterationheadline} separately and take their minimum at the end.

\paragraph{The smoothness branch.}
Lemma~\ref{lem:agrowth} states
$A_N\ge\frac14\big(\sum_{k<N}\sqrt{\lambda_k}\big)^2$, and
$\lambda_k\ge\lambda_b$ gives
\begin{equation}\label{eq:smoothbranch}
 A_N\ \ge\ \frac{\sigma N^2}{4L_1}.
\end{equation}
With Lemma~\ref{lem:potential}(i),
$f(x_N)-f^\star\le R^2/(2A_N)\le2L_1R^2/(\sigma N^2)$. The curvature branch
bounds the same $A_N$ from below by counting iterations on dyadic scales of
the mass.

\paragraph{One mass scale and its path length.}
Fix $s>0$ such that
$I=\{k\in\{0,\ldots,N-1\}:s\le A_{k+1}<2s\}$ is nonempty.
Since the mass is strictly increasing, $I=\{t,\ldots,u\}$ is an
interval. Write $n=|I|$ and $S_I=\sum_{k\in I}\lambda_k^{-1}$.
Define the constants
\begin{equation}\label{eq:fp-pathconstants}
 c_0=1+\log2,\qquad
 m_\sigma=(1+C_x)c_0,\qquad
 p_\sigma=m_\sigma+\sqrt{\frac{c_0}{1-\sigma^2}}.
\end{equation}
The first increment may enter the scale from below. Separating it and
using $(v-w)/v\le\log(v/w)$ for $v\ge w>0$ gives
\[
 \sum_{k=t}^u\frac{a_k}{A_{k+1}}
 \le1+\sum_{k=t+1}^u\log\frac{A_{k+1}}{A_k}
 =1+\log\frac{A_{u+1}}{A_{t+1}}\le c_0,
\]
because $A_{t+1}\ge s$ and $A_{u+1}<2s$. The sum after the first term is
empty when $t=u$. By \eqref{eq:ydiff}, ${\color{ReviewBlue}m}_k=(a_k/A_{k+1})\zeta_k$ with
$\zeta_k=\|z_k-x_k\|\le(1+C_x)R$ by \eqref{eq:fp-basic}, so
\[
 \sum_{k\in I}{\color{ReviewBlue}m}_k
 \le(1+C_x)R\sum_{k\in I}\frac{a_k}{A_{k+1}}
 \le m_\sigma R.
\]
Also $a_k/A_{k+1}\le1$, so
$\lambda_k/A_{k+1}=(a_k/A_{k+1})^2\le a_k/A_{k+1}$. Since
$\|r_k\|=q_k\sqrt{\lambda_k/A_{k+1}}$, Cauchy--Schwarz and
\eqref{eq:fp-basic} imply
\[
 \sum_{k\in I}\|r_k\|
 \le\left(\sum_{k\in I}q_k^2\right)^{1/2}
      \left(\sum_{k\in I}\frac{\lambda_k}{A_{k+1}}\right)^{1/2}
 \le R\sqrt{\frac{c_0}{1-\sigma^2}}.
\]
Thus
\begin{equation}\label{eq:fp-localpath}
 \sum_{k\in I}{\color{ReviewBlue}m}_k\le m_\sigma R,\qquad
 \sum_{k\in I}\delta_k\le p_\sigma R.
\end{equation}

\paragraph{Rolling windows crossing the scale boundary.}
Bound the first $\min\{\chi,n\}$ certificates in $I$ by the cap $L_1$ of
\eqref{eq:rollingcert}. Every remaining index $k\ge t+\chi$ satisfies
$k\ge\chi$, and its window $\{k-\chi+1,\ldots,k-1\}$ is contained in $I$,
so \eqref{eq:rollingcert} bounds $\bar E_k$ by $2L_2\widehat D_k+e_k$, where
$\widehat D_k$ sums displacements $\delta_j$ with $j\in I$. As in the proof
of Lemma~\ref{lem:sum}, each displacement is counted in at most
$\chi-1\le\chi$ such windows, and the finite-difference allowances over
these indices sum to at most $2b=2L_2R$. Hence \eqref{eq:fp-localpath}
yields
\[
 \sum_{k\in I}\bar E_k
 \le\chi L_1+2\chi L_2\sum_{k\in I}\delta_k+2L_2R
 \le\chi L_1+(2\chi p_\sigma+2)L_2R.
\]
Set $e_\sigma=2p_\sigma+2$. Since $\chi\ge1$ and $e_\sigma\ge1$,
\begin{equation}\label{eq:fp-localcert}
 \sum_{k\in I}\bar E_k\le e_\sigma K.
\end{equation}
This includes $n<\chi$, when all certificates are charged to the cap.
No displacement before $t$ is charged to the local path.

\paragraph{The gradient term on one scale.}
Since $\nabla f(x^\star)=0$, smoothness and \eqref{eq:fp-basic} give
$g_t\le L_1C_xR$. For $k=t+1,\ldots,u$, the condition \eqref{eq:ms} at
iteration $k-1$ and the triangle inequality give
$\lambda_{k-1}g_k\le(1+\sigma)\|r_{k-1}\|$, and since
$a_{k-1}=\sqrt{\lambda_{k-1}A_k}$,
\[
 g_k\le(1+\sigma)\frac{\|r_{k-1}\|}{\lambda_{k-1}}
       =(1+\sigma)\frac{q_{k-1}}{a_{k-1}}.
\]
Here $A_k\ge s$, because $k-1\in I$. Thus
$a_{k-1}^2=\lambda_{k-1}A_k\ge\lambda_b s$ and
\begin{align}
 \sum_{k\in I}\sqrt{g_k}
 &\le\sqrt{L_1C_xR}
 +\sqrt{1+\sigma}(\lambda_b s)^{-1/4}
       \sum_{j=t}^{u-1}q_j^{1/2}\notag\\
 &\le\sqrt{L_1C_xR}
 +\sqrt{1+\sigma}(\lambda_b s)^{-1/4}
       n^{3/4}\left(\sum_{j=t}^{u-1}q_j^2\right)^{1/4}\notag\\
 &\le\sqrt{C_x}\sqrt{L_1R}
 +\frac{\sqrt{1+\sigma}}{\sigma^{1/4}(1-\sigma^2)^{1/4}}\,
       L_1^{1/4}R^{1/2}n^{3/4}s^{-1/4}.
 \label{eq:fp-localgradient}
\end{align}
The second inequality is H\"older's inequality with exponents $4$ and
$4/3$, with $n-1\le n$ terms. The intermediate sums are zero when $n=1$.

\paragraph{The local inverse-stepsize bound.}
Sum \eqref{eq:fiveterm} over $k\in I$. Using $E_k\le\bar E_k+L_2{\color{ReviewBlue}m}_k$
from \eqref{eq:Etransport}, we obtain
\begin{align}
 S_I\le{}&
 \frac{10}{\sigma}\sum_{k\in I}\bar E_k
 +\left(\frac{15}{\sigma}+\frac12\sqrt{\frac{10}{\sigma}}\right)
        L_2\sum_{k\in I}{\color{ReviewBlue}m}_k+\sqrt{\frac{5}{\sigma}}\left[
       \sqrt{L_2}\sum_{k\in I}\sqrt{g_k}
       +\sqrt{L_2}\sum_{k\in I}\sqrt{\bar E_k{\color{ReviewBlue}m}_k}\right].
 \label{eq:fp-five-summed}
\end{align}
For the mixed term, Cauchy--Schwarz, \eqref{eq:fp-localcert} and
\eqref{eq:fp-localpath} give
\[
 \sqrt{L_2}\sum_{k\in I}\sqrt{\bar E_k{\color{ReviewBlue}m}_k}
 \le\sqrt{\left(\sum_{k\in I}\bar E_k\right)
              \left(L_2\sum_{k\in I}{\color{ReviewBlue}m}_k\right)}
 \le\sqrt{e_\sigma m_\sigma}\,K,
\]
where we used $L_2R\le K$. By the arithmetic--geometric mean inequality,
the boundary part of the gradient term satisfies
$\sqrt{C_x}\sqrt{L_1L_2R}\le(\sqrt{C_x}/2)K$. Define
\begin{equation}\label{eq:fp-momentconstants}
\begin{split}
 \alpha_\sigma={}&
 \frac{10e_\sigma}{\sigma}
 +\left(\frac{15}{\sigma}+\frac12\sqrt{\frac{10}{\sigma}}\right)m_\sigma
 +\sqrt{\frac5\sigma}
       \left(\frac{\sqrt{C_x}}2+\sqrt{e_\sigma m_\sigma}\right),\\
 \beta_\sigma={}&
 \sqrt{\frac5\sigma}\,
 \frac{\sqrt{1+\sigma}}{\sigma^{1/4}(1-\sigma^2)^{1/4}}.
\end{split}
\end{equation}
Substituting \eqref{eq:fp-localpath}--\eqref{eq:fp-localgradient} into
\eqref{eq:fp-five-summed} gives
\begin{equation}\label{eq:fp-localmoment}
\begin{split}
 S_I
 &\le\alpha_\sigma K+
       \beta_\sigma(L_1L_2^2R^2)^{1/4}n^{3/4}s^{-1/4}\le\alpha_\sigma K+
       \beta_\sigma K^{3/4}n^{3/4}s^{-1/4}.
\end{split}
\end{equation}
The last inequality uses $L_1(L_2R)^2\le(L_1+L_2R)^3\le K^3$.

\paragraph{Counting the iterations in this scale.}
The weight recursion gives
\[
 \sqrt{\lambda_k}
 =\frac{a_k}{\sqrt{A_{k+1}}}
 =(\sqrt{A_{k+1}}-\sqrt{A_k})
       \left(1+\sqrt{\frac{A_k}{A_{k+1}}}\right)
 \le2(\sqrt{A_{k+1}}-\sqrt{A_k}).
\]
Consequently
$\sum_{k\in I}\sqrt{\lambda_k}
 \le2(\sqrt{A_{u+1}}-\sqrt{A_t})\le2\sqrt{2s}$.
H\"older's inequality on $I$ with exponents $3/2$ and $3$ gives
\[
 n=\sum_{k\in I}
       (\sqrt{\lambda_k})^{2/3}(\lambda_k^{-1})^{1/3}
 \le\left(\sum_{k\in I}\sqrt{\lambda_k}\right)^{2/3}S_I^{1/3},
\]
and therefore $n^3\le8sS_I$. By \eqref{eq:fp-localmoment},
\[
 n^3\le8\alpha_\sigma Ks
       +8\beta_\sigma(Ks)^{3/4}n^{3/4}.
\]
Setting $v=n/(Ks)^{1/3}$ and dividing by $Ks>0$ yields
\[
 v^3\le8\alpha_\sigma+8\beta_\sigma v^{3/4}.
\]
If $v^3\le16\alpha_\sigma$, then $v\le(16\alpha_\sigma)^{1/3}$. Otherwise
$8\alpha_\sigma<v^3/2$, so $v^3/2\le8\beta_\sigma v^{3/4}$ and
$v\le(16\beta_\sigma)^{4/9}$. Thus, defining
\begin{equation}\label{eq:fp-countconstant}
 d_\sigma=
 \max\left\{(16\alpha_\sigma)^{1/3},
             (16\beta_\sigma)^{4/9}\right\},
 \qquad
 D_\sigma=\frac{d_\sigma}{1-2^{-1/3}},
\end{equation}
we have the explicit local count
\begin{equation}\label{eq:fp-blockcount}
 |I|\le d_\sigma(Ks)^{1/3}.
\end{equation}

\paragraph{Summing the scales.}
Let $s_j=2^jA_1$ and $I_j=\{k<N:s_j\le A_{k+1}<2s_j\}$ for
$j=0,\ldots,J$, where $J=\lfloor\log_2(A_N/A_1)\rfloor$. Since
$A_1\le A_{k+1}\le A_N<2s_J$ for every $k<N$, these sets partition the
completed iterations; empty sets contribute zero. Therefore
\begin{align*}
 N=\sum_{j=0}^J|I_j|
 &\le d_\sigma K^{1/3}\sum_{j=0}^Js_j^{1/3}\\
 &=d_\sigma(Ks_J)^{1/3}\sum_{j=0}^J2^{-(J-j)/3}\\
 &\le D_\sigma(Ks_J)^{1/3}
 \le D_\sigma(KA_N)^{1/3}.
\end{align*}
We conclude that
\begin{equation}\label{eq:fp-explicit-mass}
 A_N\ge\frac{N^3}{D_\sigma^3K}.
\end{equation}
Combining this with $f(x_N)-f^\star\le R^2/(2A_N)$ and
\eqref{eq:smoothbranch} proves
\begin{equation}\label{eq:fp-explicit-rate}
 f(x_N)-f^\star
 \le
 \min\left\{\frac{2L_1R^2}{\sigma N^2},
             \frac{D_\sigma^3KR^2}{2N^3}\right\},
\end{equation}
which is \eqref{eq:iterationheadline} with $k=N$, $K=\chi(L_1+L_2R)$ and
the constant $\max\{2/\sigma,D_\sigma^3/2\}$. All displayed constants depend
only on $\sigma$ and are independent of $N,d,\chi,L_1,L_2,R$; they have not
been optimized. No logarithm appears because the local counts form a
geometric series, and no bound on the global inverse-stepsize sum
$\sum_{k<N}\lambda_k^{-1}$ is needed. The count of three gradients per
iteration is verified in the proof of Corollary~\ref{thm:complexity} below. {\color{ReviewBlue}For a general allowance $b>0$ in place of $b=L_2R$, the same argument applies verbatim with $K$ replaced by $K_b=\chi(L_1+L_2R)+b$: the finite-difference allowances contribute $2b\le e_\sigma b$ to \eqref{eq:fp-localcert}, and the inequalities $L_2R\le K_b$ and $L_1(L_2R)^2\le K_b^3$ used afterwards continue to hold, so \eqref{eq:fp-explicit-mass} and \eqref{eq:fp-explicit-rate} hold with $K_b$ in place of $K$.}

\subsection{Proof of Corollary~\ref{thm:complexity}}\label{proof:thm:complexity}

\emph{Correctness.} Algorithm~1 returns $x_{k+1}$ as soon as
$A_{k+1}\ge R^2/(2\varepsilon)$; by Lemma~\ref{lem:potential}(i),
$f(x_{k+1})-f^\star\le R^2/(2A_{k+1})\le\varepsilon$. The test uses only
$A_{k+1}$, which the method computes. A return at $g_k=0$ is exact.

\emph{Number of iterations.} Set
\[
 B_\varepsilon=
 \min\left\{\sqrt{\frac{2L_1R^2}{\sigma\varepsilon}},
             D_\sigma\left(\frac{KR^2}{2\varepsilon}\right)^{1/3}
      \right\}.
\]
If $N\ge1$ iterations have been completed without meeting the stopping
rule, then $A_N<R^2/(2\varepsilon)$, and the two lower bounds
\eqref{eq:smoothbranch} and \eqref{eq:fp-explicit-mass} on $A_N$ give
$N<B_\varepsilon$. Consequently the method stops after at most
$\lceil B_\varepsilon\rceil$ iterations.

\emph{Gradient count.} Initialization computes $\nabla f(x_0)$. Iteration
$k$ computes the probe gradient $\nabla f(x_k+\tau_kv_{g_{i_k}})$ of
\eqref{eq:refresh}, the gradient $\nabla f(y_k)$ for the model step, and
$\nabla f(x_{k+1})$ for the momentum update; the last one is cached and
serves as the anchor gradient of iteration $k+1$, both in \eqref{eq:refresh}
and in $g_{k+1}$, and Subroutine~1.1 queries no gradient by
Lemma~\ref{lem:wellposed}. Thus each iteration uses at most three
gradients, and
\begin{equation}\label{eq:fp-explicit-queries}
 G_\varepsilon\le1+3\left\lceil
 \min\left\{\sqrt{\frac{2L_1R^2}{\sigma\varepsilon}},
             D_\sigma\left(\frac{KR^2}{2\varepsilon}\right)^{1/3}
      \right\}\right\rceil,
\end{equation}
which is \eqref{eq:rpcomplexity}.

\section{Analysis of the strongly convex method}\label{app:sc}
This appendix proves Theorem~\ref{thm:scmain}, Corollary~\ref{cor:scqueries}
and Remark~\ref{rem:scexplicit}.
Appendix~\ref{app:sc-setup} fixes notation and records the rolling
curvature certificate. Appendix~\ref{app:sc-weights} establishes the
weight identities and proves the potential decrease for admissible
steps. Appendix~\ref{app:sc-test} then shows that the scalar test
certifies admissibility and that Subroutine~2.1 returns a nearly
maximal certified stepsize. These results yield well-posedness,
bounded iterates and weight growth in Appendix~\ref{app:sc-run}.
Appendix~\ref{app:sc-path} uses these estimates to bound the total
anchor path. Combined with the rolling certificate, this gives a
bound on the inverse-stepsize moment that is independent of the
number of iterations. The final three subsections derive the
convergence rates and gradient complexity, and establish the
guarantee for the explicit stepsize rule.

\subsection{Setting and the certificate}\label{app:sc-setup}
Throughout, Assumptions~\ref{as:main} and~\ref{as:sc} hold, $0<\mu\le L_1$,
$L_2,R>0$, $\chi\ge1$, $\sigma\in(0,1)$ and $b=L_2R$. If
$g_k=\|\nabla f(x_k)\|=0$, then $x_k$ is the minimizer and Algorithm~2
returns it in Step~1; we therefore assume $g_k>0$ at every iteration
considered. Let us define
\[
v(\lambda)=\sqrt{\frac{\mu\lambda}{1+\mu\lambda}},\qquad
v_k=v(\lambda_k),\qquad
s_k=\frac{\sqrt{1+\mu\lambda_k}}{\lambda_k},\qquad
\rho_\star=1-v(\underline\lambda)\in(0,1),
\]
and $D_0=R$, $D_k=R/\sqrt{\mu A_k}$ for $k\ge1$.

At iteration $k$, hold the current state and model fixed. For a trial
stepsize $\lambda>0$, let $a(\lambda)$, $A'(\lambda)$, $\eta(\lambda)$,
$W(\lambda)$, $y(\lambda)$ and $m_k(\lambda)$ be as in
Section~\ref{sec:scmethod}, and set
\[
t=1+\mu A_k,\qquad t'(\lambda)=1+\mu A'(\lambda).
\]
The corresponding model step, trial point and HPE residual are
\[
r(\lambda)=-(B_k+\lambda^{-1}I)^{-1}\nabla f(y(\lambda)),\qquad
w(\lambda)=y(\lambda)+r(\lambda),\qquad
p(\lambda)=r(\lambda)+\lambda\nabla f(w(\lambda)).
\]
We write $\rho_k=\min\{\rho_k^{\rm R},\rho_k^{\rm E}\}$, where
$\rho_k^{\rm R}$ and $\rho_k^{\rm E}$ denote the first (resolvent) and the
second (energy) entry of the minimum in \eqref{eq:sc-rho}.

We first record that the rolling certificate remains valid when
the displacements are measured along the strongly convex iterates.

\begin{lemma}[The certificate]\label{lem:sc-cert}
At every iteration, $B_k\in\mathcal Z$, $0<\bar E_k\le L_1$ and
$\|B_k-H(x_k)\|_{\rm op}\le\bar E_k$. Moreover, for every completed horizon $N\ge1$,
\[
\sum_{k<N}\bar E_k\le\chi L_1+2\chi L_2\sum_{j<N}\delta_j+2b,
\qquad \delta_j=\|x_{j+1}-x_j\| .
\]
\end{lemma}

\begin{proof}
The construction gives $B_k\in\mathcal Z$. By the Lipschitz continuity
of the Hessian,
\[
\|H(x_{j+1})-H(x_j)\|_{\rm op}
\le L_2\|x_{j+1}-x_j\|
=L_2\delta_j.
\]
The proof of Lemma~\ref{lem:telescope} uses the trajectory only through
this inequality and the anchors of the stored blocks. Since the
refresh rule is unchanged, that proof applies with the present
displacements and gives
$\|B_k-H(x_k)\|_{\rm op}\le\bar E_k\le L_1$.
The strict positivity of $\bar E_k$ follows from
\eqref{eq:rollingcert}, since $L_1>0$ and $b>0$.
Finally, the summation argument of Lemma~\ref{lem:sum} depends only
on the formula \eqref{eq:rollingcert}, so it gives the stated bound
with these same displacements.
\end{proof}

\subsection{Weights and the potential}\label{app:sc-weights}
\begin{lemma}[Weights]\label{lem:sc-weights} 
Let $A\ge0$ and $\lambda>0$, let $a>0$ be the positive root of
$a^2=\lambda[A(1+\mu A)+a(1+2\mu A)]$, and put $A'=A+a$, $t=1+\mu A$,
$t'=1+\mu A'$, $\eta=at/t'$ and $W=A+\eta$. Then:
\begin{enumerate}[label=(\roman*)]
\item $a^2=\lambda A't'/(1+\mu\lambda)$, $a^2/t'=\lambda W$ and
      $W=A'/(1+\mu\lambda)$;
\item $\eta/W=\lambda t/a\in(0,1]$ and $\eta/W\le a/A'$;
\item $a^2\ge\lambda A'$ and $(a/A')^2=v(\lambda)^2\big(1+1/(\mu A')\big)$; in
      particular $a/A'\ge v(\lambda)$, and $a/A'\le\sqrt2\,v(\lambda)$ if
      $\mu A'\ge1$;
\item $\lambda/A'\le(a/A')^2$, and $\lambda/A'\le v(\lambda)^2/(\mu A)$ if
      $A>0$;
\item as functions of $\lambda$, $a$ and $\eta/W$ are nondecreasing, $a$ is
      concave with $a(0^+)=0$, and $\lambda/A'$ is nondecreasing.
\end{enumerate}
\end{lemma}
\begin{proof}

(i) Expanding, $A't'=A(1+\mu A)+a(1+2\mu A)+\mu a^2=a^2/\lambda+\mu a^2$,
which is the first identity. Also $At'+at=A(1+\mu A)+a(1+2\mu A)=a^2/\lambda$;
dividing by $t'$ gives $A+\eta=a^2/(\lambda t')$, which is the second
identity, and the two together give the third.

(ii) By (i), $\eta/W=(at/t')\cdot\lambda t'/a^2=\lambda t/a$. Since $t\le t'$,
we have $\eta\le a$, and $u\mapsto u/(A+u)$ is nondecreasing, so
$\eta/W=\eta/(A+\eta)\le a/(A+a)=a/A'\le1$.

(iii) The defining equation gives $a^2\ge\lambda(A+a)=\lambda A'$. By (i),
$(a/A')^2=\frac{\lambda t'}{(1+\mu\lambda)A'}
=\frac{\mu\lambda}{1+\mu\lambda}\cdot\frac{1+\mu A'}{\mu A'}$.

(iv) The first bound is (iii). By (i), $A'=W(1+\mu\lambda)\ge A(1+\mu\lambda)$,
so $\lambda/A'\le\lambda/(A(1+\mu\lambda))=v(\lambda)^2/(\mu A)$.

(v) Write
$a(\lambda)=\frac12\big[\lambda\beta+\sqrt{\beta^2\lambda^2+4\gamma\lambda}\big]$
with $\beta=1+2\mu A$ and $\gamma=A(1+\mu A)\ge0$. Both terms are
nondecreasing, and $a(\lambda)\to0$ as $\lambda\downarrow0$. With
$g(\lambda)=\beta^2\lambda^2+4\gamma\lambda$ one computes
$(\sqrt g)''=-4\gamma^2/g^{3/2}\le0$, so $a$ is concave. Hence
$a(\lambda)\ge\lambda a'(\lambda)$, and
$\frac{d}{d\lambda}\frac{\lambda}{A+a(\lambda)}=\frac{A+a-\lambda a'}{(A+a)^2}\ge0$.
Finally, $\eta=a(1+\mu A)/(1+\mu A+\mu a)$ is increasing in $a$, and
$\eta/W=\eta/(A+\eta)$ is nondecreasing in $\eta$.
\end{proof}

\begin{lemma}[Potential decrease]\label{lem:sc-contract}
Fix an iteration $k$, a stepsize $\lambda>0$ and a symmetric matrix
$B_k\succeq0$, and let
$z'=z_k+\frac{a}{t'}\big[\mu(w-z_k)-\nabla f(w)\big]$ with
$a=a(\lambda)$, $t'=t'(\lambda)$ and $w=w(\lambda)$. If
$\|p(\lambda)\|\le\sigma\sqrt{1+\mu\lambda}\,\|r(\lambda)\|$, then
\[
\Psi'+\frac{(1-\sigma^2)A'(\lambda)}{2\lambda}\|r(\lambda)\|^2\le\Psi_k,
\qquad
\Psi'=A'(\lambda)\big(f(w)-f^\star\big)+\frac{t'}{2}\|z'-x^\star\|^2 .
\]
For $\lambda=\lambda_k$, $\Psi'=\Psi_{k+1}$.
\end{lemma}
\begin{proof}
Abbreviate $x=x_k$, $z=z_k$, $y=y(\lambda)$, $g=\nabla f(w)$ and
$r=r(\lambda)=w-y$. Write $A=A_k$ and use $a$, $A'$, $t$, $t'$,
$\eta$ and $W$ as in Lemma~\ref{lem:sc-weights}.

\emph{Step 1 (the objective term).}
Since $A'=A+a$, strong convexity at $w$, applied to $x$ and to
$x^\star$, gives
\begin{align*}
A'\big(f(w)-f^\star\big)-A\big(f(x)-f^\star\big)
&=A\big(f(w)-f(x)\big)+a\big(f(w)-f^\star\big)\\
&\le A\langle g,w-x\rangle+a\langle g,w-x^\star\rangle -\frac{\mu}{2}
       \big[A\|w-x\|^2+a\|w-x^\star\|^2\big].
\end{align*}

\emph{Step 2 (the momentum term and cancellation).}
Using $z'-z=\frac{a}{t'}[\mu(w-z)-g]$ and $t'-t=\mu a$, expand
\begin{align*}
\frac{t'}2\|z'-x^\star\|^2-\frac t2\|z-x^\star\|^2
&=\frac{\mu a}{2}\|z-x^\star\|^2
  +a\langle z-x^\star,\mu(w-z)-g\rangle +\frac{a^2}{2t'}\|\mu(w-z)-g\|^2.
\end{align*}
Add this identity to the bound from Step~1. The gradient terms
involving $x^\star$ satisfy
\[
a\langle g,w-x^\star\rangle
-a\langle g,z-x^\star\rangle
=a\langle g,w-z\rangle.
\]
For the quadratic terms, the identity
\[
\|w-x^\star\|^2
=\|z-x^\star\|^2+\|w-z\|^2
 +2\langle z-x^\star,w-z\rangle
\]
gives
\begin{align*}
\Psi'-\Psi_k
&\le A\langle g,w-x\rangle+a\langle g,w-z\rangle
 +\frac{a^2}{2t'}\|\mu(w-z)-g\|^2
\quad-\frac{\mu}{2}
       \big[A\|w-x\|^2+a\|w-z\|^2\big].
\end{align*}
Expanding the remaining square, the coefficients of
$\langle g,w-z\rangle$ and $-\frac{\mu}{2}\|w-z\|^2$
both become
\[
a-\frac{\mu a^2}{t'}
=\frac{a(t'-\mu a)}{t'}
=\frac{at}{t'}=\eta.
\]
Consequently,
\begin{align*}
\Psi'-\Psi_k
&\le A\langle g,w-x\rangle+\eta\langle g,w-z\rangle
 +\frac{a^2}{2t'}\|g\|^2 -\frac{\mu}{2}
       \big[A\|w-x\|^2+\eta\|w-z\|^2\big].
\end{align*}

\emph{Step 3 (the center and the residual).}
The definition $y=(Ax+\eta z)/W$, with $W=A+\eta$, gives
\[
A(w-x)+\eta(w-z)=W(w-y)=Wr.
\]
Expanding the squared norms around $y$ and using
$A(y-x)+\eta(y-z)=0$ also gives
\[
A\|w-x\|^2+\eta\|w-z\|^2
=W\|r\|^2+\frac{A\eta}{W}\|x-z\|^2.
\]
Substituting these identities and
$a^2/t'=\lambda W$ from Lemma~\ref{lem:sc-weights}(i),
\begin{align*}
\Psi'-\Psi_k
&\le W\langle g,r\rangle+\frac{\lambda W}{2}\|g\|^2
 -\frac{\mu W}{2}\|r\|^2
 -\frac{\mu A\eta}{2W}\|x-z\|^2\\
&=\frac{W}{2\lambda}
   \big[\|r+\lambda g\|^2-(1+\mu\lambda)\|r\|^2\big]
 -\frac{\mu A\eta}{2W}\|x-z\|^2.
\end{align*}
Since $p(\lambda)=r+\lambda g$, the assumed residual condition bounds
the bracket by
$-(1-\sigma^2)(1+\mu\lambda)\|r\|^2$.
Discarding the last, nonpositive term and using
$W(1+\mu\lambda)=A'$ from Lemma~\ref{lem:sc-weights}(i), we obtain
\[
\Psi'-\Psi_k
\le-\frac{(1-\sigma^2)W(1+\mu\lambda)}{2\lambda}\|r\|^2
=-\frac{(1-\sigma^2)A'}{2\lambda}\|r\|^2,
\]
which proves the claim.
\end{proof}

\subsection{The certified test and the search}\label{app:sc-test}
We first derive two bounds on the model step. The resolvent bound
holds for every trial stepsize, while the energy bound follows from
the potential decrease for admissible steps. Lemma~\ref{lem:sc-admissible}
below uses a continuation argument to justify incorporating this
conditional bound into the certified test.

\begin{lemma}[Two bounds on the step]\label{lem:sc-steps}
Fix an iteration $k$ and $\lambda>0$.
\begin{enumerate}[label=(\roman*)]
\item $\|r(\lambda)\|\le\rho_k^{\rm R}(\lambda)$.
\item If $\Psi_k\le R^2/2$ and
      $\|p(\lambda)\|\le\sigma\sqrt{1+\mu\lambda}\,\|r(\lambda)\|$, then
      $\|r(\lambda)\|\le\rho_k^{\rm E}(\lambda)$.
\end{enumerate}
\end{lemma}

\begin{proof}

\emph{(i) The resolvent bound.}
Put $v=y(\lambda)-x_k$ and $m=\|v\|=m_k(\lambda)$.
By the fundamental theorem of calculus,
\[
\nabla f(y(\lambda))=\nabla f(x_k)+\bar Hv,
\qquad
\bar H=\int_0^1H(x_k+\tau v)\,d\tau.
\]
Adding and subtracting $B_kv$ in the model equation gives
\[
r(\lambda)=-(B_k+\lambda^{-1}I)^{-1}
\big[\nabla f(x_k)+B_kv+(\bar H-B_k)v\big].
\]
Since $B_k\succeq0$, we have
$\|(B_k+\lambda^{-1}I)^{-1}\|_{\rm op}\le\lambda$ and
$\|(B_k+\lambda^{-1}I)^{-1}B_k\|_{\rm op}\le1$. For the remaining term, Lemma~\ref{lem:sc-cert} and the Lipschitz
continuity of the Hessian give
\begin{align*}
\|\bar H-B_k\|_{\rm op}
&\le\|H(x_k)-B_k\|_{\rm op}
 +\int_0^1\|H(x_k+\tau v)-H(x_k)\|_{\rm op}\,d\tau\\
&\le\bar E_k+\int_0^1L_2\tau\|v\|\,d\tau
=\bar E_k+\frac{L_2}{2}m.
\end{align*}
Taking norms in the model equation and using
$g_k=\|\nabla f(x_k)\|$, we obtain
\begin{align*}
\|r(\lambda)\|
&\le\lambda g_k+m+\lambda\|\bar H-B_k\|_{\rm op}\,m\\
&\le\lambda g_k+m+\lambda\left(\bar E_k+\frac{L_2}{2}m\right)m\\
&=\rho_k^{\rm R}(\lambda).
\end{align*}

\emph{(ii) The energy bound.}
Under the assumed residual condition, Lemma~\ref{lem:sc-contract}
applies to the trial step and gives
\[
\frac{(1-\sigma^2)A'(\lambda)}{2\lambda}\|r(\lambda)\|^2
\le\Psi_k-\Psi'
\le\Psi_k
\le\frac{R^2}{2},
\]
where $\Psi'\ge0$ because both terms of the trial potential are
nonnegative. Rearranging and taking square roots yields
\[
\|r(\lambda)\|
\le R\sqrt{\frac{\lambda}{(1-\sigma^2)A'(\lambda)}}
=\rho_k^{\rm E}(\lambda).
\]

\end{proof}

\begin{lemma}[Monotonicity and scaling]\label{lem:sc-mono}
At every iteration, $m_k$, $\rho_k^{\rm R}$, $\rho_k^{\rm E}$, $\rho_k$ and
$\omega_k$ are continuous and nondecreasing on $(0,\infty)$, and for all
$c\ge1$ and $\lambda>0$,
\[
m_k(c\lambda)\le c\,m_k(\lambda),\quad
\rho_k^{\rm E}(c\lambda)\le\sqrt c\,\rho_k^{\rm E}(\lambda),\quad
\rho_k^{\rm R}(c\lambda)\le c^3\rho_k^{\rm R}(\lambda),\quad
\omega_k(c\lambda)\le c^3\omega_k(\lambda).
\]
Consequently, $\phi_k$ is continuous and strictly increasing, and
$\phi_k(c\lambda)\le c^4\phi_k(\lambda)$.
\end{lemma}
\begin{proof}
Fix an iteration $k$ and vary only $\lambda$.
By Lemma~\ref{lem:sc-weights}(ii),(v),
$m_k(\lambda)=\zeta_k\eta(\lambda)/W(\lambda)$
and $\lambda/A'(\lambda)$ are nondecreasing.
Hence $m_k$ and $\rho_k^{\rm E}$ are nondecreasing.
Since $a(\lambda)$ and $A'(\lambda)$ are also
nondecreasing, for every $c\ge1$,
\[
\begin{aligned}
m_k(c\lambda)
&=\frac{c\lambda t\zeta_k}{a(c\lambda)}
\le c\,m_k(\lambda),\\
\rho_k^{\rm E}(c\lambda)
&=R\sqrt{\frac{c\lambda}
{(1-\sigma^2)A'(c\lambda)}}
\le\sqrt c\,\rho_k^{\rm E}(\lambda).
\end{aligned}
\]

For the resolvent bound, write $m=m_k(\lambda)$.
Since
\[
\rho_k^{\rm R}(\lambda)
=\lambda g_k+m+\lambda\bar E_km
+\frac{L_2}{2}\lambda m^2,
\]
all its terms are nonnegative and nondecreasing.
Moreover, $m_k(c\lambda)\le cm$ implies
\[
\begin{aligned}
\rho_k^{\rm R}(c\lambda)
&\le c\lambda g_k+cm
+c^2\lambda\bar E_km
+\frac{c^3L_2}{2}\lambda m^2\\
&\le c^3\rho_k^{\rm R}(\lambda).
\end{aligned}
\]

Taking the minimum of the two step bounds,
we conclude that $\rho_k$ is nondecreasing and
\[
\rho_k(c\lambda)\le c^3\rho_k(\lambda).
\]
Consequently, $\omega_k$ is nondecreasing and
\[
\begin{aligned}
\omega_k(c\lambda)
&\le\min\{L_1,\bar E_k+
c^3L_2(m_k(\lambda)+\rho_k(\lambda))\}\\
&\le c^3\omega_k(\lambda),
\end{aligned}
\]
where the last inequality uses
$\min\{L_1,\bar E+qX\}
\le q\min\{L_1,\bar E+X\}$
for $q\ge1$ and $\bar E,X\ge0$.

Continuity follows from the explicit formula for
$a(\lambda)$ and the continuity of the operations
defining the remaining quantities. Finally,
$\lambda/\sqrt{1+\mu\lambda}$ is strictly
increasing, while $\omega_k$ is nondecreasing
and satisfies $\omega_k\ge\bar E_k>0$.
Thus $\phi_k$ is continuous and strictly
increasing. Its scaling bound follows from
\[
\begin{aligned}
\phi_k(c\lambda)
&=\frac{c\lambda\,\omega_k(c\lambda)}
{\sqrt{1+\mu c\lambda}}\\
&\le\frac{c^4\lambda\,\omega_k(\lambda)}
{\sqrt{1+\mu c\lambda}}
\le c^4\phi_k(\lambda).
\end{aligned}
\]
This completes the proof.
\end{proof}

\begin{lemma}[The test certifies admissibility]\label{lem:sc-admissible}
Fix an iteration $k$ with $\Psi_k\le R^2/2$, and let $\lambda>0$ satisfy
$\phi_k(\lambda)\le\sigma$. Then
$\|p(\lambda)\|\le\sigma\sqrt{1+\mu\lambda}\,\|r(\lambda)\|$ and
$\|r(\lambda)\|\le\rho_k(\lambda)$.
\end{lemma}
\begin{proof}
We hold the current state and model fixed and vary the trial
stepsize $h$ over $(0,\lambda]$. We first show that the weaker bound
$\|r(h)\|\le2\rho_k(h)$ implies admissibility and the stronger bound
$\|r(h)\|\le\rho_k(h)$. We then use continuity to establish the
weaker bound throughout this interval.

\emph{Step 1 (improving the step bound).}
Fix $h\in(0,\lambda]$. The model equation gives
\[
r(h)+h\nabla f(y(h))=-hB_kr(h).
\]
By the fundamental theorem of calculus,
\[
\nabla f(w(h))=\nabla f(y(h))+\bar H_h r(h),
\qquad
\bar H_h=\int_0^1H(y(h)+\tau r(h))\,d\tau.
\]
Combining these identities yields
\[
p(h)=h(\bar H_h-B_k)r(h).
\]
Since $B_k$ and $\bar H_h$ lie in $\mathcal Z$, their difference
has operator norm at most $L_1$. Lemma~\ref{lem:sc-cert} and the
Lipschitz continuity of the Hessian also give
\begin{align*}
\|\bar H_h-B_k\|_{\rm op}
&\le\|H(x_k)-B_k\|_{\rm op}
   +L_2\|y(h)-x_k\|+\frac{L_2}{2}\|r(h)\|\\
&\le\bar E_k+L_2m_k(h)+\frac{L_2}{2}\|r(h)\|.
\end{align*}
Suppose now that $\|r(h)\|\le2\rho_k(h)$. Taking the smaller of
the two bounds,
\[
\|\bar H_h-B_k\|_{\rm op}
\le\min\{L_1,\bar E_k+L_2(m_k(h)+\rho_k(h))\}
=\omega_k(h).
\]
Therefore,
\begin{align*}
\|p(h)\|
&\le h\,\omega_k(h)\|r(h)\|\\
&=\phi_k(h)\sqrt{1+\mu h}\,\|r(h)\|\\
&\le\sigma\sqrt{1+\mu h}\,\|r(h)\|,
\end{align*}
where the last inequality follows from
$\phi_k(h)\le\phi_k(\lambda)\le\sigma$ by
Lemma~\ref{lem:sc-mono}.

The step is thus admissible. Since $\Psi_k\le R^2/2$,
Lemma~\ref{lem:sc-steps}(ii) now gives
$\|r(h)\|\le\rho_k^{\rm E}(h)$.
Together with the unconditional resolvent bound of
Lemma~\ref{lem:sc-steps}(i), this gives
$\|r(h)\|\le\rho_k(h)$. We have proved
\begin{equation}\label{eq:sc-implication}
\|r(h)\|\le2\rho_k(h)
\quad\Longrightarrow\quad
\begin{cases}
\|p(h)\|\le\sigma\sqrt{1+\mu h}\,\|r(h)\|,\\
\|r(h)\|\le\rho_k(h).
\end{cases}
\end{equation}

\emph{Step 2 (small trial stepsizes).}
Because $y(h)$ is a convex combination of $x_k$ and $z_k$,
\[
G:=\max_{u\in[x_k,z_k]}\|\nabla f(u)\|<\infty.
\]
The resolvent estimate
$\|(B_k+h^{-1}I)^{-1}\|_{\rm op}\le h$ therefore gives
$\|r(h)\|\le hG$, without assuming admissibility. Consequently,
\[
\frac{\|r(h)\|}{\rho_k^{\rm E}(h)}
\le
\frac{G\sqrt{1-\sigma^2}}{R}\sqrt{hA'(h)}
\longrightarrow0
\qquad\text{as }h\downarrow0,
\]
since $A'(h)=A_k+a(h)\to A_k$ by
Lemma~\ref{lem:sc-weights}(v).
Thus $\|r(h)\|<2\rho_k^{\rm E}(h)$ for all sufficiently small $h$.
Also,
$\|r(h)\|\le\rho_k^{\rm R}(h)<2\rho_k^{\rm R}(h)$,
because $\rho_k^{\rm R}(h)\ge hg_k>0$.
Taking the minimum gives
$\|r(h)\|<2\rho_k(h)$ for all sufficiently small $h$.

\emph{Step 3 (continuation).}
The maps $h\mapsto r(h)$ and $h\mapsto\rho_k(h)$ are continuous
on $(0,\lambda]$. Moreover, $\rho_k(h)>0$, since
$\rho_k^{\rm R}(h)\ge hg_k>0$ and $\rho_k^{\rm E}(h)>0$.
Hence
\[
F(h):=\frac{\|r(h)\|}{\rho_k(h)}
\]
is continuous on $(0,\lambda]$ and is strictly below $2$ for
all sufficiently small $h$ by Step~2.

If $F(h)\ge2$ at some point of this interval, the intermediate
value theorem gives an $h_0\in(0,\lambda]$ with $F(h_0)=2$.
But \eqref{eq:sc-implication} then gives $F(h_0)\le1$, a contradiction.
Therefore $\|r(h)\|<2\rho_k(h)$ throughout $(0,\lambda]$.
Applying \eqref{eq:sc-implication} at $h=\lambda$ proves both claims.
\end{proof}

\begin{lemma}[The search]\label{lem:sc-search}
At every iteration, $\underline\lambda\le\bar\lambda_k$ and
$\phi_k(\underline\lambda)\le\sigma\le\phi_k(\bar\lambda_k)$, and
Subroutine~2.1 returns $\lambda_k$ with
\[
\underline\lambda\le\lambda_k\le\bar\lambda_k,\qquad
\frac\sigma2\le\phi_k(\lambda_k)\le\sigma,
\qquad\text{hence}\qquad
s_k\le\frac2\sigma\,\omega_k(\lambda_k),
\]
after at most
$\big\lceil\log_2\max\{1,\log(\bar\lambda_k/\underline\lambda)/\log2^{1/4}\}\big\rceil$
evaluations of $\phi_k$ and no gradient query. Moreover,
$\underline\lambda\ge\sigma/L_1$.
\end{lemma}
\begin{proof}
\emph{Step 1 (the initial bracket).}
The map $h\mapsto h/\sqrt{1+\mu h}$ is continuous, strictly increasing
and unbounded on $(0,\infty)$. By definition,
\[
\frac{\underline\lambda}{\sqrt{1+\mu\underline\lambda}}
=\frac{\sigma}{L_1},
\qquad
\frac{\bar\lambda_k}{\sqrt{1+\mu\bar\lambda_k}}
=\frac{\sigma}{\bar E_k}.
\]
Since $0<\bar E_k\le L_1$, these identities imply
$\underline\lambda\le\bar\lambda_k$.
Moreover, $\bar E_k\le\omega_k(h)\le L_1$ for every $h>0$, so
\[
\phi_k(\underline\lambda)
\le\frac{\underline\lambda L_1}
         {\sqrt{1+\mu\underline\lambda}}
=\sigma,
\qquad
\phi_k(\bar\lambda_k)
\ge\frac{\bar\lambda_k\bar E_k}
         {\sqrt{1+\mu\bar\lambda_k}}
=\sigma.
\]
Thus the initialization
$\lambda^-=\underline\lambda$, $\lambda^+=\bar\lambda_k$
satisfies the bracket invariant
\[
\phi_k(\lambda^-)\le\sigma\le\phi_k(\lambda^+).
\]
The defining identity for $\underline\lambda$ also gives
\[
\underline\lambda
=\frac{\sigma}{L_1}\sqrt{1+\mu\underline\lambda}
\ge\frac{\sigma}{L_1}.
\]

\emph{Step 2 (termination and evaluation count).}
Each pass evaluates $\phi_k$ once, at the geometric midpoint
$\sqrt{\lambda^-\lambda^+}$, and replaces one endpoint by it; this
preserves the bracket invariant, keeps both endpoints in
$[\underline\lambda,\bar\lambda_k]$, and halves
$\log(\lambda^+/\lambda^-)$. Starting from
$\log(\bar\lambda_k/\underline\lambda)$ and stopping once the width is
at most $\log 2^{1/4}$ gives the stated number of evaluations. Each
evaluation uses only stored scalars through
\eqref{eq:sc-rho}--\eqref{eq:sc-test}, so no gradient is queried.
The returned $\lambda_k=\lambda^-$ therefore satisfies
$\underline\lambda\le\lambda_k\le\bar\lambda_k$ and
$\phi_k(\lambda_k)\le\sigma$.

\emph{Step 3 (near-maximality).}
At termination,
$\lambda^+\le2^{1/4}\lambda^-=2^{1/4}\lambda_k$.
The bracket invariant, monotonicity and the scaling bound of
Lemma~\ref{lem:sc-mono} give
\[
\sigma
\le\phi_k(\lambda^+)
\le\phi_k(2^{1/4}\lambda_k)
\le(2^{1/4})^4\phi_k(\lambda_k)
=2\phi_k(\lambda_k).
\]
Hence $\sigma/2\le\phi_k(\lambda_k)\le\sigma$.
Finally, the definitions of $s_k$ and $\phi_k$ give
\[
\phi_k(\lambda_k)
=\frac{\lambda_k\omega_k(\lambda_k)}
       {\sqrt{1+\mu\lambda_k}}
=\frac{\omega_k(\lambda_k)}{s_k}.
\]
Rearranging $\phi_k(\lambda_k)\ge\sigma/2$ yields
$s_k\le2\omega_k(\lambda_k)/\sigma$.
\end{proof}

\subsection{Well-posedness, bounded iterates and the weights}\label{app:sc-run}
\begin{lemma}[Well-posedness and energy]\label{lem:sc-energy}
Algorithm~2 is well defined, every step satisfies \eqref{eq:sc-hpe},
$\lambda_k\ge\underline\lambda$ for every $k$, and every completed horizon
$N$ satisfies
\begin{equation}\label{eq:sc-energy}
\Psi_N+\frac{1-\sigma^2}{2}\sum_{k<N}\frac{A_{k+1}}{\lambda_k}\|r_k\|^2
\le\Psi_0\le\frac{R^2}{2}.
\end{equation}
Consequently, for $k\ge1$, $f(x_k)-f^\star\le R^2/(2A_k)$,
$\|x_k-x^\star\|\le D_k$ and
$\|z_k-x^\star\|\le R/\sqrt{1+\mu A_k}\le\min\{R,D_k\}$; and
$\|r_k\|\le\rho_k^{\rm E}(\lambda_k)$ for every $k$.
\end{lemma}
\begin{proof}
\emph{Well-posedness and admissibility.}
We argue by induction over the iterations completed before termination.
Since $A_0=0$ and $z_0=x_0$,
\[
\Psi_0=\frac12\|x_0-x^\star\|^2\le\frac{R^2}{2}.
\]
Suppose that iterations $0,\ldots,k-1$ are well defined and admissible,
that $\Psi_k\le\Psi_0$, and that the method continues at iteration $k$.
Lemma~\ref{lem:sc-cert} gives
$B_k\in\mathcal Z$ and a valid certificate with $0<\bar E_k\le L_1$.
Thus, by Lemma~\ref{lem:sc-search}, Subroutine~2.1 terminates and returns
\[
\lambda_k\ge\underline\lambda,
\qquad
\phi_k(\lambda_k)\le\sigma.
\]
The weight recursion gives $a_k>0$ and $A_{k+1}=A_k+a_k>0$.
Moreover, $B_k\succeq0$ and $\lambda_k>0$ imply
$B_k+\lambda_k^{-1}I\succ0$, so the model step is well defined;
the auxiliary-point update is also well defined because
$1+\mu A_{k+1}>0$.

The induction hypothesis gives $\Psi_k\le\Psi_0\le R^2/2$.
Together with $\phi_k(\lambda_k)\le\sigma$, this allows us to apply
Lemma~\ref{lem:sc-admissible}, which proves \eqref{eq:sc-hpe}.
Lemma~\ref{lem:sc-contract} then yields
\[
\Psi_{k+1}
+\frac{(1-\sigma^2)A_{k+1}}{2\lambda_k}\|r_k\|^2
\le\Psi_k.
\]
In particular, $\Psi_{k+1}\le\Psi_0$, closing the induction.

\emph{Energy and radius bounds.}
Summing the last inequality over $k=0,\ldots,N-1$ telescopes the
potential terms and gives \eqref{eq:sc-energy}.
Since both terms of $\Psi_k$ are nonnegative, it follows that
\[
A_k\big(f(x_k)-f^\star\big)\le\frac{R^2}{2},
\qquad
(1+\mu A_k)\|z_k-x^\star\|^2\le R^2.
\]
For $k\ge1$, we have $A_k>0$, so the first inequality gives
$f(x_k)-f^\star\le R^2/(2A_k)$.
Combining this with strong convexity,
\[
\frac{\mu}{2}\|x_k-x^\star\|^2
\le f(x_k)-f^\star
\le\frac{R^2}{2A_k},
\]
and hence $\|x_k-x^\star\|\le R/\sqrt{\mu A_k}=D_k$.
The second inequality gives
\[
\|z_k-x^\star\|
\le\frac{R}{\sqrt{1+\mu A_k}}
\le\min\left\{R,\frac{R}{\sqrt{\mu A_k}}\right\}
=\min\{R,D_k\}.
\]
Finally, every completed step is admissible and satisfies
$\Psi_k\le R^2/2$. Lemma~\ref{lem:sc-steps}(ii), applied with
$\lambda=\lambda_k$, therefore gives
$\|r_k\|\le\rho_k^{\rm E}(\lambda_k)$.
\end{proof}

\begin{lemma}[Bounded iterates]\label{lem:sc-bounded}
For every $k\ge0$, $\|x_k-x^\star\|\le C_xR$ with
$C_x=1+\sigma/\sqrt{1-\sigma^2}$; hence $\zeta_k\le(1+C_x)R$.
\end{lemma}
\begin{proof}
Write $p_k=r_k+\lambda_k\nabla f(x_{k+1})$. Since
$x_{k+1}+\lambda_k\nabla f(x_{k+1})=y_k+p_k$ and
$\nabla f(x^\star)=0$, strong monotonicity gives
\[
(1+\mu\lambda_k)\|x_{k+1}-x^\star\|^2
\le\langle y_k-x^\star+p_k,x_{k+1}-x^\star\rangle.
\]
By Cauchy--Schwarz,
\[
(1+\mu\lambda_k)\|x_{k+1}-x^\star\|
\le\|y_k-x^\star\|+\|p_k\|.
\]
Multiply by $W_k$ and use
$W_k(1+\mu\lambda_k)=A_{k+1}$
(Lemma~\ref{lem:sc-weights}),
$W_k\|y_k-x^\star\|
\le A_k\|x_k-x^\star\|+\eta_k\|z_k-x^\star\|$,
$\|z_k-x^\star\|\le R$ (Lemma~\ref{lem:sc-energy}),
$\eta_k\le a_k$ and \eqref{eq:sc-hpe}:
\[
A_{k+1}\|x_{k+1}-x^\star\|
\le A_k\|x_k-x^\star\|+a_kR
+\sigma\frac{A_{k+1}}{\sqrt{1+\mu\lambda_k}}\|r_k\|.
\]
Telescoping from $A_0=0$ gives
$A_N\|x_N-x^\star\|
\le A_NR+\sigma\sum_{k<N}
\frac{A_{k+1}}{\sqrt{1+\mu\lambda_k}}\|r_k\|$.
Write
$\frac{A_{k+1}}{\sqrt{1+\mu\lambda_k}}\|r_k\|
=\big(\sqrt{A_{k+1}/\lambda_k}\,\|r_k\|\big)
\sqrt{\lambda_kW_k}$.
By Cauchy--Schwarz, \eqref{eq:sc-energy} and
$\lambda_kW_k=a_k^2/t'_k\le a_k^2$,
the sum is at most
\[
\frac{R}{\sqrt{1-\sigma^2}}
\left(
\sum_{k<N}
a_k^2
\right)^{1/2}
\le\frac{R}{\sqrt{1-\sigma^2}}A_N.
\]
Dividing by $A_N>0$ proves the bound for $N\ge1$; for $N=0$,
it follows from $\|x_0-x^\star\|\le R$. Finally,
\[
\zeta_k\le\|z_k-x^\star\|+\|x_k-x^\star\|
\le(1+C_x)R,
\]
using Lemma~\ref{lem:sc-energy} and $z_0=x_0$.
\end{proof}

\begin{lemma}[Growth of the weights]\label{lem:sc-mass}
For every $k\ge1$, $A_k\ge\underline\lambda k^2/4$, and $A_1=\lambda_0\ge\underline\lambda$.
For every $k\ge0$, $A_k/A_{k+1}\le1-v_k\le\rho_\star$. Moreover, with
$K_\star=\lceil1/(1-\sqrt{\rho_\star})\rceil$ and
$T_\star=1+2(2+\log K_\star)/\sqrt{\mu\underline\lambda}$, every completed
horizon satisfies $\sum_{k=0}^ND_k\le RT_\star$, and
$T_\star=O_\sigma(\sqrt\kappa\log(2+\kappa))$.
\end{lemma}
\begin{proof}
By Lemma~\ref{lem:sc-weights}(iii),
$a_k\ge\sqrt{\lambda_kA_{k+1}}$, so
$\sqrt{A_{k+1}}-\sqrt{A_k}
=a_k/(\sqrt{A_{k+1}}+\sqrt{A_k})
\ge\sqrt{\lambda_k}/2
\ge\sqrt{\underline\lambda}/2$;
summing from $A_0=0$ gives the polynomial bound.
For $A_0=0$, \eqref{eq:sc-a} gives $a_0=\lambda_0$. By
Lemma~\ref{lem:sc-weights}(iii),
$A_k/A_{k+1}=1-a_k/A_{k+1}\le1-v_k$, and
$v_k\ge v(\underline\lambda)$ because $v$ is increasing.

For the radii, the polynomial bound gives
$D_k\le2R/(\sqrt{\mu\underline\lambda}\,k)$ for $k\ge1$.
Put $K=K_\star$. If $N\ge K$, the geometric bound gives
$A_k\ge A_K\rho_\star^{-(k-K)}$, and hence
\[
D_k\le\frac{2R}{\sqrt{\mu\underline\lambda}\,K}
\rho_\star^{(k-K)/2},
\qquad K\le k\le N.
\]
Hence
\[
\sum_{k=0}^ND_k
\le R+\frac{2R}{\sqrt{\mu\underline\lambda}}
\Big[
\sum_{k=1}^{K-1}\frac1k
+\frac{1}{K(1-\sqrt{\rho_\star})}
\Big]
\le R+\frac{2R}{\sqrt{\mu\underline\lambda}}
(2+\log K)
=RT_\star,
\]
using $\sum_{1\le k<K}1/k\le1+\log K$ and
$K(1-\sqrt{\rho_\star})\ge1$; 
If $N<K$, the same bound follows from the polynomial estimate alone.

For the order, the explicit root
$\underline\lambda
=\sigma(\sigma\mu+\sqrt{\sigma^2\mu^2+4L_1^2})/(2L_1^2)$
gives
$\sigma/\kappa\le\mu\underline\lambda
\le(\sigma+\sigma^2)/\kappa$, so
$1-\rho_\star=v(\underline\lambda)
=\Theta_\sigma(\kappa^{-1/2})$, and
$1-\sqrt{\rho_\star}\ge(1-\rho_\star)/2$
gives $K_\star=O_\sigma(\sqrt\kappa)$.
Together with
$(\mu\underline\lambda)^{-1/2}=O_\sigma(\sqrt\kappa)$,
this yields
$T_\star=O_\sigma(\sqrt\kappa\log(2+\kappa))$.
\end{proof}

\subsection{The path length}\label{app:sc-path}
\begin{lemma}[Path length]\label{lem:sc-path}
For every completed horizon $N$,
\[
\sum_{k<N}\delta_k\ \le\ P_N:=\sum_{k<N}\big(m_k(\lambda_k)+\rho^{\rm E}_k(\lambda_k)\big)
\ \le\ T_\sharp R,
\]
with $T_\sharp$ as in Section~\ref{sec:scguarantee}. The proof uses only
\eqref{eq:sc-hpe}, \eqref{eq:sc-energy} and $\lambda_k\ge\underline\lambda$.
\end{lemma}
\begin{proof}
The claim is immediate for $N=0$, so assume $N\ge1$.
Write $m_k=m_k(\lambda_k)=\|y_k-x_k\|$ and
$\rho^{\rm E}_k=\rho^{\rm E}_k(\lambda_k)$.
Since $x_{k+1}-x_k=(y_k-x_k)+r_k$,
Lemma~\ref{lem:sc-energy} gives
$\delta_k\le m_k+\|r_k\|\le m_k+\rho^{\rm E}_k$.
Let $k_1=\min\{k\ge1:\mu A_k\ge1\}$, with $k_1=\infty$
if no such $k$ exists, and $K'=\min\{N,k_1\}$.

\emph{Early indices, $k<K'$.}
By Lemma~\ref{lem:sc-weights}(ii),(iv),
$m_k=\frac{\eta_k}{W_k}\zeta_k
\le\frac{a_k}{A_{k+1}}\zeta_k$ and
$\rho^{\rm E}_k
\le\frac{a_k}{A_{k+1}}\frac{R}{\sqrt{1-\sigma^2}}$,
and $\zeta_k\le(1+C_x)R$ by Lemma~\ref{lem:sc-bounded}.
Moreover,
$\sum_{k<K'}a_k/A_{k+1}
\le2+\log_+\frac{1}{\mu\underline\lambda}$:
if $K'=1$, the sum equals one. For $K'\ge2$, the term
$k=0$ equals one because $A_0=0$, the term $k=K'-1$
is at most one, and the terms $1\le k\le K'-2$ satisfy
$a_k/A_{k+1}
=(A_{k+1}-A_k)/A_{k+1}
\le\log(A_{k+1}/A_k)$ and telescope to at most
$\log(A_{K'-1}/A_1)
\le\log_+\frac{1}{\mu\underline\lambda}$,
because $\mu A_{K'-1}<1$ and $A_1\ge\underline\lambda$.
The early indices therefore contribute at most
$\big(1+C_x+(1-\sigma^2)^{-1/2}\big)
\big(2+\log_+\frac1{\mu\underline\lambda}\big)R$.

\emph{Late indices, $k_1\le k<N$.}
If $k_1\ge N$, this contribution is zero. Otherwise,
$\mu A_{k+1}\ge\mu A_k\ge1$, so
$a_k/A_{k+1}\le\sqrt2\,v_k$
by Lemma~\ref{lem:sc-weights}(iii), and
$\zeta_k\le\|z_k-x^\star\|+\|x_k-x^\star\|
\le2D_k$ by Lemma~\ref{lem:sc-energy}.
Thus
$m_k\le\frac{a_k}{A_{k+1}}\zeta_k
\le2\sqrt2\,v_kD_k$,
and
$\rho^{\rm E}_k\le v_kD_k/\sqrt{1-\sigma^2}$
by Lemma~\ref{lem:sc-weights}(iv).
By Lemma~\ref{lem:sc-mass},
$D_{k+1}=D_k\sqrt{A_k/A_{k+1}}
\le D_k\sqrt{1-v_k}\le D_k(1-v_k/2)$,
so $v_kD_k\le2(D_k-D_{k+1})$, and the sum over late
indices telescopes to at most $2D_{k_1}\le2R$.
The late indices therefore contribute at most
$2\big(2\sqrt2+(1-\sigma^2)^{-1/2}\big)R$,
and adding the two contributions gives $P_N\le T_\sharp R$.
\end{proof}

\subsection{Proof of Theorem~\ref{thm:scmain}}\label{proof:scmain}
\begin{proof}
Well-posedness, admissibility and $\lambda_k\ge\underline\lambda\ge\sigma/L_1$
are Lemmas~\ref{lem:sc-energy} and~\ref{lem:sc-search}; the gradient count
was verified in Section~\ref{sec:scmethod}, and the evaluation count is
Lemma~\ref{lem:sc-search}.

\emph{The moment bound.} By Lemma~\ref{lem:sc-search},
$s_k\le\frac2\sigma\omega_k(\lambda_k)$, and
$\omega_k(\lambda_k)\le\bar E_k+L_2\big(m_k(\lambda_k)+\rho^{\rm E}_k(\lambda_k)\big)$
since $\rho_k\le\rho^{\rm E}_k$. Lemmas~\ref{lem:sc-cert} and~\ref{lem:sc-path}
give
\[
\sum_{k<N}s_k\le\frac2\sigma\Big[\chi L_1+2\chi L_2\sum_{k<N}\delta_k+2b+L_2P_N\Big]
\le\frac2\sigma\Big[\chi L_1+2L_2R+(2\chi+1)L_2RT_\sharp\Big]=C_\sharp ,
\]
which is \eqref{eq:sc-moment}. The order of $B_\sharp$ follows from
$\log_+\frac1{\mu\underline\lambda}\le\log\frac\kappa\sigma$.

\emph{Three lower bounds on $A_N$.}
Let $n=N-1$. If $n=0$, all three bounds reduce to
$A_1\ge\underline\lambda$. We therefore assume $n\ge1$
in the following estimates.
First, Lemma~\ref{lem:sc-mass} gives
$A_N\ge A_1\rho_\star^{-n}
\ge\underline\lambda\rho_\star^{-n}$. Next, for
$v=v(\lambda)$, $\frac{1}{1-v}=\frac{1+v}{1-v^2}=(1+\mu\lambda)(1+v)$, so
Lemma~\ref{lem:sc-mass} gives, for $k\ge1$,
\begin{equation}\label{eq:sc-logratio}
\log\frac{A_{k+1}}{A_k}\ \ge\ -\log(1-v_k)\ \ge\ \max\{\log(1+\mu\lambda_k),\ v_k\}.
\end{equation}
\emph{Superlinear bound.} Solving $s^2\lambda^2-\mu\lambda-1=0$ for
$\lambda>0$ gives $1+\mu\lambda=\big(\frac{\mu}{2s}+\sqrt{1+\frac{\mu^2}{4s^2}}\big)^2$,
that is, $\log(1+\mu\lambda_k)=2\psi(s_k)$ with
$\psi(s)=\operatorname{arsinh}(\mu/(2s))$. Direct differentiation gives
$\psi'(s)=-\alpha/(s\sqrt{s^2+\alpha^2})<0$ and
$\psi''(s)=\alpha(2s^2+\alpha^2)/(s^2(s^2+\alpha^2)^{3/2})>0$ with $\alpha=\mu/2$,
so $\psi$ is convex and decreasing. By \eqref{eq:sc-logratio}, Jensen's
inequality and \eqref{eq:sc-moment},
\[
\log\frac{A_N}{A_1}\ \ge\ 2\sum_{k=1}^{n}\psi(s_k)\ \ge\ 2n\,\psi\Big(\frac1n\sum_{k=1}^ns_k\Big)
\ \ge\ 2n\,\psi\Big(\frac{C_\sharp}{n}\Big)=2n\operatorname{arsinh}\frac{n}{2B_\sharp}.
\]
\emph{Moderate-accuracy bound.} If $\mu\lambda_k\le1$, then
$s_k/\mu=\sqrt{1+\mu\lambda_k}/(\mu\lambda_k)\ge1/(\mu\lambda_k)$ and
$v_k\ge\sqrt{\mu\lambda_k/2}\ge(2s_k/\mu)^{-1/2}$; otherwise $v_k\ge2^{-1/2}$.
Hence $v_k\ge2^{-1/2}\min\{1,(s_k/\mu)^{-1/2}\}$. Let $J=\{1\le k\le n:s_k\ge\mu\}$
and $I=\{1,\ldots,n\}\setminus J$. H\"older's inequality with exponents $3$
and $3/2$ gives
$|J|\le\big(\sum_{J}s_k/\mu\big)^{1/3}\big(\sum_J(s_k/\mu)^{-1/2}\big)^{2/3}$,
so $\sum_J(s_k/\mu)^{-1/2}\ge|J|^{3/2}B_\sharp^{-1/2}$ by \eqref{eq:sc-moment}.
With \eqref{eq:sc-logratio},
\[
\log\frac{A_N}{A_1}\ \ge\ \sum_{k=1}^nv_k\ \ge\ 2^{-1/2}\big(|I|+|J|^{3/2}B_\sharp^{-1/2}\big)
\ \ge\ 2^{-1/2}\min\Big\{\frac n2,\ \Big(\frac n2\Big)^{3/2}B_\sharp^{-1/2}\Big\},
\]
because one of $I$ and $J$ has at least $n/2$ elements.

\emph{Conclusion.} By Lemma~\ref{lem:sc-energy} and $A_1\ge\underline\lambda$,
$f(x_N)-f^\star\le R^2/(2A_N)$ is at most $R^2/(2\underline\lambda)$ times the
smallest of $\rho_\star^n$,
$\exp(-2n\operatorname{arsinh}(n/(2B_\sharp)))$ and
$\exp(-2^{-1/2}\min\{n/2,(n/2)^{3/2}B_\sharp^{-1/2}\})$. Finally,
$R^2/(2\underline\lambda)\le L_1R^2/(2\sigma)$; $\rho_\star\le q$, because $v$ is
increasing and $\underline\lambda\ge\sigma/L_1$ gives
$v(\underline\lambda)\ge\sqrt{\sigma/(\kappa+\sigma)}$; and
$\exp(2\operatorname{arsinh}x)=(x+\sqrt{1+x^2})^2\ge1+4x^2$ for $x\ge0$, which
gives $\exp(-2n\operatorname{arsinh}(n/(2B_\sharp)))\le(1+n^2/B_\sharp^2)^{-n}$.
This is \eqref{eq:sc-rate}.
\end{proof}

\subsection{Proof of Corollary~\ref{cor:scqueries}}\label{proof:scqueries}
\begin{proof}
Let $0<\varepsilon<L_1R^2/2$ and write $\mathcal L=\mathcal L_\varepsilon$; then
$\mathcal L>0$. By Lemma~\ref{lem:sc-energy}, the stopping rule
$A_N\ge R^2/(2\varepsilon)$ certifies $f(x_N)-f^\star\le\varepsilon$. By the
proof of Theorem~\ref{thm:scmain}, $A_N\ge\underline\lambda e^{\Gamma}$, where
$\Gamma$ is the largest of $n\log(1/\rho_\star)$,
$2n\operatorname{arsinh}(n/(2B_\sharp))$ and
$2^{-1/2}\min\{n/2,(n/2)^{3/2}B_\sharp^{-1/2}\}$, with $n=N-1$. Since
$\underline\lambda e^{\mathcal L}\ge\frac{\sigma}{L_1}\cdot\frac{L_1R^2}{2\sigma\varepsilon}=\frac{R^2}{2\varepsilon}$,
it suffices that $\Gamma\ge\mathcal L$, and each of the following conditions
is sufficient.
\begin{itemize}
\item $n\ge\mathcal L/\log(1/q)$, since $\rho_\star\le q$.
\item $n\ge\max\{2\sqrt2\,\mathcal L,\ 2^{4/3}\mathcal L^{2/3}B_\sharp^{1/3}\}$:
      then $n/2\ge\sqrt2\,\mathcal L$ and
      $(n/2)^{3/2}B_\sharp^{-1/2}\ge\sqrt2\,\mathcal L$.
\item $n\ge\mathcal L/\log(1+(\mathcal L/B_\sharp)^{2/3})$. Indeed, put
      $u=(\mathcal L/B_\sharp)^{2/3}$ and $m=\mathcal L/\log(1+u)$. Since
      $\log(1+u)\le u$, $m^2/B_\sharp^2=u^3/\log^2(1+u)\ge u$, so
      $m\log(1+m^2/B_\sharp^2)\ge\mathcal L$; the map
      $n\mapsto n\log(1+n^2/B_\sharp^2)$ is increasing, and
      $\exp(2n\operatorname{arsinh}(n/(2B_\sharp)))\ge(1+n^2/B_\sharp^2)^n$.
\end{itemize}
Taking the ceiling of the smallest of these three sufficient
thresholds gives $n_\varepsilon$. Thus the algorithm terminates
after at most $1+n_\varepsilon$ iterations. The
algorithm queries one initial gradient and at most three gradients per
iteration, which gives at most $4+3n_\varepsilon$ gradients. For
\eqref{eq:sc-complexity}, $\log(1/q)\ge1-q=\sqrt{\sigma/(\kappa+\sigma)}$ gives
$\mathcal L/\log(1/q)\le\sqrt{(\kappa+\sigma)/\sigma}\,\mathcal L$, and
$\max\{2\sqrt2\,\mathcal L,2^{4/3}\mathcal L^{2/3}B_\sharp^{1/3}\}\le3\big(\mathcal L+(B_\sharp\mathcal L^2)^{1/3}\big)$.
\end{proof}

\subsection{Proof of Remark~\ref{rem:scexplicit}}\label{proof:scexplicit}
\begin{proof}
Consider the rule $\lambda_k\omega_k^{\rm ex}=\sigma\sqrt{1+\mu\lambda_k}$ and
argue by induction as in Lemma~\ref{lem:sc-energy}. If $\Psi_k\le R^2/2$,
then $\|x_k-x^\star\|\le D_k$ and $\|z_k-x^\star\|\le D_k$; for $k=0$ both
points equal $x_0$. For every $\lambda>0$,
$m_k(\lambda)\le\|z_k-x_k\|\le2D_k$,
and, for $k\ge1$,
$\rho_k(\lambda)\le\rho^{\rm E}_k(\lambda)
\le D_k/\sqrt{1-\sigma^2}$
by Lemma~\ref{lem:sc-weights}(iv);
for $k=0$, $A'(\lambda)=\lambda$ and
$\rho^{\rm E}_0=R/\sqrt{1-\sigma^2}$. Hence $\omega_k(\lambda)\le\omega_k^{\rm ex}$
for every $\lambda$, so
$\phi_k(\lambda_k)\le\lambda_k\omega_k^{\rm ex}/\sqrt{1+\mu\lambda_k}=\sigma$, and
Lemma~\ref{lem:sc-admissible} gives admissibility; since
$\omega_k^{\rm ex}\le L_1$, also $\lambda_k\ge\underline\lambda$. The induction
therefore goes through, and Lemmas~\ref{lem:sc-bounded}, \ref{lem:sc-mass}
and~\ref{lem:sc-path} hold, since their proofs use only admissibility,
\eqref{eq:sc-energy} and $\lambda_k\ge\underline\lambda$. Now
$s_k=\omega_k^{\rm ex}/\sigma$, and Lemmas~\ref{lem:sc-cert}, \ref{lem:sc-path}
and~\ref{lem:sc-mass} give
\[
\sum_{k<N}s_k\le\frac1\sigma\Big[\chi L_1+2b+2\chi L_2T_\sharp R+c_\sigma L_2RT_\star\Big]
=\mu\,O_\sigma\big(\chi\kappa+\chi\kappa_2\log(2+\kappa)+\kappa_2\sqrt\kappa\log(2+\kappa)\big).
\]
The rate arguments in the proof of Theorem~\ref{thm:scmain}
therefore apply with $C_\sharp$ replaced by the explicit upper
bound above, since $\lambda_k\ge\underline\lambda$ still holds.
\end{proof}

\section{Restarting the convex method}\label{app:strongly}
Strong convexity also permits a restart scheme based on Algorithm~1:
at each stage, run the convex method until its guarantee certifies
a distance to the minimizer at most half the current radius,
then restart from the returned point.
This appendix describes the scheme and proves its guarantee,
Theorem~\ref{thm:screstart}, which
Section~\ref{sec:sccompare} compares with Corollary~\ref{cor:scqueries}.
Appendix~\ref{sec:screstartsec} states the scheme and the theorem,
Appendix~\ref{app:scstages} establishes the radius invariant, counts the
stages and records two elementary estimates, and
Appendix~\ref{proof:screstart} applies the explicit convex query bound
\eqref{eq:fp-explicit-queries} at each stage and sums the resulting costs.
Throughout, $\kappa$, $\kappa_2$, $\mathcal L_\varepsilon$ and
$\mathcal S_\varepsilon$ are as in Section~\ref{sec:strongly}.

\subsection{The restart scheme and its guarantee}\label{sec:screstartsec}
Suppose a run of Algorithm~1 starts at $w$ with $\|w-x^\star\|\le\bar R$ and
returns $w^+$ with $f(w^+)-f^\star\le\bar\varepsilon$. Strong convexity gives
$\frac\mu2\|w^+-x^\star\|^2\le\bar\varepsilon$, so the choice
$\bar\varepsilon=\mu\bar R^2/8$ yields $\|w^+-x^\star\|\le\bar R/2$
(Lemma~\ref{lem:halve}), and the mass target of the run,
$\bar R^2/(2\bar\varepsilon)=4/\mu$, is the same at every stage. The
\emph{restart scheme} runs Algorithm~1 in stages $s=0,\ldots,S-1$. Stage $s$
starts from $w_s$, with $w_0=x_0$, uses $R_s=2^{-s}R$,
$\varepsilon_s=\mu R_s^2/8$ and $b=L_2R_s$ wherever the convex method uses
$R$, $\varepsilon$ and $b$, starts with $A_0=0$ and a fresh model
($M_{-1}=0$), and returns $w_{s+1}$ once $A_{k+1}\ge4/\mu$. The number of
stages is
$S=\max\{0,\lceil\log_2(R\sqrt{L_1/(2\varepsilon)})\rceil\}\le\mathcal S_\varepsilon$,
which ensures $f(w_S)-f^\star\le L_1R_S^2/2\le\varepsilon$
(Lemma~\ref{lem:wrapper}). Applying the explicit form
\eqref{eq:fp-explicit-queries} in the proof of Corollary~\ref{thm:complexity} once per
stage gives the following bound.

\begin{theorem}[Restarted RP-QNPE]\label{thm:screstart}
Under Assumptions~\ref{as:main} and~\ref{as:sc}, the restart scheme
returns a point with $f(w_S)-f^\star\le\varepsilon$ using
\begin{equation}\label{eq:screstart}
G_\varepsilon=O_\sigma\!\left(
\mathcal S_\varepsilon+\min\big\{\sqrt{\kappa}\,\mathcal S_\varepsilon,\ (\chi\kappa)^{1/3}\mathcal S_\varepsilon+(\chi\kappa_2)^{1/3}\big\}\right)
\end{equation}
gradient queries.
\end{theorem}
\noindent The proof is in Appendix~\ref{proof:screstart}.
The terms that involve $L_1$ depend on the stage only through
$L_1R_s^2/\varepsilon_s=8\kappa$ and repeat at every stage,
while those involving $L_2R_s$ decrease geometrically with $s$
and sum to the transient $(\chi\kappa_2)^{1/3}$.
The stage-independent costs therefore accumulate linearly in
$\mathcal S_\varepsilon$. For fixed problem parameters,
Corollary~\ref{cor:scqueries} instead gives
$\mathcal L_\varepsilon^{2/3}$ dependence in the moderate-accuracy
regime and
$\mathcal L_\varepsilon/\log\mathcal L_\varepsilon$ dependence
as $\varepsilon\downarrow0$.

\subsection{Radius contraction and the number of stages}\label{app:scstages}
\begin{lemma}[One stage halves the radius]\label{lem:halve}
Under Assumptions~\ref{as:main} and~\ref{as:sc}, every $x\in\mathbb R^d$
satisfies
\[
\frac\mu2\|x-x^\star\|^2\ \le\ f(x)-f^\star\ \le\ \frac{L_1}2\|x-x^\star\|^2 .
\]
Consequently, if $\bar R>0$ and $f(w^+)-f^\star\le\mu\bar R^2/8$, then
$\|w^+-x^\star\|\le\bar R/2$.
\end{lemma}
\begin{proof}
Since $\nabla f(x^\star)=0$, Taylor's theorem with integral remainder gives,
with $v=x-x^\star$,
$f(x)-f^\star=\int_0^1(1-t)\langle H(x^\star+tv)v,v\rangle\,dt$. By
Assumptions~\ref{as:main} and~\ref{as:sc} the integrand lies between
$(1-t)\mu\|v\|^2$ and $(1-t)L_1\|v\|^2$, and $\int_0^1(1-t)\,dt=\frac12$. For
the second claim, the lower bound gives
$\frac\mu2\|w^+-x^\star\|^2\le\mu\bar R^2/8$, that is,
$\|w^+-x^\star\|^2\le\bar R^2/4$.
\end{proof}

\begin{lemma}[Correctness and number of stages]\label{lem:wrapper}
Under Assumptions~\ref{as:main} and~\ref{as:sc}, the restart scheme
satisfies $\|w_s-x^\star\|\le R_s$ for every $0\le s\le S$, every stage stops
after finitely many gradient queries, and the returned point satisfies
$f(w_S)-f^\star\le\varepsilon$. Moreover,
\[
S\ \le\ 1+\max\left\{0,\ \tfrac12\log_2\frac{L_1R^2}{2\varepsilon}\right\}
\ \le\ \mathcal S_\varepsilon .
\]
\end{lemma}
\begin{proof}
Consider a stage started at a point $w$ with $\|w-x^\star\|\le\bar R$, with
radius parameter $\bar R$ and accuracy parameter $\bar\varepsilon$.
Assumption~\ref{as:main} involves the starting point and the radius only
through $\|x_0-x^\star\|\le R$, which holds for $(w,\bar R)$, so
Corollary~\ref{thm:complexity} applies with
$(x_0,R,\varepsilon)=(w,\bar R,\bar\varepsilon)$ and bounds the number of
gradient queries of the stage. Its output $w^+$ satisfies
$f(w^+)-f^\star\le\bar\varepsilon$, certified as in the convex case: a return
at $A_{k+1}\ge\bar R^2/(2\bar\varepsilon)$ is certified by
Proposition~\ref{prop:hpeestimates}(i), and a return at $g_k=0$ is exact.

The radius invariant follows by induction on $s$. The case $s=0$ is the
hypothesis $\|x_0-x^\star\|\le R$. If $\|w_s-x^\star\|\le R_s$ for some
$s<S$, the previous paragraph with
$(w,\bar R,\bar\varepsilon)=(w_s,R_s,\varepsilon_s)$ shows that stage $s$
stops and returns $w_{s+1}$ with
$f(w_{s+1})-f^\star\le\varepsilon_s=\mu R_s^2/8$, and Lemma~\ref{lem:halve}
gives $\|w_{s+1}-x^\star\|\le R_s/2=R_{s+1}$.
Put $u=\log_2\!\big(R\sqrt{L_1/(2\varepsilon)}\big)
=\tfrac12\log_2\big(L_1R^2/(2\varepsilon)\big)$, so that
$S=\max\{0,\lceil u\rceil\}$. %
Since $S\ge u$, the radius invariant gives
\[
\|w_S-x^\star\|\le R_S=2^{-S}R
\le2^{-u}R=\sqrt{2\varepsilon/L_1},
\]
and the upper bound of
Lemma~\ref{lem:halve} gives $f(w_S)-f^\star\le\tfrac{L_1}2R_S^2\le\varepsilon$.
For the count, $\lceil u\rceil\le1+u$ gives the first inequality. For the
second, if $L_1R^2/\varepsilon\le2$, then $u\le0$ and the middle term equals
$1\le\mathcal S_\varepsilon$. Otherwise, in natural logarithms,
\[
1+u=\frac12+\frac{\log(L_1R^2/\varepsilon)}{2\log2}
\ \le\ 1+\log\frac{L_1R^2}{\varepsilon}=\mathcal S_\varepsilon ,
\]
because $2\log2>1$.
\end{proof}

\begin{lemma}[Elementary estimates]\label{lem:scfacts}
\begin{enumerate}[label=(\roman*),leftmargin=2.4em]
\item For $u,v\ge0$, $(u+v)^{1/3}\le u^{1/3}+v^{1/3}$.
\item For $0<r<1$, $\sum_{s\ge0}r^s=(1-r)^{-1}$. In particular,
      $(1-2^{-1/3})^{-1}<4.85$.
\end{enumerate}
\end{lemma}
\begin{proof}
(i) $(u^{1/3}+v^{1/3})^3=u+v+3u^{1/3}v^{1/3}(u^{1/3}+v^{1/3})\ge u+v$; now
apply the increasing map $x\mapsto x^{1/3}$.
(ii) This is the geometric series. Numerically, $2^{-1/3}<0.7938$ gives
$(1-2^{-1/3})^{-1}<(0.2062)^{-1}<4.85$.
\end{proof}

\subsection{Proof of Theorem~\ref{thm:screstart}}\label{proof:screstart}
\begin{proof}
Correctness is Lemma~\ref{lem:wrapper}. If $S=0$, no gradient is queried,
so let $S\ge1$. Stage $s$ is a run of Algorithm~1 from $w_s$ with radius
parameter $R_s$, accuracy parameter $\varepsilon_s=\mu R_s^2/8$ and
$b=L_2R_s$, and $\|w_s-x^\star\|\le R_s$ by Lemma~\ref{lem:wrapper}. With
$K_s=\chi(L_1+L_2R_s)$, the explicit bound \eqref{eq:fp-explicit-queries} in the proof of
Corollary~\ref{thm:complexity} gives, for the number $G_s$ of gradient queries
of stage $s$,
\[
G_s\le1+3\Big\lceil\min\Big\{\sqrt{\frac{2L_1R_s^2}{\sigma\varepsilon_s}},\
D_\sigma\Big(\frac{K_sR_s^2}{2\varepsilon_s}\Big)^{1/3}\Big\}\Big\rceil
\le4+3\min\{\alpha,\xi_s\},
\]
with $\alpha=4\sqrt{\kappa/\sigma}$ and
$\xi_s=D_\sigma\big(4\chi(\kappa+\kappa_2 2^{-s})\big)^{1/3}$, because
$L_1R_s^2/\varepsilon_s=8\kappa$ and $L_2R_s/\mu=\kappa_2 2^{-s}$.
Lemma~\ref{lem:scfacts}(i) gives $\xi_s\le\xi^\star+c\,2^{-s/3}$ with
$\xi^\star=D_\sigma(4\chi\kappa)^{1/3}$ and $c=D_\sigma(4\chi\kappa_2)^{1/3}$, and
Lemma~\ref{lem:scfacts}(ii) with $r=2^{-1/3}$ gives
$\sum_{s<S}\xi_s\le S\xi^\star+4.85\,c$. Since
$\sum_{s<S}\min\{\alpha,\xi_s\}$ is at most both $S\alpha$ and
$\sum_{s<S}\xi_s$, we obtain
\[
\sum_{s<S}\min\{\alpha,\xi_s\}\le\min\{S\alpha,\ S\xi^\star+4.85\,c\} .
\]
Summing $G_s$ over $s<S$,
\[
G_\varepsilon\le4S+3\min\Big\{4S\sqrt{\kappa/\sigma},\ SD_\sigma(4\chi\kappa)^{1/3}+4.85\,D_\sigma(4\chi\kappa_2)^{1/3}\Big\},
\]
and $S\le\mathcal S_\varepsilon$ by Lemma~\ref{lem:wrapper}, which gives
\eqref{eq:screstart}.
\end{proof}


\section{Analysis of the implementations}\label{app:implementation}

This appendix proves the statements of Section~\ref{sec:implementation}.
Appendix~\ref{app:repair} proves the two properties of the certified
repair, Appendix~\ref{app:inexact} the residual identity for inexact
linear solves, Appendix~\ref{proof:implementation} proves
Theorem~\ref{thm:implementation}, Appendix~\ref{app:chebyshev} the
iteration count \eqref{eq:lineariterations}, and
Appendix~\ref{app:approxproj} the accommodation of approximate spectral
projections. Throughout, $\mathcal Z=\{S\in\mathbb S^d:0\preceq S\preceq L_1I\}$.

\subsection{Certified spectral repair}\label{app:repair}

\begin{lemma}[Spectral repair from a raw error bound]\label{lem:repair}
Let $H\in\mathcal Z$, let $S\in\mathbb S^d$ satisfy
$\|S-H\|_{\rm op}\le e$ for some $e\ge0$, and let
$\mathcal R_e(S)=\frac{L_1}{L_1+2e}(S+eI)$ be the matrix of
\eqref{eq:repair}. Then $\mathcal R_e(S)\in\mathcal Z$ and
$\|\mathcal R_e(S)-H\|_{\rm op}\le2e$. Forming $\mathcal R_e(S)$ costs one
shift and one scaling of $S$.
\end{lemma}

\begin{proof}
Since $-eI\preceq S-H\preceq eI$ and $0\preceq H\preceq L_1I$, we have
$-eI\preceq S\preceq(L_1+e)I$, that is,
$0\preceq S+eI\preceq(L_1+2e)I$. Multiplying by
$\alpha=L_1/(L_1+2e)\in(0,1]$ gives $0\preceq\mathcal R_e(S)\preceq L_1I$.
For the error, write $1-\alpha=2e/(L_1+2e)$, so that
$\alpha eI-(1-\alpha)H=(1-\alpha)\big(\tfrac{L_1}2I-H\big)$ and
\[
\mathcal R_e(S)-H=\alpha(S-H)+\alpha eI-(1-\alpha)H
=\alpha(S-H)+(1-\alpha)\Big(\frac{L_1}{2}I-H\Big).
\]
Since $0\preceq H\preceq L_1I$, the matrix $\frac{L_1}{2}I-H$ has
eigenvalues in $[-\frac{L_1}{2},\frac{L_1}{2}]$. Therefore
\[
\|\mathcal R_e(S)-H\|_{\rm op}
\le\alpha e+(1-\alpha)\frac{L_1}{2}
=\frac{L_1e}{L_1+2e}+\frac{eL_1}{L_1+2e}
=\frac{2L_1e}{L_1+2e}\le2e .
\]
\end{proof}

\subsection{Inexact linear solves}\label{app:inexact}

\begin{lemma}[Linear residual and nonlinear acceptance]\label{lem:inexactresidual}
Let $F$ be convex and twice continuously differentiable with
$\mathcal L$-Lipschitz Hessian, let $B\in\mathbb S^d$ with $B\succeq0$,
$\lambda>0$, $y\in\mathbb R^d$, and let
$r_*=-(B+\lambda^{-1}I)^{-1}\nabla F(y)$ be the exact model step. Suppose
$r\in\mathbb R^d$ and $0\le\delta<1$ satisfy
$\lambda\|w\|\le\delta\|r\|$ for the linear residual
$w=(B+\lambda^{-1}I)r+\nabla F(y)$. Then:
\begin{enumerate}[label=(\alph*)]
\item $\|r\|\le\|r_*\|/(1-\delta)$;
\item $\|r+\lambda\nabla F(y+r)\|\le
\big(\delta+\lambda\|B-\nabla^2F(y)\|_{\rm op}+\frac{\mathcal L\lambda}{2}\|r\|\big)\|r\|$;
\item if moreover $0\preceq B\preceq LI$ and $0\preceq\nabla^2F(v)\preceq LI$
for every $v\in[y,y+r]$, then
$\|r+\lambda\nabla F(y+r)\|\le(\delta+\lambda L)\|r\|$.
\end{enumerate}
With $\delta=0$, (b) and (c) are the two certificates of
Lemma~\ref{lem:resid}.
\end{lemma}

\begin{proof}
By the definition of $w$, $(B+\lambda^{-1}I)r=-\nabla F(y)+w$, hence
$r=r_*+(B+\lambda^{-1}I)^{-1}w$. Since $B\succeq0$, the eigenvalues of
$(B+\lambda^{-1}I)^{-1}$ lie in $(0,\lambda]$, so
$\|r\|\le\|r_*\|+\lambda\|w\|\le\|r_*\|+\delta\|r\|$, which is (a).
For (b), let $\bar H=\int_0^1\nabla^2F(y+sr)\,ds$, so that
$\nabla F(y+r)=\nabla F(y)+\bar Hr$ by the fundamental theorem of
calculus. Substituting $\nabla F(y)=w-Br-\lambda^{-1}r$ from the
definition of $w$,
\[
r+\lambda\nabla F(y+r)
=r+\lambda\big(w-Br-\lambda^{-1}r\big)+\lambda\bar Hr
=\lambda w+\lambda(\bar H-B)r .
\]
Taking norms, using $\lambda\|w\|\le\delta\|r\|$ and
\[
\|\bar H-B\|_{\rm op}
\le\|B-\nabla^2F(y)\|_{\rm op}
+\int_0^1\|\nabla^2F(y+sr)-\nabla^2F(y)\|_{\rm op}\,ds
\le\|B-\nabla^2F(y)\|_{\rm op}+\frac{\mathcal L}{2}\|r\| ,
\]
gives (b). Under the hypotheses of (c), $\bar H$ is an average of
matrices in $[0,LI]$, so $\bar H$ and $B$ both lie in $[0,LI]$,
$-LI\preceq\bar H-B\preceq LI$, and $\|\bar H-B\|_{\rm op}\le L$; the same
identity gives (c).
\end{proof}

\subsection{Proof of Theorem~\ref{thm:implementation}}
\label{proof:implementation}

\begin{proof}
We treat the repair first and the inexact solves afterwards.

\emph{Part (i), Algorithm~1.} For $k<\chi$ the repaired method uses
$B_k=0$. This matrix lies in $\mathcal Z$, and since
$0\preceq H(x_k)\preceq L_1I$, $\|B_k-H(x_k)\|_{\rm op}\le L_1=\bar E_k$,
the first branch of \eqref{eq:rollingcert}. For $k\ge\chi$, the proof of
Lemma~\ref{lem:telescope} shows, before the projection is applied, that
$S_k=(M_k+M_k^\top)/2$ satisfies
$\|S_k-H(x_k)\|_{\rm op}\le e_k^{\rm raw}=L_2\widehat D_k+e_k/2$ with
$e_k=b/(k-\chi+2)^2$;
that argument uses only the anchors of the stored blocks, the block
finite-difference bound of Lemma~\ref{lem:group}, and the actual
displacements $\delta_j$, so it holds for the trajectory generated by the
repaired method. Lemma~\ref{lem:repair}, applied with $S=S_k$,
$H=H(x_k)\in\mathcal Z$ and $e=e_k^{\rm raw}$, gives
$B_k=\mathcal R_{e_k^{\rm raw}}(S_k)\in\mathcal Z$ and
$\|B_k-H(x_k)\|_{\rm op}\le2e_k^{\rm raw}$, and since both matrices lie in
$\mathcal Z$, also $\|B_k-H(x_k)\|_{\rm op}\le L_1$. Hence
$\|B_k-H(x_k)\|_{\rm op}\le\min\{L_1,2e_k^{\rm raw}\}=\bar E_k$, the second
branch of \eqref{eq:rollingcert}. Thus Lemma~\ref{lem:telescope}
holds for the repaired matrices with the same certificates.
Lemma~\ref{lem:sum} depends only on the formulas for $\bar E_k$.
Lemmas~\ref{lem:scalarbounds}, \ref{lem:wellposed} and~\ref{lem:fiveterm} and the proofs of
Theorem~\ref{thm:iterationrate} and Corollary~\ref{thm:complexity} use $B_k$ only
through $B_k\in\mathcal Z$ and the certificate
$\|B_k-H(x_k)\|_{\rm op}\le\bar E_k$, and the repair queries no gradient,
so all of them hold unchanged.

\emph{Part (ii), Subroutine~1.1.} Write $\lambda_{\rm b}'=3\sigma/(4L_1)$
and $\sigma'=\sigma/2$. The modified subroutine returns
$\lambda_k=\lambda_{\rm b}'$ if $\widehat\Psi_k(\lambda_{\rm b}')\ge\sigma'$;
otherwise it bisects on
$[\lambda^-,\lambda^+]=[\lambda_{\rm b}',\sigma'/\bar E_k]$ with the test
$\widehat\Psi_k(\lambda)<\sigma'$, exactly as in Subroutine~1.1, and
returns $\lambda_k=\lambda^-$. The computed displacement $r_k$ satisfies
\eqref{eq:linearstop} with $\delta=\sigma/4$ at $y=y_k$,
$\lambda=\lambda_k$ and $B=B_k$, and $x_{k+1}=y_k+r_k$.

The scalar analysis in the proof of Lemma~\ref{lem:wellposed} applies
with threshold $\sigma'$ and floor $\lambda_{\rm b}'$. The function $\widehat\Psi_k$ is
continuous and strictly increasing with
$\widehat\Psi_k(\lambda)\ge\lambda\bar E_k$, so
$\widehat\Psi_k(\sigma'/\bar E_k)\ge\sigma'$; and if the subroutine does
not return at once, then
$\sigma'>\widehat\Psi_k(\lambda_{\rm b}')\ge\frac{3\sigma}{4L_1}\bar E_k$,
so $\bar E_k<\frac23L_1$ and $\lambda^-<\lambda^+$. The bracket invariant
$\widehat\Psi_k(\lambda^-)<\sigma'\le\widehat\Psi_k(\lambda^+)$ is
maintained, each pass halves $\log(\lambda^+/\lambda^-)$, which starts at
$\log\big(2L_1/(3\bar E_k)\big)$, and the returned value satisfies
$\lambda_{\rm b}'\le\lambda_k\le\sigma'/\bar E_k$. Since
$\widehat\Psi_k(c\lambda)\le2\widehat\Psi_k(\lambda)$ for
$1\le c\le2^{1/4}$, as shown in that proof, the bisection branch gives
$\widehat\Psi_k(\lambda_k)\ge\sigma'/2=\sigma/4$, and the other branch
gives $\widehat\Psi_k(\lambda_k)\ge\sigma'$.

For acceptance, consider first $\lambda_k=\lambda_{\rm b}'$.
Lemma~\ref{lem:inexactresidual}(c) applies with $F=f$ and $L=L_1$, since
$B_k\in\mathcal Z$ and $0\preceq H\preceq L_1I$ everywhere, and gives
$\|r_k+\lambda_k\nabla f(x_{k+1})\|\le(\delta+\lambda_kL_1)\|r_k\|
=(\frac\sigma4+\frac{3\sigma}4)\|r_k\|=\sigma\|r_k\|$. In the bisection
branch, $\widehat\Psi_k(\lambda_k)<\sigma'$. Lemma~\ref{lem:scalarbounds}
gives $\|B_k-H(y_k)\|_{\rm op}\le E_k(\lambda_k)$ and bounds the exact
model step by $\|r_*\|\le\widehat\rho_k(\lambda_k)$, so
Lemma~\ref{lem:inexactresidual}(a) gives
$\|r_k\|\le\widehat\rho_k(\lambda_k)/(1-\delta)$, and
Lemma~\ref{lem:inexactresidual}(b) gives
\[
\frac{\|r_k+\lambda_k\nabla f(x_{k+1})\|}{\|r_k\|}
\le\delta+\lambda_kE_k(\lambda_k)+\frac{L_2\lambda_k\widehat\rho_k(\lambda_k)}{2(1-\delta)}
\le\delta+\frac{\widehat\Psi_k(\lambda_k)}{1-\delta}
<\frac\sigma4+\frac{\sigma/2}{1-\sigma/4}
=\frac\sigma4+\frac{2\sigma}{4-\sigma}\le\sigma ,
\]
where the second inequality uses $1-\delta\le1$ in the middle term and
$\widehat\Psi_k(\lambda)=\lambda E_k(\lambda)+\frac{L_2\lambda}{2}\widehat\rho_k(\lambda)$,
and the last inequality is equivalent to $\sigma\le4/3$. Thus every
iteration satisfies \eqref{eq:ms} with the original parameter $\sigma$,
and since the residual test uses no gradient, each iteration still uses
at most three gradients.

Finally we trace the two modified scalar facts through the proofs of
Theorem~\ref{thm:iterationrate} and Corollary~\ref{thm:complexity}. Those proofs
use Lemma~\ref{lem:wellposed} through \eqref{eq:ms}, which is unchanged,
through the floor $\lambda_k\ge\lambda_{\rm b}$, and, by way of
Lemma~\ref{lem:fiveterm}, through $\widehat\Psi_k(\lambda_k)\ge\sigma/2$.
The floor enters as $A_N\ge\frac14N^2\lambda_{\rm b}$ in
\eqref{eq:smoothbranch}, as $A_1\ge\lambda_{\rm b}>0$, which is used only
to start the scales, and as the factor $(\lambda_{\rm b}s)^{-1/4}$ in
\eqref{eq:fp-localgradient}; replacing $\lambda_{\rm b}$ by
$\frac34\lambda_{\rm b}$ multiplies the smoothness rate by $\frac43$ and
$\beta_\sigma$ by $(\frac43)^{1/4}$. The lower bound
$\widehat\Psi_k(\lambda_k)\ge\sigma/4$ replaces the threshold $\sigma/10$
in the five alternatives of Lemma~\ref{lem:fiveterm} by $\sigma/20$, which
multiplies each term of \eqref{eq:fiveterm}, and hence $\alpha_\sigma$ and
$\beta_\sigma$, by at most $2$. The remaining ingredients, the potential,
the local path length and the local certificate sum, use only
\eqref{eq:ms} and the actual displacements. Hence $d_\sigma$ and $D_\sigma$
in \eqref{eq:fp-countconstant} are multiplied by absolute constants, and
Theorem~\ref{thm:iterationrate} and Corollary~\ref{thm:complexity} follow as stated
with their $O_\sigma$ constants multiplied by absolute factors.

\emph{Part (iii), Algorithm~2.}
The repair argument applies with the actual anchor displacements
$\delta_j=\|x_{j+1}-x_j\|$, preserving feasibility and
Lemma~\ref{lem:sc-cert}. Thus the exact-solve guarantees remain
unchanged.

For inexact solves, the proof of Lemma~\ref{lem:sc-search},
with threshold $\sigma/2$ and the stated endpoints, gives
\[
 \lambda_k\ge\underline\lambda_{\rm in}\ge\frac{\sigma}{2L_1},
 \qquad
 \frac{\sigma}{4}\le\phi_k(\lambda_k)\le\frac{\sigma}{2},
 \qquad
 s_k\le\frac4\sigma\,\omega_k(\lambda_k).
\]
We prove admissibility by induction. Suppose $\Psi_k\le R^2/2$,
and let $r_{*,k}$ be the hypothetical exact model step from
the current center $y_k$. Since $\phi_k(\lambda_k)\le\sigma$,
Lemma~\ref{lem:sc-admissible} gives
$\|r_{*,k}\|\le\rho_k(\lambda_k)$.
Lemma~\ref{lem:inexactresidual}(a) therefore yields
\[
 \|r_k\|\le\frac{\rho_k(\lambda_k)}{1-\delta}
 \le2\rho_k(\lambda_k).
\]
For $\bar H_k=\int_0^1H(y_k+t r_k)\,dt$, feasibility and
Hessian Lipschitz continuity give
\[
 \|\bar H_k-B_k\|_{\rm op}
 \le\min\left\{L_1,\,
 \bar E_k+L_2m_k(\lambda_k)+\frac{L_2}{2}\|r_k\|\right\}
 \le\omega_k(\lambda_k).
\]
The linear residual identity and \eqref{eq:linearstop} imply
\[
\begin{aligned}
 \|r_k+\lambda_k\nabla f(x_{k+1})\|
 &\le\big(\delta+\lambda_k\omega_k(\lambda_k)\big)\|r_k\|\\
 &\le\left(\frac{\sigma}{4}
       +\frac{\sigma}{2}\sqrt{1+\mu\lambda_k}\right)\|r_k\|\\
 &\le\sigma\sqrt{1+\mu\lambda_k}\,\|r_k\|.
\end{aligned}
\]
The potential-decrease proof of Lemma~\ref{lem:sc-contract}
uses only this residual condition and the outer update identities,
so it applies to the computed step. Starting from $\Psi_0\le R^2/2$ completes the induction and preserves the energy estimate and mass stopping rule; in particular, Lemma~\ref{lem:sc-contract} implies that the computed step $r_k$ also satisfies $\|r_k\|\le\rho_k^{\rm E}(\lambda_k)$, so $\delta_k\le m_k(\lambda_k)+\rho_k^{\rm E}(\lambda_k)$.

For the rates, monotonicity of
$\lambda/\sqrt{1+\mu\lambda}$ gives
$\underline\lambda_{\rm in}\ge\underline\lambda/4$.
Let $T_{\sharp,{\rm in}}$ denote $T_\sharp$ with this new floor.
Since
\[
 \log_+\frac1{\mu\underline\lambda_{\rm in}}
 \le\log_+\frac1{\mu\underline\lambda}+\log4,
\]
we have $T_{\sharp,{\rm in}}\le2T_\sharp$.
The boundedness, mass-growth and path-length proofs use only
admissibility, the energy estimate and the stepsize floor,
so they apply with these constants. Consequently, the moment
argument of Theorem~\ref{thm:scmain} gives, for every completed
horizon $N$,
\[
 \sum_{k<N}s_k
 \le C_{\sharp,{\rm in}}
 :=\frac4\sigma\left[
 \chi L_1+2L_2R+(2\chi+1)L_2R T_{\sharp,{\rm in}}\right]
 \le4C_\sharp.
\]
The remaining rate and complexity arguments apply with
$\underline\lambda_{\rm in}$ and
$B_{\sharp,{\rm in}}=C_{\sharp,{\rm in}}/\mu\le4B_\sharp$,
changing only absolute constant factors.
The exact step $r_{*,k}$ is used only in the proof;
the implementation requires no additional gradient queries.
\end{proof}

\subsection{A finite iterative solver}\label{app:chebyshev}

\begin{lemma}[Chebyshev iteration with a relative residual test]\label{lem:chebyshev}
Let $B\in\mathbb S^d$ with $0\preceq B\preceq L_BI$, let $\lambda>0$,
$A=B+\lambda^{-1}I$ and $\kappa_B=1+\lambda L_B$. For every
$g\in\mathbb R^d$ and $\delta\in(0,1)$, the iteration \eqref{eq:chebrec}
below produces, in exact arithmetic, a vector $r$ with
$\lambda\|Ar+g\|\le\delta\|r\|$ using at most
\[
m\le1+\Big\lceil\frac{\sqrt{\kappa_B}}{2}\log\frac{2(\kappa_B+\delta)}{\delta}\Big\rceil
\]
matrix-vector products with $A$, which is \eqref{eq:lineariterations},
and $O(d)$ further operations per product. For dense $A$ each product
costs $O(d^2)$.
\end{lemma}

\begin{proof}
If $g=0$, return $r=0$. If $L_B=0$, then $A=\lambda^{-1}I$ and
$r=-\lambda g$ solves the system exactly. Otherwise put
$\alpha=\lambda^{-1}$ and $\beta=\lambda^{-1}+L_B$, so that the spectrum
of $A$ lies in $[\alpha,\beta]$, $\alpha<\beta$, and
$\beta/\alpha=\kappa_B$. Let $T_m$ be the Chebyshev polynomials,
$T_0(t)=1$, $T_1(t)=t$, $T_{m+1}(t)=2tT_m(t)-T_{m-1}(t)$, which satisfy
$|T_m(t)|\le1$ on $[-1,1]$ and $T_m(\cosh s)=\cosh(ms)$. The affine map
$\ell(x)=(\beta+\alpha-2x)/(\beta-\alpha)$ sends $[\alpha,\beta]$ onto
$[-1,1]$ and $0$ to $c=(\beta+\alpha)/(\beta-\alpha)=(\kappa_B+1)/(\kappa_B-1)>1$.
Put $p_m(x)=T_m(\ell(x))/T_m(c)$. Then $p_m(0)=1$, so
$q_{m-1}(x)=(1-p_m(x))/x$ is a polynomial of degree $m-1$, and
$r_m=-q_{m-1}(A)g$ satisfies $Ar_m+g=p_m(A)g$. By the spectral theorem,
\[
\|Ar_m+g\|\le\max_{x\in[\alpha,\beta]}|p_m(x)|\,\|g\|\le\frac{\|g\|}{T_m(c)} .
\]
Writing $c=\cosh s$, we have
$e^s=c+\sqrt{c^2-1}=(\sqrt{\kappa_B}+1)/(\sqrt{\kappa_B}-1)$ and
$T_m(c)=\cosh(ms)\ge\frac12e^{ms}$, so
$\|Ar_m+g\|\le2\rho^m\|g\|$ with
$\rho=(\sqrt{\kappa_B}-1)/(\sqrt{\kappa_B}+1)$. Since
$\log\frac{1+u}{1-u}=\int_0^u\frac{2\,dv}{1-v^2}\ge2u$ for $u\in[0,1)$,
taking $u=\kappa_B^{-1/2}$ gives $\rho\le e^{-2/\sqrt{\kappa_B}}$ and
therefore $\|Ar_m+g\|\le\epsilon_m\|g\|$ with
$\epsilon_m=2e^{-2m/\sqrt{\kappa_B}}$.

To pass to the relative test, suppose $\epsilon_m<1$. Then
$\|Ar_m\|\ge\|g\|-\|Ar_m+g\|\ge(1-\epsilon_m)\|g\|$ and
$\|Ar_m\|\le\beta\|r_m\|$, so $\|r_m\|\ge(1-\epsilon_m)\|g\|/\beta$ and,
using $\lambda\beta=\kappa_B$,
\[
\frac{\lambda\|Ar_m+g\|}{\|r_m\|}
\le\frac{\lambda\beta\,\epsilon_m}{1-\epsilon_m}
=\frac{\kappa_B\epsilon_m}{1-\epsilon_m} .
\]
The right-hand side is at most $\delta$ if and only if
$\epsilon_m\le\delta/(\kappa_B+\delta)$, which is below one and holds as
soon as $m\ge\frac{\sqrt{\kappa_B}}{2}\log\frac{2(\kappa_B+\delta)}{\delta}$.
Let $m_0$ be the smallest such integer.

The vectors $r_m$ are computed without forming any polynomial in $A$.
With $C=\ell(A)=\big((\beta+\alpha)I-2A\big)/(\beta-\alpha)$, set
\begin{equation}\label{eq:chebrec}
t_0=1,\ \ t_1=c,\ \ t_{m+1}=2ct_m-t_{m-1},
\ \
u_0=0,\ \ u_1=-\frac{2}{\beta-\alpha}g,\ \
u_{m+1}=2Cu_m-u_{m-1}-\frac{4t_m}{\beta-\alpha}g,
\ \ r_m=\frac{u_m}{t_m}.
\end{equation}
Indeed, $t_m=T_m(c)$, and $u_m=t_mr_m$ satisfies
$Au_m=-\big(t_mI-T_m(C)\big)g$ because $t_mp_m(A)=T_m(C)$. The
three-term recurrences of $t_m$ and $T_m(C)$ together with
$cI-C=2A/(\beta-\alpha)$ give
\[
t_{m+1}I-T_{m+1}(C)
=2C\big(t_mI-T_m(C)\big)-\big(t_{m-1}I-T_{m-1}(C)\big)+\frac{4t_m}{\beta-\alpha}A ,
\]
and multiplying by $-g$, cancelling the invertible factor $A$, which
commutes with $C$, yields the recurrence for $u_m$; the initial values
correspond to $r_0=0$ and $r_1=-\frac{2}{\beta+\alpha}g$. Each update uses
one product with $A$, through $C$, and $O(d)$ further operations. Running
$m_0$ updates and then evaluating $Ar_{m_0}+g$ once for the test uses
$m_0+1$ products, which is the claimed bound. The scalars $t_m$ grow
geometrically; rescaling the recurrence is standard, but we do not
analyze floating-point effects.
\end{proof}

\subsection{Approximate spectral projection}\label{app:approxproj}

\begin{lemma}[Accommodating approximate projection]\label{lem:approxprojection}
\begin{enumerate}[label=(\roman*)]
\item In Algorithm~1, replace $B_k=\Pi_{\mathcal Z}(S_k)$ by any
$\widetilde B_k\in\mathcal Z$ with $\|\widetilde B_k-B_k\|_{\rm op}\le\xi_k$,
and replace $\bar E_k$ by $\widetilde E_k=\min\{L_1,\bar E_k+\xi_k\}$ in
Subroutine~1.1. Then $\|\widetilde B_k-H(x_k)\|_{\rm op}\le\widetilde E_k$,
and if $\xi_k\le b/(k+1)^2$ for all $k$, Lemma~\ref{lem:sum} holds with
$2b$ replaced by $4b$, and Theorem~\ref{thm:iterationrate}
and Corollary~\ref{thm:complexity} hold as stated, also in combination with the
inexact solves of Theorem~\ref{thm:implementation}(ii).
\item For $S\in\mathbb S^d$, if $Q$ is orthogonal and $\widehat\Lambda$ is diagonal with
$\|S-Q\widehat\Lambda Q^\top\|_F\le\xi$, then
$\widetilde B=Q\,\Pi_{\mathcal Z}(\widehat\Lambda)\,Q^\top$, obtained by
clipping the diagonal entries of $\widehat\Lambda$ to $[0,L_1]$, lies in
$\mathcal Z$ and satisfies $\|\widetilde B-\Pi_{\mathcal Z}(S)\|_{\rm op}\le\xi$.
\end{enumerate}
\end{lemma}

\begin{proof}
For (i), the triangle inequality gives
$\|\widetilde B_k-H(x_k)\|_{\rm op}\le\xi_k+\bar E_k$, and
$\|\widetilde B_k-H(x_k)\|_{\rm op}\le L_1$ because both matrices lie in
$\mathcal Z$; hence $\widetilde E_k$ is a valid certificate, and
$\widetilde E_k\le L_1$ keeps the bound $NL_1$ of Lemma~\ref{lem:sum}.
Moreover
$\sum_{k<N}\widetilde E_k\le\sum_{k<N}\bar E_k+b\sum_{k\ge0}(k+1)^{-2}
<\sum_{k<N}\bar E_k+2b$, so the other bound of Lemma~\ref{lem:sum}
holds with $2b$ replaced by $4b$; on every scale of the mass, the same
estimate replaces $e_\sigma K$ in \eqref{eq:fp-localcert} by
$e_\sigma K+2L_2R\le2e_\sigma K$. Lemma~\ref{lem:wellposed}, the proofs of
Theorem~\ref{thm:iterationrate} and Corollary~\ref{thm:complexity}, and the
argument of Theorem~\ref{thm:implementation}(ii) use $B_k$ only through
feasibility, a valid certificate and the sums of the certificates, so
their constants change by absolute factors.

For (ii), the set $\mathcal Z$ is closed and convex, and by
Lemma~\ref{lem:clip} the eigenvalue clipping $\Pi_{\mathcal Z}$ is the
Frobenius-norm projection onto $\mathcal Z$; projections onto closed
convex sets are nonexpansive, so
$\|\Pi_{\mathcal Z}(S)-\Pi_{\mathcal Z}(S')\|_F\le\|S-S'\|_F$ for all
$S,S'\in\mathbb S^d$. Take $S'=Q\widehat\Lambda Q^\top$. Because $Q$ is
orthogonal, this is an eigendecomposition of $S'$ with eigenvalues the
diagonal entries of $\widehat\Lambda$, so clipping them gives
$\Pi_{\mathcal Z}(S')=Q\,\Pi_{\mathcal Z}(\widehat\Lambda)\,Q^\top=\widetilde B\in\mathcal Z$.
Hence
$\|\widetilde B-\Pi_{\mathcal Z}(S)\|_{\rm op}\le\|\widetilde B-\Pi_{\mathcal Z}(S)\|_F
\le\|S'-S\|_F\le\xi$.
\end{proof}

\section{Analysis of the nonconvex method}\label{app:nonconvex}
This appendix proves the results of Section~\ref{sec:nonconvex} in the
order in which they are stated. Appendices~\ref{proof:ncconversion}
and~\ref{proof:ncmaster} prove the two facts about the conversion,
Lemma~\ref{lem:ncconversion} and Proposition~\ref{prop:ncmaster}.
Appendices~\ref{proof:nctrust}--\ref{proof:ncregret} treat the optimistic
learner: the trust-region characterization of the implicit update, the
per-step inequality, and the regret bound. Appendix~\ref{proof:thm:ncmaster}
proves the master formula, Theorem~\ref{thm:ncmaster}.
Appendices~\ref{proof:ncpointwise} and~\ref{proof:nccertificate} prove the
pointwise error bound and the certificate of the rolling model,
Appendices~\ref{proof:ncmodel} and~\ref{proof:ncsmooth} substitute the
parameters of the two variants, and Appendix~\ref{proof:ncsolver} proves
the guarantee of the projected-gradient solver. Throughout,
Assumption~\ref{as:nonconvex} is in force, $H(x)=\nabla^2f(x)$, and we use
the two Taylor estimates
\begin{equation}\label{eq:nctaylor}
\nabla f(y)=\nabla f(x)+\bar H(x,y)(y-x),\quad
\bar H(x,y)=\int_0^1H(x+t(y-x))\,dt,\quad
\|\bar H(x,y)-H(x)\|_{\rm op}\le\frac{L_2}{2}\|y-x\|,
\end{equation}
which follow from the fundamental theorem of calculus applied to
$t\mapsto\nabla f(x+t(y-x))$ and from
$\|H(x+t(y-x))-H(x)\|_{\rm op}\le L_2t\|y-x\|$ integrated over
$t\in[0,1]$. Since $\bar H(x,y)$ is an average of symmetric matrices of
operator norm at most $L_1$, it is symmetric with
$\|\bar H(x,y)\|_{\rm op}\le L_1$.

\subsection{Proof of Lemma~\ref{lem:ncconversion}}\label{proof:ncconversion}
\begin{proof}
\emph{Descent, \eqref{eq:ncdescent}.} Fix $n\ge1$ and let
$\psi(t)=\nabla f(x_{n-1}+t\Delta_n)$ for $t\in[0,1]$, so that
$\psi(\tfrac12)=\nabla f(w_n)$ and
$f(x_n)-f(x_{n-1})=\int_0^1\langle\psi(t),\Delta_n\rangle\,dt$. The map
$\psi$ is continuously differentiable with
$\psi'(t)=H(x_{n-1}+t\Delta_n)\Delta_n$, and by the Lipschitz property of
the Hessian, $\|\psi'(s)-\psi'(t)\|\le L_2|s-t|\,\|\Delta_n\|^2$. For
each $t$, the fundamental theorem of calculus gives
\[
\psi(t)-\psi(\tfrac12)-(t-\tfrac12)\psi'(\tfrac12)
=\int_{1/2}^t\big(\psi'(s)-\psi'(\tfrac12)\big)\,ds,
\]
whose norm is at most $L_2\|\Delta_n\|^2\big|\int_{1/2}^t|s-\tfrac12|\,ds\big|
=\frac{L_2\|\Delta_n\|^2}{2}(t-\tfrac12)^2$. Since
$\int_0^1(t-\tfrac12)\,dt=0$, the term linear in $t-\tfrac12$ integrates to
zero, and
\[
\Big\|\int_0^1\psi(t)\,dt-\psi(\tfrac12)\Big\|
\le\frac{L_2\|\Delta_n\|^2}{2}\int_0^1(t-\tfrac12)^2\,dt
=\frac{L_2\|\Delta_n\|^2}{24}.
\]
By the Cauchy--Schwarz inequality,
$f(x_n)-f(x_{n-1})-\langle\nabla f(w_n),\Delta_n\rangle
=\big\langle\int_0^1\psi(t)\,dt-\psi(\tfrac12),\Delta_n\big\rangle
\le\frac{L_2}{24}\|\Delta_n\|^3\le\frac{L_2D^3}{24}$.

\emph{Averaging, \eqref{eq:ncavg}.} Fix an episode $k$ and write
$\bar w=\bar w^k$. By \eqref{eq:nctaylor} with $x=\bar w$ and $y=w_n$,
\[
\nabla f(w_n)=\nabla f(\bar w)+H(\bar w)(w_n-\bar w)+\rho_n,
\qquad
\rho_n=\big(\bar H(\bar w,w_n)-H(\bar w)\big)(w_n-\bar w),
\qquad
\|\rho_n\|\le\frac{L_2}{2}\|w_n-\bar w\|^2 .
\]
Summing over $n\in\mathcal N_k$ and using
$\sum_{n\in\mathcal N_k}(w_n-\bar w)=0$, which is the definition of
$\bar w$, gives $G^k=T\nabla f(\bar w)+\sum_{n\in\mathcal N_k}\rho_n$,
hence
\[
\|\nabla f(\bar w)\|\le\frac{\|G^k\|}{T}+\frac1T\sum_{n\in\mathcal N_k}\|\rho_n\|
\le\frac{\|G^k\|}{T}+\frac{L_2}{2}\max_{n\in\mathcal N_k}\|w_n-\bar w\|^2 .
\]
By \eqref{eq:ncpoints}, consecutive midpoints differ by
$w_{m+1}-w_m=\Delta_m+\frac12(\Delta_{m+1}-\Delta_m)=\frac12(\Delta_m+\Delta_{m+1})$,
a vector of norm at most $D$, so any two midpoints of the episode are at
distance at most $(T-1)D$, and
$\|w_n-\bar w\|\le\frac1T\sum_{m\in\mathcal N_k}\|w_n-w_m\|\le(T-1)D\le TD$.
Substituting gives \eqref{eq:ncavg}.
\end{proof}

\subsection{Proof of Proposition~\ref{prop:ncmaster}}\label{proof:ncmaster}
\begin{proof}
Summing \eqref{eq:ncdescent} over $n=1,\ldots,N$,
\[
f(x_N)-f(x_0)\le\sum_{n=1}^N\langle\nabla f(w_n),\Delta_n\rangle+\frac{NL_2D^3}{24}.
\]
By the definition \eqref{eq:nccomparator} of the regret and of the
comparators,
\[
\sum_{n=1}^N\langle\nabla f(w_n),\Delta_n\rangle
=\mathrm{Reg}_N+\sum_{k=1}^K\langle G^k,u^k\rangle
=\mathrm{Reg}_N-D\sum_{k=1}^K\|G^k\| ,
\]
where $\langle G^k,u^k\rangle=-D\|G^k\|$ also when $G^k=0$. Combining the
two displays and using $f(x_N)\ge f^\star$, that is,
$f(x_0)-f(x_N)\le\Phi_0$,
\[
D\sum_{k=1}^K\|G^k\|\le\Phi_0+\mathrm{Reg}_N+\frac{NL_2D^3}{24}.
\]
Since $k^\star$ minimizes $\|G^k\|$, we have
$\|G^{k^\star}\|\le\frac1K\sum_k\|G^k\|$, and dividing by $T$, with
$KT=N$,
\[
\frac{\|G^{k^\star}\|}{T}\le\frac{1}{N}\sum_{k=1}^K\|G^k\|
\le\frac{\Phi_0+\mathrm{Reg}_N}{DN}+\frac{L_2D^2}{24}.
\]
Inequality \eqref{eq:ncavg} for the episode $k^\star$ adds
$L_2T^2D^2/2$ and gives \eqref{eq:ncmaster}.

\end{proof}

\subsection{Proof of Lemma~\ref{lem:nctrust}}\label{proof:nctrust}

Let $\Pi_D$ denote the Euclidean projection onto
$\mathcal B_D=\{\Delta:\|\Delta\|\le D\}$, where $D>0$.
It satisfies $\Pi_D(v)=v$ if $\|v\|\le D$ and
$\Pi_D(v)=Dv/\|v\|$ otherwise.
We will use the following two properties.

\begin{lemma}[Projection onto the ball]\label{lem:ncball}
\begin{enumerate}[label=(\roman*),leftmargin=2em]
\item For every $v$ and every $u$ with $\|u\|\le D$,
$\langle v-\Pi_D(v),u-\Pi_D(v)\rangle\le0$.
\item Let $\eta>0$. A point $\Delta$ satisfies $\Delta=\Pi_D(v)$ if and only
if $\|\Delta\|\le D$ and there is $\varrho\ge0$ with
$(1+\eta\varrho)\Delta=v$ and $\varrho(\|\Delta\|-D)=0$.
\end{enumerate}
\end{lemma}
\begin{proof}
(i) Put $p=\Pi_D(v)$. For $t\in(0,1]$ the point $p+t(u-p)$ lies in the
convex set $\mathcal B_D$, so by minimality of $p$,
$\|v-p\|^2\le\|v-p-t(u-p)\|^2=\|v-p\|^2-2t\langle v-p,u-p\rangle+t^2\|u-p\|^2$.
Dividing by $t$ and letting $t\downarrow0$ gives the claim.

(ii) Suppose $\Delta=\Pi_D(v)$. If $\|v\|\le D$, then $\Delta=v$ and
$\varrho=0$ works. If $\|v\|>D$, then $\Delta=Dv/\|v\|$ has norm $D$,
and $v=(\|v\|/D)\Delta=(1+\eta\varrho)\Delta$ with
$\varrho=(\|v\|/D-1)/\eta>0$. Conversely, let $\|\Delta\|\le D$,
$\varrho\ge0$, $(1+\eta\varrho)\Delta=v$ and $\varrho(\|\Delta\|-D)=0$.
If $\varrho=0$, then $v=\Delta$ has norm at most $D$ and
$\Pi_D(v)=v=\Delta$. If $\varrho>0$, then $\|\Delta\|=D$ and
$\|v\|=(1+\eta\varrho)D>D$, so
$\Pi_D(v)=Dv/\|v\|=D\Delta/\|\Delta\|=\Delta$.
\end{proof}

\begin{proof}[Proof of Lemma~\ref{lem:nctrust}]
Write $h(\Delta)=\nabla f(z_n)+\frac12B_n(\Delta-\Delta_n)$, so that
$h_{n+1}=h(\Delta_{n+1})$ by \eqref{eq:nchint}. Since $B_n$ is symmetric,
the gradient of $\phi_n$ in \eqref{eq:nctr} is
\begin{equation}\label{eq:ncgradphi}
\nabla\phi_n(\Delta)=\nabla f(z_n)+\theta_n+\tfrac12B_n(\Delta-\Delta_n)
+\tfrac1\eta(\Delta-\Delta_n)
=h(\Delta)+\theta_n+\tfrac1\eta(\Delta-\Delta_n).
\end{equation}
Let $\|\Delta\|\le D$. For $\varrho\ge0$ put
$q=\nabla\phi_n(\Delta)+\varrho\Delta$. By \eqref{eq:ncgradphi},
\[
q=h(\Delta)+\theta_n+\tfrac1\eta(\Delta-\Delta_n)+\varrho\Delta
\quad\Longleftrightarrow\quad
(1+\eta\varrho)\Delta=\Delta_n-\eta\theta_n-\eta h(\Delta)+\eta q .
\]
Hence $\Delta$ satisfies \eqref{eq:nckkt} with multiplier $\varrho$ and
residual $q$, $\|q\|\le\nu$, if and only if
$(1+\eta\varrho)\Delta=v$ and $\varrho(\|\Delta\|-D)=0$ for
$v=\Delta_n-\eta\theta_n-\eta h(\Delta)+\eta q$, which by
Lemma~\ref{lem:ncball}(ii) is equivalent to $\Delta=\Pi_D(v)$, that is,
to \eqref{eq:ncinexact} with $h_{n+1}=h(\Delta)$. The case $\nu=0$
identifies the solutions of \eqref{eq:nclearner} with the KKT points.

For existence, $\phi_n$ is continuous and $\mathcal B_D$ is compact, so a
global minimizer $\Delta^\star$ exists. The constraint
$c(\Delta)=\frac12(\|\Delta\|^2-D^2)\le0$ has gradient
$\nabla c(\Delta)=\Delta$, which is nonzero wherever the constraint is
active, so the linear independence constraint qualification holds and
the KKT conditions are necessary at $\Delta^\star$: there is
$\varrho\ge0$ with $\nabla\phi_n(\Delta^\star)+\varrho\Delta^\star=0$ and
$\varrho\,c(\Delta^\star)=0$, the latter being
$\varrho(\|\Delta^\star\|-D)=0$. By the first part, $\Delta^\star$ solves
\eqref{eq:nclearner}.
\end{proof}

\subsection{Proof of Lemma~\ref{lem:nconestep}}\label{proof:nconestep}
\begin{proof}
By Lemma~\ref{lem:nctrust}, $\Delta_{n+1}=\Pi_D(v)$ with
$v=\Delta_n-\eta\theta_n-\eta h_{n+1}+\eta q$ and $\|q\|\le\nu$. Fix $u$
with $\|u\|\le D$.

\emph{Step 1: the projection inequality.} Lemma~\ref{lem:ncball}(i) gives
$\langle v-\Delta_{n+1},\Delta_{n+1}-u\rangle\ge0$. Substituting $v$ and
rearranging,
\[
\eta\langle\theta_n+h_{n+1}-q,\Delta_{n+1}-u\rangle
\le\langle\Delta_n-\Delta_{n+1},\Delta_{n+1}-u\rangle
=\frac{\|\Delta_n-u\|^2-\|\Delta_{n+1}-u\|^2-\|\Delta_{n+1}-\Delta_n\|^2}{2},
\]
where the identity is obtained by expanding
$\|\Delta_n-u\|^2=\|(\Delta_n-\Delta_{n+1})+(\Delta_{n+1}-u)\|^2$. Since
$\|\Delta_{n+1}-u\|\le2D$ and $\|q\|\le\nu$,
$\langle q,\Delta_{n+1}-u\rangle\le2D\nu$, and therefore
\begin{equation}\label{eq:ncstep1}
\langle\theta_n+h_{n+1},\Delta_{n+1}-u\rangle
\le\frac{\|\Delta_n-u\|^2-\|\Delta_{n+1}-u\|^2}{2\eta}
-\frac{\|\Delta_{n+1}-\Delta_n\|^2}{2\eta}+2D\nu .
\end{equation}

\emph{Step 2: decomposition of the loss.} Since
$\nabla f(w_{n+1})=h_{n+1}+\theta_{n+1}$ by \eqref{eq:nctheta},
\[
\langle\nabla f(w_{n+1}),\Delta_{n+1}-u\rangle
=\langle\theta_n+h_{n+1},\Delta_{n+1}-u\rangle
+\langle\theta_{n+1},\Delta_{n+1}-u\rangle
-\langle\theta_n,\Delta_{n+1}-u\rangle .
\]

\emph{Step 3: the previous error.} Writing
$\Delta_{n+1}-u=(\Delta_n-u)+(\Delta_{n+1}-\Delta_n)$ and using
$ab\le\eta a^2+b^2/(4\eta)$ for $a=\|\theta_n\|$ and
$b=\|\Delta_{n+1}-\Delta_n\|$,
\[
-\langle\theta_n,\Delta_{n+1}-u\rangle
=-\langle\theta_n,\Delta_n-u\rangle-\langle\theta_n,\Delta_{n+1}-\Delta_n\rangle
\le-\langle\theta_n,\Delta_n-u\rangle+\eta\|\theta_n\|^2
+\frac{\|\Delta_{n+1}-\Delta_n\|^2}{4\eta}.
\]

\emph{Conclusion.} Inserting \eqref{eq:ncstep1} and Step~3 into Step~2,
the two multiples of $\|\Delta_{n+1}-\Delta_n\|^2$ combine to
$-\frac1{2\eta}+\frac1{4\eta}=-\frac1{4\eta}$, which is
\eqref{eq:nconestep}.
\end{proof}

\subsection{Proof of Lemma~\ref{lem:ncregret}}\label{proof:ncregret}
\begin{proof}
\emph{Telescoping within an episode.} Fix $k\in\{1,\ldots,K\}$ and put
$u=u^k$, which satisfies $\|u\|\le D$. The losses of episode $k$ are
$\nabla f(w_{m+1})$ for $m=(k-1)T,\ldots,kT-1$. Applying
Lemma~\ref{lem:nconestep} at each such $m$ and summing, the quadratic
terms telescope to
$\frac{1}{2\eta}\big(\|\Delta_{(k-1)T}-u\|^2-\|\Delta_{kT}-u\|^2\big)
\le\frac{1}{2\eta}\|\Delta_{(k-1)T}-u\|^2\le\frac{2D^2}{\eta}$, because
$\|\Delta_{(k-1)T}\|\le D$ and $\|u\|\le D$.
The pairs
$\langle\theta_{m+1},\Delta_{m+1}-u\rangle-\langle\theta_m,\Delta_m-u\rangle$
telescope to the two boundary terms
$\langle\theta_{kT},\Delta_{kT}-u\rangle-\langle\theta_{(k-1)T},\Delta_{(k-1)T}-u\rangle$;
and the remaining terms add up. Hence
\begin{equation}\label{eq:ncepisode}
\begin{aligned}
\sum_{n\in\mathcal N_k}\langle\nabla f(w_n),\Delta_n-u^k\rangle
\le{}&\frac{2D^2}{\eta}
+\langle\theta_{kT},\Delta_{kT}-u^k\rangle
-\langle\theta_{(k-1)T},\Delta_{(k-1)T}-u^k\rangle\\
&+\eta\sum_{m=(k-1)T}^{kT-1}\|\theta_m\|^2
-\frac{1}{4\eta}\sum_{m=(k-1)T}^{kT-1}\|\Delta_{m+1}-\Delta_m\|^2
+2DT\nu .
\end{aligned}
\end{equation}

\emph{The boundary terms.} For any index $m$, since
$\|\Delta_m-u\|\le2D$, the Cauchy--Schwarz inequality and
$2Da\le\eta a^2+D^2/\eta$ give
$|\langle\theta_m,\Delta_m-u\rangle|\le2D\|\theta_m\|\le\eta\|\theta_m\|^2+D^2/\eta$.
The two boundary terms of \eqref{eq:ncepisode} therefore contribute at
most $2D^2/\eta$ plus $\eta\|\theta_m\|^2$ for each of the two indices
$m=(k-1)T$ and $m=kT$.

\emph{Summing over episodes.} The constant terms give
$K(2D^2/\eta+2D^2/\eta)=4KD^2/\eta$, and the subproblem terms give
$2DKT\nu=2DN\nu$. The sums over $m$ in \eqref{eq:ncepisode} concatenate
to $\eta\sum_{m=0}^{N-1}\|\theta_m\|^2$ and
$-\frac1{4\eta}\sum_{m=0}^{N-1}\|\Delta_{m+1}-\Delta_m\|^2$. Each index in
$\{0,T,2T,\ldots,KT\}$ is an endpoint of at most two episodes, so the
boundary terms add at most $2\eta\|\theta_m\|^2$ for each such $m$, and
all these indices lie in $\{0,\ldots,N\}$. Altogether the coefficient of
$\|\theta_m\|^2$ is at most $3\eta$ for every $m\in\{0,\ldots,N\}$, which
gives \eqref{eq:ncregret}.

\emph{Absorption.} Under \eqref{eq:nccert}, and since $\theta_0=0$,
\[
3\eta\sum_{n=0}^N\|\theta_n\|^2=3\eta\sum_{n=0}^{N-1}\|\theta_{n+1}\|^2
\le3\eta\Gamma^2\sum_{n=0}^{N-1}\|\Delta_{n+1}-\Delta_n\|^2
\le\frac{1}{4\eta}\sum_{n=0}^{N-1}\|\Delta_{n+1}-\Delta_n\|^2 ,
\]
where the last step is $12\eta^2\Gamma^2\le1$, that is,
$\eta\le1/(2\sqrt3\,\Gamma)$. The second and third terms of
\eqref{eq:ncregret} then have nonpositive sum, and
\eqref{eq:ncregretabs} follows.
\end{proof}

\subsection{Proof of Theorem~\ref{thm:ncmaster}}\label{proof:thm:ncmaster}
\begin{proof}
Since $\Gamma\le\omega L_2D$, the learning rate
$\eta=1/(2\sqrt3\,\omega L_2D)$ satisfies $\eta\le1/(2\sqrt3\,\Gamma)$, so
Lemma~\ref{lem:ncregret} gives $\mathrm{Reg}_N\le4KD^2/\eta+2DN\nu$.
Inserting this into \eqref{eq:ncmaster} and using $K/N=1/T$,
\begin{equation}\label{eq:ncfiveterms}
\|\nabla f(\bar w^{k^\star})\|\le
\underbrace{\frac{\Phi_0}{DN}}_{(1)}
+\underbrace{\frac{4D}{\eta T}}_{(2)}
+\underbrace{2\nu}_{(3)}
+\underbrace{\frac{L_2D^2}{24}}_{(4)}
+\underbrace{\frac{L_2T^2D^2}{2}}_{(5)} .
\end{equation}

With the stated $\eta$, $(2)=8\sqrt3\,\omega L_2D^2/T$.
The choice $T=\lceil(16\sqrt3\,\omega)^{1/3}\rceil$
balances terms $(2)$ and $(5)$ before rounding.
Write $c=(16\sqrt3)^{1/3}$. Since $\omega\ge1$,
\[
c\omega^{1/3}\le T
\le c\omega^{1/3}+1
\le(c+1)\omega^{1/3}.
\]
Using $8\sqrt3/c=c^2/2$, we obtain
\[
\begin{aligned}
(2)+(4)+(5)
&=L_2D^2\left(\frac{8\sqrt3\,\omega}{T}
              +\frac1{24}+\frac{T^2}{2}\right)\\
&\le\left(\frac{c^2}{2}+\frac1{24}
          +\frac{(c+1)^2}{2}\right)\omega^{2/3}L_2D^2\\
&=\left(c^2+c+\frac{13}{24}\right)\omega^{2/3}L_2D^2
<13\,\omega^{2/3}L_2D^2,
\end{aligned}
\]
where the last inequality follows from $c<3.05$.
Together with $(1)$ and $(3)$, this proves
\eqref{eq:ncmasterbound}. The same bound on $c$ gives
$T\le(c+1)\omega^{1/3}<4.1\,\omega^{1/3}$.
\end{proof}

\subsection{Proof of Lemma~\ref{lem:ncpointwise}}\label{proof:ncpointwise}
\begin{proof}
Fix $n\ge0$ and write $H_j=H(z_j)$.

\emph{Anchors.} For a group $g$, let $n_g$ be the last index $m\le n$
at which $g$ was probed: $n_g$ is the largest $m\in\{1,\ldots,n\}$ with
$g_{i_m}=g$ if such an $m$ exists, and $n_g=0$ otherwise, the group then
still carrying its warm-up block. Since $i_m=((m-1)\bmod\chi)+1$ takes
every value in $\{1,\ldots,\chi\}$ exactly once as $m$ runs over any
$\chi$ consecutive positive integers, we have $n_g\ge n-\chi+1$ when
$n\ge\chi$, and trivially $n_g\ge0$; thus $n_g\ge\max\{0,n-\chi+1\}$ for
every group. By \eqref{eq:ncrefresh}, the columns of $M_n$ indexed by $g$
were written at iteration $n_g$ (in the warm-up if $n_g=0$) as $P_ju_{n_g}$,
$j\in g$, with $u_{n_g}$ the gradient difference at $z_{n_g}$ along
$v_g$ with radius $\tau_g$, and they were not modified afterwards, since
$g_{i_m}\ne g$ for $n_g<m\le n$. Lemma~\ref{lem:group}, applied with
$x=z_{n_g}$ and $\tau=\tau_g=2D/(\sqrt\chi|g|)$, gives
\[
M_ne_j=H_{n_j}e_j+\phi_j\quad(j\in[d]),\qquad
\Big(\sum_{j\in g}\|\phi_j\|^2\Big)^{1/2}\le\frac{L_2\tau_g|g|}{2}
=\frac{L_2D}{\sqrt\chi}\quad\text{for every group }g,
\]
where $n_j=n_g$ for $j\in g$. Setting $\Xi_n=\sum_j\phi_je_j^\top$,
\begin{equation}\label{eq:ncrawsplit}
M_n-H_n=\sum_{j=1}^d(H_{n_j}-H_n)e_je_j^\top+\Xi_n .
\end{equation}

\emph{Exchange of summation.} For each $j$,
$H_{n_j}-H_n=-\sum_{m=n_j}^{n-1}(H_{m+1}-H_m)$, an empty sum if $n_j=n$.
Substituting into \eqref{eq:ncrawsplit} and grouping, for each $m$, the
columns with $n_j\le m$, exactly as in \eqref{eq:exchange},
\[
\sum_{j=1}^d(H_{n_j}-H_n)e_je_j^\top
=-\sum_{m=\max\{0,n-\chi+1\}}^{n-1}(H_{m+1}-H_m)\,\Pi_m,
\qquad
\Pi_m=\sum_{j:\,n_j\le m}e_je_j^\top,
\]
where the lower limit uses $\min_jn_j\ge\max\{0,n-\chi+1\}$. Each $\Pi_m$ is
a diagonal matrix with entries in $\{0,1\}$, hence
$\|\Pi_m\|_{\rm op}\le1$, and
$\|H_{m+1}-H_m\|_{\rm op}\le L_2\|z_{m+1}-z_m\|$ by
Assumption~\ref{as:nonconvex}. Therefore
\[
\Big\|\sum_{j=1}^d(H_{n_j}-H_n)e_je_j^\top\Big\|_{\rm op}
\le L_2\sum_{m=\max\{0,n-\chi+1\}}^{n-1}\|z_{m+1}-z_m\| .
\]

\emph{Finite differences.} The $\chi$ groups index disjoint sets of
columns, so the squared Frobenius norms of the blocks of $\Xi_n$ add:
\[
\|\Xi_n\|_{\rm op}\le\|\Xi_n\|_F
=\Big(\sum_g\sum_{j\in g}\|\phi_j\|^2\Big)^{1/2}
\le\Big(\chi\cdot\frac{L_2^2D^2}{\chi}\Big)^{1/2}=L_2D .
\]
With \eqref{eq:ncrawsplit},
$\|M_n-H_n\|_{\rm op}\le e_n:=L_2\sum_{m=\max\{0,n-\chi+1\}}^{n-1}\|z_{m+1}-z_m\|+L_2D$.

\emph{Symmetrization.} Since $H_n$ is symmetric and
$B_n=(M_n+M_n^\top)/2$,
$B_n-H_n=\frac12(M_n-H_n)+\frac12(M_n-H_n)^\top$, and transposition
preserves the operator norm, so $\|B_n-H_n\|_{\rm op}\le e_n$, which is
the first inequality of \eqref{eq:ncpointwise}.

\emph{The uniform bound.} The sum defining $e_n$ has at most $\chi-1$
terms, each at most $2D$ by \eqref{eq:nczstep}, so
$e_n\le L_2\big(2(\chi-1)D+D\big)=(2\chi-1)L_2D$. Finally,
$\|B_n\|_{\rm op}\le\|H_n\|_{\rm op}+\|B_n-H_n\|_{\rm op}\le L_1+(2\chi-1)L_2D$.
\end{proof}

\subsection{Proof of Lemma~\ref{lem:nccertificate}}\label{proof:nccertificate}
\begin{proof}
Fix $n\ge0$. By \eqref{eq:nctheta} and \eqref{eq:nctaylor} with $x=z_n$
and $y=z_n+s_n=w_{n+1}$,
\[
\theta_{n+1}=\nabla f(z_n+s_n)-\nabla f(z_n)-B_ns_n=(\bar H_n-B_n)s_n,
\qquad \bar H_n=\bar H(z_n,z_n+s_n),
\]
so $\|\theta_{n+1}\|\le\|\bar H_n-B_n\|_{\rm op}\|s_n\|$, and
$\|s_n\|=\frac12\|\Delta_{n+1}-\Delta_n\|$ by \eqref{eq:ncsn}. In the
model variant, the triangle inequality, \eqref{eq:nctaylor} with
$\|s_n\|\le D$, and Lemma~\ref{lem:ncpointwise} give
\[
\|\bar H_n-B_n\|_{\rm op}
\le\|\bar H_n-H(z_n)\|_{\rm op}+\|H(z_n)-B_n\|_{\rm op}
\le\frac{L_2D}{2}+(2\chi-1)L_2D\le2\chi L_2D ,
\]
hence $\|\theta_{n+1}\|\le2\chi L_2D\,\|s_n\|=\chi L_2D\,\|\Delta_{n+1}-\Delta_n\|$.
In the gradient variant, $B_n=0$ and $\|\bar H_n\|_{\rm op}\le L_1$, as
noted at the beginning of this appendix, hence
$\|\theta_{n+1}\|\le L_1\|s_n\|=\frac12L_1\|\Delta_{n+1}-\Delta_n\|$.
\end{proof}

\subsection{Proof of Theorem~\ref{thm:ncmodel}}\label{proof:ncmodel}
\begin{proof}
By Lemma~\ref{lem:nccertificate}, the certificate \eqref{eq:nccert} holds
with $\Gamma=\chi L_2D=\omega L_2D$ for $\omega=\chi\ge1$, whatever the
value of $D$. The learning rate and episode length of the theorem are
those of Theorem~\ref{thm:ncmaster} for this $\omega$, so
Theorem~\ref{thm:ncmaster} gives, for every $K\ge1$ and $N=KT$,
\[
\|\nabla f(\bar w^{k^\star})\|\le\frac{\Phi_0}{DN}+2\nu+13\,\chi^{2/3}L_2D^2
\le\frac{\Phi_0}{DN}+\frac{\varepsilon}{4}+\frac{\varepsilon}{2},
\]
because $\nu\le\varepsilon/8$ and $26\,\chi^{2/3}L_2D^2=\varepsilon$ by
the choice of $D$. This is the first claim, and
$\Phi_0/(DN)\le\varepsilon/4$ whenever $K\ge4\Phi_0/(DT\varepsilon)$.

For the count, $K=\max\{1,\lceil4\Phi_0/(DT\varepsilon)\rceil\}$ gives
$N=KT\le4\Phi_0/(D\varepsilon)+T$. Substituting $D$,
\[
\frac{4\Phi_0}{D\varepsilon}
=4\sqrt{26}\,\chi^{1/3}\,\Phi_0\sqrt{L_2}\,\varepsilon^{-3/2}
<20.4\,\Phi_0\sqrt{L_2}\,\chi^{1/3}\varepsilon^{-3/2},
\]
while, since $\chi\ge1$,
\[
T=\left\lceil(16\sqrt3\,\chi)^{1/3}\right\rceil
\le\big((16\sqrt3)^{1/3}+1\big)\chi^{1/3}
<4.04\,\chi^{1/3}.
\]
Hence $N<\chi^{1/3}\big(20.4\,\Phi_0\sqrt{L_2}\,\varepsilon^{-3/2}+4.04\big)$,
and the gradient count $3N+\chi-1$ of Section~\ref{sec:ncalgorithm} is
less than
$61.2\,\Phi_0\sqrt{L_2}\,\chi^{1/3}\varepsilon^{-3/2}+12.2\,\chi^{1/3}+\chi$,
which is below \eqref{eq:ncmodelcount}.
\end{proof}

\subsection{Proof of Theorem~\ref{thm:ncsmooth}}\label{proof:ncsmooth}
\begin{proof}
Write $\beta=L_1^2/(L_2\varepsilon)$, so that the hypothesis of the
theorem is $\beta\ge1$. In the gradient variant $B_n=0$, and by
Lemma~\ref{lem:nctrust} with $B_n=0$ the explicit update of Algorithm~3
is the exact solution of the subproblem, so $\nu=0$. By
Lemma~\ref{lem:nccertificate}, the certificate \eqref{eq:nccert} holds
with $\Gamma=L_1/2=\omega L_2D$ for $\omega=L_1/(2L_2D)$. With the stated
radius,
\[
\omega=\frac{L_1}{2L_2D}
=\frac{9\,L_1^{3/2}L_2^{1/4}}{2\,L_2\,\varepsilon^{3/4}}
=\frac92\,\beta^{3/4}\ \ge\ \frac92 .
\]
The learning rate and episode length of the theorem are those of
Theorem~\ref{thm:ncmaster} for this $\omega$, since
$2\sqrt3\,\omega L_2D=\sqrt3\,L_1$ and
$16\sqrt3\,\omega=8\sqrt3\,L_1/(L_2D)$. Theorem~\ref{thm:ncmaster} gives,
for every $K\ge1$ and $N=KT$,
\[
\|\nabla f(\bar w^{k^\star})\|\le\frac{\Phi_0}{DN}+13\,\omega^{2/3}L_2D^2 .
\]
We evaluate the last term. From the definitions,
$L_2D^2=L_2^{1/2}\varepsilon^{3/2}/(81L_1)$ and
$\omega^{2/3}=(9/2)^{2/3}\beta^{1/2}=(9/2)^{2/3}L_1L_2^{-1/2}\varepsilon^{-1/2}$,
so
\[
13\,\omega^{2/3}L_2D^2=\frac{13\,(9/2)^{2/3}}{81}\,\varepsilon
<0.44\,\varepsilon<\frac34\,\varepsilon ,
\]
using $(9/2)^{2/3}<2.73$. This is the first claim, and
$\Phi_0/(DN)\le\varepsilon/4$ whenever $K\ge4\Phi_0/(DT\varepsilon)$. For
the bound on $T$ in the statement, $16\sqrt3\,\omega=72\sqrt3\,\beta^{3/4}$
and $(72\sqrt3)^{1/3}<5$, so
$T\le(16\sqrt3\,\omega)^{1/3}+1<5\beta^{1/4}+1\le6\beta^{1/4}$.

For the count, $K=\max\{1,\lceil4\Phi_0/(DT\varepsilon)\rceil\}$ gives
$N\le4\Phi_0/(D\varepsilon)+T$, and
\[
\frac{4\Phi_0}{D\varepsilon}
=\frac{4\cdot9\,\Phi_0L_1^{1/2}L_2^{1/4}}{\varepsilon^{7/4}}
=36\,\Phi_0L_1^{1/2}L_2^{1/4}\varepsilon^{-7/4}.
\]
Hence $N\le36\,\Phi_0L_1^{1/2}L_2^{1/4}\varepsilon^{-7/4}+6\beta^{1/4}$,
and the gradient count $2N$ of Section~\ref{sec:ncalgorithm} is at most
\eqref{eq:ncsmoothcount}.
\end{proof}

\subsection{Proof of Lemma~\ref{lem:ncsolver}}\label{proof:ncsolver}
\begin{proof}
Write $c=\nabla f(z_n)+\theta_n$. Since $B_n$ is symmetric,
\eqref{eq:nctr} reads
$\phi_n(\Delta)=\langle c,\Delta\rangle+\frac12\langle A_n(\Delta-\Delta_n),\Delta-\Delta_n\rangle$
with $A_n=\frac12B_n+\eta^{-1}I$ symmetric, so
$\nabla\phi_n(\Delta)=c+A_n(\Delta-\Delta_n)$, as in \eqref{eq:ncpg}, and
for all $\Delta,\Delta'$,
\begin{equation}\label{eq:ncquad}
\phi_n(\Delta')=\phi_n(\Delta)+\langle\nabla\phi_n(\Delta),\Delta'-\Delta\rangle
+\tfrac12\langle A_n(\Delta'-\Delta),\Delta'-\Delta\rangle .
\end{equation}
All iterates lie in the ball: $\Delta^{(0)}=\Delta_n$ does, and
$\Delta^{(t+1)}$ is a projection onto it. Fix $t$, write
$\Delta=\Delta^{(t)}$, $\Delta^+=\Delta^{(t+1)}$ and
$v=\Delta-\alpha\nabla\phi_n(\Delta)$, so that $\Delta^+=\Pi_D(v)$.

\emph{The stopping rule certifies \eqref{eq:nckkt}.} By
Lemma~\ref{lem:ncball}(ii), applied with $\alpha$ in place of $\eta$,
there is $\varrho\ge0$ with $(1+\alpha\varrho)\Delta^+=v$ and
$\varrho(\|\Delta^+\|-D)=0$. The first relation is
$\nabla\phi_n(\Delta)+\varrho\Delta^+=(\Delta-\Delta^+)/\alpha$, hence
\[
\nabla\phi_n(\Delta^+)+\varrho\Delta^+
=\nabla\phi_n(\Delta^+)-\nabla\phi_n(\Delta)+\frac{\Delta-\Delta^+}{\alpha}
=\big(A_n-\alpha^{-1}I\big)(\Delta^+-\Delta),
\]
whose norm is at most
$(\|A_n\|_{\rm op}+\alpha^{-1})\|\Delta^+-\Delta\|\le\frac2\alpha\|\Delta^+-\Delta\|$,
since $\alpha\|A_n\|_{\rm op}\le1$. If $\|\Delta^+-\Delta\|\le\alpha\nu/2$,
this is at most $\nu$, and $\Delta^+$, with the multiplier $\varrho$,
satisfies \eqref{eq:nckkt}.

\emph{Descent.} Lemma~\ref{lem:ncball}(i) with $u=\Delta$, which lies in
the ball, gives $\langle v-\Delta^+,\Delta-\Delta^+\rangle\le0$, that is,
$\langle\nabla\phi_n(\Delta),\Delta^+-\Delta\rangle\le-\frac1\alpha\|\Delta^+-\Delta\|^2$.
Inserting this into \eqref{eq:ncquad} with $\Delta'=\Delta^+$ and using
$\langle A_nw,w\rangle\le\|A_n\|_{\rm op}\|w\|^2\le\alpha^{-1}\|w\|^2$,
\[
\phi_n(\Delta^{(t+1)})\le\phi_n(\Delta^{(t)})
-\frac{1}{2\alpha}\|\Delta^{(t+1)}-\Delta^{(t)}\|^2 .
\]
Let $\phi_n^\star=\min_{\|\Delta\|\le D}\phi_n$ and
$x=8\big(\phi_n(\Delta_n)-\phi_n^\star\big)/(\alpha\nu^2)$. If the
stopping rule fails for $t=0,\ldots,m-1$ with $m\ge1$, then each of these
steps decreases $\phi_n$ by more than
$\frac{1}{2\alpha}\cdot\frac{\alpha^2\nu^2}{4}=\frac{\alpha\nu^2}{8}$, and
summing,
$m\,\frac{\alpha\nu^2}{8}<\phi_n(\Delta^{(0)})-\phi_n(\Delta^{(m)})\le\phi_n(\Delta_n)-\phi_n^\star$,
that is, $m<x$. The first index $t$ at which the rule holds is therefore
at most $\lfloor x\rfloor$, and the rule is met within $1+\lfloor x\rfloor$
iterations.

\emph{The initial gap.} $\phi_n(\Delta_n)=\langle c,\Delta_n\rangle$, and
for every $\Delta$ in the ball, $\|\Delta-\Delta_n\|\le2D$ and
$\frac14\langle B_n(\Delta-\Delta_n),\Delta-\Delta_n\rangle\ge-\|B_n\|_{\rm op}D^2$,
so $\phi_n(\Delta)\ge\langle c,\Delta\rangle-\|B_n\|_{\rm op}D^2$ and
\[
\phi_n(\Delta_n)-\phi_n(\Delta)\le\langle c,\Delta_n-\Delta\rangle+\|B_n\|_{\rm op}D^2
\le2D\|c\|+\|B_n\|_{\rm op}D^2 .
\]
Taking the supremum over the ball gives the stated bound on
$\phi_n(\Delta_n)-\phi_n^\star$.
\end{proof}

\end{document}